\documentclass[english, a4paper, 11pt]{article}
\usepackage[utf8]{inputenc}
\usepackage[T1]{fontenc}
\usepackage{imakeidx}
\makeindex[options= -s Index_Confg.ist]
\usepackage[
    unicode=true,
    pdfstartview=FitV,
    colorlinks=true,
    citecolor=Turquoise,
    linkcolor=Turquoise,
    urlcolor=MFCB,
    linktoc=page,
    hyperindex=true,
    pdfcreator={},
]{hyperref}
\usepackage[main=english,french]{babel}

\usepackage{amsmath,amsfonts,amssymb}
\usepackage{amsthm}
\usepackage[english]{cleveref} 
\crefname{figure}{Figure}{Figures}
\Crefname{figure}{Figure}{Figures}
\crefname{section}{Section}{Sections}
\crefname{part}{Part}{Parts}
\crefname{remark}{Remark}{Remarks}
\crefname{Rq}{Remark}{Remarks}
\crefname{Claim}{Claim}{Claims}
\crefname{Def}{Definition}{Definitions}
\crefname{Th}{Theorem}{Theorems}
\crefname{Cor}{Corollary}{Corollaries}
\crefname{Prop}{Proposition}{Propositions}
\crefname{Lmm}{Lemma}{Lemmas}
\crefname{lemma}{Lemma}{Lemmas}
\crefname{Ex}{Example}{Examples}
\crefname{Exs}{Examples}{Examples}
\crefname{CEx}{Counter-example}{Counter-example}
\crefname{Q}{Question}{Questions}
\crefname{Nota}{Notation}{Notations}
\theoremstyle{definition}
\newtheorem{Def}{Definition}
\numberwithin{Def}{section}
\newtheorem{Th}[Def]{Theorem}
\newtheorem{Cor}[Def]{Corollary}
\newtheorem{Prop}[Def]{Proposition}
\newtheorem{Lmm}[Def]{Lemma}
\newtheorem{lemma}[Def]{Lemma}
\newtheorem{Ex}[Def]{Example}

\newtheorem{Claim}[Def]{Claim}
\newtheorem{Rq}[Def]{Remark}

\newtheorem*{Th1}{Theorem 1.2}
\newtheorem*{Th2}{Theorem 1.6}
\newtheorem*{Th3}{Theorem 1.10}
\newtheorem*{Th4}{Theorem 1.16}

\newtheorem*{Th6}{Theorem 1.13}
\numberwithin{equation}{section}
\usepackage[x11names]{xcolor} 
\usepackage{graphicx, subcaption} 
\usepackage{tikz}
\definecolor{MCB}{cmyk}{0,0.03,0.08,0.26} 
\definecolor{MFCB}{cmyk}{0,0.06,0.20,0.6} 
\colorlet{Turquoise}{DeepSkyBlue4}
\colorlet{TurquoiseClair}{DeepSkyBlue4!20}
\colorlet{Orange}{DarkOrange3!85}

\usepackage{tcolorbox} 
\usepackage{adforn} 
\usepackage{setspace} 
\usepackage{enumitem}
\setlist{nosep}

\usepackage{geometry}
\usepackage{mathabx}

\title{Graph products and measure equivalence: classification, rigidity, and quantitative aspects}
\author{Amandine Escalier and Camille Horbez}
\date{\today}

\begin{document}
\newcommand{\bC}{\mathbb{C}}
\newcommand{\bD}{\mathbb{D}}
\newcommand{\bF}{\mathbb{F}}
\newcommand{\bK}{\mathbb{K}}
\newcommand{\bL}{\mathbb{L}}
\newcommand{\bN}{\mathbb{N}}
\newcommand{\bP}{\mathbb{P}}
\newcommand{\bQ}{\mathbb{Q}}
\newcommand{\bR}{\mathbb{R}}
\newcommand{\bZ}{\mathbb{Z}}
\newcommand{\calA}{\mathcal{A}}
\newcommand{\cala}{\mathcal{A}}
\newcommand{\calB}{\mathcal{B}}
\newcommand{\calb}{\mathcal{B}}
\newcommand{\calc}{\mathcal{C}}
\newcommand{\calC}{\mathcal{C}}
\newcommand{\calF}{\mathcal{F}}
\newcommand{\calf}{\mathcal{F}}
\newcommand{\calG}{\mathcal{G}}
\newcommand{\calg}{\mathcal{G}}
\newcommand{\calH}{\mathcal{H}}
\newcommand{\calh}{\mathcal{H}}
\newcommand{\cH}{\mathcal{H}}
\newcommand{\calI}{\mathcal{I}}
\newcommand{\calk}{\mathcal{K}}
\newcommand{\caln}{\mathcal{N}}
\newcommand{\calO}{\mathcal{O}}
\newcommand{\calp}{\mathcal{P}}
\newcommand{\calP}{\mathcal{P}}
\newcommand{\calQ}{\mathcal{Q}}
\newcommand{\calq}{\mathcal{Q}}
\newcommand{\calr}{\mathcal{R}}
\newcommand{\calS}{\mathcal{S}}
\newcommand{\cals}{\mathcal{S}}
\newcommand{\calt}{\mathcal{T}}
\newcommand{\calU}{\mathcal{U}}
\newcommand{\calv}{\mathcal{V}}
\newcommand{\calV}{\mathcal{V}}
\newcommand{\calw}{\mathcal{W}}
\newcommand{\calW}{\mathcal{W}}
\newcommand{\calY}{\mathcal{Y}}
\newcommand{\grpda}{\mathbf{a}}
\newcommand{\grpdb}{\mathbf{b}}
\newcommand{\grpdg}{\mathbf{g}}
\newcommand{\grpdh}{\mathbf{h}}
\newcommand{\classC}{\mathsf{C}}
\newcommand{\supp}{\mathrm{supp}}
\newcommand{\lk}{\mathrm{lk}}
\newcommand{\st}{\mathrm{star}}
\newcommand{\reg}{\mathrm{reg}}
\newcommand{\para}{\mathrm{para}}
\newcommand{\Prob}{\mathrm{Prob}}
\newcommand{\Isom}{\mathrm{Isom}}
\newcommand{\Maps}{\mathrm{Maps}}
\newcommand{\Aut}{\mathrm{Aut}}
\newcommand{\Out}{\mathrm{Out}}
\newcommand{\Mod}{\mathrm{Mod}}
\newcommand{\Homeo}{\mathrm{Homeo}}
\newcommand{\Stab}{\mathrm{Stab}}
\newcommand{\Ind}{\mathrm{Ind}}
\newcommand{\Hig}{\mathrm{Hig}}
\newcommand{\BS}{\mathrm{BS}}
\newcommand{\ad}{\mathrm{ad}}
\newcommand{\Creg}{\mathcal{C}_{\mathrm{reg}}}
\newcommand{\Pprod}{(\mathrm{P}_{\mathrm{prod}})}
\newcommand{\Pprodb}{(\mathrm{P}'_{\mathrm{prod}})}
\newcommand{\Pspec}{(\mathrm{P}_{\mathrm{spec}})}
\newcommand{\Padm}{(\mathrm{P}_{\mathrm{cofact}})}
\newcommand{\Pfact}{(\mathrm{P}_{\mathrm{fact}})}
\newcommand{\Pfactp}{(\mathrm{P}^\prime_{\mathrm{fact}})}
\newcommand{\Pvert}{(\mathrm{P}_{\mathrm{vert}})}
\newcommand{\Pcofact}{(\mathrm{P}^{\prime}_{\mathrm{cofact}})}
\newcommand{\Pf}{(\mathrm{P}'_{\mathrm{fact}})}
\newcommand{\Pv}{(\mathrm{P}'_{\mathrm{vert}})}
\newcommand{\Pamen}{(\mathrm{P}'_{\mathrm{vert}})}
\newcommand{\Pamenb}{(\mathrm{P}''_{\mathrm{vert}})}
\newcommand{\Pclique}{(\mathrm{P}_{\mathrm{clique}})}
\newcommand{\Pcliqueb}{(\mathrm{P}_{\mathrm{clique}}')}

 \newcommand{\Qprod}{(\mathrm{Q}_{\mathrm{prod}})}
\newcommand{\Qisol}{(\mathrm{Q}_{\mathrm{thick}})}
\newcommand{\Qspec}{(\mathrm{Q}_{\mathrm{spec}})}
\newcommand{\Qadm}{(\mathrm{Q}_{\mathrm{cofact}})}
\newcommand{\Qvert}{(\mathrm{Q}_{\mathrm{vert}})}
\newcommand{\Qfact}{(\mathrm{Q}_{\mathrm{fact}})}
\newcommand{\Qtc}{(\mathrm{Q}_{\mathrm{tc}})}
\newcommand{\cdotprime}{\bullet}
\newcommand{\cpp}{c^{\prime \prime}_1}
\newcommand{\cppdeux}{c^{\prime \prime}_2}
\newcommand{\dunion}{\sqcup} 
\newcommand{\actson}{\curvearrowright}
\newcommand{\Ccom}[1]{\Cmod\marginpar{\color{blue}\tiny #1 --ch}} 
\newcommand{\Cmod}{$\textcolor{blue}{\clubsuit}$} 
\newcommand{\Ccomm}[1]{\Cmodd\marginpar{\color{red}\tiny #1 --ch}} 
\newcommand{\Cmodd}{$\textcolor{red}{\clubsuit}$} 
\newcommand{\Amod}{$\textcolor{violet}{\clubsuit}$} 
\newcommand{\Acom}[1]{\Amod\marginpar{\color{violet}\tiny #1 --a}} 
\newtcolorbox{Todo}{colback=black!20,
colframe=black,fonttitle=\bfseries,
colbacktitle=black,title=To do}

\maketitle



\begin{abstract}
    We study graph products of groups from the viewpoint of measured group theory.
    
    We first establish a full measure equivalence classification of graph products of countably infinite groups over finite simple graphs with no transvection and no partial conjugation. This finds applications to their classification up to commensurability, and up to isomorphism, and to the study of their automorphism groups. We also derive structural properties of von Neumann algebras associated to probability measure-preserving actions of graph products. Variations of the measure equivalence classification statement are given with fewer assumptions on the defining graphs.
    
    We also provide a quantified version of our measure equivalence classification theorem, that keeps track of the integrability of associated cocycles. As an application, we solve an inverse problem in quantitative orbit equivalence for a large family of right-angled Artin groups.
    
    We then establish several rigidity theorems. First, in the spirit of work of Monod--Shalom, we achieve rigidity in orbit equivalence for probability measure-preserving actions of graph products, upon imposing extra ergodicity assumptions. 
    
    Second, we establish a sufficient condition on the defining graph and on the vertex groups ensuring that a graph product $G$ is rigid in measure equivalence among torsion-free groups (in the sense that every torsion-free countable group $H$ which is measure equivalent to $G$, is in fact isomorphic to $G$). Using variations over the Higman groups as the vertex groups, we construct the first example of a group which is rigid in measure equivalence, but not in quasi-isometry, among torsion-free groups.  
\end{abstract}
\vfill

\noindent \textcolor{black!70}{This work is openly licensed via Creative Commons CC-BY 4.0}\\
\href{https://creativecommons.org/licenses/by/4.0/}{https://creativecommons.org/licenses/by/4.0/}.
\newpage

\tableofcontents

\newpage


\newpage
\section{Introduction}
Graph products of groups, introduced by Green in \cite{Gre}, form a natural class of groups that interpolates between direct products and free products. If $\Gamma$ is a finite simple graph and $(G_v)_{v\in V\Gamma}$ is a family of groups, the \emph{graph product}\index{Graph product} over $\Gamma$ with vertex groups $(G_v)_{v\in V\Gamma}$ is the group

\begin{equation*}
    G:=\left(\bigast_{v\in V\Gamma} G_v \right)\bigg/ \left\langle\left\langle{ [g,h] \mid g\in G_u, h\in G_v, (u,v)\in E\Gamma}\right\rangle\right\rangle.
\end{equation*} 
The class of graph products encompasses several particularly important families of groups, among which right-angled Artin groups (when all vertex groups are isomorphic to $\mathbb{Z}$) and right-angled Coxeter groups (when all vertex groups are isomorphic to $\mathbb{Z}/2\mathbb{Z}$).

Basic questions about graph products remain open. For instance, their classification up to isomorphism, or up to commensurability, is still open in general. Moving to the framework of geometric group theory, the problem of their quasi-isometry classification is also still wide open. In a very different direction, the notion of a graph product can be defined for other structures than groups, e.g.\ von Neumann algebras \cite{CF}; and from this viewpoint the $W^*$-classification of graph products and their measure-preserving actions is also an important open question that has recently attracted a lot of attention \cite{CKE,CDD1,CDD2}.

In the present work, we adopt the point of view of measured group theory, and obtain classification and rigidity theorems for graph products in measure equivalence. Interestingly, using this viewpoint, we are able to also obtain several consequences and draw connections with the algebraic, geometric and functional analytic approaches mentioned in the above paragraph. 

Besides being a particularly natural class of groups, and an important object of study on their own, graph products will also enable us to exhibit new phenomena in measured group theory, like examples of groups that are superrigid in measure equivalence (among torsion-free groups) but not in quasi-isometry (Part~\ref{Part:SuperrigityHigman}). They will also be a good framework to develop the theory of quantitative measure equivalence (in the sense of \cite{DKLMT}) beyond the world of amenable groups (Part~\ref{Part:QuantitativeResults}).  

\paragraph*{Measure equivalence.} Introduced by Gromov in \cite{Gro}, measure equivalence is a measurable counterpart to the notion of quasi-isometry between finitely generated groups. Two countable groups $G$ and $H$ are \emph{measure equivalent}\index{Measure equivalence} if there exists a standard measure space $(\Omega,m)$ equipped with commuting, free, measure-preserving actions of $G$ and $H$ by Borel automorphisms, having finite measure Borel fundamental domains $X_G,X_H$ respectively. The space $(\Omega,m)$ is called a \emph{measure equivalence coupling}\index{Measure equivalence!Measure equivalence coupling} between $G$ and~$H$. As an important example, any two lattices in the same locally compact second countable group $\mathbf{G}$ are measure equivalent, by letting $(\Omega,m)=(\mathbf{G},\mathrm{Haar})$, equipped with the actions of the two lattices by left/right multiplication.

Measure equivalence is intimately connected to the ergodic-theoretic notion of \emph{orbit equivalence}, which finds its roots in the work of Dye \cite{Dye1,Dye2}. Two countable groups $G$ and $H$ are \emph{orbit equivalent}\index{Orbit equivalence} if there exist free probability measure-preserving actions $G\curvearrowright X$ and $H\curvearrowright Y$ on standard probability spaces, and a measure space isomorphism $f:X\to Y$ sending orbits to orbits, i.e.\ such that for almost every $x\in X$, one has $f(G\cdot x)=H\cdot f(x)$. It turns out that orbit equivalence is a stronger notion than measure equivalence: two groups are orbit equivalent precisely when one can find a measure equivalence coupling $(\Omega,m)$ as above, with $X_G=X_H$ (see e.g.\ $\mathbf{P}_{\mathrm{ME}}\mathbf{5}$ in \cite{Gab}).

The classification of countable groups up to measure equivalence is fundamentally different in the amenable and in the non-amenable case. A celebrated theorem of Ornstein and Weiss \cite{OW80}, building on the aforementioned work of Dye, asserts that all countably infinite amenable groups are orbit equivalent. At the other extreme, certain non-amenable groups satisfy a very strong form of rigidity. Furman proved in \cite{Fur-me} that every countable group which is measure equivalent to a lattice in a higher-rank simple Lie group $\mathbf{G}$ with trivial center, is virtually isomorphic to a lattice in $\mathbf{G}$. Some (non-amenable) groups $G$ are even more rigid, in that every countable group $H$ which is measure equivalent to $G$, must be virtually isomorphic to $G$: this is the case of most mapping class groups of finite-type surfaces (Kida \cite{Kid-me}), or of $\mathrm{Out}(F_N)$ for $N\ge 3$ (Guirardel-Horbez \cite{GH-OutFn}).

\subsection{Measure equivalence classification of graph products}

We now present our first main theorem (Theorem~\ref{theo:classificationversionintro} below, proved in Part~\ref{Part:MEClassification}), giving a full measure equivalence classification of graph products with countably infinite vertex groups defined over transvection-free graphs with no partial conjugation. Let us mention that Monod-Shalom \cite{MS} previously obtained measure equivalence classification results for a wide class of direct products. In the opposite direction, Alvarez-Gaboriau \cite{AG} classified a large family of free products. As a motivating example for our work, we first review the case of right-angled Artin groups treated in \cite{HH21}.

\subsubsection{The case of right-angled Artin groups: review} \label{Sec:IntroRAAG}

In \cite{HH21}, Huang and the second-named author proved the following measure equivalence classification theorem for right-angled Artin groups, which turns out to match the quasi-isometry classification for the groups under consideration \cite[Theorem~1.1]{Hua}.

\begin{Th}[{\cite[Theorem 1]{HH21}}]\label{Th:ClassificationRAAG}
    Let $A_1$ and $A_2$ be two right-angled Artin groups such that $\Out(A_i)$ is finite for every $i\in\{1,2\}$. Then $A_1$ and $A_2$ are measure equivalent if and only if they are isomorphic.
\end{Th}

The condition on the finiteness of the outer automorphism groups is important, and can be read off from the defining graphs as follows. Recall that if $\Gamma$ is a finite simple graph, and $v\in V\Gamma$, then 
\begin{itemize}
\item the \emph{link} of $v$\index{Link (of a vertex)}\label{Def:Link} is the subgraph $\lk(v)$ of $\Gamma$ induced by all vertices adjacent to $v$;
\item the \emph{star} of $v$\index{Star (of a vertex)}\label{Def:Star} is the subgraph $\st(v)$ of $\Gamma$ induced by $v$ itself and its neighbors. 
\end{itemize}
A theorem of Laurence \cite{Laurence}, confirming a conjecture of Servatius \cite{Servatius}, ensures that if $A_\Gamma$ is the right-angled Artin group over $\Gamma$, then  $|\Out(A_\Gamma)|<+\infty$ if and only if 
\begin{itemize}
    \item $\Gamma$ is \emph{transvection-free}, that is, for any two distinct vertices $v,w\in V\Gamma$, one has $\lk(v)\nsubseteq\st(w)$, and
    \item $\Gamma$ has \emph{no partial conjugation}, that is, for every vertex $v\in V\Gamma$, the subgraph induced by $V\Gamma\backslash V\st(v)$ is connected.
\end{itemize}

\subsubsection{Main theorem and discussion on the hypotheses} 

Our first main theorem is proved in \cref{Part:MEClassification}. Its statement vastly generalizes Theorem~\ref{Th:ClassificationRAAG} to all graph products of countably infinite groups, and its proof relies on different tools (see the discussion at the beginning of Part~\ref{Part:MEClassification}).

\begin{Th} 
\label{theo:classificationversionintro}
   Let $\Gamma_G,\Gamma_H$ be two finite simple graphs, not reduced to one vertex, with no transvection and no partial conjugation. Let $G,H$ be graph products of countably infinite groups over $\Gamma_G,\Gamma_H$, respectively. 
   Then the following assertions are equivalent.
   \begin{enumerate}
       \item The groups $G$ and $H$ are measure equivalent.
       \item The groups $G$ and $H$ are orbit equivalent.
       \item There exists a graph isomorphism $\sigma:\Gamma_G\to\Gamma_H$ such that for every $v\in V\Gamma_G$, the groups $G_v$ and $H_{\sigma(v)}$ are orbit equivalent.
   \end{enumerate}
\end{Th}

The implication $2\Rightarrow 1$ is obvious, and $3\Rightarrow 2$ was proved in \cite[Proposition~4.2]{HH21} by adapting an argument of Gaboriau \cite{Gab} from free products to graph products. Our contribution in the present paper is the (hardest) implication $1\Rightarrow 3$.

Examples illustrating the above theorem (and some variations obtained in \cref{Part:MEClassification}) are provided in \cref{Sec:ExemplesMEVRP}. 

\paragraph*{Necessity of all assumptions.}

The assumptions in Theorem~\ref{theo:classificationversionintro} can not be dropped.

First, the theorem does not hold in general if the vertex groups are finite. For instance, for $n\ge 5$, the right-angled Coxeter group $W_n$ with defining graph an $n$-gon acts properly and cocompactly on the real hyperbolic plane $\mathbb{H}^2_{\mathbb{R}}$, as the group generated by the reflections over a right-angled hyperbolic $n$-gon. Therefore all groups $W_n$ are cocompact Fuchsian, so they are all measure equivalent (in fact commensurable). The classification of right-angled Coxeter groups up to measure equivalence is a widely open problem. We refer to the survey \cite[Sections~5 and~6]{Dan} and the references therein for an overview of the current knowledge regarding their commensurability and quasi-isometry classification.

The assumption that both $\Gamma_G$ and $\Gamma_H$ have no partial conjugation can not be dropped either, even in the case of right-angled Artin groups. Indeed, every non-cyclic right-angled Artin group $G$ with $|\Out(G)|<+\infty$ has proper finite-index subgroups that are themselves right-angled Artin groups, defined over graphs which are transvection-free but have partial conjugations. An example is given by the two leftmost graphs in Figure~\ref{fig:Rk56}, \cpageref{fig:Rk56}, with all vertex groups isomorphic to $\mathbb{Z}$. More generally, we refer to \cite[Theorem~1.5 and Section~6]{Hua} for a full description of all finite-index right-angled Artin subgroups of $G$.

Furthermore, assuming that $\Gamma_G,\Gamma_H$ have no transvection is also important. For instance, when $\Gamma_G=\Gamma_H=:\Gamma$ is either a complete graph or a square, the same group $G$ can be defined as a graph product over $\Gamma$ in two different ways, and therefore the third conclusion of the theorem cannot hold in general. One can obtain counterexamples by considering a graph product over a square with appropriate vertex groups $A,B,C\ast D,E$ (in cyclic order), which is isomorphic to the graph product over a square with vertex groups $A\ast C,B,D,E$. On the other hand, we mention that it is not possible to decompose the same group $G$ in two different ways over a finite simple simple graph (not reduced to a point) with no transvection and no partial conjugation (see Corollary~\ref{cor:isomorphism}, \cpageref{cor:isomorphism}).

\begin{Rq}[New instances of the measure equivalence classification of RAAGs]
    Even among right-angled Artin groups, our main theorem yields new classification results that were not covered by \cite{HH21}, and were not accessible to our knowledge by previously available techniques. 
    For example, let $\Gamma$ be a pentagon and let $G$ and $H$ be the graph products over $\Gamma$ with respective vertex groups $(F_2)_{v\in V\Gamma}$ and $(F_3)_{v\in V\Gamma}$. By our main result $G$ and $H$ are not measure equivalent, because the vertex groups $F_2$ and $F_3$ are not orbit equivalent \cite{Gab-cost}. But {this was not possible to tell using hitherto known techniques}. For instance $G$ and $H$ both have cost $1$ because they have chain-commuting generating sets, see \cite[Critères VI.24(1)]{Gab-cost}. They have proportional $\ell^2$-Betti numbers: these all vanish except in dimension $2$, as follows from the work of Davis and Leary which relates the $\ell^2$-Betti numbers of a right-angled Artin group, to the classical homology of the flag completion of its defining graph \cite[Corollary~2]{DL}. And they have infinite outer automorphism groups (which contain transvections), so \cite{HH21} does not apply --~one could also prove that they have isomorphic extension graphs in the sense of Kim-Koberda \cite{KK}, so the situation is very different from \cite{HH21}. 
    We finally remark that $G$ and $H$ are quasi-isometric (see also our next remark on this point).
    
    Sticking to the problem of the measure equivalence classification of right-angled Artin groups, we also refer to \cref{Sec:AmenalbeUntransvectable}, where we establish the \emph{untransvectable extension graph} as a new invariant, thereby improving the results from \cite{HH21}.
\end{Rq}

\begin{Rq}[Measure equivalence versus quasi-isometry]\label{rk:qi-me}
No analogue of Theorem~\ref{theo:classificationversionintro} for quasi-isometry is known. An easy observation (see e.g.\ Lemma~\ref{lemma:qi}, \cpageref{lemma:qi}) shows that if $G$ and $H$ are graph products over the same finite simple graph $\Gamma$, and if the vertex groups $G_v$ and $H_v$ are bi-Lipschitz equivalent for every $v\in V\Gamma$ (i.e.\ there exists a bi-Lipschitz bijection from $G_v$ to $H_v$), then $G$ and $H$ are bi-Lipschitz equivalent. But no converse of this statement is known, except in some specific situations, see e.g.\ \cite[Corollary~A.3]{HKS} and the conjecture \cite[Remark~A.4]{HKS}. 

Notice that, in contrast to the situation from Theorem~\ref{Th:ClassificationRAAG}, the measure equivalence and quasi-isometric classifications do not match for graph products in general. This is because the difference between quasi-isometry and bi-Lipschitz equivalence is significantly looser than the difference between measure equivalence and orbit equivalence. Indeed, a theorem of Whyte \cite[Theorem~2]{Why} ensures that two non-amenable groups which are quasi-isometric are in fact bi-Lipschitz equivalent. On the other hand, measure equivalent groups need not be orbit equivalent: for example, a group $G$ with a positive $\ell^2$-Betti number is never orbit equivalent to any proper finite-index subgroup \cite{Gab-l2}. As a consequence, if $G$ is a graph product of non-amenable groups, then changing one vertex group $G_v$ to $G_v\times\mathbb{Z}/2\mathbb{Z}$ yields a group which is always quasi-isometric to $G$, but not measure equivalent to $G$ in general.

See also Theorem~\ref{theo:qi-me} for an example of a group which is rigid in measure equivalence but not in quasi-isometry among torsion-free groups.
\end{Rq}

\begin{Rq}[Fundamental groups and von Neumann algebras]\label{rk:von-neumann}
When proving \cref{theo:classificationversionintro}, we actually show (see \cref{prop:me-oe}, \cpageref{prop:me-oe}) that all measure equivalence couplings between $G$ and $H$ are actually orbit equivalence couplings, i.e.\ there is a common Borel fundamental domain for $G$ and $H$.

A consequence is the following. Given a group $G$ as in Theorem~\ref{theo:classificationversionintro}, and a free, ergodic, measure-preserving $G$-action on a standard probability space $X$, the orbit equivalence relation $\calr$ of the action \emph{has trivial fundamental group}, i.e.\ it is not isomorphic to $\calr_{|U}$ for any positive measure Borel subset $U\subseteq X$.

Now, to every action $G\actson X$ as above, one can associate a von Neumann algebra (a $\mathrm{II}_1$ factor), via Murray and von Neumann's group measure space construction \cite{MvN}. By combining the above with a Cartan rigidity theorem established for graph products by Chifan and Kunnawalkam Elayavalli \cite{CKE} in the framework of Popa's deformation/rigidity theory, we also deduce that the von Neumann algebra $L(G\actson X)$ has trivial fundamental group. This is proved in Section~\ref{sec:von-neumann}, \cpageref{sec:von-neumann}. See also Section~\ref{sec:von-neumann-2}, \cpageref{sec:von-neumann-2}, for another application of our work to a $W^*$-rigidity theorem.
\end{Rq}

Finally, let us mention that we also obtain variations of \cref{theo:classificationversionintro}, for instance \cref{Cor:VRPandGraphMorphism}, \cpageref{Cor:VRPandGraphMorphism}. We refer to \cref{Sec:ExemplesMEVRP} for examples of applications.

\subsubsection{On the commensurability classification}

Recall that two groups $G$ and $H$ are \emph{commensurable}\index{Commensurable} if they share isomorphic finite-index subgroups. Following \cite{JS}, we say that they are \emph{strongly commensurable}\index{Strongly!Strongly commensurable}\index{Commensurable!Strongly commensurable} if there exist isomorphic finite-index subgroups $G^0\subseteq G$ and $H^0\subseteq H$ such that $[G:G^0]=[H:H^0]$. Januszkiewicz and \'{S}wi\c{a}tkowski proved in \cite[Theorem~1]{JS} that if $G,H$ are two graph products over finite simple graphs $\Gamma_G$ and $\Gamma_H$ respectively, and if there exists a graph isomorphism $\sigma:\Gamma_G\to\Gamma_H$ such that for every $v\in V\Gamma_G$, the vertex groups $G_v$ and $H_{\sigma(v)}$ are strongly commensurable, then $G$ and $H$ are (strongly) commensurable. In \cite[Section~5]{JS}, they discuss the necessity of strong commensurability of the vertex groups (as opposed to mere commensurability) in various examples.

It turns out that commensurability (resp.\ strong commensurability) is exactly the property of having a measure equivalence coupling (resp.\ orbit equivalence coupling) on a countable set $\Omega$, see Lemma~\ref{lemma:commensurability}, \cpageref{lemma:commensurability}. By specifying Theorem~\ref{theo:classificationversionintro} to the setting of discrete couplings, we attain the following theorem which provides a converse to the work of Januszkiewicz and \'{S}wi\c{a}tkowski when the defining graphs have no transvection and no partial conjugation.

\begin{Th} 
\label{theo:classification-commensurability}
   Let $\Gamma_G,\Gamma_H$ be two finite simple graphs, not reduced to one vertex, with no transvection and no partial conjugation. Let $G,H$ be graph products of countably infinite groups over $\Gamma_G,\Gamma_H$, respectively. 
   Then the following assertions are equivalent.
   \begin{enumerate}
       \item The groups $G$ and $H$ are commensurable.
       \item The groups $G$ and $H$ are strongly commensurable.
       \item There exists a graph isomorphism $\sigma:\Gamma_G\to\Gamma_H$ such that for every $v\in V\Gamma_G$, the groups $G_v$ and $H_{\sigma(v)}$ are strongly commensurable.
   \end{enumerate}
\end{Th}

\begin{Rq}[On the classification of graph products up to isomorphism]\label{rk:isomorphism}
The same theorem is also true if ``commensurable'' and ``strongly commensurable'' are replaced by ``isomorphic''. In this case one can altogether remove the assumption that the vertex groups are countably infinite. See Corollaries~\ref{cor:isomorphism} and~\ref{cor:isomorphism-join}, \cpageref{cor:isomorphism,cor:isomorphism-join}. This classification theorem up to isomorphism was essentially proved by Genevois \cite[Theorem~8.2]{Gen}, with the extra assumption that vertex groups do not further split as graph products. Our approach, which was inspired from the setting of measure equivalence, is different, and enables us to remove this assumption.
\end{Rq}

\subsubsection{A word on proof techniques}

A complete outline of the proof is available in the introduction of Part~\ref{Part:MEClassification}, whose reading can be completed with the introduction of Section~\ref{sec:strongly-reduced}. We therefore refer the reader to these sections, and here we only describe our main tools. 

A measure equivalence coupling $\Omega$ between $G$ and $H$ yields two actions $G\actson H\backslash\Omega$ and $H\actson G\backslash\Omega$, whose orbits coincide on a common positive measure subset $U$ of the two spaces \cite{Fur}. This can be reformulated in the language of measured groupoids (see Section~\ref{sec:background-groupoids}), by saying that we have a measured groupoid $\calg$ over $U$, equipped with two cocycles $\rho_G:\calg\to G$ and $\rho_H:\calg\to H$. Our proof strategy can be decomposed in several steps.
\begin{itemize}
    \item We first show that subgroupoids of the form $\rho_G^{-1}(G_v)$ for some $v\in V\Gamma_G$, are also of the form $\rho_H^{-1}(hH_wh^{-1})$ for some $w\in V\Gamma_H$ and $h\in H$, up to a countable Borel partition of $U$.
    \item Following a strategy of Kida \cite{Kid-me} (reviewed in Section~\ref{sec:blueprint}, \cpageref{sec:blueprint}), we use the above to find a Borel $(G\times H)$-equivariant map $\Omega\to\Isom(\Gamma_G^e,\Gamma_H^e)$, where $\Gamma_G^e,\Gamma_H^e$ are the extension graphs of $G,H$ in the sense of Kim-Koberda \cite{KK} -- they play for us the same role as the curve graph in Kida's proof of the measure equivalence rigidity of surface mapping class groups.
    \item Another combinatorial object associated to a graph product $G$ is its right-angled building $\bD_G$, introduced by Davis \cite{Dav}. In fact, by comparison of the automorphism groups of $\Gamma_G^e$ and $\bD_G$, the above yields a $(G\times H)$-equivariant Borel map from $\Omega$ to $\Isom(\bD_G,\bD_H)$. Using that $G$ and $H$ have a common fundamental domain on $\bD_G\approx\bD_H$, and therefore on $\Isom(\bD_G,\bD_H)$, we deduce by pulling back a common fundamental domain on $\Omega$. Therefore $G$ and $H$ are orbit equivalent. And we also use the right-angled building to show that $\Omega$ induces an orbit equivalence coupling at the level of the vertex groups.
\end{itemize}

\begin{Rq}
The first step is the translation to measured groupoids of the following group-theoretic statement: every isomorphism $f:G\to H$ sends vertex groups to conjugates of vertex groups. Its proof requires new combinatorial arguments, which we believe to be of independent interest for the study of graph products. We have written a proof of the above group-theoretic statement, with assumptions that are slightly more general than those from Theorem~\ref{theo:classificationversionintro}, in the appendix of the paper (see Theorem~\ref{theo:conjugating-automorphism}, \cpageref{theo:conjugating-automorphism}). This is used to obtain some new results on graph products: this is the key for the isomorphism classification as in Remark~\ref{rk:isomorphism}, and also allows us to slightly improve a theorem of Genevois regarding the acylindrical hyperbolicity of their automorphism groups \cite{Gen}. 
\end{Rq}

\subsection{Quantitative point of view}\label{Sec:QuantitatifIntro}

As previously mentioned, some classes of groups turn out to be very rigid from the measure equivalence point of view while amenable groups or lattices in $\mathrm{SL}_2(\bR)$, on the other hand, are quite flexible: the class of groups that are measure equivalent to one of these groups is actually very large. 

This motivates the study of refinements of measure (or orbit) equivalence, that are more sensitive to the geometry of the groups inside these flexible classes. To that end, an \emph{integrable} version of these notions has been defined and allowed Bader, Furman and Sauer \cite{BFSIntegrability} to obtain rigidity phenomena for lattices in $\mathrm{SO}(n,1)$ when $n\geq 2$. 

This refined version has recently been generalized by Delabie, Koivisto, Le Maître and Tessera \cite{DKLMT} to what they called \emph{quantitative} measure (and orbit) equivalence (see \cref{Def:QuantitativeME} below). In particular, geometric notions such as the growth (see Bowen's appendix to Austin's article \cite{AustinBowen}, see also \cite[Theorem~3.1]{DKLMT}) or the isoperimetric profile \cite[Theorem~1.1]{DKLMT} are captured by this quantitative version of measure equivalence, and can thus be used to distinguish amenable groups.

But quantitative measure equivalence has not been extensively studied beyond the amenable world (where growth or isoperimetric profiles no longer provide meaningful invariants), apart from some strong rigidity theorems in integrable measure equivalence for some specific families of groups \cite{BFSIntegrability,Bow,HH-L1}. 

In \cref{Part:QuantitativeResults} of the present work, we extend our main classification theorem for graph products to the quantitative framework (\cref{Prop:CouplageGraphProd,Prop:DescenteDuProfil}), thereby providing the first systematic study of quantitative measure equivalence in a non-amenable setting. As an application, we answer the inverse problem of the quantification for a very large family of right-angled Artin groups (\cref{Th:PrescribedCouplingRAAG}). But let us first recall some notations and terminology introduced in \cite{DKLMT}.

\subsubsection{Quantification of couplings}

Let $G$ and $H$ be two measure equivalent groups over some measured space $(\Omega,m)$ with respective fundamental domains $X_G$ and $X_H$. In this framework one can define two \emph{measure equivalence cocycles}\index{Cocycle!Measure equivalence cocycle}\index{Measure equivalence!Measure equivalence cocycle} denoted $c:G\times X_H\to H$ and $c':H\times X_G\rightarrow G$ by letting $c(g,x)$ be the unique element of $H$ that brings $g*x$ back to the fundamental domain $X_H$ (see \cref{fig:Orbites}) and $c'$ is defined analogously by reverting the roles of $G$ and $H$.
In this section we will need to keep track of the fundamental domains $X_G$ and $X_H$, therefore we will denote by $(\Omega,m,X_G,X_H)$ a {measure equivalence coupling}\index{Measure equivalence!Measure equivalence coupling} from $G$ to $H$.

\begin{figure}[htbp]
    \centering
    \includegraphics[width=\textwidth]{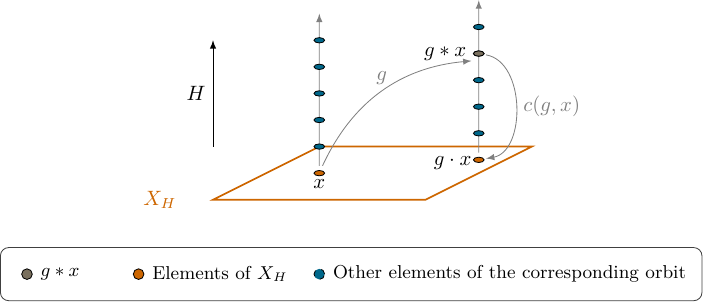}
  \caption{Definition of $c(g,x)$} 
  \label{fig:Orbites}
\end{figure}

If $S_G$ denotes a finite generating set of $G$, we denote by $|g|_{S_G}$ the word length of $g$ with respect to $S_G$.

\begin{Def} \label{Def:QuantitativeME} Let $G=\langle{S_G}\rangle$ and $H=\langle{S_H}\rangle$ be two finitely generated groups. 
Given a measure equivalence coupling $(\Omega,m,X_G,X_H)$ and two non-decreasing functions $\varphi$ and $\psi$, one says that the coupling is \emph{$(\varphi,\psi)$-integrable}\index{phi integrable@$(\varphi,\psi)$-integrable} {(from $G$ to~$H$)} if
for all $g\in G$ and all $h\in H$, there exist $\kappa_g,\kappa_h>0$ such that
\[\int_{X_H}\varphi\left(\frac{1}{\kappa_g}|c(g,x)|_{S_H}\right)dm(x)<+\infty \quad \text{and} \quad 
\int_{X_G}\psi\left(\frac{1}{\kappa_h}|c'(h,x)|_{S_G}\right)dm(x)<+\infty.\] 
\end{Def}

We will say that the coupling is $(\varphi,L^0)$-integrable when there are no conditions on the integrability of the cocycle $c^\prime$.

In the particular case where $X_G=X_H=:X$, the coupling $(\Omega,m,X,X)$ will be called a \emph{$(\varphi,\psi)$-integrable orbit equivalence coupling.} 

\subsubsection{Main quantification results for graph products}

We first show in \cref{Sec:FactToProd} that the integrability between vertex groups passes to the graph products. The following theorem gives a quantitative version of the implication “$3\Rightarrow 2$” of \cref{theo:classificationversionintro}, valid with no condition on the defining graph.

\begin{Th} 
\label{Prop:CouplageGraphProd}
 Let $G$ and $H$ be two graph products over a finite simple graph $\Gamma$, with finitely generated vertex groups $(G_v)_{v\in V\Gamma}$ and $(H_v)_{v\in V\Gamma}$, respectively.
 
  Let $\varphi,\psi :[1,+\infty[\rightarrow[1,+\infty[ $ be two non-decreasing functions. Assume that for every $v\in V\Gamma$, there exists a $(\varphi,\psi)$-integrable orbit equivalence coupling from $G_v$ to~$H_v$. 
  
  Then there exists a $(\varphi,\psi)$-integrable orbit equivalence coupling from
  $G$ to $H$.
\end{Th}

In fact the coupling was already constructed in \cite{HH21}, using an argument of Gaboriau \cite{Gab}, and here we only need to check its integrability. 

Remark that by \cref{theo:classificationversionintro}, orbit equivalence can not be replaced by measure equivalence in the above theorem. For a basic example, consider $G:=\{1\}*\bZ/2\bZ$ and $H=\bZ/2\bZ * \bZ/2\bZ$. Even though $\{1\}$ and $\bZ/2\bZ$ are measure equivalent, the two graph products $G$ and $H$ are not: $G$ is finite and $H$ infinite amenable (see e.g.\ $\mathbf{P}_{\mathrm{ME}}\mathbf{1}$ and $\mathbf{P}_{\mathrm{ME}}\mathbf{10}$ in \cite{Gab}). 

\medskip

Our main theorem in Part~\ref{Part:QuantitativeResults} is a converse to Theorem~\ref{Prop:CouplageGraphProd}, which provides a quantitative version of the implication “$1\Rightarrow 3$” in Theorem~\ref{theo:classificationversionintro}. 

\begin{Th}[see Theorem~\ref{Th:ProdToFactWithVRP}, \cpageref{Th:ProdToFactWithVRP}]\label{Prop:DescenteDuProfil} 
 Let $G,H$ be graph products with infinite finitely generated vertex groups over finite simple graphs $\Gamma_G,\Gamma_H$ with no transvection and no partial conjugation. Let $\varphi,\psi :[1,+\infty[\rightarrow[1,+\infty[ $ be two non-decreasing functions.
 
 If there exists a $(\varphi,\psi)$-integrable measure equivalence coupling from $G$ to $H$,
 then there exists a graph isomorphism $\sigma:\Gamma_G\to\Gamma_H$, such that for every $v\in V\Gamma_G$, there exists a $(\varphi,\psi)$-integrable measure equivalence coupling from $G_v$ to $H_{\sigma(v)}$.
\end{Th}

We refer to \cref{Th:ProdToFactWithVRP} for an even more general statement, with variations on the assumptions made on the defining graphs $\Gamma_G,\Gamma_H$. 

Notice that Theorem~\ref{theo:classificationversionintro} already gives us a measure equivalence coupling from $G_v$ to $H_{\sigma(v)}$. We crucially use the convexity of the vertex groups inside the graph product (exploited in the form of Green's normal form theorem for graph products \cite{Gre}) to obtain the desired integrability.

\begin{Rq}[Discussion on the asymmetry orbit/measure equivalence]\label{Rq:AsymmetryOEMEQt} We do not know whether, as could be suggested by  Theorem~\ref{theo:classificationversionintro}, the conclusion of Theorem~\ref{Prop:DescenteDuProfil} can be improved to a $(\varphi,\psi)$-integrable orbit equivalence coupling from $G_v$ to $H_{\sigma(v)}$. The difficulty is that having a $(\varphi,\psi)$-integrable measure equivalence coupling $(\Omega,m,X_G,X_H)$, and knowing that there also exists a common fundamental domain $Y$ for the actions of $G$ and $H$, is not enough to deduce that $(\Omega,m,Y,Y)$ is $(\varphi,\psi)$-integrable. Whence the slight asymmetry in our statements of Theorems~\ref{Prop:CouplageGraphProd} and~\ref{Prop:DescenteDuProfil}, one with orbit equivalence, and the second with measure equivalence.
\end{Rq}

\subsubsection{An application to the inverse problem for quantifications}\label{Sec:OptimIPIntro}

In \cref{Sec:InverseProblem} we construct couplings with right-angled Artin groups with prescribed integrability. Formally, we will say that a $(\Phi,L^0)$-integrable coupling {from} $G$ to $H$ is \emph{optimal}\index{Optimal coupling}\index{Coupling!Optimal} if any $(\varphi,L^0)$-integrable coupling from $G$ to $H$ verifies $\varphi \preccurlyeq \Phi$ (that is to say if $\varphi(x)=\mathcal{O}(\Phi(Cx))$ for some $C>0$, as $x$ goes to $+\infty$). In \cref{Sec:InverseProblem} we investigate the following question —~also called “the inverse problem”: given a group $H$ and an increasing function $\varphi$, can one find a group $G$ and a $(\varphi,L^0)$-integrable coupling from $G$ to $H$ such that the integrability is optimal?

The construction of prescribed couplings with optimal (or close to be) integrability has been solved by the first named author in \cite{Esc1} for various classes of amenable groups (including $\bZ$ or the usual lamplighter group). We offer here an extension of these results for the same class of functions as in the latter reference but \emph{outside of the amenable framework}. This is done by constructing couplings with prescribed integrability with a right-angled Artin group $H$. To do so we rely on \cref{Prop:CouplageGraphProd} to obtain a coupling from a graph product $G$ with amenable vertex groups, to $H$, with the wanted integrability. Finally, relying on the geometry of the vertex groups, we prove that these couplings are optimal (up to some logarithmic error). This is formalized in \cref{Th:PrescribedCouplingRAAG} below. 

A vertex $v\in V\Gamma$ is \emph{untransvectable}\index{Untransvectable!Untransvectable vertex} if there does not exist any $w\in V\Gamma\backslash\{v\}$ such that $\lk(v)\subseteq \st(w)$.

\begin{Th}\label{Th:PrescribedCouplingRAAG}
Let $H$ be a right-angled Artin group over a finite simple graph $\Gamma$ which contains an untransvectable vertex. Let $\rho:[1,+\infty) \rightarrow [1,+\infty)$ be a non-decreasing map such that $x\mapsto x/\rho(x)$ is non-decreasing. 
Then there exist a group $G$ such that
\begin{itemize}
    \item there exists an orbit equivalence coupling that is $(\varphi_{\varepsilon},\exp \circ \rho)$-integrable from $G$ to $H$ for all 
    $\varepsilon>0$, where $\varphi_{\varepsilon}(x):={\rho\circ
    \log(x)}/{\left( \log\circ\rho\circ\log(x) \right)^{1+\varepsilon}}$;
    \item and if there exists a $(\varphi,L^0)$-integrable measure equivalence coupling from $G$ to $H$, then $\varphi\preccurlyeq \rho \circ \log$ (so the coupling given by the first point has optimal integrability up to a log factor).
\end{itemize}
\end{Th}

Theorem~\ref{Th:PrescribedCouplingRAAG} applies, for instance, to every right-angled Artin group with finite outer automorphism group (see the above \cref{Sec:IntroRAAG}). 
But Theorem~\ref{Th:PrescribedCouplingRAAG} also applies, for instance, to all right-angled Artin groups whose defining graph is a tree of diameter at least $3$, a class which is not covered by the results in \cite{HH21}. 

\subsection{Rigidity}

All the results presented so far are classification theorems within the class of graph products. We will now consider measure equivalence couplings between a graph product $G$ and an arbitrary countable group $H$, and aim for rigidity theorems, i.e.\ give conditions ensuring that $H$ has to be virtually isomorphic to $G$. There are two ways in which rigidity can be achieved, which occupy Part~\ref{Part:Superrigidity} and Part~\ref{Part:SuperrigityHigman} of the present work, respectively. 
\begin{itemize}
    \item In Part~\ref{Part:Superrigidity}, we impose extra ergodicity conditions on the coupling between $G$ and $H$. This is in the spirit of the work of Monod--Shalom on rigidity for irreducible actions of direct products \cite{MS}, and \cite{HHI} for right-angled Artin groups --~in fact we do follow the strategy of \cite{HHI} very closely.
    \item In Part~\ref{Part:SuperrigityHigman}, we impose some strong form of measure equivalence rigidity for the vertex groups (but no extra conditions on the coupling), and use graph products to obtain examples of groups that are superrigid in measure equivalence, but not in quasi-isometry. 
\end{itemize}

\subsubsection{Orbit equivalence rigidity under ergodicity assumptions}

In \cref{Part:Superrigidity}, we consider probability measure-preserving actions of graph products that satisfy extra ergodicity conditions, and obtain strong rigidity results. Here we will state a simple version of our theorems where we assume all actions to be mixing. In Part~\ref{Part:Superrigidity}, we investigate various sets of ergodicity assumptions that ensure rigidity. We refer in particular to Theorem~\ref{theo:strong-rigidity}, \cpageref{theo:strong-rigidity} and Theorem~\ref{theo:superrigidity}, \cpageref{theo:superrigidity} for the full statements.

Recall that a free, measure-preserving action of a countable group $G$ on a standard probability space $(X,\mu)$ is \emph{mixing} if for any Borel subsets $U,V\subseteq X$, and any sequence $(g_n)_{n\in\mathbb{N}}\in G^{\mathbb{N}}$ consisting of pairwise distinct elements, one has \[\lim_{n\to +\infty}\mu(g_n\cdot U\cap V)=\mu(U)\mu(V).\] Recall also that two probability measure-preserving actions $G\actson X$ and $H\actson Y$ are \emph{conjugate}\index{Conjugate actions} if there exist a group isomorphism $\alpha:G\rightarrow H$ and a measure space isomorphism $f:X\to Y$ such that for all $g\in G$ and a.e.\ $x\in X$ one has $f(g\cdot x)=\alpha(g)\cdot f(x)$.

\begin{Th}\label{theo:superrigidityintro}
  Let $G,H$ be two torsion-free countable groups. Let $G\actson X$ and $H\actson Y$ be two mixing measure-preserving actions on standard probability spaces. Assume that $G$ splits as a graph product over a connected finite simple graph $\Gamma$, with no vertex joined to every other vertex.
  
  If the two actions $G\actson X$ and $H\actson Y$ are orbit equivalent, then they are conjugate.
\end{Th}

Part~\ref{Part:Superrigidity} is organized as follows.
\begin{itemize}
    \item In Section~\ref{Sec:StrongRigidity}, we consider the case where both $G$ and $H$ are graph products, over connected graphs with no vertex joined to every other vertex. We prove rigidity results that are internal to the class of graph products, under softer ergodicity conditions. For instance, assuming that the actions $G\actson X$ and $H\actson Y$ are both vertexwise ergodic (i.e.\ every vertex group acts ergodically), and that $G$ and $H$ have the unique root property (if $g_1^k=g_2^k$ with $k>0$, then $g_1=g_2$), we prove that (stably) orbit equivalent actions $G\actson X$ and $H\actson Y$ are automatically conjugate. Several variations over this statement are given.
    \item In Section~\ref{Sec:Superrigidity}, we consider the superrigidity problem, where the group $H$ is arbitrary. As in the work of Monod--Shalom, a mild mixing condition on the action $H\actson Y$ enables to transfer the rigidity within the class of graph products, to a superrigidity theorem, through composition of couplings. 
\end{itemize}

\begin{Rq}
As in Remark~\ref{rk:von-neumann}, combining the above theorem with theorems regarding the uniqueness of the Cartan subalgebra also yields $W^*$-rigidity theorems for actions of graph products satisfying extra ergodicity conditions. See Section~\ref{sec:von-neumann-2}.
\end{Rq}

\paragraph*{Comparison with previously known results.}
\begin{itemize}
    \item In \cite{MS}, Monod and Shalom proved rigidity theorems for actions of products of groups having some negative curvature property (class $\Creg$ – see \cref{ex:creg}, \cpageref{ex:creg} for the definition), assuming that each factor acts ergodically. Our proof deeply relies on their work, using the fact that graph products contain many product subgroups. Interestingly, we do not need to assume that the vertex groups belong to the class $\Creg$ in our theorems: the graph product structure alone is enough to ensure the rigidity.
    \item In \cite{HHI}, Huang, Ioana and the second-named author proved the above theorem (and the variations presented in Part~\ref{Part:Superrigidity}) in the case of right-angled Artin groups. We follow their strategy very closely, and crucially use our work from Part~\ref{Part:MEClassification} to replace several arguments from \cite{HHI}.
    \item Finally, for Bernoulli actions, a particular case of \emph{malleable} actions (see \cite[Definition 4.3]{Pop}), very strong rigidity theorems were proved by Popa in the framework of his deformation/rigidity theory \cite{Pop,Popa-SpectralGap}. His theorems cover the Bernoulli actions of numerous graph products, and in fact many more classes (including, in a different direction, all Bernoulli actions of Property~(T) groups). 
\end{itemize}

\subsubsection{Measure equivalence rigidity for graph products}

The final chapter of this work (Part~\ref{Part:SuperrigityHigman}) addresses the question of the rigidity of graph products whose vertex groups are rigid in measure equivalence. We say that a countable group $G$ is \emph{rigid in measure equivalence} if every countable group $H$ which is measure equivalent to $G$, is virtually isomorphic to $G$. Our main results are the following.
\begin{itemize}
    \item In Section~\ref{sec:rigidity-criterion}, we give a sufficient criterion to ensure that a graph product $G$ is rigid in measure equivalence. This criterion requires to know that every vertex group $G_v$ is rigid in a very strong sense, that is, every self measure equivalence coupling of $G_v$ factors through the tautological one (by left-right multiplication on $G_v$ itself, or at least on a finite index extension of $G_v$). We also need extra assumptions on the vertex groups, including having some control on possible homomorphisms between them. See Theorem~\ref{theo:rigidity-criterion}, \cpageref{theo:rigidity-criterion}.
    \item In Section~\ref{sec:higman}, using the above criterion with Higman groups as vertex groups, we construct a graph product which, to our knowledge, is the first example of a group which is rigid (among torsion-free groups) in measure equivalence, but not in quasi-isometry.
\end{itemize}

As mentioned, our example involves families of Higman groups \cite{Hig}, whose rigidity in measure equivalence was proven by Huang and the second-named author \cite{HH-Higman}. Given $k\ge 4$, let 
\[\Hig_k=\langle a_1,\dots, a_k \mid a_ia_{i+1}a_i^{-1}=a_{i+1}^{2} \quad \forall i~(\text{mod~} k)\rangle.\] 
We attain the following theorem.

\begin{Th}
\label{theo:qi-me}
    Let $\Gamma$ be a finite simple graph not reduced to one vertex, with no transvection and no partial conjugation, such that $\Aut(\Gamma)=\{1\}$. Let $(k_v)_{v\in V\Gamma}$ be a family of integers, and for every $v\in V\Gamma$, let $G_v=\Hig_{k_v}$. Assume that for every $v\neq w$, the integer $k_v$ does not divide $k_w$. Let $G$ be the graph product over $\Gamma$ with vertex groups $(G_v)_{v\in V\Gamma}$. Then
    \begin{enumerate}
        \item Every torsion-free countable group $H$ which is measure equivalent to $G$, is isomorphic to $G$.
        \item There exists an infinite family of finitely generated torsion-free groups that are all quasi-isometric to $G$, and pairwise not commensurable.
    \end{enumerate}
\end{Th}

To build a group that is quasi-isometric to $G$ but not commensurable, one only needs to replace one vertex groups $G_v$ by $G_v\times\mathbb{Z}/n\mathbb{Z}$. See Remark~\ref{rk:qi-me}. We now say a word about our proof of the measure equivalence rigidity of $G$, a more detailed outline is given at the beginning of Section~\ref{sec:rigidity-criterion}.

\paragraph{A word on the proof.} Using the work from Part~\ref{Part:MEClassification}, we know that for every self measure equivalence coupling $\Omega$ of $G$, there is a $(G\times G)$-equivariant measurable map $\Omega\to \Aut(\bD_G)$, where $\bD_G$ is the right-angled building of the graph product. We use a standard technique that dates back to the work of Furman \cite{Fur-me}, to transfer rigidity from self-couplings to arbitrary couplings. We deduce a homomorphism $H\to\Aut(\bD_G)$, i.e.\ an action of $H$ on $\bD_G$. And we prove that the cell stabilizers for the $H$-action on $\bD_G$ are measure equivalent to their respective $G$-stabilizers. Using the measure equivalence rigidity of the Higman groups, in fact $G$ and $H$ act on $\bD_G$ with isomorphic cell stabilizers. 

The crucial point is to use a volume (coupling index) argument to show that $G$ and $H$ act on $\bD_G$ with a common strict fundamental domain $Y$. For this, it is crucial to know that $\Hig_\sigma$ has no proper finite-index subgroup \cite{Hig}.

The existence of a strict fundamental domain for the $H$-action ensures that $H$ has the structure of simple complex of groups over $Y$, in the sense of Bridson--Haefliger \cite[Chapter~II.12]{BH}. The inclusion homomorphisms are inclusions between direct products of Higman groups, and these are very constrained: they must be factor inclusions. Consequently, the complex of group structure of $H$ is nothing but a graph product structure, in other words $H$ is isomorphic to $G$.

\subsection{Structure of the paper}

The present article is decomposed in five parts and one appendix. The first part contains the necessary material about graph products, measure and orbit equivalence, and measured groupoids and --~depending on the reader’s background~-- can be skipped on first reading.

\cref{Part:MEClassification} revolves around our measure equivalence classification of graph products, namely the proof of \cref{theo:classificationversionintro}. In particular we develop in \cref{sec:strongly-reduced} most of the groupoid theoretic machinery needed to prove the latter result. Let us mention that a translation of this argument in the language of groups can be found in \Cref{Appendix:Isom}, where it is used to show the isomorphic classification theorem analogous to \cref{theo:classificationversionintro}. \\
\cref{Part:MEClassification} is also where we prove our corollaries about commensurability (\cref{theo:classification-commensurability}) and fundamental groups of equivalence relations and von Neumann algebras (\cref{cor:fundamental}, \cpageref{cor:fundamental}). Both of the latter proofs can be found in \cref{Sec:ApplicationsMainTheorem}, \cpageref{Sec:ApplicationsMainTheorem}. The last section of the part is devoted to examples (\cref{Sec:ExemplesMEVRP}, \cpageref{Sec:ExemplesMEVRP}). 

In \cref{Part:QuantitativeResults} we study the quantitative version of this measure equivalence classification and prove \cref{Prop:CouplageGraphProd,Prop:DescenteDuProfil,Th:PrescribedCouplingRAAG}.

The last two parts deal with rigidity behaviours. In \cref{Part:Superrigidity} we impose some ergodicity assumptions on the actions and obtain some rigidity statements that are in the same spirit as Monod-Shalom rigidity theorem for direct products (Theorem~\ref{theo:superrigidityintro}). \cref{Part:SuperrigityHigman} focuses on exhibiting an example of a group that is superrigid for measure equivalence but not for quasi-isometry (Theorem~\ref{theo:qi-me}).

Finally, in \Cref{Appendix:Isom} we obtain an isomorphic classification theorem (\Cref{cor:isomorphism}, \cpageref{cor:isomorphism}) of graph products. We use it to slightly improve a theorem of Genevois and show the acylindrical hyperbolicity of $\mathrm{Aut}(G)$ whenever $G$ does not split as a direct product, and is not an infinite dihedral group (\Cref{theo:acyl-hyp}, \cpageref{theo:acyl-hyp}). For the reader less familiar with the language of groupoids, the proof of \Cref{cor:isomorphism} can also serve as a guide through the proofs of \cref{sec:strongly-reduced}.

\section*{Acknowledgements}

We warmly thank Jingyin Huang for many stimulating discussions that inspired several ideas from the present work. We also thank the referee for their careful reading and valuable comments.

\medskip

This project was initiated during the trimester \emph{Groups acting on fractals, Hyperbolicity and Self-Similarity} held at the Institut Henri Poincaré (UAR 839 CNRS-Sorbonne Université) during Spring 2022. Both authors thank the IHP for its hospitality and support, through LabEx CARMIN, ANR-10-LABX-59-01.

Both authors were supported by the European Union (ERC, Artin-Out-ME-OA, 101040507). Views and opinions expressed are however those of the authors only and do not necessarily reflect those of the European Union or the European Research Council. Neither the European Union nor the granting authority can be held responsible for them.

The first-named author was also funded by the Deutsche Forschungsgemeinschaft (DFG, German Research Foundation) – Project-ID 427320536 – SFB 1442, as well as under Germany’s Excellence Strategy EXC 2044 –390685587, Mathematics Münster: Dynamics–Geometry–Structure.

\newpage\part{Graph products and measured groupoids}
This first part contains background and preliminary facts regarding graph products and their geometry (Section~\ref{sec:background-graph-products}), measure and orbit equivalence and its quantitative versions (Section~\ref{sec:ME}), and measured groupoids and their use in measure equivalence classification problems (Section~\ref{sec:background-groupoids}).

\section{Graph products and their geometry}
\label{sec:background-graph-products}
In the following $\Gamma$ will always denote a finite simple graph (namely a graph with no loop edge and no multiple edges between two vertices) and $(G_v)_{v\in V\Gamma}$ a family of non-trivial vertex groups. We let $G$ be the graph product over $\Gamma$ with vertex groups $(G_v)_{v\in V\Gamma}$. An \emph{induced subgraph}\index{Induced subgraph} $\Lambda\subseteq\Gamma$ is a subgraph such that two vertices of $\Lambda$ are joined by an edge in $\Lambda$ if and only if they are joined by an edge in $\Gamma$. Recall that the \emph{link}\index{Link (of a vertex)} of a vertex $v\in V\Gamma$ is the induced subgraph with vertex set all vertices adjacent to $v$. We denote it $\lk(v)$. The \emph{star}\index{Star (of a vertex)} of $v\in V\Gamma$ is the subgraph of $\Gamma$ induced by $v$ itself and its neighbors and is denoted by $\st(v)$. 

Recall that a finite simple graph $\Gamma$ is \emph{transvection-free}\index{Transvection-free} if there do not exist two distinct vertices $v,w\in V\Gamma$ such that $\lk(v)\subseteq\st(w)$. We say that $\Gamma$ \emph{has no partial conjugation}\index{No partial conjugation} if there does not exist any vertex $v\in V\Gamma$ such that $\st(v)$ disconnects~$\Gamma$. This terminology comes from the setting of right-angled Artin groups, where these conditions prevent having transvections and partial conjugations in $\Out(G_\Gamma)$ -- and in fact any finite simple graph $\Gamma$ which is transvection-free and has no partial conjugation determines a right-angled Artin group $G_\Gamma$ with $|\Out(G_\Gamma)|<+\infty$ by \cite{Laurence,Servatius}.

A finite simple graph $\Gamma$ is called \emph{reducible}\index{Reducible graph} if it splits as a join of (at least) two non-empty subgraphs. We will call \emph{irreducible}\index{Irreducible graph}\label{Def:Irreducible} a graph which is not reducible.

\begin{Rq}\label{rk:transvection-free-factors}
    Let $\Gamma$ be a finite simple graph, not reduced to one vertex, and let $\Gamma=\Gamma_1\circ\dots\circ\Gamma_k$ be its join decomposition into irreducible subgraphs (well-defined up to permutation of the factors). If $\Gamma$ is transvection-free, then no $\Gamma_i$ is reduced to one vertex, and every $\Gamma_i$ is transvection-free. And if $\Gamma$ has no partial conjugation, then for every $i\in\{1,\dots,k\}$, the graph $\Gamma_i$ has no partial conjugation.
\end{Rq}

\subsection{Normal form for graph products}\label{Sec:NormalForms}

This section reviews work of Green \cite{Gre}. 
Given $g\in G$, a \emph{word}\label{def:word}\index{Word}
representing $g$ is a tuple $\underline{\mathsf{w}}=(g_1,\dots,g_k)$, where every $g_i$ is an element of some vertex group, such that $g=g_1\dots g_k$.
Following \cite[Definition~3.5]{Gre}, we say that a word $\underline{\mathsf{w}}$ is \emph{reduced}\index{Reduced word} if, letting $v_i\in V\Gamma$ be such that $g_i\in G_{v_i}$, the following two conditions hold: 
\begin{itemize}
    \item for every $i\in\{1,\dots,k\}$, one has $g_i\neq 1$;
    \item for every $i<j$, if $v_i=v_j$, then there exists $i<k<j$ such that $v_k$ is neither equal nor adjacent to $v_i$. 
\end{itemize}
The second condition implies in particular that $v_{i+1}\neq v_i$ for every $i\in\{1,\dots,k\}$. Notice that a subword of a reduced word is again reduced, and that if $(g_1,\dots,g_n)$ is reduced, so is $(g^{-1}_n,\dots,g^{-1}_1)$.

By Green’s normal form theorem \cite[Theorem~3.9]{Gre},\label{Th:GreenNormalForm} every element of $G$ is represented by a reduced word. In addition, two reduced words $\underline{\mathsf{w}}=(g_1,\dots,g_k)$ and $\underline{\mathsf{w}}'=(g'_1,\dots,g'_\ell)$ represent the same element if and only if $k=\ell$, and $\underline{\mathsf{w}}$ and $\underline{\mathsf{w}}'$ are obtained from one another by successive applications of the following operation: if $g_i$ and $g_{i+1}$ belong to adjacent vertex groups, swap them. 
More generally, in her proof, Green also showed (see \cite[Definition~3.10 and Property~(a) on p.30]{Gre}) that, starting from any word $\underline{\mathsf{w}}$ representing $g$, one can obtain a reduced word representing $g$ by applying a sequence of the following operations:
\begin{itemize}
    \item remove some $g_i=1$;
    \item if $g_i$ and $g_{i+1}$ belong to the same vertex group, replace them by the product $g_ig_{i+1}$;
    \item if $g_i$ and $g_{i+1}$ belong to adjacent vertex groups, swap them.
\end{itemize}

For every $v\in V\Gamma$, let $S_v$ be a finite generating set of $G_v$, and let $S=\cup_{v\in V\Gamma}S_v$, which is a finite generating set of $G$. Given a word $\underline{\mathsf{w}}=(g_1,\dots,g_k)$, we let $|\underline{\mathsf{w}}|_S=\sum_{i=1}^k|g_i|_{S_{v_i}}$, where $v_i\in V\Gamma$ is such that $g_i\in G_{v_i}$ -- here $|g_i|_{S_{v_i}}$ denotes the word length of $g_i$ in the generating set $S_{v_i}$. Observe that if $\underline{\mathsf{w}}$ is a reduced word representing $g$, then $|g|_S=|\underline{\mathsf{w}}|_S$ (see e.g.\ \cite[Corollary~2.5]{BR}).

Given $g\in G$, the \emph{head}\label{Def:Head}\index{Head of a word} of $g$ is the set of all elements $h\in G$ such that there exists a reduced word $\underline{\mathsf{w}}=(g_1,\dots,g_k)$ representing $g$, with $g_1=h$. Likewise, the \emph{tail}\label{Def:Tail}\index{Tail of a word} of $g$ is the set of all elements $h\in G$ such that there exists a reduced word $\underline{\mathsf{w}}=(g_1,\dots,g_k)$ representing $g$, with $g_k=h$. We say that an element $z\in G$ is a \emph{syllable}\index{Syllable of a word} of $g$ if for some (equivalently, any) reduced word $\underline{\mathsf{w}}=(g_1,\dots,g_k)$ representing $g$, there exists $i\in\{1,\dots,k\}$ such that $z=g_i$. 

\begin{Lmm}\label{Lmm:LengthAndNormalForm}
Let $g,h\in G$, let $v\in V\Gamma$, and let $z\in G_v\setminus\{1\}$. Assume that the heads of $g$ and $h$ do not contain any element of $G_v$.

Then $z$ is a syllable of $g^{-1}zh$. In particular, $|g^{-1}zh|_S\ge |z|_S$.
\end{Lmm}

\begin{proof}
Given a reduced word $\underline{\mathsf{w}}=(g_1,\dots,g_k)$ representing $g$,
we let
\begin{equation*}
    i_0(\underline{\mathsf{w}})=\max\{i \, : \, \forall j\leq i, \ g_j\notin G_v, \text{and}\ g_j \ \text{centralizes}\ G_v\}.
\end{equation*}
Namely $i_0$ is the largest $i$ such that for every $j\le i$, the element $g_j$ belongs to a vertex group $G_{w_j}$ with $w_j$ adjacent to $v$.
We now choose $\underline{\mathsf{w}}_g$ so as to maximize the value of $i_0$. Let $\underline{\mathsf{w}}_h=(h_1,\dots,h_\ell)$ be a reduced word representing $h$, chosen with the same constraint (we denote by $j_0$ the index defined by the same property). Then $(g_k^{-1},\dots,g_{i_0}^{-1},\dots,g_1^{-1},z,h_1,\dots,h_{j_0},\dots,h_\ell)$ is a word that represents $g^{-1}zh$, and so is 
\begin{equation*}
    (g_k^{-1},\dots,g_{i_0+1}^{-1},z,g_{i_0}^{-1},\dots,g_1^{-1},h_1,\dots,h_{j_0},\dots,h_\ell).
\end{equation*}
The subword $\underline{\mathsf{w}}'=(g_{i_0}^{-1},\dots,g_1^{-1},h_1,\dots,h_{j_0})$ might not be reduced. Let $\underline{\mathsf{w}}''=(k_1,\dots,k_r)$ be a reduced word that represents the same element as $\underline{\mathsf{w}}'$, and observe that all $k_p$ belong to vertex groups $G_{w_p}$ with $w_p$ adjacent to $v$ (in particular, distinct from $v$). Now $g^{-1}zh$ is represented by $(g_k^{-1},\dots,g_{i_0+1}^{-1},z,k_1,\dots,k_r,h_{j_0+1},\dots,h_\ell)$. Notice that $g_{i_0+1}$ and $h_{j_0+1}$ do not belong to vertex groups adjacent to $G_v$ by definition of $i_0$ and $j_0$, and they do not belong to $G_v$ either by our assumption on the heads of $g$ and $h$. Therefore the latter word is reduced, which concludes the proof. 
\end{proof}

\subsection{Parabolic subgroups}\label{Subsec:ParabolicSubgroup}

\subsubsection{Definition and first facts}\label{sec:first-facts}

Let $\Lambda\subseteq\Gamma$ be an induced subgraph, possibly empty. 
Let $G_\Lambda\subseteq G$ be the subgroup generated by all subgroups $G_v$ with $v\in V\Lambda$, note that when $\Lambda=\emptyset$ we have $G_\Lambda=\{e\}$. As a consequence of the normal form theorem for graph products reviewed in the previous section, the group $G_\Lambda$ is isomorphic to the graph product over $\Lambda$ with vertex groups $(G_v)_{v\in V\Lambda}$.

A \emph{parabolic subgroup}\index{Parabolic!Parabolic subgroup} of $G$ is a subgroup $P$ which is conjugate to $G_\Lambda$ for some induced subgraph $\Lambda\subseteq\Gamma$.\footnote{The notion of parabolic subgroup does not merely depend on $G$, but also of its decomposition as a graph product. Throughout the paper, we always assume that $G$ comes equipped with this decomposition.} Notice that $\Lambda$ is uniquely determined by $P$: if $gG_\Lambda g^{-1}=hG_{\Lambda'}h^{-1}$, then $\Lambda=\Lambda'$, as follows from \cite[Corollary~3.8]{AM}. We say that $\Lambda$ is the \emph{type}\index{Type!Type of a parabolic subgroup} of $P$.

An intersection of parabolic subgroups of $G$ is again a parabolic subgroup: this is \cite[Corollary~3.6]{AM} for finite intersections, and follows from \cite[Corollary~3.18]{AM} for infinite intersections. Remark also that if $P\subseteq G$ is a parabolic subgroup of type~$\Lambda$, then $P$ is naturally a graph product over $\Lambda$; parabolic subgroups of $P$ for this graph product structure are then naturally identified with the parabolic subgroups of $G$ that are contained in $P$. 

Given an induced subgraph $\Lambda\subseteq\Gamma$, we denote by $\Lambda^{\perp}$\label{notation:lambdaperp} the induced subgraph spanned by all vertices of $\Gamma\setminus\Lambda$ that are joined to all vertices of $\Lambda$ (see \cref{fig:GraphOrthog} for two examples). Remark that in particular if $\Lambda$ is reduced to one vertex $v$ then $\Lambda^\perp=\lk(v)$. Given a parabolic subgroup $P=gG_\Lambda g^{-1}$ (with $g\in G$, and $\Lambda\subseteq\Gamma$ an induced subgraph), we let $P^{\perp}=gG_{\Lambda^{\perp}}g^{-1}$. That this is well-defined (i.e.\ does not depend on the choice of $g$) follows from \cite[Proposition~3.13]{AM}. The normalizer of $P$ in $G$ is equal to $P\times P^{\perp}$, and is therefore again a parabolic subgroup \cite[Proposition~3.13]{AM}.
\begin{figure}[htbp]
  \centering
  \includegraphics[width=0.4\textwidth]{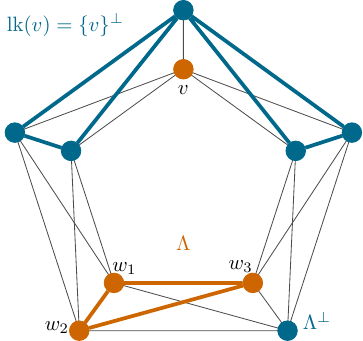}
  \caption{Induced subgraphs and orthogonal}
  \label{fig:GraphOrthog}
\end{figure}

Given an induced subgraph $\Lambda\subseteq\Gamma$, there is a natural retraction\index{Retraction to a parabolic subgroup} $r_\Lambda:G_\Gamma\to G_\Lambda$, sending every $G_v$ with $v\in V\Lambda$ to itself, and every $G_w$ with $w\notin V\Lambda$ to $\{1\}$. More generally, let $P$ be a parabolic subgroup, and let $\Lambda\subseteq\Gamma$ be its type. Then there exists a unique element $g\in G$ such that $P=gG_\Lambda g^{-1}$ and the tail of $g$ does not contain any element in $G_\Lambda\times G_{\Lambda}^{\perp}$. For such a $g$ we then define the retraction $r_P:G\to P$ to be the map sending $gG_vg^{-1}$ to itself if $v\in V\Lambda$, and $gG_wg^{-1}$ to $\{1\}$ if $w\notin V\Lambda$.  

\subsubsection{Splittings of a graph product}\label{sec:trees}

Every vertex $v\in V\Gamma$ determines a splitting of $G$ as $G=G_{\st(v)}\ast_{G_{\lk(v)}}G_{\Gamma\setminus\{v\}}$. We let $T_v$\label{NotationTv} be the Bass--Serre tree of this splitting, see \cite{Ser}. It is equipped with an isometric $G$-action whose vertex stabilizers are exactly the conjugates of $G_{\st(v)}$ and of $G_{\Gamma\setminus\{v\}}$.

The above splitting is non-trivial (i.e.\ the $G$-action on $T_v$ has no global fixed point) unless $v$ is joined to all other vertices of $\Gamma$ by an edge.

We will need a few standard facts regarding group actions on trees. Given an isometric action of a group $G$ on a tree $T$, we say that a subgroup $H\subseteq G$ is \emph{elliptic}\index{Elliptic subgroup} if it fixes a point in $T$. Notice that this happens as soon as $H$ has a finite orbit in $T$: indeed $H$ then fixes the circumcenter of this finite orbit. In particular, if $H$ has a finite-index subgroup $H^0$ which is elliptic in $T$, then $H$ is also elliptic in $T$ (because if $x\in T$ is a point fixed by $H^0$, then the $H$-orbit of $x$ is finite). 

\subsubsection{More properties of parabolic subgroups}

We now establish a few extra properties of parabolic subgroups that will be used throughout the paper. They are based on the following results of Antol\'in--Minasyan.

\begin{Prop}[{Antol\'in--Minasyan \cite[Proposition~3.4 and Lemma~3.9]{AM}}]\label{prop:am}
Let $G$ be a graph product over a finite simple graph $\Gamma$.
\begin{enumerate}
\item Let $\Lambda_1,\Lambda_2\subseteq\Gamma$ be induced subgraphs, and let $g\in G$. Then there exist an induced subgraph $\Upsilon\subseteq\Lambda_1\cap\Lambda_2$, and $h\in G_{\Lambda_2}$, such that $gG_{\Lambda_1}g^{-1}\cap G_{\Lambda_2}=hG_{\Upsilon}h^{-1}$. 
\item Let $P\subseteq G$ be a parabolic subgroup, and $g\in G$. If $gPg^{-1}\subseteq P$, then $gPg^{-1}=P$.
\end{enumerate}
\end{Prop}

\begin{Lmm}\label{lemma:finite-index-0}
Let $G$ be a graph product over a finite simple graph $\Gamma$ such that no vertex $v$ of $\Gamma$ satisfies $\Gamma=\st(v)$. Let $H\subseteq G$ be a subgroup. 
If $H$ has a finite-index subgroup $H^0$ contained in a proper parabolic subgroup $P$ of $G$, then $H$ itself is contained in a proper parabolic subgroup $Q$ of $G$.
\end{Lmm}

\begin{proof}
The group $P$ is conjugate to $G_\Lambda$ for some proper induced subgraph $\Lambda\subsetneq\Gamma$. Let $v$ be in $V\Gamma\setminus V\Lambda$. Then $H^0$ is elliptic in the tree $T_v$ because it is conjugate into $G_{\Gamma\setminus\{v\}}$. Therefore $H$ is also elliptic in $T_v$, and the conclusion follows (see the discussion in \cref{sec:trees}).
\end{proof}

\begin{Lmm}\label{lemma:finite-index}
Let $G$ be a graph product over a finite simple graph $\Gamma$ with infinite vertex groups. Let $P,Q\subseteq G$ be two parabolic subgroups, and assume that some finite-index subgroup of $P$ is contained in $Q$.

Then $P\subseteq Q$.
\end{Lmm}

\begin{proof}
Up to a global conjugation, we can (and shall) assume that there exist two induced subgraphs $\Upsilon,\Lambda\subseteq\Gamma$ and $g\in G$ such that $Q=G_\Lambda$ and $P=gG_\Upsilon g^{-1}$. By Proposition~\ref{prop:am}(1), the intersection $P\cap Q$ is of the form $hG_{\Theta}h^{-1}$, with $h\in Q$ and $\Theta\subseteq\Lambda\cap\Upsilon$. Then for every $v\in V\Upsilon\setminus V\Theta$, the subgroup $gG_vg^{-1}$ intersects $P\cap Q$ trivially, again by Proposition~\ref{prop:am}(1). But $gG_vg^{-1}$ is contained in $P$, so by assumption $gG_vg^{-1}$ has a finite-index (whence infinite) subgroup contained in $Q$. This proves that $\Theta=\Upsilon$. We have $P\cap Q=hG_\Theta h^{-1}$ and $P=gG_\Theta g^{-1}$, so Proposition~\ref{prop:am}(2) implies that $P\cap Q=P$, i.e.\ $P\subseteq Q$.
\end{proof}

In the case where all vertex groups are torsion-free (or equivalently that $G$ is torsion-free \cite[Corollary~3.28]{Gre}), we can extend \cref{lemma:finite-index} to all subgroups $H$ of $G$.

\begin{lemma}\label{lemma:root}
Let $G$  be a graph product over a finite simple graph $\Gamma$, with infinite torsion-free vertex groups. Let $P\subseteq G$ be a parabolic subgroup, and let $H\subseteq G$ be a subgroup.

If $H$ has a finite-index subgroup $H^0$ contained in $P$, then $H\subseteq P$.
\end{lemma}

\begin{proof}
Without loss of generality, we will assume that $P$ is the smallest parabolic subgroup that contains $H^0$, namely $P$ is the intersection of all parabolic subgroups that contain $H^0$. Up to conjugation, we will assume that $P=G_\Lambda$ for some induced subgraph $\Lambda\subseteq\Gamma$. Let $v\in V\Gamma\setminus V\Lambda$. Then $H^0$ is elliptic in $T_v$, so $H$ is elliptic in $T_v$. So either $H$ is conjugate into $G_{\Gamma\setminus\{v\}}$, or else $H$ is conjugate into $G_{\st(v)}$.

Let us first assume that $H$ is conjugate into $G_{\st(v)}$. In particular, there exists $g\in G$ such that $H^0\subseteq G_\Lambda\cap gG_{\st(v)}g^{-1}$, namely $g^{-1}H^0g\subseteq g^{-1}G_\Lambda g\cap G_{\st(v)}$. By Proposition~\ref{prop:am}(1), it follows that $H^0\subseteq gG_{\lk(v)}g^{-1}$. Since $G_{\st(v)}=G_v\times G_{\lk(v)}$ and since $G_v$ is torsion-free, it follows that $H\subseteq gG_{\lk(v)}g^{-1}$: indeed otherwise, some element $h\in H$ has a nontrivial projection to $gG_vg^{-1}$, but then all powers of $h$ have a nontrivial projection to $gG_vg^{-1}$, contradicting that some power is contained in the finite-index subgroup $H^0$. 

So in all cases, we can assume that $H$ is conjugate into $G_{\Gamma\setminus\{v\}}$. As this is true for every $v\in V\Gamma\setminus V\Lambda$, we can apply Proposition~\ref{prop:am}(1) again and deduce that $H$ is contained in a conjugate $hPh^{-1}$. In particular $H^0$ is contained in $P\cap hPh^{-1}$, and by minimality of $P$ this implies that $hPh^{-1}=P$. Therefore $H\subseteq P$, as desired.
\end{proof}

Using the first point of \cref{prop:am} we can show that if $P=P_1\times P_2$ is a product of two parabolic subgroups $P_1,P_2$ then its type $\Lambda$ can be written as the join $\Lambda=\Lambda_1 \circ \Lambda_2$, where $\Lambda_i$ is the type of $P_i$ for every $i\in \{1,2\}$.
Moreover if $P=gG_{\Lambda}g^{-1}$ for some $g\in G$, then $P_i=gG_{\Lambda_i}g^{-1}$. We can thus deduce the following two lemmas -- details are left to the reader.

\begin{Lmm}\label{lemma:product-parabolic-subgroup}
    Let $G$ be a graph product over a finite simple graph $\Gamma$. Let $P=P_1\times P_2$ be a parabolic subgroup that splits as a direct product of two parabolic subgroups $P_1,P_2$. Let $Q\subseteq P$ be a parabolic subgroup.

    Then $Q=(Q\cap P_1)\times (Q\cap P_2)$.
\qed
\end{Lmm}

\begin{Lmm}\label{lemma:adjacency-extension}
    Let $G$ be a graph product over a finite simple graph $\Gamma$. Let $P,Q$ be two distinct parabolic subgroups which are conjugate to vertex groups.

    Then $P$ and $Q$ commute if and only if there exist $g\in G$ and two adjacent vertices $v,w\in V\Gamma$ such that $P=gG_vg^{-1}$ and $Q=gG_wg^{-1}$. 
    \qed
\end{Lmm}

A vertex $v\in V\Gamma$ is \emph{untransvectable}\index{Untransvectable!Untransvectable vertex}\label{Def:Untransvectable} if there does not exist any vertex $w\neq v$ such that $\lk(v)\subseteq\st(w)$.

\begin{Lmm}\label{lemma:parabolic-untransvectable}
    Let $\Gamma$ be a finite simple graph which is not reduced to one vertex, let $G$ be a graph product over $\Gamma$. Let $C\subseteq G$ be a parabolic subgroup which is conjugate to $G_{\Upsilon}$ for some complete subgraph $\Upsilon\subseteq \Gamma$. Let $v\in V\Gamma$ be an untransvectable vertex.

    If $G_v\times G_v^{\perp}\subseteq C\times C^{\perp}$, then $C=G_v$.
\end{Lmm}

\begin{proof} Let $g\in G$ be such that $C=gG_{\Upsilon}g^{-1}$.
Since $G_v\times G_v^{\perp}\subseteq C\times C^{\perp}$ then by \cref{prop:am} we also have inclusion of their types, namely $\st(v)\subseteq \Upsilon\circ \Upsilon^\perp$.  

Now let $w\in V\Upsilon$. Since $\Upsilon$ is a clique, we get $\Upsilon \circ \Upsilon^\perp\subseteq \st(w)$. In particular $\st(v)\subseteq \st(w)$, but $v$ is assumed to be untransvectable therefore $w=v$ and hence $C=gG_{v}g^{-1}$.
Using \cref{prop:am} leads to $g\in G_{\st(v)}$ and therefore $C=G_v$.
\end{proof}

\subsection{Extension graphs, right-angled buildings and their automorphisms}\label{subsec:ExtensionGraphs}

In this section, we first review the definitions of the extension graph $\Gamma_G^e$ and of the right-angled building $\bD_G$ associated to a graph product $G$. We then establish some results regarding the comparison of their automorphism groups.

\subsubsection{Definitions}\label{sec:def-building}

\paragraph{Extension graphs.} They were introduced by Kim and Koberda in the context of right-angled Artin groups \cite{KK}. Their definition naturally extends to graph products, as already observed for instance in \cite[Definition~3.1]{CRKNG}.

\begin{Def}[{Extension graph \cite{KK}}]\label{de:extension-graph}
Let $G$ be a graph product over a finite simple graph $\Gamma$. The \emph{extension graph}\index{Extension graph} $\Gamma_G^e$ is the graph whose vertices are the conjugates of the vertex groups of $G$, where two distinct vertices are joined by an edge if the corresponding subgroups commute.
\end{Def}

Notice that $\Gamma_G^e$ contains a natural subgraph isomorphic to $\Gamma$, spanned by the vertices corresponding to the vertex groups. The action of $G$ by conjugation on the set of conjugates of vertex groups, induces a $G$-action by graph automorphisms on $\Gamma_G^e$. For this action $\Gamma$ is a fundamental domain, as follows using Lemma~\ref{lemma:adjacency-extension}.

We will call a \emph{clique}\index{Clique} a complete subgraph of $\Gamma$ (possibly empty and in particular not necessarily maximal). A \emph{standard clique subgroup}\index{Standard!Clique subgroup} of $G$ is a parabolic subgroup of $G$ equal to $G_\Upsilon$ for some clique subgraph $\Upsilon\subseteq \Gamma$. A \emph{standard clique coset}\index{Standard!Clique coset} is a left coset $gG_\Upsilon$ of a standard clique subgroup of $G$. Notice that the stabilizer of the standard clique coset $gG_\Upsilon$, for the $G$-action by left translation on itself, is equal to $g G_\Upsilon g^{-1}$. From this, one deduces that if $\Upsilon_1,\Upsilon_2$ are two clique subgraphs of $G$ and $g_1,g_2\in G$, then the standard clique cosets $g_1G_{\Upsilon_1}$ and $g_2G_{\Upsilon_2}$ are equal if and only if $\Upsilon_1=\Upsilon_2$ and $g_1^{-1}g_2\in G_{\Upsilon_2}$. In particular we can define the \emph{type} of a clique coset as the clique subgraph $\Upsilon\subseteq\Gamma$ associated to the coset. 

\paragraph{Right-angled buildings.} In \cite{Dav}, Davis associated a right-angled building to any graph product, as follows.

As usual, let $G$ be a graph product over a finite simple graph $\Gamma$. Let $\mathfrak{F}$\label{Def:FrakF} be the set of all standard clique cosets. Then $\mathfrak{F}$ is a partially ordered set, ordered by inclusion. Given $F_1\subseteq F_2$ in $\mathfrak{F}$, the \emph{interval}\index{Interval in $\mathfrak{F}$} $I_{F_1,F_2}$ is defined as \[I_{F_1,F_2}=\{F\in\mathfrak{F}\mid F_1\subseteq F\subseteq F_2\}.\] Every interval in $\mathfrak{F}$ is a Boolean lattice of finite rank, so there exists a unique (up to isomorphism) cube complex $\bD_G$ whose poset of cubes is isomorphic to the poset of intervals of $\mathfrak{F}$, see e.g.\ \cite[Proposition~A.38]{AB}. More concretely,
\begin{itemize}
    \item the vertex set of $\bD_G$ is the set of standard clique cosets of $G$;
    \item there is an edge between any two standard clique cosets of the form $gG_{\Lambda},gG_{\Upsilon}$, where $\Lambda\subseteq\Upsilon$ and $|V\Upsilon\setminus V\Lambda|=1$;
    \item for $k\geq 2$ one glues a $k$-cube to every subgraph isomorphic to the $1$-skeleton of a $k$-cube, with the obvious face identifications. 
\end{itemize}

\begin{Def}[{Right-angled building \cite{Dav}}]\label{de:building}
The cube complex $\bD_G$ is called the \emph{Davis realization of the right-angled building}\index{Right-angled building of a graph product} of $G$. 
\end{Def}
  
In the sequel, we will simply say that $\bD_G$ is the \emph{right-angled building} of $G$. We mention that $\bD_G$ is a $\mathrm{CAT}(0)$ cube complex by \cite[Corollary~11.7]{Dav}. 

By construction, every vertex $v\in V\bD_G$ corresponds to a left coset of the form $gG_\Lambda$, where $\Lambda\subseteq\Gamma$ is a complete subgraph. The \emph{rank} of $v$ is the cardinality of the clique $\Lambda$ (this is well-defined: if $gG_\Lambda=g'G_{\Lambda'}$, then $\Lambda=\Lambda'$). Notice that rank $0$ vertices correspond to elements of $G$.

\subsubsection{Comparison with right-angled Artin groups}

In this section, we will prove that the extension graph and the right-angled building of a graph product of countably infinite groups over a finite simple graph $\Gamma$, are isomorphic to the same objects associated to the right-angled Artin group over $\Gamma$. This will allow us to derive certain facts about these spaces from the analogous facts in the case of right-angled Artin groups.

More generally, let $\Gamma$ be a finite simple graph, and let $G$ and $H$ be two graph products over $\Gamma$ such that for every $v\in V\Gamma$, the vertex groups $G_v$ and $H_v$ have the same cardinality. For every $v\in V\Gamma$, we fix a bijection $\theta_v:G_v\to H_v$ such that $\theta_v(e)=e$.

Through normal forms (see Section~\ref{Sec:NormalForms}), this induces a bijection $\theta:G\to H$, in the following way. Let $\underline{\mathsf{w}}_g=(g_1,\dots,g_k)$ be a reduced word that represents $g$. For every $i\in\{1,\dots,k\}$, let $v_i\in V\Gamma$ be such that $g_i\in G_{v_i}$. Let $\underline{\mathsf{w}}_H=(\theta_{v_1}(g_1),\dots,\theta_{v_k}(g_k))$. Our condition that $\theta_{v_i}(e)=e$ ensures that $\underline{\mathsf{w}}_H$ is again a reduced word, which represents an element of $H$ denoted $\theta(g)$. Notice that $\theta(g)$ does not depend on the choice of a reduced word that represents $g$, and that $\theta:G\to H$ is indeed a bijection (its inverse is obtained by considering the bijections $\theta_v^{-1}$).

We say that the bijection $\theta$ defined above is the bijection \emph{combined from $(\theta_v)_{v\in V\Gamma}$}\index{Combined bijection}.

\begin{Lmm}\label{lemma:extension-graph}
    Let $\Gamma$ be a finite simple graph, and let $G,H$ be two graph products over $\Gamma$. Assume that for every $v\in V\Gamma$, the vertex groups $G_v$ and $H_v$ have the same cardinality, and choose a bijection $\theta_v:G_v\to H_v$ with $\theta_v(e)=e$. Let $\theta:G\to H$ be the bijection combined from $(\theta_v)_{v\in V\Gamma}$.

    Then for every $g\in G$ and every $v\in V\Gamma_G$, one has $\theta(gG_vg^{-1})=\theta(g)H_{v}\theta(g)^{-1}$; in particular $\theta$ induces a graph isomorphism $\theta^e:\Gamma_G^e\to\Gamma_H^e$.
\end{Lmm}

\begin{proof}
Write $g=g_1\dots g_jg_{j+1}\dots g_k$, where each $g_i$ belongs to a vertex group, and $g_{j+1},\dots,g_k\in G_v\times G_v^{\perp}$, with $g_j\notin G_v\times G_v^{\perp}$. Then $gG_vg^{-1}=g_1\dots g_jG_vg_j^{-1}\dots g_1^{-1}$, so \[\theta(gG_vg^{-1})=\theta(g_1)\dots\theta(g_j) H_v\theta(g_j)^{-1}\dots\theta(g_1)^{-1}=\theta(g)H_v\theta(g)^{-1}.\]

   The map $\theta^e:V\Gamma_G^e\to V\Gamma_H^e$ defined by $\theta^e(gG_vg^{-1})=\theta(gG_vg^{-1})$ is a bijection. That it induces a graph isomorphism follows from Lemma~\ref{lemma:adjacency-extension}.
\end{proof}

The analogous statement for the right-angled building is as follows.

\begin{Lmm}\label{lemma:building}
   Let $\Gamma$ be a finite simple graph, and let $G,H$ be two graph products over $\Gamma$. Assume that for every $v\in V\Gamma$, the vertex groups $G_v$ and $H_v$ have the same cardinality, and choose a bijection $\theta_v:G_v\to H_v$ with $\theta_v(e)=e$. Let $\theta:G\to H$ be the bijection combined from $(\theta_v)_{v\in V\Gamma}$.

   Then for every standard clique coset $gG_\Upsilon$, one has $\theta(gG_\Upsilon)=\theta(g)H_\Upsilon$; in particular $\theta$ induces a cubical isomorphism $\theta_\bD:\bD_G\to\bD_H$.
\end{Lmm}

\begin{proof}
   Let us first prove that for all $g\in G$ and all $\Upsilon \subseteq \Gamma$ we have $\theta(gG_\Upsilon)=\theta(g)H_\Upsilon$.\\
   Let $\underline{\mathsf{w}}=(g_1,\dots,g_k)$ be a reduced word that represents $g$, chosen so that there exists $i\in\{1,\dots,k\}$ such that all $g_j$ with $j>i$ belong to $G_\Upsilon$, and every $h\in G_\Upsilon$ that belongs to the tail of $g$ is one of the $g_j$ with $j>i$. In this way $gG_\Upsilon=g_1\dots g_i G_\Upsilon$, and $\theta(gG_{\Upsilon})=\theta(g_1)\dots\theta(g_i)H_\Upsilon$. On the other hand $\theta(g)=\theta(g_1)\dots\theta(g_i)\theta(g_{i+1})\dots\theta(g_k)$, and the elements $\theta(g_{i+1}),\dots,\theta(g_k)$ all belong to $H_\Upsilon$. This proves that $\theta(gG_\Upsilon)=\theta(g)H_\Upsilon$, as claimed. 

In particular, the above proves that the map $gG_\Upsilon\mapsto \theta(g)H_\Upsilon$ is well-defined, and by construction it induces a cubical isomorphism from $\bD_G$ to $\bD_H$.
\end{proof}

\begin{Cor}\label{cor:rank}
    Every automorphism of $\bD_G$ preserves ranks of vertices.
\end{Cor}

\begin{proof}
    This follows from the case of right-angled Artin groups, established in \cite[Lemma~2.5]{HH-L1}, together with Lemma~\ref{lemma:building}.
\end{proof}

\subsubsection{Automorphisms of the extension graph and the right-angled building}

We equip $\Isom(\Gamma_G^e,\Gamma_H^e)$ and $\Isom(\bD_G,\bD_H)$ with the compact-open topology (equivalently, the topology of pointwise convergence on their vertex sets), and with the $(G\times H)$-action by left/right multiplication, i.e.\ for $f\in\Isom(\Gamma_G^e,\Gamma_H^e)$ and $(g,h)\in G\times H$, we let $(g,h)\cdot f=gfh^{-1}$. And similarly if $f\in\Isom(\bD_G,\bD_H)$. The main result of the present section is the following proposition -- this was established as \cite[Lemma~2.6]{HH-L1} for right-angled Artin groups, based on work of Huang \cite{Hua}.

\begin{Prop}\label{theo:autos}
   Let $G$ and $H$ be two graph products with countably infinite vertex groups over a common finite simple graph $\Gamma$ with no transvection and no partial conjugation.

   Then there exists a $(G\times H)$-equivariant continuous map from $\Isom(\Gamma_G^e,\Gamma_H^e)$ to $\Isom(\bD_G,\bD_H)$.
\end{Prop}

The main part of the proof is the definition of the map from $\Isom(\Gamma^e_G,\Gamma^e_H)$ to $\Isom(\bD_G,\bD_H)$. The crucial step for this is to extend Huang's \cite[Lemma~4.12]{Hua} to graph products. So let us start with $\beta\in \Isom(\Gamma^e_G,\Gamma^e_H)$ and let us define its image in $\Isom(\bD_G,\bD_H)$ denoted by $\Psi(\beta)$.
First, remark that the map sending a maximal standard clique coset $gG_\Upsilon$ to its stabilizer $gG_\Upsilon g^{-1}$, is a bijection between maximal standard clique cosets and maximal cliques in the extension graph of $G$. 
So, if $\beta(gG_\Upsilon g^{-1})=h H_\Lambda h^{-1}$ we define $\Psi(\beta)(gG_\Upsilon):=hH_{\Lambda}$. This only defines $\Psi(\beta)$ on maximal vertices. The goal is now to extend this definition to an element in $\Isom(\bD_G,\bD_H)$. To do so, let us define the image of some $g\in G$ viewed as a rank $0$ vertex of $\bD_G$. 

\begin{lemma}\label{lemma:intersection-cosets}
    Let $\Gamma$ be a finite simple graph with no transvection and no partial conjugation, not reduced to one vertex. Let $G$ and $H$ be graph products over $\Gamma$ with countably infinite vertex groups. Let $\beta\in\Isom(\Gamma_G^e,\Gamma_H^e)$. Let $g\in G$, and let $B_1,\dots,B_n$ be the maximal standard clique cosets containing $g$.

    Then $\Psi(\beta)(B_1)\cap\dots\cap\Psi(\beta)(B_n)$ is a singleton denoted by $\{\Psi(\beta)(g)\}$.

    Moreover the map $g\mapsto \Psi(\beta)(g)$ is a bijection from $G$ to $H$ that sends standard clique cosets to standard clique cosets of the same rank.
\end{lemma}

\begin{proof}
We start with the first assertion. In the case of right-angled Artin groups, this was proved by Huang  \cite[Lemmas~4.12 and~4.17]{Hua}, see in particular \cite[Equation~4.13]{Hua}. We explain how to derive the case of arbitrary graph products with countably infinite vertex groups, from the case of right-angled Artin groups.
In the following, we denote by $A$ the right-angled Artin group defined over $\Gamma$.

    For every $v\in V\Gamma$, we fix bijections $\theta_G^v:G_v\to\mathbb{Z}$ and $\theta_H^v:H_v\to\mathbb{Z}$. Through normal forms, this induces bijections $\theta_G:G\to A$ and $\theta_H:H\to A$ -- here $\theta_G$ is combined from $(\theta_G^v)_{v\in V\Gamma}$, and $\theta_H$ is combined from $(\theta_H^v)_{v\in V\Gamma}$, in the sense of the previous section.

    Let $\theta_G^e:\Gamma_G^e\to\Gamma_A^e$ and $\theta_H^e:\Gamma_H^e\to\Gamma_A^e$ be the induced bijections provided by Lemma~\ref{lemma:extension-graph}. Let $\beta_A\in\Aut(\Gamma_A^e)$ be defined by $\beta_A=\theta_H^e\beta(\theta_G^e)^{-1}$ (see also \cref{fig:diagrammepreuveDavisCplx}).
    
\begin{figure}[htbp]
    \centering
    \begin{tikzpicture}[scale=1.25]
        \node (gammaeG) at (0,0) {$\Gamma^e_G$};
        \node (gammaeH) at (2,0) {$\Gamma^e_H$};
        \draw[->,>=latex] (gammaeG)--(gammaeH) node[midway,above]{$\beta$};
        \node (gammaA1) at (0,-2) {$\Gamma^e_A$};
        \node (gammaA2) at (2,-2) {$\Gamma^e_A$};
        \draw[->,>=latex, dashed] (gammaA1)--(gammaA2)node[midway,above]{$\beta_A$};
        \draw[->,>=latex] (gammaeG)--(gammaA1) node[midway,left]{$\theta^e_G$};
        \draw[->,>=latex] (gammaeH)--(gammaA2) node[midway,right]{$\theta^e_H$};
    \end{tikzpicture}
    \caption{Definition of $\beta_A$}
    \label{fig:diagrammepreuveDavisCplx}
\end{figure}
Similarly as in the discussion above the lemma, we can define $\Psi_A(\beta_A)$ as a bijection of the set of maximal standard clique cosets of $A$. Using \cref{lemma:extension-graph,lemma:building} and the definitions of $\Psi(\beta)$ and $\Psi_A(\beta_A)$, we get

\begin{equation}\label{eq:PsiBeta}
    \Psi(\beta)=\theta^{-1}_H \circ \Psi_A(\beta_A) \circ \theta_G.
\end{equation}

Let $B=gG_\Upsilon$ be a maximal standard clique coset of $G$, where $g\in G$ and $\Upsilon\subseteq\Gamma$ is a maximal clique. By Lemma~\ref{lemma:building}, we have $\theta_G(B)=\theta_G(g)A_\Upsilon$.  

The maximal standard clique cosets $B_i$ containing $g$ coincide with the cosets $gG_{\Upsilon_i}$, with $\Upsilon_i$ ranging over the set of maximal cliques of $\Gamma$. But remark that 
\begin{equation*}
    \theta_G(B_1)\cap\dots\cap\theta_G(B_n)=\theta_G(g)A_{\Upsilon_1}\cap \cdots\cap \theta_G(g)A_{\Upsilon_n}=\theta_G(g)A_{\cap_i \Upsilon_i}.
\end{equation*}
Now, since the graph is transvection-free there is no vertex in $\Gamma$ linked to all other vertices of $\Gamma$, thus $\cap_i \Upsilon_i=\emptyset$. Therefore $\theta_G(B_1)\cap\dots\cap\theta_G(B_n)=\{\theta_G(g)\}$.

The cosets $\theta_G(B_i)$ are exactly the maximal standard clique cosets in $A$ containing $\theta_G(g)$. In the language of \cite{Hua}, these are called the \emph{maximal standard flats} containing $\theta_G(g)$. It then follows from Huang's aforementioned work \cite[Lemma~4.12, Equation~4.13 and Lemma~4.17]{Hua} that $\Psi_A(\beta_A)(\theta_G(B_1))\cap\dots\cap\Psi_A(\beta_A)(\theta_G(B_n))$ is a singleton, denoted by $\{\Psi_A(\beta_A)\circ\theta_G(g)\}$. 

Applying $\theta^{-1}_H$ shows that $\theta_H^{-1}\Psi_A(\beta_A)\theta_G(B_1)\cap\dots\cap\theta_H^{-1}\Psi_A(\beta_A)\theta_G(B_n)$ is a singleton. This by construction (see \cref{eq:PsiBeta}) is exactly $\Psi(\beta)(B_1)\cap\dots\cap\Psi(\beta)(B_n)$, which proves the first assertion. Notice that, by construction, the map $g\mapsto \Psi(\beta)(g)$ still verifies Equation~\eqref{eq:PsiBeta}.

Let us now show that $g\mapsto \Psi(\beta)(g)$ is a bijection from $G$ to $H$ that sends standard clique cosets to standard clique cosets.

To see that it is a bijection, remark that one can define in an analogous way a map $\Psi^{\prime}(\beta^{-1}):H\rightarrow G$ and note that it is the inverse of $\Psi(\beta)$.

Using the construction made in \cite{HH-L1}, above Lemma 2.6, we obtain that $\Psi_{A}(\beta_A)$ sends standard clique cosets to standard clique cosets. Moreover, by induction on the rank, using the bijectivity of $\Psi_A(\beta_A)$ (see again \cite{HH-L1}) and what precedes, we can show that $\Psi_A(\beta_A)$ preserves the rank. Using \cref{eq:PsiBeta} defining $\Psi(\beta)$, we obtain the wanted assertion.
\end{proof}

\begin{proof}[Proof of Proposition~\ref{theo:autos}]
    Let us start with the definition of our $(G\times H)$-equivariant map 
    \[\Psi:\Isom(\Gamma_G^e,\Gamma_H^e)\to\Isom(\bD_G,\bD_H).\] 
    Let $\beta\in\Isom(\Gamma_G^e,\Gamma_H^e)$. Let $g\in G$. Let $B_1,\dots,B_n$ be the maximal standard clique cosets that contain $g$. By Lemma~\ref{lemma:intersection-cosets}, the cosets $\Psi(\beta)(B_1)\cap\dots\cap\Psi(\beta)(B_n)$ intersect in a single point $\Psi(\beta)(g)$, and the map $g\mapsto\Psi(\beta)(g)$ is a bijection that sends every standard clique coset to a standard clique coset. 
    So $\Psi(\beta)$ can be viewed as a map in $\Isom(\bD_G,\bD_H)$. The assignment $\beta\mapsto \Psi(\beta)$ yields the desired $(G\times H)$-equivariant map \[\Psi:\Isom(\Gamma_G^e,\Gamma_H^e)\to\Isom(\bD_G,\bD_H).\]
    The continuity of $\Psi$ is now a consequence of the following observation: given any two rank $0$ vertices $g\in V\bD_G$ and $h\in V\bD_H$, one has $\Psi(\beta)(g)=h$ if and only if, denoting by $\Upsilon_1,\dots,\Upsilon_k$ the maximal cliques of $\Gamma$, one has $\beta(\{gG_{\Upsilon_i}g^{-1}\})=\{h H_{\Upsilon_i}h^{-1}\}$. And any vertex of rank at least $1$ is determined by finitely many adjacent rank $0$ vertices.
\end{proof}

\begin{Rq}
It is in fact possible to show that the map $\Psi$ constructed in the above proof is a homeomorphism, see \cite[Section~2.3]{HH-L1} where this is explained in the case of right-angled Artin groups. We will however never need this fact in the sequel.
\end{Rq}

\subsection{Amenable actions on boundaries of trees}

\subsubsection{Review on Borel amenability of group actions}\label{sec:amenability}
We now review some basic facts on Borel amenability of group actions. 
For general background on the topic we refer to \cite{ADR,Ozawa-ICM}. 

Let $G$ be a countable group, acting by Borel automorphisms on a standard Borel set~$\Delta$. The $G$-action on~$\Delta$ is \emph{Borel amenable}\index{Borel!Borel amenable action} if there exists a sequence of measurable maps $\mu_n:\Delta\to\Prob(G)$ such that for every $\delta\in\Delta$ and every $g\in G$, one has \[\lim_{n\to +\infty}||\mu_n(g\delta)-g\mu_n(\delta)||_1=0.\] 
Here $\Prob(G)$ is the space of probability measures on $G$, equipped with the topology of pointwise convergence. We say that the sequence of maps $\mu_n$ is \emph{asymptotically equivariant}\index{Asymptotically equivariant}.

For a first example of an amenable action of a non-amenable group, we also refer to \cite[Example 2.2]{Ozawa-ICM} for the case of an action of $F_2$ on the boundary of a tree. We will use a variation over this argument in our proof of Lemma~\ref{lemma:decomposition-boundary} below.

We mention a few elementary facts about Borel amenability of group actions that we will use in the sequel.
\begin{itemize}
    \item A countable group $G$ is amenable if and only if its action on a point is Borel amenable.
    \item If $G_1,\dots,G_k$ are countable groups with Borel amenable actions on standard Borel sets $\Delta_1,\dots,\Delta_k$, then the product action of $G_1\times\dots\times G_k$ on $\Delta_1\times\dots\times\Delta_k$ is Borel amenable.
    \item If the $G$-action on $X$ is Borel amenable, then —~letting $\calp_{\le 2}(X)$ be the standard Borel set consisting of all non-empty subsets of $X$ of cardinality at most $2$~— the $G$-action on $\calp_{\le 2}(X)$ is Borel amenable.
\end{itemize}

\subsubsection{Action on the boundary of the tree $T_v$}

Recall that given some $v\in V\Gamma$, we denote by $T_v$ the Bass-Serre tree associated to the splitting $G=G_{\st(v)}*_{G_{\lk(v)}} G_{\Gamma\backslash\{v\}}$ (see \cref{sec:trees}).

Given a vertex $v\in V\Gamma$, the $G$-action on the tree $T_v$ extends to a $G$-action by homeomorphisms on $\partial_\infty T_v$ (equipped with its visual topology). Let $\mathbb{P}_G$ be the (countable) set of all parabolic subgroups of $G$. We will now build a $G$-equivariant Borel map $\theta_v:\partial_\infty T_v\to\mathbb{P}_G$.

Given $\xi\in\partial_\infty T_v$, we define the \emph{elliptic stabilizer} of $\xi$ as the subgroup of $\Stab_G(\xi)$ consisting of all elements that act elliptically on $T_v$ (i.e.\ with a fixed point).

\begin{lemma}\label{lemma:elliptic-stabilizer}
    Let $v\in V\Gamma$ be such that $\Gamma\neq\st(v)$. For every $\xi\in \partial_\infty T_v$, the elliptic stabilizer of $\xi$ is a proper parabolic subgroup of $G$.
\end{lemma}

\begin{proof}
  Let $x\in T_v$ be a vertex, and let $\xi\in\partial_\infty T_v$. For every $n\in\mathbb{N}$, let $\gamma_n(x,\xi)$ be the segment in $T_v$ consisting of the union of the first $n$ edges of the ray from $x$ to $\xi$. Let $G_n(x,\xi)$ be the pointwise stabilizer of $\gamma_n(x,\xi)$. Notice that edge stabilizers of $T_v$ are proper parabolic subgroups of $G$. Therefore, for every $n\in\mathbb{N}$, the subgroup $G_n(x,\xi)$ is a proper parabolic subgroup of $G$, being an intersection of proper parabolic subgroups.

  As $n$ increases, the subgroups $G_n(x,\xi)$ form a non-increasing chain of proper parabolic subgroups, which is therefore stationary. We denote by $G_\infty(x,\xi)$ the ultimate value of this sequence.

  Now, if $x_1,x_2,\dots$ form a ray towards $\xi$, then as $n$ increases, the subgroups $G_\infty(x_n,\xi)$ form a non-decreasing chain of proper parabolic subgroups of $G$. Therefore it is stationary, and we denote by $\theta_v(\xi)$ its ultimate value. Notice that the value of $\theta_v(\xi)$ does not depend on the choice of a ray towards $\xi$, because any two such rays eventually coincide.

  Finally, note that being in the elliptic stabilizer is equivalent to fixing pointwise a ray towards $\xi$. Therefore $\theta_v(\xi)$ is indeed the elliptic stabilizer of $\xi$.
\end{proof}

We now let $\theta_v:\partial_\infty T_v\to\mathbb{P}_G$ be the map sending $\xi$ to its elliptic stabilizer. Notice that this map is Borel and $G$-equivariant. We let $(\partial_\infty T_v)^{\reg}:=\theta_v^{-1}(\{1\})$.\label{Def:TvReg}

\begin{lemma}\label{lemma:reg-nonempty}
    Let $v\in V\Gamma$. Then $(\partial_\infty T_v)^{\reg}\neq\emptyset$ if and only if $\Gamma$ is irreducible. 
\end{lemma}

\begin{proof}
    First assume that $\Gamma$ is reducible and decompose it as $\Gamma=\Gamma_1\circ \Gamma_2$. Assuming without loss of generality that $v\in V\Gamma_1$, the parabolic subgroup $G_{\Gamma_2}$ fixes all the vertices of $T_v$ and therefore is contained in the elliptic stabilizer of any $\xi\in \partial_\infty T_v$. 

    We now assume that $\Gamma$ is irreducible.
    \begin{itemize}
    \item Let us first show that there exists $g\in G$ that is not contained in any proper parabolic subgroup of $G$.\\
    Let $\{w_1,\ldots,w_m\}=V\Gamma$ be an enumeration of $V\Gamma$. For all $i\in \{1,\ldots,m\}$ take some $g_i\in G_{w_i}\setminus\{1\}$ and consider $g=g_{w_1}\cdots g_{w_m}$. Using a normal form argument, we now obtain that $g$ is not contained in any proper parabolic subgroup of $G$.
    \item Let $g$ be as above, then $g$ acts loxodromically on $T_v$ and we let $\xi$ be the attracting fixed point of $g$ in $\partial_\infty T_v$. Let $P=\theta_v(\xi)$. Then $P$ fixes a subray $r$ of the axis of $g$. As in the proof of Lemma~\ref{lemma:elliptic-stabilizer}, we can find a compact segment $I\subseteq r$ whose stabilizer is equal to $P$. Then the stabilizer of $gI$ is equal to $gPg^{-1}$, so $P\subseteq gPg^{-1}$. It follows from Proposition~\ref{prop:am} that $gPg^{-1}=P$, i.e.\ $g\in P\times P^{-1}$. Since $\Gamma$ is irreducible and $g$ is not contained in any proper parabolic subgroup of $G$, we deduce that $P=\{1\}$. In particular $\xi\in (\partial_\infty T_v)^{\reg}$ and the lemma is proved. \qedhere
    \end{itemize}
\end{proof}

\begin{lemma}\label{lemma:decomposition-boundary}
  For every $v\in V\Gamma$, the $G$-action on $(\partial_\infty T_v)^{\reg}$ is Borel amenable.
\end{lemma}

\begin{proof}
Let $\mathbb{S}$ be the countable set of all oriented segments of $T_v$ which start and end at vertices and have trivial $G$-stabilizer. The $G$-action on $T_v$ induces a $G$-action on $\mathbb{S}$, and stabilizers for the latter action are trivial. Therefore, it is enough to construct an asymptotically equivariant sequence of Borel maps $\mu_n:(\partial_\infty T_v)^{\reg}\to\Prob(\mathbb{S})$ (see e.g.\ \cite[Proposition~11]{Oza} or \cite[Proposition~2.12]{BGH}).

Given any vertex $x\in VT_v$ and any $\xi\in (\partial_\infty T_v)^{\reg}$, we denote by $\alpha(x,\xi)\in\mathbb{S}$ the smallest (oriented) segment of the ray $[x,\xi)$ originating at $x$, ending at a vertex, and with trivial stabilizer. This exists, and it is a finite subsegment, by definition of $(\partial_\infty T_v)^{\reg}$. Notice that, for a fixed $x$, the map $\alpha(x,\cdot)$ is continuous. 

We now fix a basepoint $o\in VT_v$. Given $\xi\in (\partial_\infty T_v)^\reg$ let $\{o=v_1(\xi),v_2(\xi),\cdots\}$ be the set of vertices encountered along the ray from $o$ to $\xi$. Given $n\in\mathbb{N}$ and $\xi\in (\partial_\infty T_v)^{\reg}$, we let \[\mu_n(\xi)=\frac{1}{n}\sum_{i=1}^n\delta_{\alpha(v_i(\xi),\xi)},\] a weighted sum of Dirac masses, so $\mu_n(\xi)\in\Prob(\mathbb{S})$, and $\mu_n$ depends measurably on $\xi$. Now remark that $\mu_n(g\xi)={1}/{n}\sum_{i=1}^n\delta_{\alpha(v_i(g\xi),g\xi)}$ where the $\alpha\big(v_i(g\xi),g\xi\big)$ are subintervals of the ray $[o,g\xi)$. Moreover $g\mu_n(\xi)={1}/{n}\sum_{i=1}^n\delta_{g\alpha(v_i(\xi),\xi)}$ where the $g\alpha\big(v_i(\xi),\xi\big)$ are subintervals of the ray $[go,g\xi)$. Since any two rays towards $g\xi$ eventually coincide, this implies that $\mu_n$ is asymptotically equivariant (see also \cref{fig:MoyennabiliteBord}).
\end{proof}
\begin{figure}
    \centering
    \includegraphics[width=\textwidth]{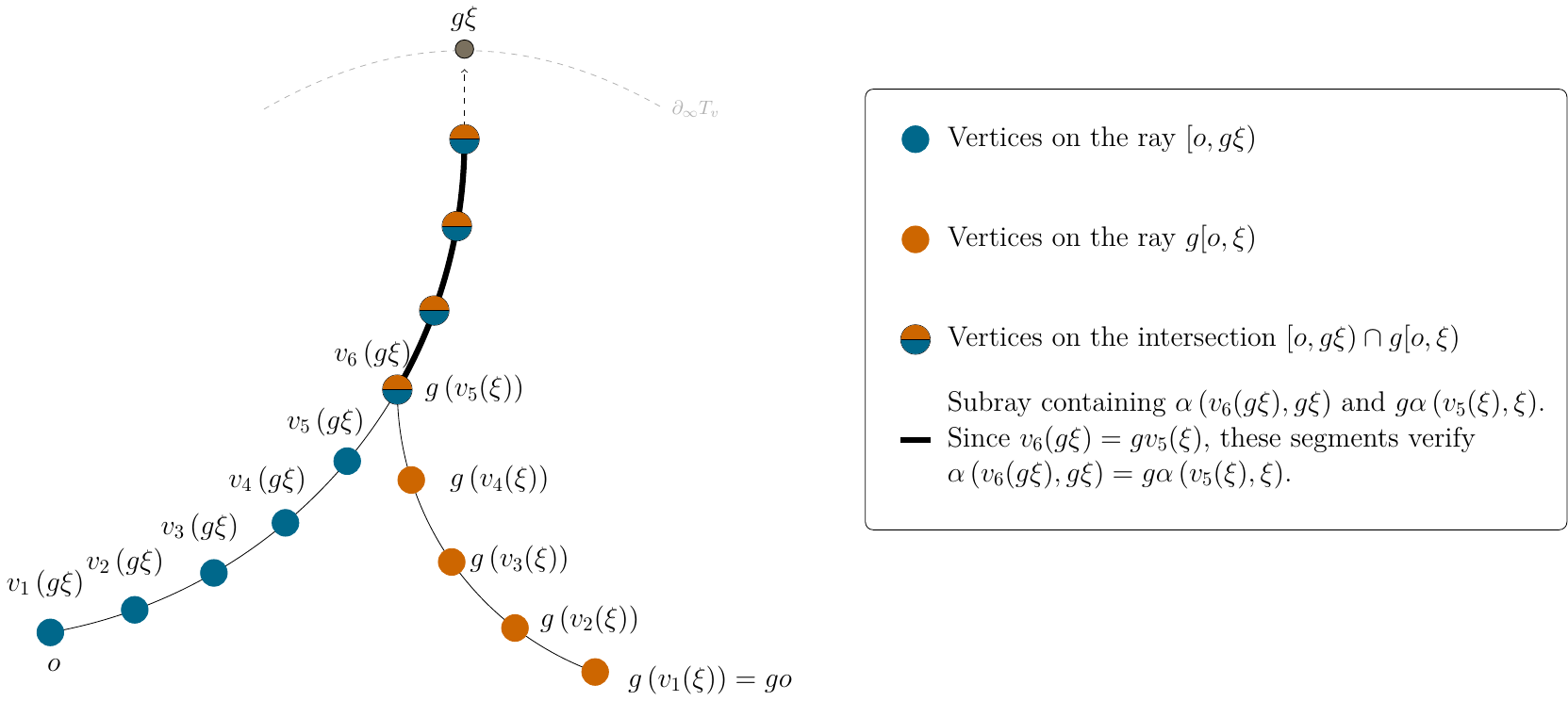}
    \caption{Illustration of the proof of \cref{lemma:decomposition-boundary}}
    \label{fig:MoyennabiliteBord}
\end{figure}

\section{Measure and orbit equivalence}
\label{sec:ME}
This section contains general facts about measure equivalence and orbit equivalence. After reviewing the basic definitions in Section~\ref{sec:me-background}, we compare measure and orbit equivalence with their discrete counterparts, namely commensurability and strong commensurability (Section~\ref{sec:commensurability}) -- this will be useful in Part~\ref{Part:MEClassification} to derive commensurability classification theorems for graph products, from their versions in measure equivalence. Finally, in Section~\ref{lemma:me-subgroups}, we establish a lemma regarding how to restrict measure equivalence couplings to subgroups, which will be useful in Part~\ref{Part:QuantitativeResults} of our work.   

\subsection{Background and definitions}\label{sec:me-background}

We recall here some general material about measure equivalence and refer to \cite{Fur-survey,Gab-survey} for general surveys. We also refer to Furman's original article \cite{Fur} for more information on the dictionary between measure equivalence and stable orbit equivalence, and the associated cocycles.

\subsubsection{Measure equivalence and (stable) orbit equivalence}\label{sec:me-background-1}

A \emph{standard measure space}\index{Standard!Standard measure space} $(\Omega,m)$ is a measurable space whose $\sigma$-algebra $\mathcal{B}(\Omega)$ consists of the Borel subsets coming from some Polish (\emph{viz.} separable and completely metrizable) topology on $\Omega$ and equipped with a non-zero $\sigma$-finite measure $m$ on $\mathcal{B}(\Omega)$. One says that an action of a countable group $G$ on a standard measure space $(\Omega,m)$ is \emph{measure-preserving}\index{Measure preserving action} if the action map $(g,x)\mapsto g*x$ is Borel and if $\mu(g*E)=\mu(E)$ for all $g\in G$ and all Borel subsets $E\subseteq \Omega$. In this setting, a \emph{fundamental domain}\index{Fundamental!Fundamental domain} for the action of $G$ on $(\Omega,m)$ is a Borel subset $X_G\subseteq \Omega$ which intersects almost every $G$-orbit at exactly one point.

Recall from the introduction that two countable groups $G$ and $H$ are \emph{measure equivalent}\index{Measure equivalence} if there exists a standard measure space $(\Omega,m)$, equipped with a measure-preserving action of $G\times H$, such that the actions of $G$ and $H$ on $\Omega$ are both free and have finite measure fundamental domains, denoted respectively by $X_G$ and $X_H$. The space $(\Omega,m)$ (sometimes simply denoted $\Omega$) is called a \emph{measure equivalence coupling}\index{Measure equivalence!Measure equivalence coupling}\index{Coupling!Measure equivalence coupling} between $G$ and $H$. If we need to keep track of the fundamental domains, we will sometimes call “coupling” from $G$ to $H$ the tuple $(\Omega,m,X_G,X_H)$. 
When additionally one can take $X_G=X_H$ in the above definition, the groups $G$ and $H$ are said to be \emph{orbit equivalent}\index{Orbit equivalent}.

In general, the \emph{compression constant}\index{Compression constant}\label{Def:CompressionConstant} (from $G$ to $H$) of the measure equivalence coupling $\Omega$ is defined as \[[G:H]_\Omega:=\frac{m(X_H)}{m(X_G)}\] 
(this does not depend on the choice of fundamental domains). The compression constant of an orbit equivalence coupling is always equal to $1$.

By a theorem of Furman \cite{Fur} and Gaboriau \cite[$\mathbf{P}_{\mathrm{ME}}\mathbf{5}$]{Gab}, two countable groups are orbit equivalent if and only if they admit free, measure-preserving actions on a standard probability space $(X,\mu)$ with the same orbits (on a conull subset). We will call such an $(X,\mu)$ an \emph{orbit equivalence pairing}\footnote{we use the word \emph{pairing} instead of \emph{coupling} to distinguish it from measure equivalence couplings where the two fundamental domains coincide}\index{Orbit equivalence!Orbit equivalence pairing}.

Likewise, measure equivalence is characterized by having \emph{stably orbit equivalent}\index{Stable!Stable orbit equivalence}\index{Orbit equivalence!Stable orbit equivalence} free, measure-preserving actions on standard probability spaces, in the following sense. Given two free, measure-preserving actions $G\actson (X,\mu)$ and $H\actson (Y,\nu)$ on standard probability spaces, a \emph{stable orbit equivalence} between the two actions is a measure-scaling isomorphism $f:U\to V$ between positive measure subsets $U\subseteq X$ and $V\subseteq Y$ (i.e.\ $f$ induces a measure space isomorphism after renormalizing $U$ and $V$ to probability spaces) such that for almost every $x\in U$, one has $f((G\cdot x)\cap U)=(H\cdot f(x))\cap V$. If a map $f$ with these properties exists, the two actions are said to be stably orbit equivalent. The \emph{compression constant}\index{Compression constant} of $f$ is defined as\label{Def:CompressionConstantf} \[\kappa(f):=\frac{\nu(V)}{\mu(U)}.\] 
In fact, two countable groups admit a measure equivalence coupling with compression constant $\kappa$ (from $G$ to $H$) if and only if there is a stable orbit equivalence from a free, measure-preserving $G$-action on a standard probability space $X$, to a measure-preserving $H$-action on a standard probability space $Y$, with compression constant $\kappa$, see \cite[Lemma~3.2 and Theorem~3.3]{Fur}. 

\subsubsection{Cocycles}\label{sec:cocycles}

We now review the notions of measure equivalence and (stable) orbit equivalence cocycles. In the sequel, we will also work with a notion of cocycle in the groupoid theoretic framework -- we refer to \cref{Def:CocycleGroupoid}, \cpageref{Def:CocycleGroupoid} for more details on the groupoid viewpoint.

\paragraph{Measure equivalence cocycles.}
\label{Nota:ActionDomFond} 
A measure equivalence coupling $(\Omega,m,X_G,X_H)$ comes with two induced actions: one of $G$ on the fundamental domain $X_H$ and one of $H$ on $X_G$. We will denote both induced actions by “$\cdot$” to distinguish them from the action of the groups on the whole space $\Omega$, denoted by “$*$”. For all $g\in G$ and a.e. $x\in X_H$, the element $g\cdot x$ is defined as the (unique) element of $X_H$ that belongs to the $H$-orbit of $g*x$, namely $\left\{g\cdot x\right\}=X_H\cap\big(H*(g*x)\big)$ (see \cref{fig:Orbites}, \cpageref{fig:Orbites}). In the following we will (slightly) abuse notations and write $H*g*x$ instead of $H*(g*x)$. Similarly for all $h\in H$ and a.e. $x^\prime \in X_G$ we have $\{h\cdot x^\prime\}=X_G\cap \big(G*h*x^\prime\big)$. 

We thus define $c:G\times X_H\rightarrow H$ and $c^\prime:G\times X_H \rightarrow H$ the \emph{measure equivalence cocycles}\index{Cocycle!Measure equivalence cocycle}, by (see also \cref{fig:Orbites})
\begin{align*}
    \forall g\in G, \ \text{a.e.}\ x\in X_H \qquad &c(g,x)*g*x=g\cdot x, \\   
    \forall h\in H, \ \text{a.e.}\ x\in X_G\ \qquad &c'(h,x)*h*x=h\cdot x.
\end{align*}
They verify the \emph{cocycle relation}, namely 
\begin{equation*}
    \forall g,g^\prime \in G,\ a.e.\ x\in X \qquad
    c(gg^\prime,x)=c(g,g^\prime\cdot x)c(g^\prime,x).
\end{equation*}

Two cocycles $c:G\times X_H\rightarrow H$ and $c^\prime:G\times X_H \rightarrow H$ are said to be \emph{cohomologous}\index{Cohomologous cocycles}\index{Cocycle!Cohomologous cocycles}\label{Def:CohomCocyclesME} if there exists a measurable map $\alpha:X_H\rightarrow H$ such that for all $g\in G$ and a.e.\ $x\in X_H$ one has
\begin{equation*}
    c^\prime(g,x)=\alpha(g \cdot x)c(g,x)\alpha(x)^{-1}.
\end{equation*}

\begin{Rq}\label{Rq:CohomologousCocyclesFundamentalDomains}
    Identifying $X_H$ with $\Omega/H$ allows to see the measure equivalence cocycle $c$ as defined over $G\times \Omega/H$. With this identification, changing the fundamental domain $X_H$ amounts to considering cohomologous cocycles. In the following, given two fundamental domains $X_H$ and $X^\prime_H$ and their associated respective cocycles $c:G\times X_H\rightarrow H$ and $c^\prime:G\times X^\prime_H\rightarrow H$, we will consider $c$ and $c^\prime$ to be cohomologous if they are cohomologous via the natural identifications $X_H\simeq \Omega/H \simeq X^\prime_H$.
\end{Rq}

\paragraph{(Stable) orbit equivalence cocycles.}\label{Par:SOE}

Assume now that $G$ and $H$ are orbit equivalent over some standard probability space $(X,\mu)$. One can define two \emph{orbit equivalence cocycles}\index{Cocycle!Orbit equivalence cocycle}\index{Orbit equivalence!Cocycle} denoted $c:G\times X\to H$ and $c':H\times X\rightarrow G$ by letting $c(g,x)$ (resp.\ $c'(h,x)$) be the unique element of $H$ (resp.\ $G$) such that $c(g,x)\cdot x=g\cdot x$ and $c^\prime(h,x) \cdot x=h\cdot x$ (this is well-defined for a.e.\ $x\in X$). When $\Omega$ is an orbit equivalence coupling between $G$ and $H$, with a common fundamental domain $X=X_G=X_H$, the associated measure equivalence cocycles $c:G\times X\to H$ and $c':H\times X\to G$ are nothing but the orbit equivalence cocycles associated to the actions $G\actson X$ and $H\actson X$.

More generally, if $G\actson (X,\mu)$ and $H\actson (Y,\nu)$ are two free, ergodic, measure-preserving actions on standard probability spaces, and if $f:U\to V$ (with $U\subseteq X$ and $V\subseteq Y$ Borel subsets of positive measure) is a stable orbit equivalence, we say that a cocycle $c:G\times X\to H$ is a \emph{stable orbit equivalence (SOE) cocycle}\index{Stable!Stable orbit equivalence cocycle}\index{Cocycle!SOE cocycle} associated to $f$ if there exists a Borel map $p:X\to U$ such that 
\begin{itemize}
    \item for almost every $x\in X$, one has $p(x)\in G\cdot x$, and
    \item for every $g\in G$ and almost every $x\in X$, one has $f\circ p (g\cdot x)=c(g,x)\cdot (f\circ p(x))$.
\end{itemize}
Notice that $p$ can always be chosen so that $p(x)=x$ whenever $x\in U$. Changing $p$ to another projection map yields a cohomologous cocycle.

We also observe that if $f:U\to V$ is a stable orbit equivalence between $G\actson (X,\mu)$ and $H\actson (Y,\nu)$, then for every positive measure Borel subset $W\subseteq U$, the map $f_{|W}:W\to f(W)$ is again a stable orbit equivalence between $G\actson (X,\mu)$ and $H\actson (Y,\nu)$, with the same compression constant. In addition, any SOE cocycle for $f_{|W}$ is also an SOE cocycle for $f$, and any SOE cocycle for $f$ is cohomologous to an SOE cocycle for~$f_{|W}$. 

We also observe that if $P_G\subseteq G$ and $P_H\subseteq H$ are subgroups, such that
\begin{itemize}
    \item for almost every $x\in U$, one has $f((P_G\cdot x)\cap U)=(P_H\cdot f(x))\cap V$,
    \item $P_G$ acts ergodically on $X$,
\end{itemize}
then the map $p$ can be chosen so that $p(x)\in P_G\cdot x$ for almost every $x\in X$ (by ergodicity). Moreover $f$ is a stable orbit equivalence between the actions $P_G\actson X$ and $P_H\actson Y$. Any SOE cocycle $c:G\times X\to H$ for $f$ is $H$-cohomologous to a cocycle $c'$ such that $c'(P_G\times X)\subseteq P_H$, and $c'$ is an SOE cocycle for $f$, viewed as a stable orbit equivalence between $P_G\actson X$ and $P_H\actson Y$. 

\subsubsection{Quantitative versions}\label{Sec:QuantitativeDefinitions}

Recall from the introduction (Definition~\ref{Def:QuantitativeME}) that given a measure equivalence coupling $(\Omega,m,X_G,X_H)$ (equipped with fundamental domains $X_G,X_H$) between finitely generated groups $G,H$, and two non-decreasing functions $\varphi$ and $\psi$, one says that the coupling is \emph{$(\varphi,\psi)$-integrable}\index{phi integrable@$(\varphi,\psi)$-integrable} {(from $G$ to $H$)} if
for all $g\in G$ and all $h\in H$, there exist $\kappa_g,\kappa_h>0$ such that
\begin{equation}\label{eq:integrability}
\int_{X_H}\varphi\left(\frac{1}{\kappa_g}|c(g,x)|_{S_H}\right)dm(x)<+\infty \ \text{and} \
\int_{X_G}\psi\left(\frac{1}{\kappa_h}|c'(h,x)|_{S_G}\right)dm(x)<+\infty.
\end{equation}
where $c:G\times X_H\to H$ and $c':H\times X_G\to G$ denote the measure equivalence cocycles, and $|\cdot|_G,|\cdot|_H$ denote word lengths on $G,H$ associated to finite generating sets (the integrability does not depend on the choice of generating sets). To prove that a coupling is $(\varphi,\psi)$-integrable, it is in fact enough to check Equation~\ref{eq:integrability} with $g$ and $h$ varying among generating sets of $G$ and $H$, see \cite[Proposition~2.22 and Lemma~2.18]{DKLMT}.

When $\varphi(x)=x^p$ for some $p\in [0,+\infty)$ we will talk about $(L^p,\psi)$-integrable couplings instead of $(\varphi,\psi)$-integrable couplings, to stay consistent with the literature. In particular $L^0$ means that we are not imposing any integrability condition on the corresponding cocycle. 

Likewise, an orbit equivalence pairing $(X,\mu)$ between two finitely generated groups $G$ and $H$ (i.e.\ a standard probability space on which the two groups have the same orbits) is \emph{$(\varphi,\psi)$-integrable} from $G$ to $H$ if the associated orbit equivalence cocycles $c:G\times X\to H$ and $c':H\times X\to G$ satisfy the integrability conditions~\eqref{eq:integrability} (with $X=X_G=X_H$).  

Using the dictionary between measure equivalence couplings and orbit equivalence pairings, the existence of a $(\varphi,\psi)$-integrable orbit equivalence pairing from $G$ to $H$ is equivalent to the existence of a $(\varphi,\psi)$-integrable measure equivalence coupling of the form $(\Omega,m,X_G,X_H)$ with $X_G=X_H$.

\subsection{Link with commensurability}\label{sec:commensurability} 

Two groups are \emph{commensurable}\index{Commensurable} if they have isomorphic finite index subgroups. Following \cite{JS}, we say that they are \emph{strongly commensurable}\index{Commensurable!Strongly commensurable}\index{Strongly!Strongly commensurable} if they admit isomorphic finite-index subgroups of the same index. The following lemma shows that commensurability and strong commensurability are the respective analogues of measure equivalence and orbit equivalence, when one restricts to discrete (i.e.\ countable) couplings.

\begin{Lmm}\label{lemma:commensurability}
    Two countable groups $G_1$ and $G_2$ are 
    \begin{enumerate}
        \item commensurable if and only if there exists a non-empty countable set $\Omega$ equipped with an action of $G_1\times G_2$ such that for every $i\in\{1,2\}$, the action of $G_i$ on $\Omega$ is free and cofinite.
        \item strongly commensurable if and only if there exists a non-empty countable set $\Omega$ equipped with an action of $G_1\times G_2$ such that for every $i\in\{1,2\}$, the action of $G_i$ on $\Omega$ is free, and the actions of $G_1$ and $G_2$ on $\Omega$ have a common finite fundamental domain.
    \end{enumerate}
\end{Lmm}

For comparison, notice also that two countable groups $G_1$ and $G_2$ are isomorphic if and only if there exists a non-empty countable set $\Omega$ equipped with an action of $G_1\times G_2$, such that for every $i\in\{1,2\}$, the action of $G_i$ on $\Omega$ is free and transitive.
Indeed, the direct implication is clear by considering the left-right action of the group on itself. Conversely, identifying $\Omega$ to $G_1$ and $G_2$ via the orbit maps, gives the needed isomorphism.

\begin{proof}
    Assume first that $G_1$ and $G_2$ are commensurable. Let $H_1\subseteq G_1$ and $H_2\subseteq G_2$ be isomorphic finite-index subgroups, and let $\theta:H_1\to H_2$ be an isomorphism. The group $H_1\times H_2$ acts on $\tilde{\Omega}=G_1\times H_1\times G_2$ via
    \begin{equation*}
    (h_1,h_2)\cdot (g_1,h,g_2)=\left(g_1h_1^{-1},h_1h(\theta^{-1}(h_2))^{-1},h_2g_2\right).
    \end{equation*}
    And the action of $G_1\times G_2$ on $\tilde{\Omega}$ given by $(g'_1,g'_2)\cdot (g_1,h,g_2)=(g'_1g_1,h,g_2(g'_2)^{-1})$ commutes with that of $H_1\times H_2$, hence descends to an action on the quotient $\Omega=(H_1\times H_2)\backslash \tilde{\Omega}$. One checks that the actions of $G_1$ and $G_2$ on $\Omega$ are both free. We will denote by $[g_1,h,g_2]$ the class of $(g_1,h,g_2)$ in $\Omega$. Let $X_1\subseteq G_1$ be a set of representatives of the left cosets of $H_1$ in $G_1$, and $X_2\subseteq G_2$ be a set of representatives of the right cosets of $H_2$ in $G_2$. Then $\{[e,e,g_2]\mid g_2\in X_2\}$ is a finite fundamental domain for the $G_1$-action on $\Omega$, and $\{[g_1,e,e]\mid g_1\in X_1\}$ is a finite fundamental domain for the $G_2$-action on $\Omega$.

    In the above, if we further assume that $G_1$ and $G_2$ are strongly commensurable, then $|X_1|=|X_2|$. Let $\alpha:X_1\to X_2$ be a bijection. Then $\{[g_1,e,\alpha(g_1)]\mid g_1\in X_1\}$ is a common fundamental domain for the actions of $G_1$ and $G_2$ on $\Omega$.

    Conversely, let us now assume that there exists a non-empty countable set $\Omega$ equipped with an action of $G_1\times G_2$ such that for every $i\in\{1,2\}$, the action of $G_i$ on $\Omega$ is free and admits a finite fundamental domain $Y_i$. 
    Let $y_2\in Y_2$, and let $H_1\subseteq G_1$ be the stabilizer of $y_2$ for the induced action of $G_1$ on $Y_2\approx G_2\backslash\Omega$. Then $H_1$ is a finite-index subgroup of $G_1$ which acts freely with finitely many orbits on $G_2y_2$. Using the fact that the actions of $G_1$ and $G_2$ commute, one checks that $\{g_2\in G_2\mid g_2y_2\in H_1y_2\}$ is a subgroup $H_2$ of $G_2$, and that its (finite) index is equal to the number of $H_1$-orbits on $G_2y_2$. Now $H_1$ and $H_2$ have free and transitive actions on $H_1y_2$, from which it follows that they are isomorphic. 

    To get the characterization of strong commensurability, in the above, assume in addition that $Y_1=Y_2$. Up to changing $\Omega$ to a $(G_1\times G_2)$-invariant subset (which does not affect the fact that the two actions have a common fundamental domain), we can (and will) assume that the induced actions of $G_1$ on $G_2\backslash\Omega$ and of $G_2$ on $G_1\backslash\Omega$ are transitive. Then in the above, we have $[G_1:H_1]=|Y_2|$. And $G_1$ acts on $\Omega$ with $|Y_1|$ orbits, and each of these orbits meets $G_2y_2$. So $H_1=\mathrm{Stab}_{G_1}(y_2)$ acts on $G_2y_2$ with $|Y_1|$ orbits, which by the above implies that $[G_2:H_2]=|Y_1|$. So $[G_1:H_1]=[G_2:H_2]$, which proves that $G_1$ and $G_2$ are strongly commensurable.  
\end{proof}

\subsection{Restricting couplings to subgroups}

\begin{lemma}\label{lemma:me-subgroups} Let $G$ and $H$ be two countable groups and $M\subseteq G$ and $N\subseteq H$ be subgroups. 
    Let $(\Omega,m,X_G,X_H)$ be a measure equivalence coupling from $G$ to $H$ and denote by $c_1:G\times X_H\to H$ and $c_2:H\times X_G\to G$ the coresponding measure equivalence cocycles. Assume that there exists a positive measure Borel subset $U\subseteq X_G\cap X_H$ such that 
    \begin{itemize}
    \item for every $g\in M$ and a.e.\ $x\in U$, if $gx\in U$, then $c_1(g,x)\in N$;
    \item for every $h\in N$ and a.e.\ $x\in U$, if $hx\in U$, then $c_2(h,x)\in M$.
    \end{itemize}

    Then there exist an $(M\times N)$-invariant Borel subspace $\Omega'\subseteq\Omega$, and Borel fundamental domains $X_M,X_N$ for the actions of $M,N$ on $\Omega'$, such that
    \begin{itemize}
        \item $(\Omega',m,X_M,X_N)$ is a measure equivalence coupling from $M$ to $N$, and
        \item $X_M,X_N$ are contained in fundamental domains for the actions of $G,H$ on $\Omega$.
    \end{itemize}
\end{lemma}

We refer to \cref{Rq:CouplageRestriction} for the groupoid theoretic translation of the above lemma.

\begin{proof}
   We denote by $c_1:G\times X_H\to H$ and $c_2:H\times X_G\to G$ the associated measure equivalence cocycles. Our assumption on $U$ ensures that for a.e.\ $x,y\in U$, if $g\in M$ sends $x$ to $y$, then $c_1(g,x)$ belongs to $N$. 
  Remark that $M\cdot U \subseteq X_H$ and $N\cdot U \subseteq X_G$.
  Let $\kappa:M\cdot U\to M$ be a measurable map such that $\kappa(x)\cdot x\in U$ for all $x\in M\cdot U$,
  and extend it to $X_H$ by letting $\kappa(x)=e_{M}$ for all $x\in X_H\backslash M \cdot U$. Similarly, define $\lambda:N\cdot U\to N$ a measurable map such that $\lambda(x)\cdot x\in U$ for every $x\in N\cdot U$,
  and extend it to $X_G$ by letting $\lambda(x)=e_{N}$ for all $x\in X_G\backslash N\cdot U$. Finally remark that these maps can (and will) be chosen such that for all $x\in U$ we have $\kappa(x)=e_{M}$ and $\lambda(x)=e_{N}$.

  \begin{description}
  \item[Shifted fundamental domains for $G$ and $H$] First let
  \begin{align*}
    X^\prime_G
    &:=\left\{ c_2\big(\lambda(x),x\big)*x \mid x\in X_G \right\},\\
    X^\prime_H
    &:=\left\{ c_1\big(\kappa(x),x\big)*x \mid x\in X_H \right\}.
  \end{align*}
  Remark that $X^\prime_G$ and $X^\prime_H$ are Borel subsets of $\Omega$ and both contain $U$.
  Now since $X^\prime_G$ and $X^\prime_H$ are fundamental domains for the respective actions of $G$ and $H$ on $\Omega$, we also have an induced action of $G$ (resp.\ $H$) on $X^\prime_H$ (resp.\ $X^\prime_G$). We will denote these actions by~“$\cdotprime$” 
  (see also \cref{eq:ActionDotPrime} below).
  \item[Cocycles] Now define for all $g\in M$ and a.e.\ $x\in M\cdot U$
  \begin{align*}
      c^\prime_1(g,x)&:= c_1\big(\kappa(g\cdot x),g\cdot x \big) c_1(g,x) c_1\big(\kappa(x),x\big)^{-1},\\
      &=c_1\left(\kappa(g\cdot x) g \kappa(x)^{-1},\kappa(x)\cdot x\right).
  \end{align*}
  In other words ${c}^\prime_1(g,x)$ verifies ${c}^\prime_1(g,x)\cdot \big(\kappa(x)\cdot x \big)=\kappa(g\cdot x) \cdot (g\cdot x)$ (see \cref{fig:DefcPrimeUn}).
    \begin{figure}[htbp]
    \centering
    \includegraphics[width=0.45\textwidth]{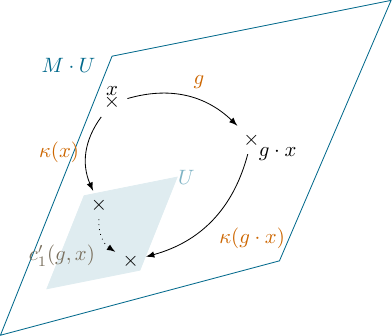}
    
    {\small \itshape The action of $c^\prime_1(g,x)$ is represented by the dotted arrow, it corresponds to the element of $H$ sending $\big(\kappa(x)\cdot x \big)$ to $\kappa(g\cdot x) \cdot (g\cdot x)$.
    Actions of elements of $G$ (in \textcolor{Orange}{orange}) are represented by the other arrows.}
    \caption{Definition of $c^\prime_1$}
    \label{fig:DefcPrimeUn}
  \end{figure}
  But by definition of $\kappa$, the product $\kappa(g\cdot x)g\kappa(x)^{-1}$ belongs to $M$ and both 
  $\kappa(x)\cdot x$ and $\kappa(g\cdot x) \cdot (g\cdot x)$ belong to $U$. 
  Therefore, ${c}^\prime_1(g,x)\in N$.

  Similarly, define for all $h\in N$ and all $x\in N\cdot U$
  \begin{equation*}
      c^\prime_2(h,x):= c_2\big(\lambda(h\cdot x),h\cdot x \big) c_2(h,x) c_2\big(\lambda(x),x\big)^{-1}.
  \end{equation*}
  As above ${c}^\prime_2$ verifies 
  ${c}^\prime_2(h,x)\cdot \big(\lambda(x)\cdot x \big)=\lambda(h\cdot x) \cdot (h\cdot x)$ and $c^\prime_2(h,x)\in M$.

  Remark that, if we denote by $x^\prime=c_1\big(\kappa(x),x\big)*x$ an element of $X^\prime_H$, for all $g\in M$ and a.e.\ $x\in M\cdot U$ we have (see also \cref{fig:DefXprime})
  \begin{equation}\label{eq:ActionDotPrime}
   g\cdotprime x^\prime = c^\prime_1(g,x)*g*x^\prime.
  \end{equation}
  And similarily for $c^\prime_2$: for all $h\in N$ and a.e.\ $y\in N\cdot U$, if we let $y':=c_2(\lambda(y),y)\ast y$, we have $ h\cdotprime y^\prime = c^\prime_2(h,y)*h*y^\prime$.
  
  \begin{figure}[htbp]
      \centering
      \includegraphics{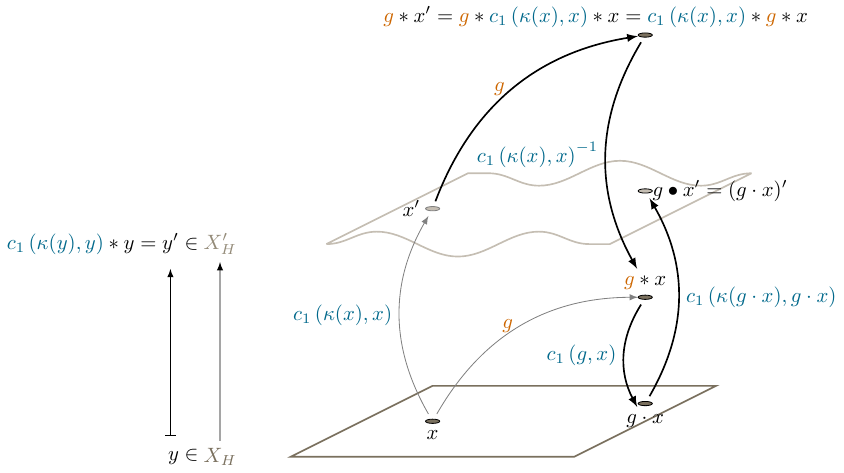}
      \caption{Definition of $g \cdotprime x^\prime$ (see \cref{eq:ActionDotPrime}). Here $X^\prime_H$ is represented wavy to emphasize that it is not a translate of $X_H$.}
      \label{fig:DefXprime}
  \end{figure}
  \item[Definition of the space] Let us now define the following two Borel spaces
  \begin{equation*}
    \Omega_{M,N}:=N*\left(M \cdotprime U\right),
    \qquad \Omega_{N,M}:=M*\left(N\cdotprime U\right).
  \end{equation*}
  Remark that $\Omega_{M,N}$ (resp.\ $\Omega_{N,M}$) is $N$-invariant (resp.\ $M$-invariant). In addition $N\cdotprime U$ and $M\cdotprime U$ are contained in the fundamental domains $X'_G$ and $X'_H$ respectively. Thus, $M\cdotprime U$ (resp.\ $N\cdotprime U$) is a Borel fundamental domain of finite measure for the $N$-action on $\Omega_{M,N}$ (resp.\ for the $M$-action on $\Omega_{N,M}$).
  Therefore, to obtain a measure equivalence coupling between $M$ and $N$, it only remains to prove that $\Omega_{M,N}=\Omega_{N,M}$: indeed, the lemma then follows by letting $\Omega'=\Omega_{M,N}$, and letting $X_M=N\cdotprime U$ and $X_N=M\cdotprime U$, which are respectively contained in $X'_G$ and $X'_H$. By symmetry, it is enough to prove the inclusion $\Omega_{M,N}\subseteq \Omega_{N,M}$.
  
  So let $h\in N$ and $u\in U$ and $g\in M$ and consider $h*(g\cdotprime u)\in \Omega_{M,N}$. Since $g\in M$ we have $c^\prime_1(g,u)\in N$ and therefore $h^\prime:=hc^\prime_1(g,u)$ also belongs to $N$. Remark that $h^\prime$ is the element of $N$ sending $g*u$ to $h*(g\cdotprime u)$: indeed using \cref{eq:ActionDotPrime} (see also \cref{fig:Omegavwwv}) we obtain
  \begin{equation}\label{eq:hprime}  
    h^\prime*g*u= hc^\prime_1(g,u)*g*u=  h*(g\cdotprime u).
  \end{equation}
  Note moreover that since $h^\prime$ belongs to $N$ then $c^\prime_2(h^\prime,u)$ is in $M$. Then, using \cref{eq:hprime} and that the actions of $H$ and $G$ commute we get (see also \cref{fig:Omegavwwv})
  \begin{align*}
      h*(g\cdotprime u)
      &= h^\prime*g*u,\\
      &= g*h^\prime*u,\\
      &=g*c^\prime_2(h^\prime,u)^{-1}*(h^\prime \cdotprime u)\in M*(N\cdotprime U).
  \end{align*}
  
  \begin{figure}[htbp]
      \centering
      \includegraphics[width=0.85\textwidth]{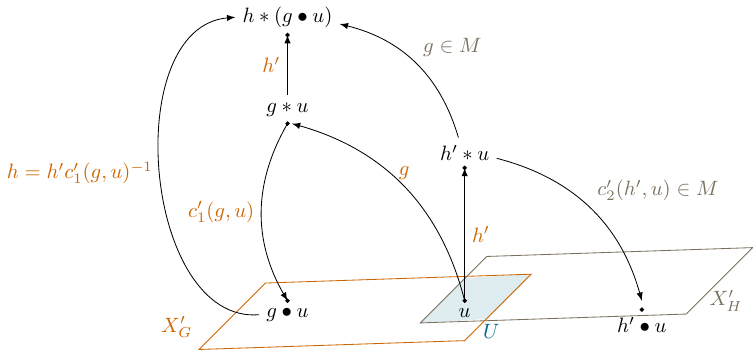} 
      \caption{Showing that $h*(g \cdotprime u)$ belongs to $\Omega_{N,M}=M*(N \cdotprime U)$}
      \label{fig:Omegavwwv}
  \end{figure}
  Therefore $\Omega_{M,N}=\Omega_{N,M}$. Hence denoting again $m$ the measure induced by $m$ on $\Omega_{M,N}$ we obtain that $\big(\Omega_{M,N},m, N\cdotprime U,M\cdotprime U \big)$ is a measure equivalence coupling from $M$ to $N$. \qedhere
  \end{description}
\end{proof}

\section{Measured groupoids}
\label{sec:background-groupoids}
Since the work of Kida on the measure equivalence rigidity of mapping class groups \cite{Kid-me}, proofs of measure equivalence classification and rigidity theorems have often taken advantage of the language of measured groupoids. In this section, we review the relevant notions for the present work and emphasize the relationship to measure equivalence. In particular, we finish this section by providing a general framework, phrased in the language of measured groupoids, that encompasses Kida's strategy, and will serve as a blueprint in Part~\ref{Part:MEClassification} of the present work (Section~\ref{sec:blueprint}). General references about measured groupoids include \cite[Section~2.1]{AD}, \cite{Kid-survey} or \cite[Section~3]{GH-OutFn}.

\subsection{Measured groupoids}
In this section we recall some necessary material about measured groupoids and give numerous examples coming from the framework of groups acting on measure spaces.

A \emph{Borel groupoid}\index{Borel!Borel groupoid}\index{Groupoid!Borel groupoid} is a standard Borel space $\calg$ endowed with two Borel maps $s,r:\calG \rightarrow X$ towards a standard Borel space $X$ and equipped with a measurable (partially defined) composition law, a measurable inverse map, and a unit element for all $x\in X$ denoted by~$e_x$.
We call $X$ the \emph{base space}\index{Base space (of a groupoid)}\index{Groupoid!Base space} of $\calg$ and if $\grpdg \in \calg$ we call $s(\grpdg)$ the \emph{source}\index{Source (of a groupoid element)} and $r(\grpdg)$ the \emph{range}\index{Range (of a groupoid element)} of $\grpdg$. Two elements $\grpdg,\grpdh\in\calg$ are composable (i.e.\ one can form the product $\grpdg\grpdh$) if and only if $r(\grpdh)=s(\grpdg)$. Let $B_1,B_2\subseteq \calG$, we will denote by $B_1B_2$ the set\label{Def:ProduitGroupoides}
\begin{equation*}
      B_1B_2 =\left\{ \grpdb_1 \grpdb_2 \, : \, \grpdb_1\in B_1, \grpdb_2\in B_2,\  r(\grpdb_2)=s(\grpdb_1) \right\}.
\end{equation*}
All Borel groupoids considered in this article are \emph{discrete}\index{Discrete groupoid}, that is to say there are at most countably many elements with a given source or range. 

\begin{Ex}\label{Ex:MeasuredGroupoidAssociatedToAnAction}
    Let $G$ be a countable group acting on a standard probability space $(X,\mu)$. Then $G\times X$ has a natural structure of Borel groupoid over $X$. The associated source and range maps are then defined by $s(g,x)=x$ and $r(g,x)=g\cdot x$, the unit elements are given by $e_x:=(e_G,x)$ and the inverse map by $(g,x)^{-1}:=(g^{-1},g\cdot x)$. We denote this groupoid by $G \ltimes X$. 
\end{Ex}

Let $U\subseteq X$ be a Borel subset. The \emph{restriction}\index{Restriction (groupoid)}\index{Groupoid!Restriction of a groupoid} of $\calg$ to $U$ is defined as 
\begin{equation*}
    \calg_{|U}:=\left\{ \grpdg\in \calg \ :
    \ s(\grpdg),\, r(\grpdg)\in U\right\},
\end{equation*}
which is naturally a Borel groupoid over the base space $U$.

A Borel subset $B\subseteq \calG$ is called a \emph{bisection}\index{Bisection} if $s_{|B}$ and $r_{|B}$ are Borel isomorphisms to Borel subsets $s(B),r(B)\subseteq X$ respectively.

In the sequel of the paper we will work with measured groupoids\index{Measured groupoid}, defined as follows. A \emph{measured groupoid}\index{Groupoid!Measured groupoid} over a standard measure space $(X,\mu)$ is a discrete Borel groupoid $\calg$ over $X$ for which the measure $\mu$ is quasi-invariant in the following sense: for every bisection $B\subseteq\calg$, one has $\mu(r(B))=0$ if and only if $\mu(s(B))=0$.

The example of interest for us is when $G$ acts measure preservingly on a probability space $X$: the groupoid $G\ltimes X$ defined in \cref{Ex:MeasuredGroupoidAssociatedToAnAction} is then a measured groupoid. 

We mention that there is a natural notion of measured subgroupoid of $\calG$, as well as of groupoid isomorphism and refer to \cite[Section 3]{GH-OutFn} for more details.

Let us now illustrate how to construct measured groupoids from measure or orbit equivalence couplings.

\begin{Ex}\label{Ex:GroupoidsAssociatedToCouplings}
\begin{itemize}
    \item If $G$ and $H$ are orbit equivalent over some probability space $(X,\mu)$, then the corresponding measured groupoids $G\ltimes X$ and $H\ltimes X$ are isomorphic. In the following, when $G$ and $H$ are orbit equivalent, we will therefore identify $G\ltimes X$ and $H\ltimes X$ to the same measured groupoid.
    \item If two countable groups $G$ and $H$ are measure equivalent over some measure space $(\Omega,m)$, one can choose the fundamental domains $X_G$ and $X_H$ such that $m(X_G\cap X_H)>0$. For such a choice, let $X:=X_G\cap X_H$ and consider $(G \ltimes X_H)_{|X}$ and $(H \ltimes X_G)_{|X}$. Remark that the orbit equivalence relations on $X$ induced by the $G$-action on $X_H$ and $H$-action on $X_G$ coincide. In fact the two measured groupoids are isomorphic, via the identification $(g,x)\mapsto \big(c(g,x),x \big)$, where $c:G\times X_H\to H$ is the associated measure equivalence cocycle.
\end{itemize}    
\end{Ex}

\subsubsection{Finite index and infinite type groupoids} 
The \emph{index}\index{Index of a measured subgroupoid} of a measured subgroupoid was defined by Kida \cite[Section~3.2]{Kid-BS}, generalizing a notion of Feldman--Sutherland--Zimmer for equivalence relations \cite{FSZ}. Let $\calH \subseteq \calG$ be a measured subgroupoid and define the equivalence relation $\sim_\calH$ on $\calG$ as $\grpdg_1\sim_{\calH} \grpdg_2$ if and only if $\grpdg_2\grpdg^{-1}_1\in \calH$. 

\begin{itemize}
    \item One says that $\calH$ has \emph{finite index}\index{Finite index subgroupoid}\index{Subgroupoid of finite index} in $\calG$ if for a.e.\ $x\in X$ there are only finitely many equivalence classes modulo $\sim_{\calH}$ on $s^{-1}(x)$.
    \item One says that $\calg$ is \emph{of infinite type}\index{Infinite type groupoid}\index{Groupoid!Groupoid of infinite type} if there does not exist any positive measure Borel subset $U\subseteq X$ such that the trivial subgroupoid has finite index in $\calG_{|U}$, that is to say  
    if for every Borel subset $U\subseteq X$ of positive measure and a.e.\ $x\in U$ there are infinitely many $\grpdg\in \calG_{|U}$ such that $ s(\grpdg)=x$. 
\end{itemize}

\subsubsection{Stably something groupoids} \label{Subsec:StablySomething} In the present paper we will often work with properties verified \emph{up to some subdivision of the base space.} If \textsf{A} denotes some adjective, we will say that the groupoid is \emph{stably \textsf{A}}\index{Stably} if there exist a conull Borel subset $X^*\subseteq X$ and a partition $X^*=\sqcup_{i\in I} X_i $ into at most countably many Borel subsets such that $\calG_{|X_i}$ is \textsf{A}. For example:
\begin{itemize}
    \item We say that $\calg$ is \emph{stably trivial}\index{Stably!Stably trivial groupoid} if there exist a conull Borel subset $X^*\subseteq X$ and a partition $X^*=\sqcup_{i\in I} X_i $ into at most countably many Borel subsets such that $\calg_{|X_i}=\{ e_x \ : \ x\in X_i\}$ for every $i\in I$.
    \item Let $\calh,\calh'$ be two measured subgroupoids of $\calg$. We say that $\calh$ is \emph{stably contained}\index{Stably!Stably contained} in $\calh'$ (resp.\ \emph{stably equal}\index{Stably!Stably equal} to $\calh'$) if there exist a conull Borel subset $X^*\subseteq X$ and a partition $X^*=\sqcup_{i\in I} X_i $ into at most countably many Borel subsets such that $\calh_{|X_i}\subseteq \calh'_{|X_i}$ (resp.\ $\calh_{|X_i} = \calh'_{|X_i}$) for all $i\in I$.
\end{itemize}

We will also work with the notion of \emph{stably normalized} subgroupoid (see \cref{Subsec:NormalisationGroupoid}).

\subsubsection{Cocycles}\label{Def:CocycleGroupoid}
We now introduce the notion of cocycle for {measured groupoids}. In order to avoid confusion with measure and orbit equivalence cocycles defined above, we will denote by $c$ the latter and by $\rho$ a cocycle defined on a groupoid. The link between the different notions is made in the examples below.

Let $\calg$ be a measured groupoid over some standard measure space $X$ and $G$ be a countable group. A map $\rho:\calG \rightarrow G$ is called a \emph{cocycle}\index{Cocycle} if there exists a conull Borel subset $X^*\subseteq X$ such that for all
$\grpdg_1,\grpdg_2\in \calg_{|X^*}$ such that $s(\grpdg_1)=r(\grpdg_2)$, it verifies $\rho(\grpdg_1\grpdg_2)=\rho(\grpdg_1)\rho(\grpdg_2)$. Up to taking a conull subset, we can (and will) always assume that the equality holds everywhere.

\begin{Ex} \label{Ex:CocycleAssociatedToAnAction}
\begin{itemize}
    \item Let $G$ and $X$ be as in \cref{Ex:MeasuredGroupoidAssociatedToAnAction}. The groupoid $G\ltimes X$ comes with a cocycle $\rho:G\ltimes X \rightarrow G$ defined as $\rho(g,x)=g$. We call it the \emph{natural cocycle}\index{Natural cocycle}\index{Cocycle!Natural cocycle associated to an action} associated to the action.
    \item \label{Ex:CocycleAssociatedToAnOE}
    Let $G$ and $H$ be two groups that are orbit equivalent over some probability space~$X$. Recall that we have an orbit equivalence cocycle $c:G\times X \rightarrow H$ where $c(g,x)$ is defined as the (unique) element of $H$ such that $c(g,x)\cdot x=g\cdot x$. 
    Therefore this map $c$ induces a cocycle $\rho$ from the groupoid $\calG=G\ltimes X$ to the group $H$, defined by $\rho(g,x)=c(g,x)$. Note that since $G\ltimes X$ and $H\ltimes X$ are isomorphic, if $c^\prime:H\times X\rightarrow G$ denotes the other cocycle associated to the orbit equivalence, one can also define a cocycle $\rho^{\prime}:\calG\rightarrow G$ by letting $\rho^\prime(h,x)=c^\prime(h,x)$.
    \item \label{Ex:CocycleGrAssociatedToAME}
    Similarly, assume that $G$ and $H$ are two measure equivalent groups over some measure space $\Omega$ and $X$ is as in the second point of \cref{Ex:GroupoidsAssociatedToCouplings}. One can thus consider $\calG:=(G\ltimes X_H)_{|X}$ and define a cocycle $\rho:\calG \rightarrow H$ by letting $\rho(g,x):=c(g,x)$ where $c(g,x)$ denotes here the measure equivalence cocycle associated to the coupling $(\Omega,m,X_G,X_H)$.
\end{itemize}
\end{Ex}

\begin{Rq}\label{Rq:CouplageRestriction}
    For future use, notice that the assumption made in Lemma~\ref{lemma:me-subgroups} can be reformulated in terms of measured subgroupoids. Indeed, let $\calg$ be the measured groupoid on $X_G\cap X_H$ naturally associated to the measure equivalence coupling (as in the second point of \cref{Ex:GroupoidsAssociatedToCouplings}), and let $\rho_G:\calg\to G$ and $\rho_H:\calg\to H$ be the natural cocycles. Then our assumption on the cocycles $c_1$ and $c_2$ in Lemma~\ref{lemma:me-subgroups} is equivalent to requiring that $\rho_G^{-1}(M)_{|U}=\rho_H^{-1}(N)_{|U}$. 
\end{Rq}

The \emph{kernel}\index{Kernel!Kernel of a cocycle} of $\rho:\calG\rightarrow G$ is the set of all $\grpdg \in \calG$ such that $\rho(\grpdg)=e_G$. We say that $\rho$ has \emph{trivial kernel}\index{Trivial kernel}\index{Kernel!Trivial kernel} if $\rho^{-1}(e_G)$ only consists of units of $\calg$.
A cocycle $\rho$ is called \emph{nowhere trivial}\index{Nowhere!Nowhere trivial (cocycle)} if there does not exist any Borel subset $U\subseteq X$ of positive measure such that $\rho(\calG_{|U})=\{e_G\}$. 
The cocycle $\rho$ is called \emph{stably trivial}\index{Stably!Stably trivial cocycle}\label{Def:Stably-Trivial-Cocycle} if there exist a conull Borel subset $X^*\subseteq X$ and a Borel partition $X^*=\sqcup_{i\in I} X_i $ into at most countably many Borel subsets such that $\rho\left(\calG_{|X_i} \right)=\{e_G\}$ for all $i\in I$.

\paragraph{Equivariant maps} 
Let $\rho: \calG \rightarrow G$ be a cocycle and assume that the group $G$ acts on a standard Borel space $\Delta$ by Borel automorphisms. One says that a map $\Phi:X\rightarrow \Delta$ is \emph{$(\calG,\rho)$-equivariant}\index{Equivariant map} if there exists a conull Borel subset $X^*\subseteq X$ such that for all $\grpdg \in \calG_{|X^*}$ one has 
\begin{equation*}
    \Phi\left(r(\grpdg) \right)=\rho(\grpdg)\Phi\left(s(\grpdg) \right).
\end{equation*}
\begin{Ex}
    Let $G$ be a countable group acting measure-preservingly on a standard measure space $X$. Denote by $\calG:=G \ltimes X$ the associated groupoid and $\rho:\calG \rightarrow G$ the natural cocycle. If $G$ acts by measurable automorphisms on a standard Borel space $\Delta$, then any $G$-equivariant map from $X$ to $\Delta$ is $(\calG,\rho)$-equivariant.
\end{Ex}

\subsection{Normalization and amenable groupoids}
A crucial tool in our proofs is the study of amenable subgroupoids and their normalizers (see for example \cref{Sec:ExploitingNormalAmenable}). We recall in this section the main definitions and properties needed regarding the aforementioned two notions.

\subsubsection{Normalization}\label{Subsec:NormalisationGroupoid} A notion of normal subgroupoid of a measured groupoid was formalized by Kida \cite[Section~2.4]{Kid-me}, building upon a similar notion for measured equivalence relations due to Feldman--Sutherland--Zimmer \cite{FSZ}. We recall here the reformulation given in \cite[Section~3.2]{GH-OutFn}.

Let $\calH$ be a measured subgroupoid of $\calG$. Given a Borel subset $B\subseteq \calG$, one says that $\calH$ is \emph{$B$-invariant}\index{Invariant subgroupoids} if there exists a conull Borel subset $X^*\subseteq X$ such that for all $\grpdg_1,\grpdg_2 \in B\cap \calG_{|X^*}$ and all $\grpdh\in \calG_{|X^*}$ verifying $s(\grpdh)=s(\grpdg_1)$ and $r(\grpdh)=s(\grpdg_2)$, one has $\grpdh\in \calH$ if and only if $\grpdg_2\grpdh\grpdg^{-1}_1 \in \calH$.

Given two measured subgroupoids $\calH,\caln$, one says that $\calH$ is \emph{normalized}\index{Normalized (groupoids)} by $\caln$ if there exists a covering of $\caln$ by countably many Borel subsets $B_n\subseteq \calG$ so that $\calH$ is $B_n$-invariant for all $n\in \bN$. 

\begin{Ex}\label{ex:normal}
    Let $\rho:\calG\rightarrow G$ be a cocycle towards a countable group $G$. 
    If $H$ and $N$ are two subgroups of $G$ such that $H$ is normalized by $N$ then $\rho^{-1}(H)$ is normalized by~$\rho^{-1}(N)$.
\end{Ex}

Note that similarly as in \cref{Subsec:StablySomething}, one can define the notion of \emph{stably normalized}\index{Stably!Stably normalized (subgroupoid)} subgroupoids.

\subsubsection{Amenability} Generalizing Zimmer's notion of amenability of a group action \cite{Zim-amen}, Kida introduced the notion of amenable measured groupoid\index{Amenable measured groupoid}\index{Groupoid!Amenable measured groupoid}. We refer to \cite[Section~4]{Kid-survey} for the formal definition and only recall here the properties needed for our proofs. As an example, if $(X,\mu)$ is a standard probability space equipped with a measure-preserving action of a countable group $G$, then the groupoid $G\ltimes X$ is amenable if and only if the $G$-action on $X$ is amenable in the sense of Zimmer. 

\begin{Lmm}[{see \cite[Corollary~3.39]{GH-OutFn}}]\label{lemma:amenable-subgroup-subgroupoid}
    Let $\calG$ be a measured groupoid, and $G$ be an amenable countable group.
    If there exists a cocycle $\calG \rightarrow G$ with trivial kernel, then $\calG$ is amenable.
\end{Lmm}

In the following, if $K$ is a compact metrizable space we denote by $\mathrm{Prob}(K)$ the space of Borel probability measures on $K$ equipped with the weak-$*$ topology.

\begin{Prop}[{\cite[Proposition 4.14]{Kid-survey}}]\label{prop:amenable-prob}
    Let $\calG$ be an amenable measured groupoid over some base space $X$ and $\rho:\calG \rightarrow G$ be a cocycle toward $G$ a countable group. Assume that $G$ acts by homeomorphisms on a compact metrizable space $K$. Then there exists a $(\calG,\rho)$-invariant Borel map from $X$ to $\mathrm{Prob}(K)$.
\end{Prop}

One says that $\calG$ is \emph{everywhere non-amenable}\index{Everywhere non-amenable groupoid} if for every Borel subset $U\subseteq X$ of positive measure, the restriction $\calG_{|U}$ is not amenable. 

\subsection{Action-like cocycles} 

We will work with the following definition.

\begin{Def}[Action-like cocycle]\index{Cocycle!Action-like cocycle}\label{de:action-like}
    Let $\calg$ be a measured groupoid over a standard probability space, and let $G$ be a countable group. A cocycle $\rho:\calg\to G$ is \emph{action-like}\index{Action-like cocycle}\index{Cocycle!Action-like cocycle} if
    \begin{enumerate}
        \item $\rho$ has trivial kernel;
        \item whenever $H_1\subseteq H_2$ is an infinite index inclusion of subgroups of $G$, for every positive measure Borel subset $U\subseteq X$, the subgroupoid $\rho^{-1}(H_1)_{|U}$ does not have finite index in $\rho^{-1}(H_2)_{|U}$;
        \item for every nonamenable subgroup $H\subseteq G$, the groupoid $\rho^{-1}(H)$ is everywhere nonamenable.
    \end{enumerate}
\end{Def}

\begin{Rq}\label{rk:action-like}
We warn the reader that, although we are using the same terminology, our notion is slightly different from the one in \cite[page~53]{HH-Higman}, where the third requirement was dropped --~it is also slightly more restrictive than the notion of \emph{action-type} cocycle from \cite[Definition~3.20]{GH-OutFn}.
\end{Rq}

\begin{Rq}\label{rk:action-like-restriction}
We observe that if $\rho:\calg\to G$ is action-like, then for every positive measure Borel subset $U\subseteq X$, the restriction $\rho:\calg_{|U}\to G$ is again action-like. Also, if $H\subseteq G$ is a subgroup, then the restricted cocycle $\rho^{-1}(H)\to H$, viewed as a cocycle towards $H$, is action-like.
\end{Rq}

Building on Example~\ref{Ex:CocycleAssociatedToAnAction}, the main example of importance to us is the following.

\begin{Lmm}\label{lemma:action-like-1}
Let $(X,\mu)$ be a standard probability space, equipped with a measure-preserving action of a countable group $G$. Let $\calg=G\ltimes X$ be the corresponding measured groupoid.

Then the natural cocycle $\rho:\calg\to G$ is action-like. 
\end{Lmm}

\begin{proof}
    The first two items of the definition are checked in \cite[Lemma~B.3]{HH-Higman} (we emphasize here, as recalled in Remark~\ref{rk:action-like}, that the notion of an action-like cocycle used in \cite{HH-Higman} only had the first two items from~\cref{de:action-like}). We now prove that if $H\subseteq G$ is a nonamenable subgroup, then $\rho^{-1}(H)$ is everywhere nonamenable. So let $U\subseteq X$ be a positive measure Borel subset, and assume that $\rho^{-1}(H)_{|U}$ is amenable. Letting $H\cdot U$ be the saturation of $U$ under the $H$-action, it follows from \cite[Theorem~4.16(iii)]{Kid-survey} that $\rho^{-1}(H)_{|H\cdot U}$ is amenable. But since $H\cdot U$ is an $H$-invariant positive measure subset on which $H$ acts in a measure-preserving way, then \cite[Proposition~4.3.3]{Zim} implies that $H$ is amenable, which is a contradiction. 
\end{proof}

We will often use the following observation.

\begin{Lmm}\label{lemma:action-like}
Let $\calg$ be a measured groupoid, let $G$ be a countable group, and let $\rho:\calg\to G$ be an action-like cocycle. Let $H,H'$ be two subgroups of $G$.
If $\rho^{-1}(H)$ is stably contained in $\rho^{-1}(H')$, then $H$ has a finite index subgroup contained in $H'$.
\end{Lmm}

\begin{proof}
Let $U\subseteq X$ be a positive measure Borel subset such that $\rho^{-1}(H)_{|U}\subseteq\rho^{-1}(H')_{|U}$. Then $\rho^{-1}(H\cap H')_{|U}=\rho^{-1}(H)_{|U}$, and it follows from the definition of an action-like cocycle that $H\cap H'$ has finite index in $H$.     
\end{proof}

As a consequence, when an action-like cocycle takes its values in a graph product $G$, one can relate the inclusion of two parabolic subgroups of $G$ (see \cref{Subsec:ParabolicSubgroup} for the definition) with the inclusion of groupoids.

\begin{Lmm}\label{lemma:inclusion-parabolics}
    Let $G$ be a graph product of countably infinite groups over a finite simple graph $\Gamma$. Let $\calg$ be a measured groupoid over a standard probability space $X$, equipped with an action-like cocycle $\rho:\calg\to G$, and let $U\subseteq X$ be a positive measure Borel subset. Let $P_1,P_2\subseteq G$ be two parabolic subgroups.

    If $\rho^{-1}(P_1)_{|U}\subseteq\rho^{-1}(P_2)_{|U}$, then $P_1\subseteq P_2$.
\end{Lmm}

\begin{proof}
  As $\rho^{-1}(P_1)_{|U}\subseteq\rho^{-1}(P_2)_{|U}$, Lemma~\ref{lemma:action-like} (applied to the groupoid $\calg_{|U}$) shows that $P_1$ has a finite-index subgroup contained in $P_2$. The conclusion then follows from Lemma~\ref{lemma:finite-index}, \cpageref{lemma:finite-index}.
\end{proof}

\subsubsection{Support}

In this section specifically, we consider measured groupoids equipped with cocycles towards a graph product, and review the notion of support introduced in \cite[Section~3.3]{HH21}. It was defined in a more general context than graph products in \cite{HH21}, but we will specify to this setting since this will be the only one of interest to us in the present paper.

\paragraph{Tight support}\label{tight-support}
Let $G$ be a graph product of countable groups over a finite simple graph. Recall that $\mathbb{P}_G$ denotes the set of parabolic subgroups of $G$. Let $\calg$ be a measured groupoid over a standard probability space $X$, and let $\rho:\calg\to G$ be a cocycle. Following \cite[Section~3.3]{HH21}, given a parabolic subgroup $P\in\mathbb{P}_G$, we say that $(\calg,\rho)$ is \emph{tightly $P$-supported}\index{Tightly $P$-supported} if 
\begin{enumerate}
\item there exists a conull Borel subset $X^*\subseteq X$ such that $\rho(\calg_{|X^*})\subseteq P$, and
\item for every proper parabolic subgroup $Q\subsetneq P$, there does not exist any positive measure Borel subset $U\subseteq X$ such that $\rho(\calg_{|U})\subseteq Q$.
\end{enumerate}

Given a pair $(\calG,\rho)$, if there exists some parabolic subgroup $P\in\mathbb{P}_G$ such that $(\calG,\rho)$ is tightly $P$-supported then such a $P$ is unique. Notice also that if $(\calg,\rho)$ is tightly $P$-supported, then for every positive measure Borel subset $U\subseteq X$, the pair $(\calg_{|U},\rho)$ is tightly $P$-supported. The following lemma provides examples of tightly supported pairs.

\begin{Lmm}\label{lemma:tight-support}
    Let $G$ be a graph product of countably infinite groups over a finite simple graph. Let $\calg$ be a measured groupoid over a standard probability space, equipped with an action-like cocycle $\rho:\calg\to G$. Let $P\subseteq G$ be a parabolic subgroup.

    Then $(\rho^{-1}(P),\rho)$ is tightly $P$-supported. 
\end{Lmm}

\begin{proof}
    The first part of the definition obviously holds. So it is enough to prove that if $Q\subsetneq P$ is a parabolic subgroup, and $U\subseteq X$ is a positive measure Borel subset, we have $\rho^{-1}(P)_{|U}\nsubseteq\rho^{-1}(Q)_{|U}$. This is a consequence of Lemma~\ref{lemma:inclusion-parabolics}.
\end{proof}

The next lemma formalizes the fact that up to a countable partition of the base space, we can work with tightly supported pairs. It follows from \cite[Lemma~3.7]{HH21}, by noticing that the collection $\mathbb{P}_G$ of parabolic subgroups of $G$ is admissible in the sense of \cite{HH21} (it is countable, stable under conjugation, intersections, and every descending chain of parabolic subgroups is finite).

\begin{Lmm}[{\cite[Lemma~3.7]{HH21}}]\label{lemma:support}
Let $G$ be a graph product of countably infinite groups over a finite simple graph. Let $\calg$ be a measured groupoid over a standard probability space $X$, and let $\rho:\calg\to G$ be a cocycle.

Then there exist a partition $X=\dunion_n X_n$ into at most countably many Borel subsets, and for every $n$, a parabolic subgroup $P_n\subseteq G$, such that $(\calg_{|X_n},\rho)$ is tightly $P_n$-supported.
\end{Lmm}

\paragraph{Tight support and normalizers}
We will often make use of the following lemma from \cite{HH21}.

\begin{Lmm}[{\cite[Lemma~3.8]{HH21}}]\label{lemma:support-normal}
Let $G$ be a graph product of countably infinite groups over a finite simple graph. Let $\calg$ be a measured groupoid over a standard probability space~$X$, and let $\rho:\calg\to G$ be a cocycle. Let $\cala,\calh$ be measured subgroupoids of $\calg$, with~$\cala$ normalized by $\calh$.

If $(\cala,\rho)$ is tightly $P$-supported for some parabolic subgroup $P\subseteq G$, then there exists a conull Borel subset $X^*\subseteq X$ such that $\rho(\calh_{|X^*})\subseteq P\times P^{\perp}$.
\end{Lmm}

\begin{Def}\label{def:ProperlySupported}
    Let $G$ be a graph product over a finite simple graph, let $\calg$ be a measured groupoid over a standard probability space $X$, and let $\rho:\calg\to G$ be a cocycle.

    We say that $(\calg,\rho)$ is \emph{properly $\mathbb{P}_G$-supported}\index{Properly $\mathbb{P}_G$-supported} if there exist a partition $X^*=\dunion_{i\in I} X_i$ of a conull Borel subset $X^*\subseteq X$ into at most countably many Borel subsets, and for every $i\in I$, a \emph{proper} parabolic subgroup $P_i\subsetneq G$ such that $\rho\left(\calg_{|X_i}\right)\subseteq P_i$.
\end{Def}
We refer to \cpageref{Def:Irreducible} for the definition of irreducible graph.
\begin{Lmm}\label{lemma:properly-supported}
    Let $G$ be a graph product of countably infinite groups over a finite irreducible simple graph. Let $\calg$ be a measured groupoid over a standard probability space, and let $\rho:\calg\to G$ be a cocycle. Let $\cala,\calh\subseteq\calg$ be measured subgroupoids, with $\cala$ normalized by $\calh$.

    If $(\cala,\rho)$ is properly $\mathbb{P}_G$-supported and $\rho_{|\cala}$ is nowhere trivial, then $(\calh,\rho)$ is properly $\mathbb{P}_G$-supported.
\end{Lmm}

\begin{proof}
    Up to a countable Borel partition of $X$, we can assume that $(\cala,\rho)$ is tightly $P$-supported for some proper parabolic subgroup $P\subsetneq G$, and aim to show that there exists a conull Borel subset $X^*\subseteq X$ such that $\rho(\calh_{|X^*})$ is contained in a proper parabolic subgroup of $G$. 
    
    As $\rho_{|\cala}$ is nowhere trivial, we have $P\neq\{1\}$. By Lemma~\ref{lemma:support-normal}, there exists a conull Borel subset $X^*\subseteq X$ such that $\rho(\calh_{|X^*})\subseteq P\times P^{\perp}$. As $G$ is irreducible and $P$ is neither equal to $\{1\}$ nor to $G$, the parabolic subgroup $P\times P^{\perp}$ is proper, which completes our proof.
\end{proof}

\subsection{Kida's blueprint}\label{sec:blueprint}

In this section, we review Kida's general strategy towards measure equivalence classification and rigidity theorems. This originates from his work on mapping class groups \cite{Kid-memoir,Kid-me} and was also used in subsequent measure equivalence rigidity theorems. We extract an abstract setting from Kida's work, that we will follow in Part~\ref{Part:MEClassification}. We start with a very brief review of the mapping class group case that will serve as a motivating example -- we refer to the references below for basic definitions and background, but will never need anything about mapping class groups in the sequel of the present work.

\paragraph*{Review.} Recall that the \emph{mapping class group} of a finite-type orientable surface $\Sigma$ is the group $\Mod(\Sigma)$ of isotopy classes of orientation-preserving homeomorphisms of $\Sigma$. It acts by automorphisms on the \emph{curve graph} $\calc(\Sigma)$, defined as follows: vertices of $\calc(\Sigma)$ are isotopy classes of essential simple closed curves on $\Sigma$, and edges correspond to the existence of disjoint representatives. An important theorem of Ivanov \cite{Iva} asserts that the automorphism group of $\calc(\Sigma)$ is exactly the \emph{extended} mapping class group $\Mod^{\pm}(\Sigma)$, where orientation-reversing homeomorphisms are allowed.

A key step in Kida's proof of the measure equivalence rigidity of the mapping class group $\Mod(\Sigma)$ of a finite-type surface \cite{Kid-me}, was to establish that every self-coupling $\Omega$ of $\mathrm{Mod}(\Sigma)$, factors through $\mathrm{Aut}\left(\mathcal{C}(\Sigma)\right)\simeq\Mod^{\pm}(\Sigma)$. From there, a standard argument allows to derive the measure equivalence rigidity of $\Mod(\Sigma)$, from this rigidity statement for self-couplings. 

Kida's strategy for proving this statement about self-couplings was in part inspired by earlier work of Ivanov \cite{Iva}. Ivanov had proved that the abstract commensurator of $\Mod(\Sigma)$ is equal to $\Mod^{\pm}(\Sigma)$: this means that every isomorphism between finite-index subgroups of $\Mod(\Sigma)$, is given by the conjugation by an element of $\Mod^{\pm}(\Sigma)$. Ivanov's proof goes as follows: he showed that, if $H_1,H_2\subseteq\Mod(\Sigma)$ are two finite-index subgroups of $\Mod(\Sigma)$, then any isomorphism $f:H_1\to H_2$ sends curve stabilizers (in $H_1$) to curve stabilizers (in $H_2$), and sends stabilizers of disjoint curves (up to isotopy) to the stabilizers of disjoint curves. 
In other words $f$ induces an automorphism of the curve graph. Using that $\Aut(\calc(\Sigma))\simeq\Mod^{\pm}(\Sigma)$, this precisely says that $f$ is the conjugation by an element of $\Mod^{\pm}(\Sigma)$. 

Kida's strategy can be viewed as a groupoid-theoretic version of Ivanov's argument. Given a measured groupoid over a standard probability space with two action-like cocycles towards $\Mod(\Sigma)$, he proves that subgroupoids that correspond (in an appropriate sense) to curve stabilizers for the first cocycle, also correspond to curve stabilizers for the second cocycle. Through the dictionary between measure equivalence couplings and measured groupoids, this establishes the desired factorization of self-couplings of $\Mod(\Sigma)$.

Kida's strategy has later been used in several contexts to establish new measure equivalence rigidity theorems (see for example \cite{ChiKid,Kid-amalgam,HeH,HH-Artin,GH-OutFn}). We now formulate it in an abstract setting, which will be applied to the extension graph of the graph product in Part~\ref{Part:MEClassification} of the present work.  

\paragraph*{Preliminary tool: from measured groupoids to couplings.}

\begin{lemma}\label{lemma:groupoids-to-couplings}
    Let $G,H$ be two countable groups. Let $\Delta$ be a Polish space, equipped with an action of $G\times H$ by measurable automorphisms. Assume that for any measured groupoid $\calg$ over a standard probability space $X$, and any action-like cocycles $\rho_G:\calg\to G$ and $\rho_H:\calg\to H$, there exists a $(\rho_G,\rho_H)$-equivariant Borel map $\theta:X\to \Delta$, i.e.\ such that (up to restricting $\calg$ to a conull Borel subset) for every $\grpdg \in\calg$, one has 
    \[\theta(r(\grpdg))=\rho_H(\grpdg)\theta(s(\grpdg))\rho_G(\grpdg)^{-1}.\]
    Then for every measure equivalence coupling $\Omega$ between $G$ and $H$, there exists a $(G\times H)$-equivariant Borel map $\Omega\to \Delta$.
\end{lemma}

As usual, the equivariance of the map $\Omega\to \Delta$ is understood almost everywhere.

\begin{proof} 
The argument of this proof comes from the work of Kida, see the demonstration of \cite[Theorem~5.6]{Kid-me}.

    Recall that we denote by $*$ the actions of $G$ and $H$ on $\Omega$, and by $\cdot$ the induced actions on the fundamental domains.
    
    Let $X_G$ and $X_H$ be fundamental domains for the actions of $G$ and $H$ respectively. We can choose $X_G$ and $X_H$ such that their intersection $X:=X_G\cap X_H$ verifies $(G\times H)*X=\Omega$ up to a null set, see e.g.\ \cite[Lemma~2.27]{Kid-survey}. Note that this last equality implies in particular that $X$ has positive measure.

    Let $\calg$ be the measured groupoid over $X$ given by the restrictions of the $G$- and $H$-actions on $X_H$ and $X_G$ respectively, as in the second point of Example~\ref{Ex:GroupoidsAssociatedToCouplings}. We thus have two action-like cocycles $\rho_G: \calg \rightarrow G$ and $\rho_H:\calG \rightarrow H$ (see Example~\ref{Ex:CocycleAssociatedToAnAction} for their definition, and Lemma~\ref{lemma:action-like-1} and \cref{rk:action-like-restriction} for the fact that they are action-like). So by assumption, we have a $(\rho_G,\rho_H)$-equivariant Borel map $\theta:X\rightarrow \Delta$. We extend it to a map $\bar{\theta}:\Omega\to \Delta$ by defining $\bar{\theta}\big( (g,h)*x \big)=h\theta(x)g^{-1}$, for a.e.\ $x\in X$ and all $(g,h)\in G\times H$. 

     It only remains to verify that the definition of $\bar{\theta}$ does not depend on the choice of the representatives $g,h$ and $x$ --~its equivariance then follows from its definition. So let us show that for all $g_1,g_2\in G$, all $h_1,h_2\in H$ and a.e.\ $x_1,x_2\in X$ verifying $(g_1,h_1)*x_1=(g_2,h_2)*x_2$, we have $h_1\theta(x_1)g^{-1}_1=h_2\theta(x_2)g^{-1}_2$. 
     Since the actions of $G$ and $H$ commute, our assumption that $(g_1,h_1)*x_1=(g_2,h_2)*x_2$ rewrites as  $(h_2^{-1}h_1)*(g_2^{-1}g_1)*x_1=x_2$. This translates in terms of actions on the fundamental domains as 
    \begin{equation*}
        (h_2^{-1}h_1)\cdot x_1=x_2=(g_2^{-1}g_1)\cdot x_1.
    \end{equation*}
    The $(\rho_G,\rho_H)$-equivariance of $\theta$ thus gives $\theta(x_2)=h_2^{-1}h_1\theta(x_1)g_1^{-1}g_2$ almost surely, and hence $\bar{\theta}$ is well defined. 
\end{proof}

\paragraph*{Abstract setting.} Let $G$ be a countable group, let $\mathfrak{C}_G$\label{Nota:FrakC} be a conjugacy-invariant countable set of subgroups of $G$, and let $\mathfrak{I}_G$ be a graph with vertex set $\mathfrak{C}_G$, such that the conjugation action of $G$ on $\mathfrak{C}_G$ extends to a graph automorphism of $\mathfrak{I}_G$. 

\begin{Rq}
    In Kida's setting, $\mathfrak{C}_G$ is the set of all stabilizers in $\mathrm{Mod}(\Sigma)$ of isotopy classes of essential simple closed curves on $\Sigma$, and $\mathfrak{I}_G$ is the curve graph --adjacency corresponds to the disjointness of the curves up to isotopy, or more algebraically to the fact that the centers of their stabilizers commute. In our setting $\mathfrak{C}_G$ will be the set of conjugates of vertex groups, and $\mathfrak{I}_G$ will be the extension graph.  
\end{Rq}

Let $\calg$ be a measured groupoid over a standard probability space $X$, equipped with a cocycle $\rho:\calg\to G$. Let $\calh\subseteq\calg$ be a measured subgroupoid. 

We say that $(\calh,\rho)$ is \emph{stably $\mathfrak{C}_G$}\index{Stably! Stably $\mathfrak{C}_G$} if there exists a countable Borel partition $X^*=\dunion_{i\in I}X_i$ of a conull Borel subset $X^*\subseteq X$ into at most countably many Borel subsets, such that for every $i\in I$, there exists $C_i\in\mathfrak{C}_G$ such that $\calh_{|X_i}=\rho^{-1}(C_i)_{|X_i}$.

Let now $\calh_1,\calh_2$ be two measured subgroupoids such that both $(\calh_1,\rho)$ and $(\calh_2,\rho)$ are stably $\mathfrak{C}_G$. Then (after refining the two partitions given by the above definition) there exists a countable Borel partition $X^*=\dunion_{i\in I}X_i$ of a conull Borel subset $X^*\subseteq X$ into at most countably many Borel subsets, such that for every $i\in I$, there exist $C_{1,i},C_{2,i}\in\mathfrak{C}_G$ such that $(\calh_1)_{|X_i}=\rho^{-1}(C_{1,i})_{|X_i}$ and $(\calh_2)_{|X_i}=\rho^{-1}(C_{2,i})_{|X_i}$. 

We say that $(\calh_1,\rho)$ and $(\calh_2,\rho)$ are \emph{$\mathfrak{I}_G$-adjacent}\index{IG-adjacent@$\mathfrak{I}_G$-adjacent}\label{Nota:FrakI} if the above partition and subgroups can be chosen such that for every $i\in I$ such that $X_i$ has positive measure, the groups $C_{1,i}$ and $C_{2,i}$ are distinct and adjacent in $\mathfrak{I}_G$. 

\begin{Def}[ME-witness]
Let $G,H$ be two countable groups, with countable sets of subgroups $\mathfrak{C}_G,\mathfrak{C}_H$, and graphs $\mathfrak{I}_G,\mathfrak{I}_H$ with vertex sets $\mathfrak{C}_G,\mathfrak{C}_H$. We say that $(\mathfrak{I}_G,\mathfrak{I}_H)$ is an \emph{ME-witness}\label{Def:MEWitness}\index{ME-Witness} for $G,H$ if for every measured groupoid $\calg$ over a standard probability space $X$, any two action-like cocycles $\rho_G:\calg\to G$ and $\rho_H:\calg\to H$, and any measured subgroupoids $\calC_1,\calC_2\subseteq\calg$, the following two properties hold:
\begin{enumerate}
    \item $(\calC_1,\rho_G)$ is stably $\mathfrak{C}_G$ if and only if $(\calC_1,\rho_H)$ is stably $\mathfrak{C}_H$, and
    \item if $(\calC_1,\rho_G)$ and $(\calC_2,\rho_G)$ are stably  $\mathfrak{C}_G$, then they are $\mathfrak{I}_G$-adjacent if and only if $(\calC_1,\rho_H)$ and $(\calC_2,\rho_H)$ are $\mathfrak{I}_H$-adjacent. 
\end{enumerate}
\end{Def}

\begin{Rq}
    Kida shows that the curve graph is an ME-witness for an appropriate subgroup of $\Mod(\Sigma)$. We will prove in Part~\ref{Part:MEClassification} that the extension graphs of graph products of countably infinite groups defined over transvection-free graphs with no partial conjugations, are also ME-witnesses. 
\end{Rq}

In the following proposition, and throughout the paper, we equip the set $\Isom(\mathfrak{I}_G,\mathfrak{I}_H)$ of all graph isomorphisms between $\mathfrak{I}_G$ and $\mathfrak{I}_H$, with the compact-open topology, or equivalently the topology of pointwise convergence. The following proposition is inspired from \cite[Section~5]{Kid-me}.

\begin{Prop}\label{prop:witness}
    Let $G,H$ be two countable groups. Let $\mathfrak{C}_G,\mathfrak{C}_H$ be conjugacy-invariant countable sets of pairwise non-commensurable subgroups, and let $\mathfrak{I}_G,\mathfrak{I}_H$ be graphs with vertex sets $\mathfrak{C}_G,\mathfrak{C}_H$, for which the actions of $G$ and $H$ by conjugation extend to graph automorphisms. Assume that $(\mathfrak{I}_G,\mathfrak{I}_H)$ is an ME-witness for $G,H$.

    Then for every measure equivalence coupling $\Omega$ between $G$ and $H$, there exists a $(G\times H)$-equivariant measurable map $\Omega\to\Isom(\mathfrak{I}_G,\mathfrak{I}_H)$. 
\end{Prop}

\begin{proof}
In view of Lemma~\ref{lemma:groupoids-to-couplings}, it is enough to prove that for every measured groupoid $\mathcal{G}$ over a standard probability space $X$, and any two action-like cocycles $\rho_G:\calg\to G$ and $\rho_H:\calg\to H$, there exists a Borel map $\theta:X\to\Isom(\mathfrak{I}_G,\mathfrak{I}_H)$ which is $(\rho_G,\rho_H)$-equivariant in the following sense: there exists a conull Borel subset $X^*\subseteq X$ such that for every $\grpdg\in \calg_{|X^*}$, one has $\theta(y)=\rho_H(\grpdg)\theta(x)\rho_G(\grpdg)^{-1}$, where $x=s(\grpdg)$ and $y=r(\grpdg)$. We will now prove this fact.

   \begin{description}
   \item[Observation] We start by showing that if $C_1,C_2\in \mathfrak{C}_G$ verify $\rho_G^{-1}(C_1)_{|U}=\rho_G^{-1}(C_2)_{|U}$ for some Borel subset $U\subseteq X$ of positive measure, then $C_1=C_2$.\\
   For this, observe that the measured subgroupoid $\rho_G^{-1}(C_1\cap C_2)_{|U}$ is equal to both $\rho_G^{-1}(C_1)_{|U}$ and to $\rho_G^{-1}(C_2)_{|U}$. 
   As $\rho_G$ is action-like, it follows that $C_1\cap C_2$ has finite index in both $C_1$ and $C_2$. Thus $C_1$ and $C_2$ are commensurable, which by assumption on $\mathfrak{C}_G$, implies that $C_1=C_2$, as desired. \\
   The same observation also holds for $C'_1,C'_2\in\mathfrak{C}_H$, with respect to the cocycle $\rho_H$.
   \item[Definition of $\theta$] Fix $C\in V\mathfrak{I}_G=\mathfrak{C}_G$ and let $\mathcal{C}:=\rho^{-1}_G(C)$. Note that in particular $(\calC,\rho_G)$ is stably~$\mathfrak{C}_G$.
    
    Since $(\mathfrak{I}_G,\mathfrak{I}_H)$ is an ME-witness, $(\calc,\rho_H)$ is of type $\mathfrak{C}_H$. This means that there exists a partition $X^*=\sqcup_{i\in I}X_i$ of a conull Borel subset $X^*\subseteq X$ into at most countably many Borel subsets, such that for all $i\in I$ there exists ${C^\prime_i}\in V\mathfrak{I}_H=\mathfrak{C}_H$ such that $\calC_{|X_i}=\rho^{-1}_H(C^\prime_i)_{|X_i}$ (and by the above observation, $C'_i$ is unique whenever $X_i$ has positive measure). We then define for a.e.\ $x\in X$
    \begin{equation}\label{eq:DefTheta}
        \theta(x):C\in V\mathfrak{I}_G \mapsto C^\prime_i \quad \text{if} \ x\in X_i. 
    \end{equation}

     We now prove that, up to null sets, our definition of $\theta(x)(C)$ does not depend on the choice of partition. So assume that, starting from $C$, we are given two partitions as above. Let $X^*=\dunion_{i\in I}X_i$ be a common refinement, with associated subgroups $C'_{1,i}$ and $C'_{2,i}$. By construction, for every $i\in I$, we have $\rho^{-1}_H(C'_{1,i})_{|X_i}=\rho_G^{-1}(C)_{|X_i}=\rho^{-1}_H(C'_{2,i})_{|X_i}$. By the observation made at the beginning of the proof $C^\prime_{1,i}=C^\prime_{2,i}$.
    \item[Injectivity of $\theta(x)$] Let us now show, for a.e.\ $x\in X$, the injectivity of $\theta(x)$.
    For all $C_1,C_2\in V\mathfrak{I}_G$ let us define:
    \begin{equation*}
        U_{C_1,C_2}:=\left\{ x\in X \ : \ \theta(x)(C_1)=\theta(x)(C_2) \right\}.
    \end{equation*}
    We aim at showing that if $U_{C_1,C_2}$ has positive measure then $C_1=C_2$. So choose $C_1,C_2\in V\mathfrak{I}_G$ and let $U:=U_{C_1,C_2}$.
    Up to restricting to positive measure Borel subset of $U$, we will assume that $\theta(x)(C_1)=\theta(x)(C_2)$ is constant on $U$, with value denoted by $C^\prime$. We can assume that the partitions of $X$ given by $C_1$ and $C_2$ coincide and denote this common partition $X^*=\sqcup_{i\in I} X_i$. Now let $i\in I$ such that $U_i:=U\cap X_i$ has positive measure. By our definition of $\theta$, we have
    \begin{equation*}
       \rho^{-1}_G (C_1)_{|U_i}=\rho^{-1}_H(C^\prime)_{|U_i}=
       \rho^{-1}_G (C_2)_{|U_i}.
    \end{equation*}
    By the observation made at the beginning of the proof, this implies that $C_1=C_2$. 
    \item[Surjectivity of $\theta(x)$] 
    Take a vertex $C^\prime$ in $\mathfrak{I}_H$. Since $(\mathfrak{I}_G,\mathfrak{I}_H)$ is an ME-witness, there exist a conull Borel subset $X^*\subseteq X$ and a partition $X^*=\sqcup_{j\in J} Y_j$ into at most countably many Borel subsets such that for all $j\in J$ there exists $C_j\in V\mathfrak{I}_G$ such that $\rho^{-1}_H(C^\prime)_{|Y_j}=\rho^{-1}_G(C_j)_{|Y_j}$. 
    Now, for a.e.\ $x\in X$ there exists $j\in J$ such that $x$ belong to $Y_j$. Therefore, by the above definition of $\theta$, it verifies $\theta(x)(C_j)=C^\prime$ almost surely. Hence the surjectivity.
    \item[Adjacency] We now prove that for a.e.\ $x\in X$ the map $\theta(x)$ is a graph isomorphism. So consider $C_1,C_2\in V\mathfrak{I}_G$. Up to taking a common refinement of the two partitions, we can assume that the partitions of (a conull Borel subset of) $X$ given by $C_1$ and $C_2$ coincide. We denote it by $X^*=\sqcup_{i\in I}X_i$. Now let $\calC_1:=\rho^{-1}_G(C_1)$ and $\calC_2:=\rho^{-1}_G(C_2)$. For $i\in I$, denote by $C^\prime_{1,i}$ and $C^\prime_{2,i}$ the vertices of $\mathfrak{I}_H$ such that $(\mathcal{C}_1)_{|X_i}=\rho^{-1}_H(C^\prime_{1,i})_{|X_i}$ and $(\calC_2)_{|X_i}=\rho^{-1}_H(C^\prime_{2,i})_{|X_i}$. In other words $C^\prime_{1,i}=\theta(x)(C_1)$ and $C^\prime_{2,i}=\theta(x)(C_2)$ whenever $x\in X_i$. Since $(\mathfrak{I}_G,\mathfrak{I}_H)$ is an ME-witness, we get that ${C^\prime_{1,i}}$ and ${C^\prime_{2,i}}$ are distinct and adjacent if and only if $C_1$ and $C_2$ are distinct and adjacent, which is the desired conclusion.
    \item[Equivariance of $\theta$] 
    First for all $g\in G$ and all $h\in H$, let
    \begin{equation*}
        B_{g,h}:=\left\{ \grpdg \in \calg \ | \ \rho_G(\grpdg)=g, \rho_H(\grpdg)=h \right\},
    \end{equation*}
    and remark that $\calg=\sqcup_{g\in G,h\in H}B_{g,h}$. Now consider $g\in G$ and $h\in H$ and denote $B:=B_{g,h}$. Now define $U:=s(B)$ and $V:=r(B)$, which are Borel subsets of $X$ (see \cite[Theorem~18.10]{Kec}). Since $\rho_G$ is action-like, the map $f:=r\circ s^{-1}$ is a Borel isomorphism from $U$ to $V$. To prove the equivariance on $B_{g,h}$, we need to show that $\theta(r(\grpdg))=\rho_H(\grpdg)\theta(s(\grpdg))\rho_G(\grpdg)^{-1}$, that is to say for a.e.\ $x\in U$ and for all $C\in V\mathfrak{I}_G$
    \begin{equation}\label{eq:PreuveEquivarianceTheta}
        \theta\big(f(x)\big)(gCg^{-1})=h \Big( \theta(x)(C)\Big)h^{-1}.
    \end{equation}
    We now fix $C\in V\mathfrak{I}_G$. Recall that the definition of $\theta$ does not depend on the choice of partition. So we can partition $U=\sqcup_{i\in I}U_i$ and $V=\sqcup_{i\in I}V_i$ into Borel subsets, such that for all $i\in I$, we have $f(U_i)=V_i$, and the maps $x\mapsto \theta(x)(C)$ and $x\mapsto \theta(f(x))(gCg^{-1})$ are constant a.e.\ on $U_i$. We denote by $C'_{U,i},C^\prime_{V,i} \in V\mathfrak{I}_H$ their respective values. Hence by definition of $\theta$, we have
    \begin{equation}\label{eq:jaiplusdenompourcetteegalite}
      \rho^{-1}_G(C)_{|U_i}=\rho^{-1}_H(C'_{U,i})_{|U_i}
      \quad \text{and} \quad
      \rho^{-1}_G(gCg^{-1})_{|V_i}=\rho^{-1}_H(C'_{V,i})_{|V_i}.
    \end{equation}
    Now, to prove \cref{eq:PreuveEquivarianceTheta} we need to show that $C^\prime_{V,i}=h C'_{U,i}h^{-1}$ for all $i$. Using the definition of $B_{g,h}$ for the first and third equalities and \cref{eq:jaiplusdenompourcetteegalite} for the second one, we have
    \begin{align*}
        \rho^{-1}_G\left(gCg^{-1}\right)_{|V_i} 
        = B\rho^{-1}_G(C)_{|U_i} B^{-1}
        &= B\rho^{-1}_H(C'_{U,i})_{|U_i} B^{-1}\\
        &=\rho^{-1}_H\left(hC^\prime_{U,i}h^{-1} \right)_{|V_i}.
    \end{align*}
    Combining the above two equations, we deduce that $\rho_H^{-1}(C'_{V,i})_{|V_i}=\rho_H^{-1}(hC'_{U,i}h^{-1})_{|V_i}$. 
    The first observation of the proof thus yields $C^\prime_{V,i}=hC^\prime_{U,i}h^{-1}$ as desired.
    Hence the equivariance. \qedhere
    \end{description}
\end{proof}

\begin{figure}[htbp]
    \begin{center}
    \includegraphics[width=0.9\textwidth]{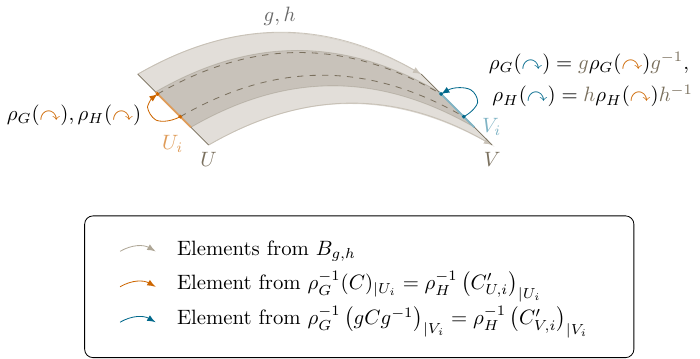}
    \end{center}
    
    \caption{Illustration for the proof of the equivariance of $\theta$}
    \label{fig:EquivarianceTheta}
\end{figure}

\newpage\part{Measure equivalence classification of graph products} \label{Part:MEClassification}
The goal of this part is to prove the following theorem, which recalls Theorem~\ref{theo:classificationversionintro} from the introduction.

\begin{Th1}
   Let $\Gamma_G,\Gamma_H$ be two finite simple graphs, not reduced to one vertex, with no transvection and no partial conjugation. Let $G,H$ be graph products of countably infinite groups over $\Gamma_G,\Gamma_H$, respectively. 
   Then the following assertions are equivalent.
   \begin{enumerate}
       \item The groups $G$ and $H$ are measure equivalent.
       \item The groups $G$ and $H$ are orbit equivalent.
       \item There exists a graph isomorphism $\sigma:\Gamma_G\to\Gamma_H$ such that for every $v\in V\Gamma_G$, the groups $G_v$ and $H_{\sigma(v)}$ are orbit equivalent.
   \end{enumerate}
\end{Th1}

\section*{Strategy of the proof} Let us describe the idea of the proof. The implication $3\Rightarrow 2$ was proved in \cite{HH21} and the implication $2\Rightarrow 1$ is obvious. We explain the implication $1\Rightarrow 2$, and give a hint about $2\Rightarrow 3$. 

So let $G$ and $H$ be as in the theorem, and assume that $G$ and $H$ are measure equivalent. Let $\Omega$ be a measure equivalence coupling between $G$ and $H$, and let $\calg$ be a measured groupoid associated to $\Omega$ as in \cref{Ex:GroupoidsAssociatedToCouplings}, coming with two action-like cocycles $\rho_G:\calg\to G$ and $\rho_H:\calg\to H$.

\paragraph{General idea and main tools} The proof that $G$ and $H$ are orbit equivalent is decomposed in two main steps.
\begin{itemize}
    \item In the first step of the proof, we work with the extension graphs $\Gamma_G^e$ and $\Gamma_H^e$, associated to $G$ and $H$. Following Kida's blueprint recalled in Section~\ref{sec:blueprint}, our goal in this step is to show that $(\Gamma_G^e,\Gamma_H^e)$ is an ME-witness between $G$ and $H$ in the sense of Definition~\ref{Def:MEWitness}. 
    
    Recall that vertices of $\Gamma_G^e,\Gamma_H^e$ are the conjugates of the vertex groups $G_v,H_w$. Thus, the key point is to show that a subgroupoid $\calv\subseteq \calg$ verifies $\calv=\rho^{-1}_G(G_v)$ for some $v\in V\Gamma_G$ if and only if (up to a countable partition of the base space) there exist $w\in V\Gamma_H$ and $h\in H$ such that $\calv=\rho^{-1}_H(hH_wh^{-1})$. This property is what we formalize under the name of \emph{Vertex Recognition Property} in \cref{de:VRP} below. Such a subgroupoid $\calv$ is called of \emph{vertex type}, with respect to either $\rho_G$ or $\rho_H$ (see \cref{Def:VertexType} for the formal definition). The proof of this Vertex Recognition Property goes through a characterization of subgroupoids of vertex type which is independent from the cocycle -- in other words, we show that there exists $v\in V\Gamma_G$ such that $\calv=\rho^{-1}_G(G_v)$ if and only if $\calv$ verifies some purely groupoid theoretic conditions. This is done in \cref{sec:strongly-reduced,sec:reducible}; we will say more about this below.
    
    \item In the second step of the proof, we use the actions of $G,H$ on their right-angled buildings $\bD_G,\bD_H$ (see \cref{subsec:ExtensionGraphs} for the definition). These are used to build a fundamental domain of $G$ and $H$ on $\Omega$ (showing that $G$ and $H$ are orbit equivalent), and also induce orbit equivalence couplings at the level of the vertex groups, in the following way.
    
    First, through Proposition~\ref{prop:witness}, \cpageref{prop:witness}, the previous step gives us a measurable $(G\times H)$-equivariant map $\Omega\to\Isom(\Gamma_G^e,\Gamma_H^e)$, where $G,H$ act by pre- and post-composition on $\Isom(\Gamma_G^e,\Gamma_H^e)$. In particular $\Gamma_G^e$ and $\Gamma_H^e$ are isomorphic. This automatically implies that $\Gamma_G$ and $\Gamma_H$ are isomorphic, using that $\Gamma_G$ and $\Gamma_H$ have no transvection and no partial conjugation \cite{Hua}. 
    
    Now, in view of the comparison between $\Isom(\Gamma_G^e,\Gamma_H^e)$ and $\Isom(\bD_G,\bD_H)$ made in Proposition~\ref{theo:autos}, \cpageref{theo:autos}, we deduce a measurable $(G\times H)$-equivariant map $\Theta$ from $\Omega$ to $\Isom(\bD_G,\bD_H)$.

    Notice that $G,H$ act transitively on the sets of rank $0$ vertices of $\bD_G,\bD_H$. Thus, the subset $Y\subseteq\Isom(\bD_G,\bD_H)$ consisting of all isomorphisms that send $\{e_G\}$ to $\{e_H\}$ (viewed as rank $0$ vertices) is a fundamental domain for both the $G$- and the $H$-action on $\Isom(\bD_G,\bD_H)$. So $\Theta^{-1}(Y)$ is a common fundamental domain for the $G$- and $H$-actions on $\Omega$. This shows that $G$ and $H$ are orbit equivalent. 

    We finally say a word about the implication $2\Rightarrow 3$.    Starting with an isomorphism $\sigma':\Gamma_G\to\Gamma_H$, the idea is to identify $v$ and $\sigma'(v)$ to rank $1$ vertices of $\bD_G,\bD_H$, and let $\Omega'_v:=\Theta^{-1}(\{f|f(v)=\sigma'(v)\})$. We show that there exists a choice of $\sigma'$ for which $\Omega^\prime_v$ has positive measure for all $v\in V\Gamma_G$. For this choice, the set $\Omega'_v$ (which is $(G_v\times H_{\sigma'(v)})$-invariant) is in fact the desired orbit equivalence coupling between $G_v$ and $H_{\sigma'(v)}$. We refer to Proposition~\ref{prop:coupling-restricted}, \cpageref{prop:coupling-restricted} for the details.  
\end{itemize}

\paragraph{More on the Vertex Recognition Property: A zoom-in process}

 As we explained in the description of the first step of the proof, in order to prove the Vertex Recognition Property, we need to give a characterization of subgroupoids of $\calg$ of vertex type that is independent from the cocycle. 
However, since we do not have any hypothesis on the vertex groups --~besides the fact that they are countably infinite~-- we cannot hope to obtain a characterization {via} their \emph{intrinsic} properties. Instead, we rely on how they behave \emph{inside} the bigger groupoid $\calG$, to get the needed characterization. 

The central combinatorial idea behind our proof is a \emph{zoom-in} process: starting from~$G$, vertex groups are obtained by successively passing to maximal (parabolic) product subgroups, and then restricting to factors. More precisely, given a vertex group $G_v$, we can find two finite sequences $(F_j)_j$ and $(P_j)_j$ of parabolics subgroups of $G$ such that 
\begin{itemize}
    \item $G=F_0\supseteq P_1\supseteq F_1\supseteq\dots\supseteq P_n\supseteq F_n=G_v$;
    \item And $P_{j+1}$ is a maximal product inside $F_j$;
    \item And $F_j$ is a factor of $P_j$ (and $F_n$ is just the clique factor of $P_n$).
\end{itemize} 
We refer to Lemma~\ref{lemma:graph}, \cpageref{lemma:graph} for a precise statement. Conversely, we manage to use this zoom-in process as a starting point towards an algebraic characterization of (conjugates of) vertex groups. 

One important feature here is that the properties we use in the process (being a maximal product, being a factor) can in fact be completely characterized just in terms of amenability of certain subgroups, and inclusion or normalization of one subgroup by another --~this allows for a translation to the groupoid-theoretic setting, using the notions of amenability and normality recalled in Section~\ref{sec:background-groupoids}. More precisely, we characterize --~independently from the cocycle $\rho_G$~-- subgroupoids of the form $\calp=\rho^{-1}_G(P)$ for $P$ some maximal product in $G$, using properties of their behavior inside of $\calg$. Likewise, we then characterize subgroupoids of the form $\calf=\rho^{-1}_G(F)$ where $F$ is a factor of $P$. Iterating this process on the sequences of parabolics given by the above zoom-in process, we are then able to characterize vertex type groupoids and show our Vertex Recognition Property.

The details of the arguments used to “recognize” maximal products and factors are explained in the introductory part of Section~\ref{sec:strongly-reduced}, to which the reader is refered.
Readers less familiar with measured groupoids can also read \cref{Appendix:Isom} where we prove the analogue of \cref{theo:classificationversionintro} with regard to isomorphism classification. The proof in this appendix can serve as a guideline (written in the group theoretic framework) through \cref{sec:strongly-reduced}.

\section*{\texorpdfstring{Structure of \cref{Part:MEClassification}}{Structure of the part}}
 Our proof of Theorem~\ref{theo:classificationversionintro} occupies Sections~\ref{sec:proof-from-VRP} to~\ref{sec:reducible} -- Section~\ref{Sec:AmenalbeUntransvectable} is not needed for the main theorem, and is only a slight variation over the Vertex Recognition Property, in preparation for Part~\ref{Part:QuantitativeResults} of the present work. Examples illustrating Theorem~\ref{theo:classificationversionintro} and variations obtained along the way, are provided in \cref{Sec:ExemplesMEVRP}.
 
  In Section~\ref{sec:proof-from-VRP}, we formulate the Vertex Recognition Property, and explain how to prove Theorem~\ref{theo:classificationversionintro} assuming this property has been established. In other words, Section~\ref{sec:proof-from-VRP} carries Step~2 from the above proof outline, which exploits the right-angled buildings of the graph products. Section~\ref{sec:proof-from-VRP} also contains the proofs of the commensurability classification theorem  (Theorem~\ref{theo:classification-commensurability}), as well as the application to fundamental groups of associated von Neumann algebras (Corollary~\ref{cor:fundamental}, \cpageref{cor:fundamental}), both derived from the statement in measure equivalence.
  
  The next three sections are devoted to the proof of the Vertex Recognition Property. Section~\ref{sec:graph} contains preparatory tools of two sorts. First, it contains some combinatorial lemmas, namely:
\begin{itemize}
    \item we establish the combinatorial lemma that provides the zoom-in process described above (Section~\ref{sec:zoom-in});
    \item we show in Section~\ref{sec:strongly-reduced-def} that if $\Gamma$ is transvection-free and has no partial conjugation, then either $\Gamma$ splits as a join (and $G$ splits as a product), or else $\Gamma$ is strongly reduced in the sense that it provides a minimal decomposition of $G$ as a graph product. Our proof of the Vertex Recognition Property will be split into this two cases in later sections.
\end{itemize}
Section~\ref{sec:graph} also contains some groupoid-theoretic tools in preparation for later sections. 

Our proof now follows the dichotomy provided by the second combinatorial lemma. Namely, in Section~\ref{sec:strongly-reduced}, we establish the Vertex Recognition Property when the defining graph is strongly reduced. And in Section~\ref{sec:reducible}, we extend this to include graph products over graphs that split as a join.

Finally \cref{Sec:ExemplesMEVRP} provides some examples, illustrating the main results and the different cases treated in the present part.

\section{From vertex recognition to rigidity}\label{sec:proof-from-VRP}
In the following section we formalize the aforementioned Vertex Recognition Property and show how to use it to obtain our main classification theorem (\cref{theo:classificationversionintro}).
We start by applying Kida’s blueprint (detailed in \cref{sec:blueprint}) to factorize our coupling through the extension graphs (see \cref{sec:factoring}) and prove our main theorem in \cref{subsec:ProofClassificationTh}. We then turn in \cref{Sec:ApplicationsMainTheorem} to applications of the Vertex Recognition Property to the classifications up to commensurability and isomorphisms and to the study of fundamental groups of equivalence relations and associated von Neumann algebras.

\subsection{The Vertex Recognition Property}\label{subsec:VRP}

Let $G$ be a graph product. Let $\calg$ be a measured groupoid over a standard probability space $X$, equipped with a cocycle $\rho:\calg\to G$, and let $\calv\subseteq\calg$ be a measured subgroupoid.
We say that $(\calv,\rho)$ is \emph{of vertex type}\label{Def:VertexType}\index{Vertex type (groupoid framework)} if there exists a countable Borel partition $X^*=\dunion_{i\in I}X_i$ of a conull Borel subset $X^*\subseteq X$, and for every $i\in I$, a parabolic subgroup $P_i\subseteq G$ that is conjugate to a vertex group, such that $\calv_{|X_i}=\rho^{-1}(P_i)_{|X_i}$. 

A crucial tool in our proof of Theorem~\ref{theo:classificationversionintro} is a recognition technique phrased in the language of measured groupoids.

\begin{Def}[Vertex Recognition Property]\label{de:VRP}
A class $\classC$ of graph products satisfies the \emph{Vertex Recognition Property}\index{Vertex Recognition Property} if for every $G,H\in \classC$, every measured groupoid $\calg$ over a standard probability space, equipped with action-like cocycles $\rho_G:\calg\to G$ and $\rho_H:\calg\to H$, and every measured subgroupoid $\calv\subseteq\calg$, the pair $(\calv,\rho_G)$ is of vertex type if and only if $(\calv,\rho_H)$ is of vertex type.
\end{Def}

\begin{Rq}\label{rk:VRP-transvection-free}
   Note that the Vertex Recognition Property for the class $\classC$ implies that for every $G\in\classC$, every automorphism $\varphi$ of $G$ sends every vertex group to a conjugate of a vertex group: indeed, this is proved by taking $G=H$, $\calg=G$, $\rho_G=\mathrm{id}$ and $\rho_H=\varphi$. In fact our notion is a groupoid-theoretic extension of this algebraic property. 
\end{Rq}

The Vertex Recognition Property holds in the context of interest to us.

\begin{Prop}\label{prop:vertex-recognition}
The class of all graph products of countably infinite groups over finite simple graphs not reduced to one vertex, which are transvection-free and have no partial conjugation, has the Vertex Recognition Property.
\end{Prop}

The proof of Proposition~\ref{prop:vertex-recognition} is postponed to Sections~\ref{sec:strongly-reduced} (where it is proved in the irreducible case) and~\ref{sec:reducible} (where it is proved in the reducible case). In the present section, we explain how we derive our classification theorem (Theorem~\ref{theo:classificationversionintro}) from the Vertex Recognition Property. In other words, we carry out the second point from the strategy of the proof described at the beginning of Part~\ref{Part:MEClassification} of the present work.

With the language and notations from Section~\ref{sec:blueprint} (Kida’s blueprint), we will let $\mathfrak{C}_G$ be the set of all conjugates of vertex groups of $G$, and let $\mathfrak{I}_G=\Gamma_G^e$ be the extension graph of $G$, a graph acted upon by $G$ whose vertex set is precisely $\mathfrak{C}_G$. In this language, a subgroupoid of vertex type is exactly a subgroupoid which is stably $\mathfrak{C}_G$. The first step in our proof of Theorem~\ref{theo:classificationversionintro}, carried in Section~\ref{sec:factoring}, consists in showing that $(\mathfrak{I}_G,\mathfrak{I}_H)$ is an ME-witness for $G,H$ (see \cpageref{Def:MEWitness} for the definition). This requires showing that adjacency in the extension graph is visible at the groupoid level. The second step, carried in Section~\ref{subsec:ProofClassificationTh}, is where we depart from Kida's strategy for measure equivalence classification: we will use the right-angled buildings of the graph products as a tool for keeping track of volumes and passing from measure equivalence to orbit equivalence.

\subsection{Factoring the self-coupling}\label{sec:factoring}
The purpose of this section is to apply Kida’s blueprint (\cref{sec:blueprint}). The first step is to obtain a characterization of adjacency in the extension graph $\Gamma^e_G$ (i.e.\ commutation) in terms of measured groupoids. More precisely, given two groupoids we are able to recognize when their image by a cocycle will give adjacent parabolic subgroups \emph{independently from the cocycle}.

\begin{Lmm}\label{lemma:adjacency}
    Let $G$ be a graph product of countably infinite groups over a finite simple graph. Let $\calg$ be a measured groupoid over a standard probability space $X$, and let $\rho:\calg\to G$ be an action-like cocycle. Let $P,Q$ be two parabolic subgroups that are conjugate to vertex groups. Let $\calp=\rho^{-1}(P)$ and $\calQ=\rho^{-1}(Q)$. The following assertions are equivalent.
    \begin{enumerate}
        \item The groups $P$ and $Q$ are distinct and commute.
        \item The groupoid $\calQ$ normalizes $\calp$ and is different from $\calp$. 
        \item There exists a positive measure Borel subset $U\subseteq X$ such that $\calQ_{|U}$ normalizes $\calp_{|U}$, and $\calQ_{|U}\neq\calp_{|U}$.
    \end{enumerate}
\end{Lmm}

\begin{proof}
    The assertion $1\Rightarrow 2$ follows from Example~\ref{ex:normal}, \cpageref{ex:normal} (for the fact that $\calq$ normalizes $\calp$), and from Lemma~\ref{lemma:inclusion-parabolics}, \cpageref{lemma:inclusion-parabolics} (for the fact that $\calq\neq\calp$). The assertion $2\Rightarrow 3$ is clear. For $3\Rightarrow 1$, notice first that $(\calQ,\rho)$ is tightly $Q$-supported (Lemma~\ref{lemma:tight-support}), and therefore $(\calQ_{|U},\rho)$ is tightly $Q$-supported. As $\calQ_{|U}$ normalizes $\calp_{|U}$, Lemma~\ref{lemma:support-normal} implies that (up to restricting to a conull Borel subset) $\calq_{|U}\subseteq\rho^{-1}(P\times P^\perp)_{|U}$. It thus follows that $Q\subseteq P\times P^{\perp}$. As $Q$ is conjugate to a vertex group, this implies that $Q$ is either equal to $P$ or commutes with $P$. The fact that $\calq_{|U}\neq\calp_{|U}$ implies that $P\neq Q$, using Lemma~\ref{lemma:inclusion-parabolics}.
\end{proof}

For the following lemma, we recall that $\Isom(\Gamma_G^e,\Gamma_H^e)$ is equipped with the topology of pointwise convergence, and with the action of $G\times H$ given by pre- and post-composition. 

\begin{Lmm}\label{Lmm:CouplingAndGraphIsomorphism}
Let $\classC$ be a class of graph products that satisfies the Vertex Recognition Property. Let $G,H\in\classC$. 
     
     For every measure equivalence coupling $\Omega$ between $G$ and $H$, there exists a $(G\times H)$-equivariant Borel map $\Omega\to\Isom(\Gamma_G^e,\Gamma^e_H)$. (In particular, if $G$ and $H$ are measure equivalent, then $\Gamma_G^e$ and $\Gamma_H^e$ are isomorphic.)
\end{Lmm}

\begin{proof}
    Let $\mathfrak{C}_G$ be the set of conjugates of vertex groups of $G$, and let $\mathfrak{I}_G=\Gamma_G^e$ be the extension graph of $G$. Define $\mathfrak{C}_H$ and $\mathfrak{I}_H$ in the same way. The Vertex Recognition Property, together with the recognition of adjacency given by Lemma~\ref{lemma:adjacency}, ensure that $(\mathfrak{I}_G,\mathfrak{I}_H)$ is an ME-witness for $G,H$ (see Definition~\ref{Def:MEWitness}). The conclusion thus follows from Proposition~\ref{prop:witness}.
    \end{proof}

In the sequel, it will be useful to have a version of the above lemma at the level of the right-angled buildings of the graph products (see Section~\ref{subsec:ExtensionGraphs} for definitions). Similarly to $\Isom(\Gamma_G^e,\Gamma_H^e)$, we will always equip $\Isom(\bD_G,\bD_H)$ with the compact-open topology, and with the action of $G\times H$ by pre- and post-composition.

\begin{Cor}\label{cor:factor-through-building}
Let $\classC$ be a class of graph products that satisfies the Vertex Recognition Property. Let $G,H\in\classC$. 
     
     For every measure equivalence coupling $\Omega$ between $G$ and $H$, there exists a $(G\times H)$-equivariant Borel map $\Omega\to\Isom(\bD_G,\bD_H)$.
\end{Cor}

\begin{proof}
    This is a consequence of Lemma~\ref{Lmm:CouplingAndGraphIsomorphism}, and the existence of a $(G\times H)$-equivariant Borel map $\Isom(\Gamma_G^e,\Gamma_H^e)\to\Isom(\bD_G,\bD_H)$ (Proposition~\ref{theo:autos}).
\end{proof}

We now deduce another consequence of Lemma~\ref{lemma:adjacency}. We mention that this will not be used to prove \cref{theo:classificationversionintro}, but will allow us to obtain classification results beyond the case covered by the latter theorem (see for example \cref{Sec:ExemplesMEVRP}) as well as quantitative statements (see \cref{Sec:ProdToFact}).

\begin{Cor}\label{Cor:VRPandGraphMorphism}
    Let $\classC$ be a class of graph products that satisfies the Vertex Recognition Property. Let $G,H\in\classC$ defined over $\Gamma_G$ and $\Gamma_H$ respectively.
    
    If $G$ and $H$ are measure equivalent, then there exists a graph homomorphism $\sigma:\Gamma_G \rightarrow \Gamma_H$ such that for all $v\in V\Gamma_G$, the group $G_v$ is measure equivalent to $H_{\sigma(v)}$.
\end{Cor}
The proof uses the Vertex Recognition Property to define the wanted map $\sigma$, and \cref{lemma:adjacency} to show that the latter is a graph homomorphism.

\begin{proof}
Let $\Omega$ be a measure equivalence coupling between $G$ and $H$. As in \cref{Ex:GroupoidsAssociatedToCouplings,Ex:CocycleGrAssociatedToAME}, let $X\subseteq \Omega$ be a Borel subset of finite positive measure, and $\calg$ be a measured groupoid over $X$, coming with action-like cocycles $\rho_G:\calg\to G$ and $\rho_H:\calg\to H$. 
\begin{itemize}
        \item Let us first show that there exists a Borel subset $U\subseteq X$ of positive measure such that for all $v\in V\Gamma_G$ there exist $h_v\in H$ and $\sigma(v)\in V\Gamma_H$ such that
        \begin{equation*}
            {\rho^{-1}_G(G_v)}_{|U}=\rho^{-1}_H\left(h_vH_{\sigma(v)}h^{-1}_v\right)_{|U}.
        \end{equation*}
        Since $\classC$ has the Vertex Recognition Property, for all $v\in V\Gamma_G$ there exist a countable Borel partition $X^*(v)=\dunion_{i\in I(v)}X_i(v)$ of a conull Borel subset $X^*(v)\subseteq X$ (depending a priori on $v$), and for every $i\in I(v)$, a parabolic subgroup $P_i(v)\subseteq H$ that is conjugate to a vertex group, such that 
        \begin{equation}\label{eq:VRPpourMorphisme}
        {\rho^{-1}_G(G_v)}_{|X_i(v)}=\rho^{-1}_H\left(P_i(v)\right)_{|X_i(v)}.
        \end{equation}
        Since $V\Gamma_G$ is finite, up to taking the common refinement of all the partitions $(X_i(v))_i$ and the intersection of all the subsets $X^*(v)$, we can obtain 
        a countable Borel partition $X^*=\dunion_{i\in I}X_i$ of a conull Borel subset $X^*$ such that \cref{eq:VRPpourMorphisme} is verified and that does not depend on $v$. Remark that since $X^*$ has full measure, there thus exists $i\in I$ such that $\mu(X_i)>0$. For such an $i\in I$, for all $v\in V\Gamma_G$ there exist (by \cref{eq:VRPpourMorphisme}) an element $h_v\in H$ and $\sigma(v)\in V\Gamma_H$ such that 
        \begin{equation}\label{eq:chaispas}
        {\rho_G^{-1}(G_v)}_{|X_i}=\rho_H^{-1}\left(h_vH_{\sigma(v)}h^{-1}_v\right)_{|X_i}.
        \end{equation}
        Taking $U=X_i$ thus leads to the wanted assertion. 
        \item The above \cref{eq:VRPpourMorphisme} moreover implies that $G_v$ and $h_vH_{\sigma(v)}h^{-1}_v$ are measure equivalent for all $v\in V\Gamma_G$ in view of Lemma~\ref{lemma:me-subgroups}, \cpageref{lemma:me-subgroups} and Remark~\ref{Rq:CouplageRestriction}, \cpageref{Rq:CouplageRestriction}. We can thus define a measure equivalence between $G_v$ and $H_{\sigma(v)}$ for all $v\in V\Gamma_G$.
        \item Let us now prove that $\sigma$ is a graph homomorphism.\\
        If $v,w\in V\Gamma_G$ are adjacent in $\Gamma_G$, then $G_v$ and $G_w$ are distinct and commute and therefore by \cref{lemma:adjacency,eq:chaispas}, $h_vH_{\sigma(v)}h^{-1}_v$ and $h_wH_{\sigma(w)}h^{-1}_w$ are also distinct and commute. In particular \cref{lemma:adjacency-extension}, \cpageref{lemma:adjacency-extension} implies that $\sigma(v)$ and $\sigma(w)$ are adjacent in $\Gamma_G$. Hence $\sigma$ is a graph homomorphism.\qedhere
    \end{itemize}
\end{proof}

\begin{Rq}\label{Rq:56}
The class of groups covered by Corollary~\ref{Cor:VRPandGraphMorphism} is larger than those covered by our main theorem (Theorem~\ref{theo:classificationversionintro}): indeed, it applies to all (direct products of) graph products of countably infinite groups over transvection-free strongly reduced\footnote{see Definition~\ref{de:strongly-reduced}, \cpageref{de:strongly-reduced}.} graphs (not reduced to a vertex), without the necessity to require that the graphs have no partial conjugation (see \cref{prop:vertex-recognition-join}, \cpageref{prop:vertex-recognition-join}). 

However, the conclusion is weaker in two ways. First, the homomorphism $\sigma$ need not be an isomorphism. Second, we only get a measure equivalence between the vertex groups, not an orbit equivalence.

In fact, the conclusion of Corollary~\ref{Cor:VRPandGraphMorphism} cannot be strengthened (even if we had assumed $G$ and $H$ to be orbit equivalent to start with). Here is an example, illustrated in Figure~\ref{fig:Rk56}. The groups $G_1$ and $G_2$ are two index $2$ subgroups of $G$. Indeed $G_1$ is isomorphic to the kernel of the homomorphism $G\to\mathbb{Z}/2\mathbb{Z}$, sending the first generator $a$ of $G_v\simeq F_2$ to $1$ and the second generator $b$ to $0$, and sending all other vertex groups to $\{0\}$. And $G_2$ is isomorphic to the kernel of the homomorphism $G\to\mathbb{Z}/2\mathbb{Z}$ sending the generator of $G_w\simeq\mathbb{Z}$ to $1$, and sending all other vertex groups to $\{0\}$. That $G_1$ and $G_2$ are indeed as described can be checked, for instance, by using the Reidemeister--Schreier method to obtain a presentation of a finite-index subgroup of $G$, see e.g.\ \cite[Proposition~II.4.1]{LS}. 

As a consequence of the second point of \cref{lemma:commensurability}, \cpageref{lemma:commensurability}, being subgroups of the same finite index in $G$, the groups $G_1$ and $G_2$ are orbit equivalent. And they satisfy the conclusion of Corollary~\ref{Cor:VRPandGraphMorphism}, however the homomorphism $\sigma$ that appear cannot be taken to be an isomorphism, and the vertex groups $F_2$ and $F_3$ are measure equivalent, but not orbit equivalent \cite{Gab-cost}.
\end{Rq}

\begin{figure}
    \centering
    \includegraphics[width=\textwidth]{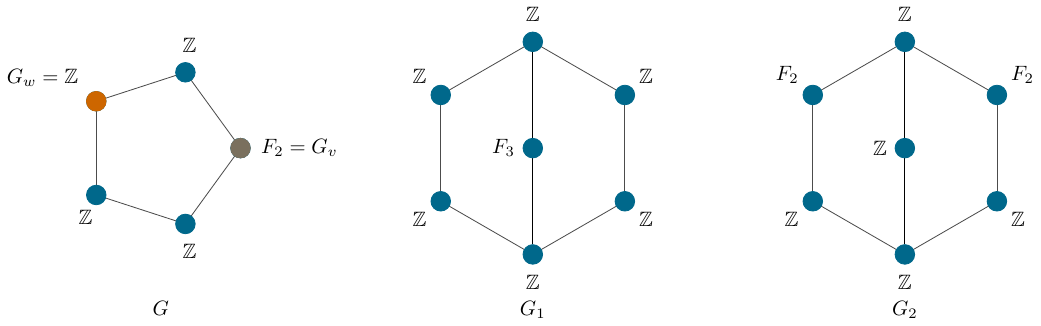}
    \caption{Groups from \cref{Rq:56}}
    \label{fig:Rk56}
\end{figure}
In Example~\ref{Ex:PentaDeplie}, we will further illustrate Corollary~\ref{Cor:VRPandGraphMorphism} by obtaining new examples of right-angled Artin groups that can be distinguished in measure equivalence.

\subsection{Proof of the measure equivalence classification theorem} \label{subsec:ProofClassificationTh}
We now turn to the proof of \cref{theo:classificationversionintro} and start with the following statement.

\begin{Prop}\label{prop:me-oe}
    Let $G$ and $H$ be graph products of countably infinite vertex groups over finite simple graphs with no transvection and no partial conjugation, not reduced to one vertex.

    Then every measure equivalence coupling $\Omega$ between $G$ and $H$ is an orbit equivalence coupling between $G$ and $H$, i.e.\ $G$ and $H$ have a common fundamental domain on $\Omega$.
\end{Prop}

\begin{proof}
    First note that  by \cref{prop:vertex-recognition}, the groups $G$ and $H$ belong to a same class $\classC$ that verifies the Vertex Recognition Property. 

    By Corollary~\ref{cor:factor-through-building}, there exists a $(G\times H)$-equivariant Borel map $\Theta:\Omega \rightarrow \Isom(\bD_G,\bD_H)$. 
    
     Recall from Corollary~\ref{cor:rank} that every isomorphism from $\bD_G$ to $\bD_H$ sends rank $0$ vertices to rank $0$ vertices. Additionally the respective actions of $G$ and $H$ on rank $0$ vertices of $\bD_G$ and $\bD_H$ are transitive. Therefore, identifying the neutral elements $e_G$ and $e_H$ with the associated rank $0$ vertices of $\bD_G$ and $\bD_H$, we deduce that \[Y:=\left\{f\in \Isom(\bD_G,\bD_H)\ : \ f(e_{G})=e_{H} \right\}\] is a fundamental domain for both the $G$- and the $H$-actions on $\Isom(\bD_G,\bD_H)$. Hence, by equivariance of $\Theta$, the set $X:=\Theta^{-1}(Y)$ is a common Borel fundamental domain for the actions of $G$ and $H$ on $\Omega$. This shows that $\Omega$ is an orbit equivalence coupling.
\end{proof}

We now prove the following statement which shows the implication $1 \Rightarrow 3$ of our main theorem. 

\begin{Prop}\label{prop:coupling-restricted}
Let $G,H$ be graph products of countably infinite vertex groups over finite simple graphs $\Gamma_G,\Gamma_H$ with no transvection and no partial conjugation, not reduced to one vertex. Let $(\Omega,m)$ be a measure equivalence coupling between $G$ and $H$.

Then there exists a graph isomorphism $\sigma:\Gamma_G\to\Gamma_H$ such that for every $v\in V\Gamma_G$, there exists a $(G_v\times H_{\sigma(v)})$-invariant Borel subset $\Omega_v\subseteq\Omega$, which is an orbit equivalence coupling between $G_v$ and $H_{\sigma(v)}$.
\end{Prop}

\begin{Rq}\label{rk:requirement}
    We will see that the coupling we construct satisfies $m(g\Omega_v\cap\Omega_v)=0$ for every $g\in (G\setminus G_v)\cup (H\setminus H_{\sigma(v)})$. In particular, every Borel fundamental domain for the $G_v$-action (resp.\ $H_{\sigma(v)}$-action) on $\Omega_v$ is contained in a Borel fundamental domain for the $G$-action (resp.\ $H$-action) on $\Omega$.
\end{Rq}

\begin{proof}
Let $\Omega$ be a measure equivalence coupling between $G$ and $H$. Proposition~\ref{prop:me-oe} implies that $\Omega$ is in fact an orbit equivalence coupling. As $G$ and $H$ are measure equivalent, the extension graphs $\Gamma_G^e$ and $\Gamma_H^e$ are isomorphic (Lemma~\ref{Lmm:CouplingAndGraphIsomorphism}). As the graphs $\Gamma_G$ and $\Gamma_H$ are transvection-free and have no partial conjugations, it follows from \cite[Corollary~4.16 and Lemma~4.17]{Hua} that $\Gamma_G$ and $\Gamma_H$ are isomorphic --~so $\Isom(\Gamma_G,\Gamma_H)$ is non-empty.
    
As in the previous proof, we identify the neutral elements $e_G,e_H$ with rank $0$ vertices of $\bD_G,\bD_H$, and recall that the set 
\[Y:=\left\{f\in \Isom(\bD_G,\bD_H)\, : \, f(e_G)=e_H \right\}\] is a fundamental domain for both the $G$- and the $H$-action on~$\Isom(\bD_G,\bD_H)$. We also identify every vertex $v\in V\Gamma_G$ (resp.\ $w\in V\Gamma_H$) with the rank $1$ vertex of $\bD_G$ (resp.\ $\bD_H$) corresponding to the coset $G_v$ (resp.\ $H_w$). We let $\Theta:\Omega \rightarrow \Isom(\bD_G,\bD_H)$ be a $(G\times H)$-equivariant Borel map given by Corollary~\ref{cor:factor-through-building}. Now, for all $\sigma \in \Isom(\Gamma_G,\Gamma_H)$ and all $v\in V\Gamma_G$ let 
        \begin{align*}
            Y^v_\sigma&:=\left\{f\in \Isom(\bD_G,\bD_H)\, : \, f(v)=\sigma(v), \, f(e_{G})=e_{H} \right\},\\
             \Isom_{v\mapsto \sigma(v)}(\bD_G,\bD_H)&:=\left\{f\in \Isom(\bD_G,\bD_H)\, : \, f(v)=\sigma(v) \right\},\\
             \Omega^v_{\sigma}&:=\Theta^{-1}\left(\Isom_{v\mapsto \sigma(v)}(\bD_G,\bD_H) \right),\\
            X^v_{\sigma}&:=\Theta^{-1}(Y^v_\sigma).
        \end{align*}
        We claim that there exists $\sigma_0\in \Isom(\Gamma_G,\Gamma_H)$ such that for all $v\in V\Gamma_G$, one has $m(X^v_{\sigma_0})>0$. Indeed, for $\sigma \in \Isom(\Gamma_G,\Gamma_H)$, let \[Y_\sigma:=\left\{f\in \Isom(\bD_G,\bD_H)\, : \, f(v)=\sigma(v), \ \forall v\in V\Gamma_G \right\}.\] 
        Observe that $Y=\sqcup_{\sigma} Y_\sigma$. Indeed, we have $\sqcup_{\sigma} Y_\sigma\subseteq Y$ because $e_G$ (resp.\ $e_H$) is the unique rank $0$ vertex adjacent to all rank $1$ vertices corresponding to cosets of the form $G_v$ (resp.\ $H_w$). For the reverse inclusion, notice that any isomorphism $f\in\Isom(\bD_G,\bD_H)$ induces a bijection $\sigma$ from the set of rank $1$ vertices adjacent to $e_G$, to the set of rank $1$ vertices adjacent to $f(e_G)$. Now if $f$ belongs to $Y$, namely $f(e_G)=e_H$, we view $\sigma$ as a bijection from $V\Gamma_G$ to $V\Gamma_H$. Observe that this bijection preserves adjacency and non-adjacency: this follows from the fact that $f$ preserves ranks of vertices (Corollary~\ref{cor:rank}), after observing that two rank one vertices $v_1,v_2$ correspond to adjacent vertices of $\Gamma_G$ if and only if they are adjacent to a common rank $2$ vertex in $\bD_G$. This proves that $Y=\sqcup_{\sigma} Y_\sigma$. 
        
        Since $Y$ is a fundamental domain for the $G$- and $H$-actions on $\Isom(\bD_G,\bD_H)$ and $\Theta$ is equivariant, the Borel set $\Theta^{-1}(Y)$ is a fundamental domain for the $G$- and $H$-actions on $\Omega$. Therefore, there exists $\sigma_0\in \Isom(\Gamma_G,\Gamma_H)$ such that $m\left(\Theta^{-1}(Y_{\sigma_0})\right)>0$. Noting that for all $v\in V\Gamma_G$ we have $Y_{\sigma_0} \subset Y^v_{\sigma_0}$, we obtain that $m\left(X^v_{\sigma_0}\right)>0$ for all $v\in V\Gamma_G$, which proves our claim. 

       Notice that for every $v\in V\Gamma_G$, the set $\Omega^{v}_{\sigma_0}$ is invariant under $G_v\times H_{\sigma_0(v)}$. Let us now prove that $X^v_{\sigma_0}$ is a common fundamental domain for the actions of $G_v$ and $H_{\sigma_0(v)}$ on $\Omega_v:=\Omega^{v}_{\sigma_0}$. By equivariance of $\Theta$, it is enough to verify that $Y^v_{\sigma_0}$ is a fundamental domain for both actions on $\Isom_{v\mapsto \sigma_0(v)}(\bD_G,\bD_H)$. First, since the action of $H$ on the set of rank $0$ vertices of $\bD_H$ is free, the $H_{\sigma_0(v)}$-translates of $Y_{\sigma_0}^v$ are pairwise disjoint (the translate of $Y^v_{\sigma_0}$ by an element $h$ consists of isomorphisms $f\in \Isom(\bD_G,\bD_H)$ sending $e_G$ to $h$). And likewise for their $G_v$-translates.
        Second, pick $f\in \Isom_{v\mapsto \sigma_0(v)}(\bD_G,\bD_H)$, then $f(v)=\sigma_0(v)$. Since $e_G$ and $v$ are adjacent (as vertices of $\bD_G$), so are $f(e_G)$ and $f(v)=\sigma_0(v)$. Therefore $f(e_G)$ being of rank $0$, it corresponds to an element --~denoted by $h$~-- of $H_{\sigma_0(v)}$ (see \cref{fig:DeligneAndOE}, \cpageref{fig:DeligneAndOE}). Hence $h^{-1}\cdot f$ belongs to $Y^v_{\sigma_0}$ and $Y^v_{\sigma_0}$ is thus a fundamental domain for the action of $H_{\sigma_{0}(v)}$. 
        Similarly, for the $G_v$-action, since $e_H$ and $\sigma_0(v)$ are adjacent so are $f^{-1}(e_H)$ and $f^{-1}(\sigma_0(v))=v$. Therefore $f^{-1}(e_H)$ is an element of $G_v$ (see again \cref{fig:DeligneAndOE}) and as above, we get that $Y^v_{\sigma_0}$ is a fundamental domain for the $G_v$-action.
        
        Thus, for all $v\in V\Gamma_G$, the subspace $\Omega_v$ is an orbit equivalence coupling between $G_v$ and $H_{\sigma_0(v)}$ (and by construction it also satisfies the requirement from Remark~\ref{rk:requirement}).
\end{proof}

\begin{figure}
    \centering
    \includegraphics[width=\textwidth]{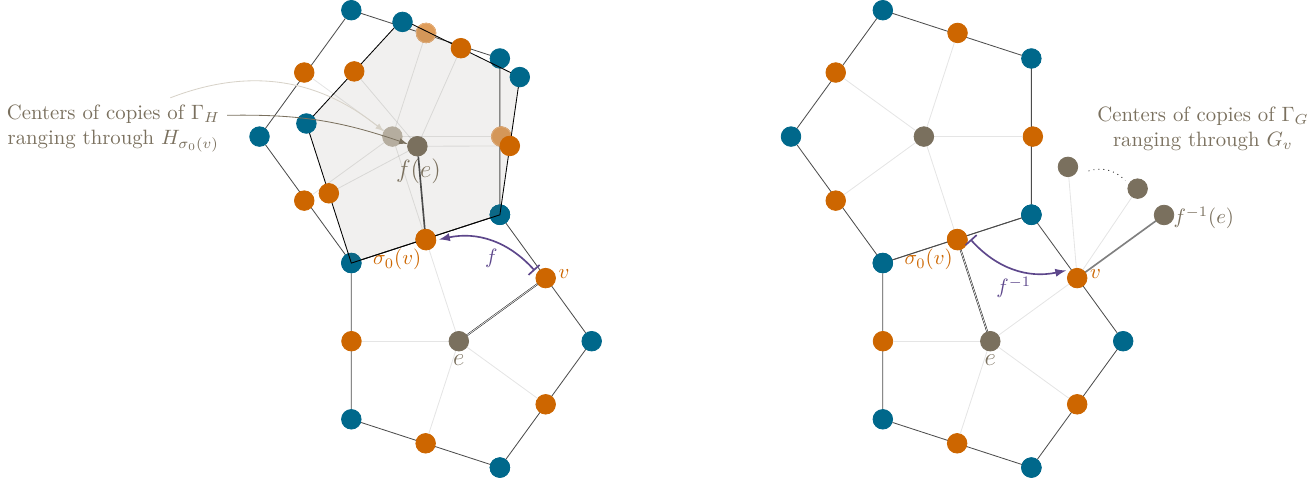}
    \caption{Illustration of the end of the proof of \cref{prop:coupling-restricted} for $\Gamma_G$ and $\Gamma_H$ pentagons}
    \label{fig:DeligneAndOE}
\end{figure}

We can finally complete the proof of our main measure equivalence classification theorem.

\begin{proof}[Proof of Theorem~\ref{theo:classificationversionintro}]
The implication $2 \Rightarrow 1$ is immediate and $3 \Rightarrow 2$ comes from \cite[Proposition~4.2]{HH21}. 
Finally $1\Rightarrow 3$ is the contents of Proposition~\ref{prop:coupling-restricted}.
\end{proof}

\subsection{Applications: Commensurability and von Neumann algebras}\label{Sec:ApplicationsMainTheorem}

We now prove \cref{theo:classification-commensurability} regarding the commensurability classification. We also give applications to fundamental groups and von Neumann algebras.

\subsubsection{Application to commensurability}

By restricting to discrete measure equivalence couplings (i.e.\ $\Omega$ is countable), we complete our proof of the commensurability classification theorem that we recall below.

\begin{Th2} 
   Let $\Gamma_G,\Gamma_H$ be two finite simple graphs, not reduced to one vertex, with no transvection and no partial conjugation. Let $G,H$ be graph products of countably infinite groups over $\Gamma_G,\Gamma_H$, respectively. Then the following assertions are equivalent.
   \begin{enumerate}
       \item The groups $G$ and $H$ are commensurable.
       \item The groups $G$ and $H$ are strongly commensurable.
       \item There exists a graph isomorphism $\sigma:\Gamma_G\to\Gamma_H$ such that for every $v\in V\Gamma_G$, the groups $G_v$ and $H_{\sigma(v)}$ are strongly commensurable.
   \end{enumerate}
\end{Th2}

\begin{proof}
    The implication $2\Rightarrow 1$ is obvious, and the implication $3\Rightarrow 2$ was proved by Januszkiewicz--\'Swi\c{a}tkowski \cite[Theorem~1]{JS}. For $1\Rightarrow 3$, since $G$ and $H$ are commensurable, there exists a discrete measure equivalence coupling $\Omega$ between $G$ and $H$ (Lemma~\ref{lemma:commensurability}). Proposition~\ref{prop:coupling-restricted} thus yields a graph isomorphism $\sigma:\Gamma_G\to\Gamma_H$ such that for every $v\in V\Gamma_G$, there exists a $(G_v\times H_{\sigma(v)})$-invariant subset $\Omega_v\subseteq\Omega$ which is a (discrete) orbit equivalence coupling between $G_v$ and $H_{\sigma(v)}$, i.e.\ $G_v$ and $H_{\sigma(v)}$ have a common (finite) fundamental domain. This exactly says that $G_v$ and $H_{\sigma(v)}$ are strongly commensurable in view of Lemma~\ref{lemma:commensurability}.
\end{proof}

\begin{Rq}
    Similarly as in \cref{Cor:VRPandGraphMorphism} we can also show that if $G$ and $H$ are commensurable and belong to a same class of graph products having the Vertex Recognition Property, then there exists a graph homomorphism $\sigma:\Gamma_G\rightarrow \Gamma_H$ such that $G_v$ is commensurable to $H_{\sigma(v)}$ for all $v\in V\Gamma_G$.
\end{Rq}

\begin{Rq}
By restricting to countable couplings $\Omega$ on which the actions of $G$ and $H$ are free and transitive, we can derive that the same conclusion as in Theorem~\ref{theo:classification-commensurability} holds after replacing “commensurable” and “strongly commensurable” by “isomorphic”. However, for isomorphisms, the statement can be made more general, in particular vertex groups do not need to be assumed countably infinite, and we can consider more general classes of graphs $\Gamma_G,\Gamma_H$. For this reason, we postpone the proof of the analogous isomorphic classification theorem to the appendix of this work (Corollary~\ref{cor:isomorphism}).
\end{Rq}

\begin{Rq} Another algebraic consequence of Proposition~\ref{prop:me-oe} is the following. Let $G$ be a graph product with countably infinite vertex groups over a finite simple graph with no transvection and no partial conjugation (and not reduced to one vertex). Then any two isomorphic finite-index subgroups of $G$ have the same index. Indeed, if $G_1$ and $G_2$ are two isomorphic finite-index subgroups, then for every $i\in\{1,2\}$, the coupling between $G$ and $G_i$ by left-right multiplication has coupling index $[G:G_i]$. By composition of couplings, one can construct a self measure equivalence coupling of $G$ whose coupling index equals $[G:G_2]/[G:G_1]$, see e.g.\ Item~c on page 300 of \cite{Fur-survey}. Proposition~\ref{prop:me-oe} therefore implies that $[G:G_1]=[G:G_2]$.

This phenomenon (any two isomorphic finite index subgroups have the same index) was called \emph{finite index rigidity} by Lazarovich, who established it for all non-elementary hyperbolic groups \cite{Laz}. Finite index rigidity also occurs for all groups having at least one positive $\ell^2$-Betti number, since these are multiplicative with the index. In the case of graph products, $\ell^2$-Betti numbers were computed by Davis--Okun \cite{DO}, so in many cases the statement was known (though our work does cover new cases).  
\end{Rq}

\subsubsection{Fundamental groups and von Neumann algebras}\label{sec:von-neumann}

Recall that to every free, ergodic, probability measure-preserving action $G\actson X$ of a countable group $G$ on a standard probability space $X$, one associates a von Neumann algebra $L(G\actson X)$, in fact a $\mathrm{II}_1$ factor, via Murray and von Neumann's group measure space construction \cite{MvN}. 

The \emph{fundamental group}\index{Fundamental!Fundamental group} of an ergodic measured equivalence relation $\calr$ over a standard probability space $X$ (resp.\ a $\mathrm{II}_1$ factor $M$) is the subgroup of $\mathbb{R}_+^\ast$ consisting of all $t>0$ such that the amplification $\calr^t$ (resp.\ $M^t$) is isomorphic to $\calr$ (resp.\ $M$). In particular $t<1$ belongs to the fundamental group of $\calr$ if for any Borel subset $U\subseteq X$ of measure~$t$, the restriction $\calr_{|U}$ is isomorphic to $\calr$ (up to rescaling the measure). And if $L(\calr)$ is the $\mathrm{II}_1$ factor associated to $\calr$, then $t$ belongs to the fundamental group of $L(\calr)$ if for any $U$ as above, $L(\calr_{|U})$ is isomorphic to $L(\calr)$.

This notion was introduced by Murray and von Neumann in \cite{MvN2} as an invariant to distinguish between $\mathrm{II}_1$ factors, and they proved that the fundamental group of the hyperfinite $\mathrm{II}_1$ factor is $\mathbb{R}_+^*$, but besides this example no computations were available for an extended period of time. Connes famously proved that the fundamental group of any $\mathrm{II}_1$ factor with Property~(T) is countable \cite{Con}. It is only much later than the first example of a $\mathrm{II}_1$ factor with trivial fundamental group was exhibited by Popa \cite{Pop2}, solving a long-standing open problem of Kadison, in the framework of his deformation/rigidity theory.

Theorem~\ref{theo:classificationversionintro}, combined with a Cartan rigidity theorem established for graph products by Chifan and Kunnawalkam Elayavalli \cite[Theorem~1.3]{CKE} in the framework of Popa's deformation/rigidity theory, yields the following corollary.

\begin{Cor}\label{cor:fundamental}
    Let $G$ be a graph product of countably infinite vertex groups over a finite simple graph $\Gamma$ with no transvection and no partial conjugation. Let $G\curvearrowright X$ be a free, ergodic, measure-preserving action of $G$ by Borel automorphisms on a standard probability space $X$. Let $\calr(G\curvearrowright X)$ be the associated orbit equivalence relation, and $L(G\curvearrowright X)$ be the associated von Neumann algebra.

    Then $\calr(G\curvearrowright X)$ and $L(G\curvearrowright X)$ have trivial fundamental groups.
\end{Cor}

\begin{proof} Let $\calr=\calr(G\actson X)$. 
We first prove that the fundamental group of $\calr$ is trivial. So let $t\le 1$ such that $\calr$ and $\calr^t$ are isomorphic. This yields a stable orbit equivalence with compression constant $t$ between two free, ergodic, measure-preserving actions of $G$ on standard probability spaces. By \cite[Theorem~3.3]{Fur}, this in turn gives a self measure equivalence coupling $\Omega$ of $G$ such that the ratio of the fundamental domains of the two actions of $G$ is equal to $t$. Proposition~\ref{prop:me-oe} implies that $t=1$ shows that the fundamental group of $\calr(G \actson X)$ is trivial.

We now prove that the fundamental group of $L(G\actson X)=L(\calr)$ is trivial. Let $t\le 1$, and assume that $L(\calr)^t=L(\calr^t)$ is isomorphic to $L(\calr)$. Let $U\subseteq X$ be a Borel subset of measure $t$, so that $\calr^t$ is isomorphic to the restriction of $\calr$ to $U$. By \cite[Theorem~1.3]{CKE}, up to unitary conjugacy, $L^\infty(X)$ is the unique Cartan subalgebra of $L(\calr)$. Therefore any isomorphism from $L(\calr)$ to $L(\calr^t)$ can be unitarily conjugated so as to send $L^\infty(X)$ to $L^\infty(U)$. It thus follows from \cite{FM} that $\calr$ and $\calr^t$ are isomorphic, and by the above $t=1$.
\end{proof}

\section{Combinatorial and groupoid-theoretic tools}\label{sec:graph}
The proof of the Vertex Recognition Property (Proposition~\ref{prop:vertex-recognition}) occupies the next two sections. This section contains preparatory tools.  

\subsection{Combinatorial tools}

This subsection contains several combinatorial tools that will be crucial in the proof of the Vertex Recognition Property. First, in Section~\ref{sec:strongly-reduced-def}, we introduce the notion of a strongly reduced graph, and show that every finite simple graph with no transvection and no partial conjugation is either strongly reduced or splits non-trivially as a join (Lemma~\ref{lemma:strongly-reduced} below). This will allow us to decompose the proof of the Vertex Recognition Property into these two cases, that will be treated separately in Sections~\ref{sec:strongly-reduced} and~\ref{sec:reducible} below.

Second, in Section~\ref{sec:zoom-in}, we explain how to recognize conjugates of untransvectable vertex groups inside a graph product, by iteratively taking maximal direct products and passing to factors. This is the contents of Lemma~\ref{lemma:graph} below. This lemma will be a starting point and a guide in the next two sections: proving the Vertex Recognition Property will indeed require establishing a groupoid-theoretic version of Lemma~\ref{lemma:graph}. Before this, in Section~\ref{sec:products-and-factors}, we provide general definitions and facts regarding products and factors in graph products.

\subsubsection{Strongly reduced graphs}\label{sec:strongly-reduced-def}

The following combinatorial condition will be crucial in the present work.

\begin{Def}[Strongly reduced graph]\label{de:strongly-reduced}
Let $\Gamma$ be a finite simple graph. An induced subgraph $\Lambda\subseteq\Gamma$ is \emph{collapsible}\index{Collapsible (subgraph)} if for any $x,y\in V\Lambda$, one has $\lk(x)\cap(\Gamma\setminus\Lambda)=\lk(y)\cap(\Gamma\setminus\Lambda)$. We refer to \cref{fig:Collapsible}, \cpageref{fig:Collapsible} for an illustration. 

A graph $\Gamma$ is \emph{strongly reduced}\index{Strongly!Strongly reduced graph} if it does not contain any proper induced subgraph on at least two vertices which is collapsible.
\end{Def}

Note that if $\Lambda\subseteq\Gamma$ is collapsible, then every graph product over $\Gamma$ is also a graph product over the graph obtained from $\Gamma$ by collapsing $\Lambda$ to a single vertex $v_\Lambda$, and joining $v_\Lambda$ by an edge to every vertex in the common link of all vertices of $\Lambda$ (see the two graphs $\Gamma$ and $\Gamma^\prime$ of \cref{Ex:CollapsibleGraphs}, \cpageref{fig:Collapsible}). Every graph product over a finite simple graph $\Gamma$ which is not reduced to one vertex, can therefore be represented as a graph product over a strongly reduced finite simple graph $\overline{\Gamma}$ not reduced to one vertex.

\begin{Ex}\label{Ex:CollapsibleGraphs}\leavevmode
\begin{itemize}
    \item If $\Gamma$ is a join then we can collapse it to an edge (which is strongly reduced).
    \item If $\Gamma$ is disconnected, then we can collapse it to a disjoint union of two vertices (which is strongly reduced).
    \item As a more interesting example, the subgraph $\Lambda$ of $\Gamma$ represented in \cref{fig:Collapsible} is collapsible. Indeed the set $\lk(x)\cap(\Gamma\setminus\Lambda)$ --~represented by the blue vertices in the leftmost picture~-- coincide with $\lk(y)\cap(\Gamma\setminus\Lambda)$. Collapsing $\Lambda$ into $\bar{\Lambda}$ leads to the graph $\Gamma^\prime$ in the middle of \cref{fig:Collapsible}. The corresponding strongly reduced finite graph is the pentagon $\bar{\Gamma}$ represented on the right most side.
\end{itemize}
\end{Ex}

\begin{figure}[htbp]
    \centering
    \includegraphics[width=\textwidth]{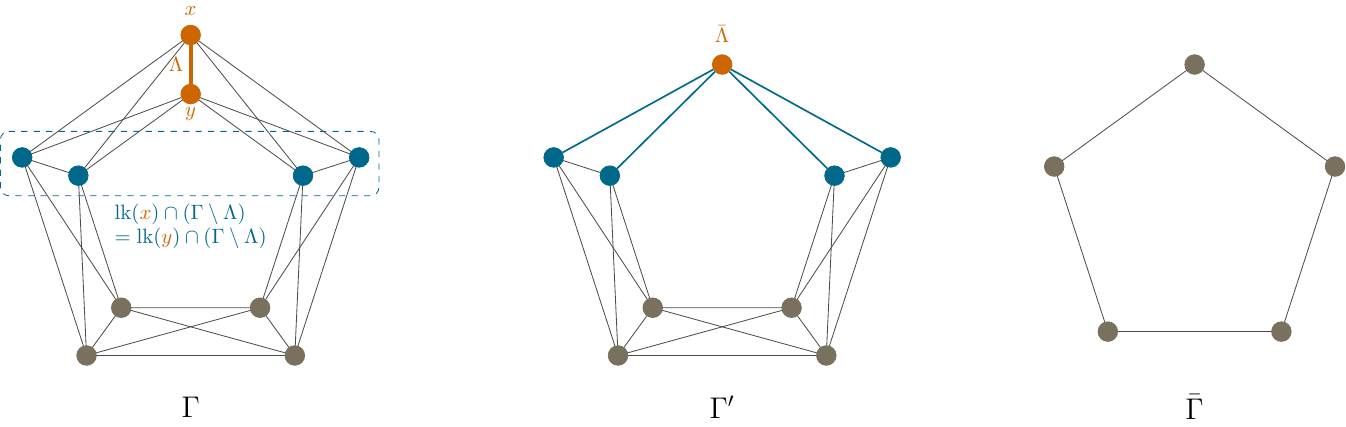}
    \caption{Collapsible and strongly reduced graphs (\cref{Ex:CollapsibleGraphs})}
    \label{fig:Collapsible}
\end{figure}

\begin{Lmm}\label{lemma:strongly-reduced}
    Let $\Gamma$ be a finite simple graph. If $\Gamma$ is transvection-free and has no partial conjugation, then either $\Gamma$ is strongly reduced or $\Gamma$ is a join.
\end{Lmm}

The idea behind the lemma is the following: if a finite simple graph $\Gamma$ is irreducible and contains a collapsible subgraph $\Lambda$, then either $\Lambda$ is a clique and in this case $\Gamma$ has transvections, or $\Lambda$ is not a clique but in this case, there exists $v\in V\Lambda$ such that $\Gamma\backslash \st(v)$ is disconnected. The second case is illustrated in \cref{fig:penta-de-penta}. In this figure, the graph $\Gamma$ on the left is irreducible but contains a collapsible subgraph $\Lambda$ (in \textcolor{orange}{orange}). Removing the star of the vertex $v$ (drawn in \textcolor{black!55}{grey} on the right most graph) produces a graph with two different connected components. 

\begin{figure}
    \centering
    \includegraphics{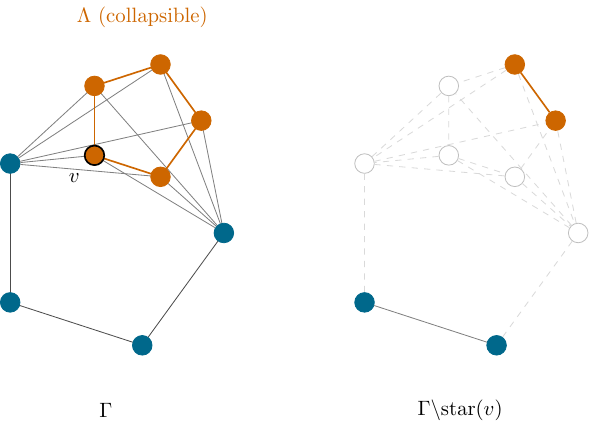}
    
    \caption{Example of an irreducible graph $\Gamma$ with partial conjugation}
    \label{fig:penta-de-penta}
\end{figure}

\begin{proof}
    Let us assume that $\Gamma$ is irreducible and show that $\Gamma$ is strongly reduced. 
    So assume towards a contradiction that there exists a proper collapsible subgraph $\Lambda\subseteq\Gamma$ on at least $2$ vertices. Let $w\in V\Lambda$. We observe that $\Lambda\neq\st(w)$. Otherwise, as $\Lambda$ is collapsible, for every $v\in V\Lambda$, we would have $\lk(v)\subseteq \st(w)$, and as $|V\Lambda|\ge 2$ this would give a transvection in $\Gamma$. We can thus choose $u\in V\Lambda\setminus\st(w)$.

    Let us show that there exists a vertex $u'\in V\Gamma\setminus V\Lambda$ which is not contained in $\st(w)$. By collapsibility of $\Lambda$ we have $\st(w)\subseteq \Lambda \circ \Lambda^\perp$, therefore  $\Lambda\cup \st(w)=\Lambda \circ \Lambda^\perp$. But since $\Lambda$ is proper and $\Gamma$ assumed to be irreducible, then $\Lambda\cup \st(w)$ is a proper subgraph of $\Gamma$. Whence the existence of $u^\prime$, as desired. 
    
    We claim that $\st(w)$ disconnects $u$ from $u'$, which will contradict the fact that $\Gamma$ has no partial conjugation -- and thus complete our proof. Indeed, let $u=u_1,\dots,u_k=u'$ be a path (i.e.\ any two consecutive vertices $u_i,u_{i+1}$ are adjacent), and let us prove that one of the vertices along this path is contained in $\st(w)$. Let $i\in\{1,\dots,k-1\}$ be such that $u_{i}\in V\Lambda$ and $u_{i+1}\notin V\Lambda$. Then $u_{i+1}$ is in $\lk(u_i)$, so by collapsibility it is also in $\lk(w)$, as desired. Hence $\st(w)$ disconnects $u$ from $u'$.
\end{proof}

In the sequel, we will use Lemma~\ref{lemma:strongly-reduced} in the form of the following corollary.

\begin{Cor}\label{cor:strongly-reduced}
Let $\Gamma$ be a finite simple graph with no transvection and no partial conjugation, not reduced to a vertex.

Then $\Gamma$ splits as $\Gamma=\Gamma_1\circ\dots\circ\Gamma_k$ (possibly with $k=1$), in such a way that every $\Gamma_j$ is strongly reduced, transvection-free, and not reduced to a vertex.
\end{Cor}

\begin{proof}
 Write $\Gamma=\Gamma_0\circ\Gamma_1\circ\dots\circ\Gamma_k$, where $\Gamma_0$ is the clique factor of $\Gamma$, and each $\Gamma_j$ with $j\ge 1$ has at least two vertices and does not split non-trivially as a join.
 
 Since $\Gamma$ is transvection-free and not reduced to a vertex, one has $\Gamma_0=\emptyset$.
 
 In view of Remark~\ref{rk:transvection-free-factors}, for every $j\in\{1,\dots,k\}$, the graph $\Gamma_j$ is transvection-free and has no partial conjugation. Therefore Lemma~\ref{lemma:strongly-reduced} shows that $\Gamma_j$ is strongly reduced, which completes our proof.
\end{proof}

\subsubsection{Products and factors}\label{sec:products-and-factors}

Let $G$ be a graph product of countably infinite vertex groups over a finite simple graph. A \emph{product parabolic subgroup} $P\subseteq G$ is \index{Product!Product parabolic subgroup} a parabolic subgroup which splits as a direct product of two non-trivial parabolic subgroups. We say that $P$ is a \emph{maximal} product parabolic subgroup \index{Maximal product!Maximal product parabolic subgroup} if it is maximal for inclusion (among product parabolic subgroups). We say that $P$ is of \emph{isolated clique type}\index{Isolated clique type!Parabolic subgroup} if $\Gamma$ has a connected component $C$ which is a clique, and $P$ is conjugate to $G_C$.  

\begin{Lmm}\label{lemma:maximal-product-nonamenable}
 Let $G$ be a graph product with countably infinite vertex groups over a finite simple graph $\Gamma$. 
 Let $P$ be a maximal product parabolic subgroup, and write $P=P_1\times P_2$, where $P_1$ and $P_2$ are non-trivial parabolic subgroups.

 Then either $P$ is of isolated clique type, or else $P_1$ or $P_2$ contains a non-abelian free subgroup.
\end{Lmm}

 \begin{proof}
 We first assume that the type $\Lambda$ of $P$ is not a clique. We can assume without loss of generality that the type $\Lambda_1$ of $P_1$ contains two non-adjacent vertices $v,w$. Being infinite, the groups $G_v$ and $G_w$ both contain either an infinite order element, or else a finite subgroup of order at least $3$. The free product $G_v*G_w$ (which has a conjugate contained in $P_1$) therefore contains a non-abelian free subgroup.
 
 We now assume that the type $\Lambda$ of $P$ is a clique, and aim to prove that it is isolated. Otherwise, $\Lambda$ is a proper subgraph of a connected component of $\Gamma$. We can therefore find $v\in V\Lambda$ such that there exists a vertex $v'\in\lk(v)\setminus V\Lambda$. Then $G_v\times G_v^{\perp}$ splits as a product of infinite parabolic subgroups and contains $P$ properly. This contradicts that $P$ is a maximal product parabolic subgroup. 
 \end{proof}

Given a finite simple graph $\Gamma$, and an induced subgraph $\Lambda\subseteq\Gamma$, we define the \emph{clique factor}\index{Clique!Clique factor of a subgraph} of $\Lambda$ as the subgraph spanned by all vertices that are joined by an edge to every other vertex of $\Lambda$. If $G$ is a graph product over $\Gamma$, and if $P=gG_\Lambda g^{-1}$, we define the \emph{clique factor}\index{Clique!Clique factor of a parabolic subgroup} of $P$ as the parabolic subgroup $gG_{\Lambda_0}g^{-1}$, where $\Lambda_0$ is the clique factor of~$\Lambda$. 

Let $P$ be a parabolic subgroup, and write $P=C_0\times F_1\times\dots\times F_n$, where $C_0$ is the clique factor of $P$, and no $F_j$ splits as a direct product of parabolic subgroups. We say that a subgroup $F\subseteq P$ is a \emph{factor}\index{Factor} if it is equal to $C_0$ or to one of the subgroups $F_j$ of the above decomposition. 

\subsubsection{Recognizing untransvectable vertices}\label{sec:zoom-in}
Recall that given a finite simple graph $\Gamma$, a vertex $v\in V\Gamma$ is \emph{untransvectable} if there does not exist any vertex $w\in V\Gamma$ distinct from $v$ such that $\lk(v)\subseteq\st(w)$. The goal of this section is to give a combinatorial/algebraic characterization of untransvectable vertices, which will be our starting point for proving the Vertex Recognition Property in later sections. 

\begin{lemma}\label{lemma:graph}
    Let $\Gamma$ be a finite simple graph, and let $G$ be a graph product over $\Gamma$. A vertex $v\in V\Gamma$ is untransvectable if and only if there exists a chain of parabolic subgroups
    \[G=F_0\supseteq P_1\supseteq F_1\supseteq\dots\supseteq P_n\supseteq F_n=G_v\]
    such that
    \begin{enumerate}
        \item for every $j\in\{1,\dots,n\}$, $P_j$ is a maximal product parabolic subgroup of $F_{j-1}$ and not of isolated clique type; 
        \item for every $j\in\{1,\dots,n-1\}$, the clique factor of $P_j$ is trivial, and $F_j$ is a factor of $P_j$, and
        \item $F_n$ is the clique factor of $P_n$.
    \end{enumerate}
\end{lemma}

At first sight, it might not seem intuitive why the conditions given in the lemma characterize untransvectable vertices and their conjugates. One useful observation is that if there exist two vertices $v,w\in V\Gamma$ such that $\lk(v)\subseteq \st(w)$, then any maximal product containing $v$ will also contain $w$ (so our process cannot pinpoint a transvectable vertex). Before reading the proof, the reader is invited to work out the conditions on an example, see Example~\ref{Ex:Zoom-in} below.

\begin{proof}
    We first assume that $v$ is untransvectable, and construct parabolic subgroups $P_j,F_j$ as in the statement. 

    Starting from $F_0=G$, our goal is to define inductively two finite sequences of parabolic subgroups $(P_j)_j=(G_{\Lambda_j})_j$ and $(F_j)_j=(G_{\Upsilon_j})_j$, such that if we have defined $P_j$ and $F_j$ for some $j\in \bN$ then 
    \begin{enumerate}
        \item Either $F_j=G_v$ and we stop.
        \item Or else 
        \begin{enumerate}
            \item We can choose $P_{j+1}$ a maximal product parabolic subgroup of $F_j$ containing $G_{\st(v)}\cap F_j$ which is not of isolated clique type;
            \item And either $G_v$ is the clique factor of $P_{j+1}$ or $P_{j+1}$ has trivial clique factor;
            \item And we denote by $F_{j+1}$ the factor of $P_{j+1}$ containing $G_v$;
             \item Either $F_{j+1}=G_v$, or $v$ is untransvectable in $\Upsilon_{j+1}$ and in particular it is not contained in an isolated clique of $\Upsilon_{j+1}$.
        \end{enumerate}
    \end{enumerate}
    The process has to stop eventually as we are constructing a chain of proper inclusions of parabolic subgroups (except perhaps $F_0=P_1$ if $F_0$ splits as a product). 
    
    So let $j\in \bN$ and assume that we have defined $F_j$ verifying that $v$ is untransvectable in $\Upsilon_j$. If $F_j=G_v$, we stop. 
     
    Now assume that $F_j \neq G_v$. In particular $v$ is not isolated in $F_j$, so we can choose a maximal product parabolic subgroup $P_{j+1}$ of $F_j$ that contains $G_{\st(v)}\cap F_j$. Without loss of generality, we will assume that $P_{j+1}=G_{\Lambda_{j+1}}$ for some induced subgraph $\Lambda_{j+1}\subseteq\Upsilon_j$.
    
    Being an untransvectable vertex of $\Upsilon_j$, the vertex $v$ cannot be contained in an isolated clique of $\Upsilon_j$, so $P_{j+1}$ is not of isolated clique type. Using again that $v$ is untransvectable, we get that either $P_{j+1}=G_{\st(v)}$ and $G_v$ is the clique factor of $P_{j+1}$ (in which case we define $F_{j+1}=G_v$), or else $P_{j+1}$ has trivial clique factor. We assume that the latter holds, and let $F_{j+1}$ be the factor of $P_{j+1}$ that contains $G_v$. 

    Let $\Upsilon_{j+1}\subseteq\Lambda_{j+1}$ be the induced subgraph such that $F_{j+1}=G_{\Upsilon_{j+1}}$. Let $\Upsilon_{j+1}^{\perp}$ be the orthogonal of $\Upsilon_{j+1}$ inside $\Lambda_{j+1}$. We are left with showing that $v$ is an untransvectable vertex of $\Upsilon_{j+1}$. So let $w\in V\Upsilon_{j+1}$ be such that $\lk_{\Upsilon_{j+1}}(v)\subseteq\st_{\Upsilon_{j+1}}(w)$. We aim to prove that $\lk_{\Upsilon_j}(v)\subseteq \st_{\Upsilon_j}(w)$ which will imply $w=v$. As $G_{\st(v)}\cap F_j\subseteq P_{j+1}$, we have 
    \begin{equation*}
      \st_{\Upsilon_j}(v):=\st(v)\cap \Upsilon_j\subseteq \Lambda_{j+1}.
    \end{equation*}
    Since $\Lambda_{j+1}\subseteq \Upsilon_j$ we thus have $\st_{\Lambda_{j+1}}(v)=\st_{\Upsilon_j}(v)$. Therefore, using that $P_{j+1}$ is a product, we have
    \[\lk_{\Upsilon_j}(v)=\lk_{\Lambda_{j+1}}(v)=\lk_{\Upsilon_{j+1}}(v)\circ \Upsilon_{j+1}^{\perp}.\] 
    We deduce that 
    \[\lk_{\Upsilon_j}(v)\subseteq\st_{\Upsilon_{j+1}}(w)\circ\Upsilon_{j+1}^{\perp}\subseteq \st_{\Upsilon_j}(w).\] 
    As $v$ is untransvectable when viewed as a vertex of $\Upsilon_j$, it follows that $w=v$, as desired. 

This finishes our inductive construction, and completes the proof of one implication of the lemma.

\medskip

Conversely, let us assume that there exist two sequences of parabolic subgroups $(F_j)_j$ and $(P_j)_j$ as in the statement. Up to a global conjugation, we can assume that all $F_j$ and $P_j$ are of the form $G_{\Upsilon_j},G_{\Lambda_j}$ for some induced subgraphs $\Upsilon_j,\Lambda_j\subseteq\Gamma$. 
    
     Let $w\in V\Gamma$ be such that $\lk(v)\subseteq\st(w)$. We aim to prove that $w=v$. 
     
     We first prove inductively that $w\in V\Upsilon_j$ and $w\in V\Lambda_{j+1}$ for every $j\in\{0,\dots,n-1\}$, as follows. 
     \begin{itemize}
     \item First, $w\in V\Upsilon_0=V\Gamma$. 
     \item Assuming that $w\in V\Upsilon_j$, we prove that $w\in V\Lambda_{j+1}$. Indeed, assume towards a contradiction that $w\notin V\Lambda_{j+1}$. Write 
     \begin{equation*}
         \Lambda_{j+1}=\Lambda_{j+1}^1\circ\dots\circ\Lambda_{j+1}^k
     \end{equation*} 
     as a maximal join, with $v\in V\Lambda_{j+1}^1$ (namely, $\Lambda_{j+1}^1=\Upsilon_{j+1}$). Since $\lk(v)\subseteq\st(w)$, we have 
     \begin{equation*}
         \Lambda_{j+1}^2\circ\cdots\circ\Lambda_{j+1}^k\subseteq\st(w).
     \end{equation*}
     So let $\langle\Lambda_{j+1}^1,w\rangle$ be the subgraph of $\Upsilon_j$ induced by $\Lambda_{j+1}^1$ and $w$. Then 
     \begin{equation*}
         \langle\Lambda_{j+1}^1,w\rangle\circ\Lambda_{j+1}^2\circ\cdots\circ\Lambda_{j+1}^k
     \end{equation*}
     is again a join, which properly contains $P_{j+1}$. This contradicts the maximality of~$P_{j+1}$.
     \item Assuming that $w\in V\Lambda_{j+1}$ with $j\in\{0,\dots,n-2\}$, we prove that $w\in V\Upsilon_{j+1}$. Indeed, write $\Lambda_{j+1}=\Lambda_{j+1}^1\circ\dots\circ\Lambda_{j+1}^k$ as a maximal join, with $v\in V\Lambda_{j+1}^1=V\Upsilon_{j+1}$. By contradiction, if $w\in V\Lambda_{j+1}^i$ with $i\ge 2$, then using that $\lk(v)\subseteq\st(w)$, we see that $\Lambda_{j+1}^i\subseteq\st(w)$, and therefore $\Lambda_{j+1}\subseteq\st(w)$. This contradicts the fact that $P_{j+1}$ has trivial clique factor.
     \end{itemize}
     
     In particular we have proved that $w\in V\Lambda_n$. The same argument as in the last point shows that $w$ belongs to the clique factor of $P_n$. Since the clique factor of $P_n$ is $G_v$ we thus obtain that $w=v$.
\end{proof}

\begin{figure}[htbp]
    \centering
    \includegraphics[width=0.3\textwidth]{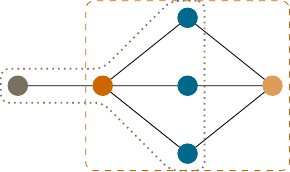}
    \caption{Graph for \cref{Ex:Zoom-in}}
    \label{fig:Example_Zoom-In}
\end{figure}

\begin{Ex}\label{Ex:Zoom-in}
This example illustrates the statement of Lemma~\ref{lemma:graph}.

    Consider $\Gamma$ to be the graph in \cref{fig:Example_Zoom-In} and $G$ a graph product over $\Gamma$. The only untransvectable vertex in $\Gamma$ is the one drawn in \textcolor{Orange}{dark orange} on the leftmost part of the graph. Let us describe the possible chains of parabolic subgroups in $G$ verifying the conditions of \cref{lemma:graph}. 
    
    Up to conjugation, there are exactly two maximal product parabolic subgroups in $G$, so there are \emph{a priori} two possible choices for $P_1$.
    \begin{itemize}
    \item If $P_1$ is taken to be a maximal product with type encompassed by the \textcolor{MFCB}{brown} dots, then $P_1$ has a clique factor $F_1$ equal to a conjugate of a vertex group. This vertex group corresponds to the \textcolor{Orange}{dark orange} vertex, which is untransvectable. Hence here taking $F_1=F_n$ gives the wanted sequence.
    \item On the other hand, if $P_1$ is taken to be the maximal product of $G$ with type encompassed by the \textcolor{Orange}{orange} dashes, then $P_1$ can be decomposed as $P_1=F_1\times F^\prime_1$ where $F_1$ is the factor whose type is induced by the three \textcolor{Turquoise}{blue} vertices and $F^\prime_1$ has type induced by the two (dark and light) \textcolor{Orange}{orange} vertices. Then $P_1$ has trivial clique factor and neither $F_1$ nor $F^\prime_1$ contains a product parabolic subgroup. In particular, we can not find a maximal product $P_2$ in any of the factors of $P_1$ verifying point 1 of \cref{lemma:graph}.
    \end{itemize}
    Hence, up to conjugation, the only possible sequence of parabolic subgroups verifying the conditions of \cref{lemma:graph} is $G=F_0\supseteq P_1\supseteq F_1=F_n$, where $P_1$ is a maximal product parabolic subgroup as in the first point.
\end{Ex}

\subsection{Groupoid-theoretic tools}\label{sec:groupoid-tools}

We now establish a few statements regarding how to exploit the amenability of a groupoid $\cala$ equipped with a cocycle $\rho$ towards a graph product, and in particular get information on the normalizer of $\cala$. The arguments developed in this section are largely inspired by work of Adams \cite{Ada} regarding the indecomposability of equivalence relations generated by probability measure-preserving actions of Gromov hyperbolic groups. Adams's strategy was later developed by Kida in the context of mapping class groups \cite{Kid-memoir} and has already found many applications towards measure equivalence classification/rigidity results in many contexts.   

In Section~\ref{sec:elementary-groupoids}, we introduce the notion of elementarity of $(\cala,\rho)$, an analogue for groupoids to the elementarity of a group action on a tree (i.e.\ having a fixed point or a finite orbit at infinity). See e.g.\ \cite[Section~5.2]{HH-Higman} for a similar treatment in a slightly different context. This notion will only be used in Section~\ref{sec:reducible}. 

Section~\ref{Sec:ExploitingNormalAmenable} contains a lemma in a similar spirit, that will only be used in the proof of Lemma~\ref{lemma:maximal-product}, and is inspired from the case of right-angled Artin groups \cite[Lemma~3.10]{HH21}.

\subsubsection{Elementary groupoids with a cocycle towards a graph product}\label{sec:elementary-groupoids}

Let $\Gamma$ be a finite simple graph. For every $v\in V\Gamma$, let $T_v$ be the Bass--Serre tree of the splitting of $G$ as $G=G_{\st(v)}\ast_{G_{\lk(v)}}G_{\Gamma\setminus\{v\}}$, as in Section~\ref{sec:trees}. We equip $\partial_\infty T_v\cup V(T_v)$ with the observers' topology, for which connected components of complements of points in $T_v$ form a subbasis of open sets. For this topology $\partial_\infty T_v\cup V(T_v)$ is compact and metrizable, and $\partial_\infty T_v$ and $V(T_v)$ are Borel subsets. The induced topology on $\partial_\infty T_v$ is nothing but the visual topology.

Given a standard Borel space $Z$, we denote by $\calp_{<\infty}(Z)$ the set of all non-empty finite subsets of $Z$, equipped with its natural Borel structure. We denote by $\calp_{\le 2}(Z)$ the set of all non-empty subsets of $Z$ of cardinality at most $2$. 

We also recall from \cpageref{Def:TvReg} that $(\partial_\infty T_v)^{\reg}$ is the set of elements in $\partial_\infty T_v$ having trivial elliptic stabilizer.  We also refer to Definition~\ref{def:ProperlySupported}, \cpageref{def:ProperlySupported} for the definition of properly $\mathbb{P}_G$-supported pairs.

\begin{Def}[(Boundary) $G$-elementary]\label{de:elementary}
Let $G$ be a graph product of countably infinite groups over a finite simple graph $\Gamma$, and let $\calg$ be a measured groupoid over a standard probability space, equipped with a cocycle $\rho:\calg\to G$. 

We say that $(\calg,\rho)$ is \emph{boundary $G$-elementary} if for every $v\in V\Gamma$, there exists a stably $(\calg,\rho)$-equivariant Borel map $X\to\calp_{\le 2}((\partial_\infty T_v)^{\reg})$.

We say that $(\calg,\rho)$ is \emph{$G$-elementary} if there exists a Borel partition $X=X_1\dunion X_2$ such that $(\calg_{|X_1},\rho)$ is properly $\mathbb{P}_G$-supported, and $(\calg_{|X_2},\rho)$ is boundary $G$-elementary.
\end{Def}

\begin{Rq}\label{Rq:PartitionGElementary}
    When $(\calg,\rho)$ is $G$-elementary, there exists a Borel partition $X=X'_1\dunion X'_2$ such that $(\calg_{|X'_1},\rho)$ is properly $\mathbb{P}_G$-supported and $(\calg_{|X'_2},\rho)$ is tightly $G$-supported and boundary $G$-elementary.
\end{Rq}

Here is a basic example of a non-elementary pair $(\calg,\rho)$.

\begin{lemma}\label{lemma:nonelementary-example}
    Let $G$ be a graph product of countably infinite groups over a finite simple graph $\Gamma$ which is not a clique. Let $\calg$ be a measured groupoid over a standard probability space, and let $\rho:\calg\to G$ be an action-like cocycle.

    Then $(\calg,\rho)$ is not $G$-elementary.
\end{lemma}

\begin{proof}
   By Lemma~\ref{lemma:tight-support}, \cpageref{lemma:tight-support}, $(\rho^{-1}(G),\rho)$ is tightly $G$-supported, thus since $\calG=\rho^{-1}(G)$ we obtain that for every positive measure Borel subset $U\subseteq X$, the pair $(\calg_{|U},\rho)$ is tightly $G$-supported. It is therefore enough to prove that $(\calg,\rho)$ is not boundary $G$-elementary. So assume towards a contradiction that for every $v\in V\Gamma$, there exists a stably $(\calg,\rho)$-equivariant Borel map $X\to\calp_{\le 2}((\partial_\infty T_v)^{\reg})$. By Lemma~\ref{lemma:decomposition-boundary}, \cpageref{lemma:decomposition-boundary}, the $G$-action on $(\partial_\infty T_v)^{\reg}$ is Borel amenable, and therefore so is the $G$-action on $\calp_{\le 2}((\partial_\infty T_v)^{\reg})$. As $\rho$ has trivial kernel, it follows from \cite[Proposition~3.38]{GH-OutFn} (recasting \cite[Proposition~4.33]{Kid-memoir}) that $\rho^{-1}(G)$ is amenable. Since $\Gamma$ is not a clique, $G$ is not amenable, and we obtain a contradiction to the fact that $\rho$ is action-like (third point of Definition~\ref{de:action-like}, \cpageref{de:action-like}).
\end{proof}

In contrast, Lemma~\ref{lemma:alternative} and Remark~\ref{rk:amenable} below will show that any amenable groupoid is $G$-elementary, with respect to any cocycle. 

Recall that whenever $K$ is a compact metrizable space, we equip $\Prob(K)$ with the weak-$*$ topology, coming from the duality with the space of continuous real-valued functions on $K$ given by the Riesz-Markov-Kakutani theorem. The reader is refered to \cite[Chapter~17]{Kec} for a detailed analysis of the induced Borel structure on $\Prob(K)$, which will justify the measurability of all maps appearing in the proof of the following lemma.

\begin{Lmm}\label{lemma:alternative}
  Let $G$ be a graph product of countably infinite groups over a finite irreducible simple graph $\Gamma$ not reduced to a vertex. Let $\calg$ be a measured groupoid over a standard probability space $X$, and let $\rho:\calg\to G$ be a cocycle.
    
  If for all $v\in V\Gamma$ there exists a $(\calg,\rho)$-equivariant Borel map $X\to\Prob(V(T_v)\cup\partial_\infty T_v)$, then $(\calg,\rho)$ is $G$-elementary. 
\end{Lmm}

\begin{Rq}\label{rk:amenable}
    In view of Proposition~\ref{prop:amenable-prob}, \cpageref{prop:amenable-prob}
    the assumption of the lemma regarding the existence of a $(\calg,\rho)$-equivariant Borel map $X\to \Prob(V(T_v)\cup\partial_\infty T_v)$ is satisfied whenever $\calg$ is amenable.
\end{Rq}

\begin{proof}
Let $X_1\subseteq X$ be a Borel subset of maximal measure such that $(\calg_{|X_1},\rho)$ is properly $\mathbb{P}_G$-supported. Let $X_2=X\setminus X_1$. Then $(\calg_{|X_2},\rho)$ is tightly $G$-supported. We will prove that $(\calg_{|X_2},\rho)$ is boundary $G$-elementary. Let $v\in V\Gamma$. 

\smallskip

\noindent\textbf{Step 1} Let us prove that for all positive measure Borel subsets $U\subseteq X_2$, there does not exist any $(\calg_{|U},\rho)$-equivariant Borel map $U\to\Prob(V(T_v))$.
\smallskip

Assume towards a contradiction that there exist a positive measure Borel subset $U\subseteq X_2$, and a $(\calg_{|U},\rho)$-equivariant Borel map $U\to\Prob(V(T_v))$. We first build a $(\calG_{|U},\rho)$-equivariant Borel map from $U$ to $V(T_v)$, the construction is summed up in \cref{fig:DemoElementary}.

As $V(T_v)$ is countable, there exists a $G$-equivariant Borel map $\Prob(V(T_v))\to\calp_{<\infty}(V(T_v))$, sending a probability measure $\mu$ to the nonempty finite set consisting of all vertices with maximal $\mu$-measure. 
Using that every bounded set of a $\mathrm{CAT}(0)$ space has a unique circumcenter \cite[Proposition~II.2.7]{BH}, we have a $G$-equivariant map $\calp_{<\infty}(V(T_v))\to T_v$.  
As the action of $G$ on $T_v$ is without edge inversion, there is also a $G$-equivariant map $T_v\to V(T_v)$: one can indeed choose a $G$-invariant orientation of the edges of $T_v$, send every vertex $v$ to itself, and every point in the interior of an edge $e$ to the origin of $e$. 
Altogether, we obtain a Borel $(\calg_{|U},\rho)$-equivariant map $U\to V(T_v)$. 

Let now $U'\subseteq U$ be a positive measure Borel subset where this map is constant (this exists because $V(T_v)$ is countable). As no vertex of $\Gamma$ is joined to all other vertices, vertex stabilizers of $T_v$ are proper parabolic subgroups of $G$ (see \cpageref{NotationTv}). Therefore $\rho(\calg_{|U'})$ is contained in a proper parabolic subgroup of $G$, which contradicts the fact that $(\calg_{|X_2},\rho)$ is tightly $G$-supported. This completes Step~1.

\begin{figure}
    \centering
    \begin{tikzpicture}
  \node (mu) at (-0.25,2) {$\mu$};
  \node[rotate=-90] (muin) at (-0.25,1.5) {$\in$};
  \node (immu) at (4.6,2) {$\left\{ w\in V(T_v)\, : \, \mu(w)=\max_{\omega\in V(T_v)}\mu(\omega) \right\}$};
  \node[rotate=-90] (in) at (3.8,1.5) {$\in$};
  \node (circum) at (10,2){circumcenter};
  \node[rotate=-90] (in) at (10,1.5) {$\in$};
  \draw[{Bar[]}->,>=latex,black!50] (mu)--(immu);
  \draw[{Bar[]}->,>=latex,black!50] (immu)--(circum);
  \begin{scope}[shift={(0,-0.5)}]
    \node (XU) at (-0.25,0) {$U$};
    \node (V) at (10,0) {$V\left(T_v  \right)$};
    \node (Up) at (-0.25,-1) {$U^\prime$};
    \node (vzero) at (10,-1){$\left\{v_0  \right\}$};
    \draw[DeepSkyBlue4,{Bar[]}->,>=latex,dashed] (Up)--(vzero);
  \end{scope}
  \node (proba) at (-0.25,1) {$\mathrm{Prob}\left( V\left(T_v  \right) \right)$};
  \node (pcal) at (3.9,1) {$\mathcal{P}_{<\infty}\left( V\left(T_v \right) \right)$};
  \node (VE) at (10,1) {$T_v$};
  \draw[->,>=latex,black!50] (XU)--(proba)node[midway,anchor=east]{$(\mathcal{G}_{|U},\rho)$-equiv.};
  \draw[->,>=latex,black!50] (proba)--(pcal)node[midway,below,anchor=north]{$G$-equiv.};
  \draw[->,>=latex,black!50] (pcal)--(VE) node[midway,above,anchor=south]{$G$-equiv.};
  \draw[->,>=latex,black!50] (VE)--(V)node[midway,anchor=west]{$G$-equiv.};
  \draw[->,>=latex,DeepSkyBlue4,dashed] (XU) --(V)node[midway,anchor=north,below] {$\left( \mathcal{G}_{|U},\rho \right)$-equivariant};
\end{tikzpicture}
    \caption{Construction of the equivariant map from $U\subseteq X_2$ to $V(T_v)$ in Step~1 of the proof of Lemma~\ref{lemma:alternative}.}
    \label{fig:DemoElementary}
\end{figure}

\smallskip

By assumption we have a $(\calg,\rho)$-equivariant Borel map $X\to \Prob(V(T_v)\cup\partial_\infty T_v)$. Therefore, by Step 1, up to replacing $X_2$ by a conull Borel subset, we have a $(\calg_{|X_2},\rho)$-equivariant Borel map $\nu:X_2\to\Prob(\partial_\infty T_v)$.

\smallskip

\noindent\textbf{Step 2} We prove that for almost every $x\in X_2$, the probability measure $\nu(x)$ gives full measure to $(\partial_\infty T_v)^{\reg}$. 
\smallskip

Let $(\partial_\infty T_v)^{\para}:=\partial_\infty T_v\setminus (\partial_\infty T_v)^{\reg}$\label{Def:Tvpara}. Recall that there exists a $G$-equivariant Borel map $\theta_v:(\partial_\infty T_v)^{\para}\to\mathbb{P}_G\setminus\{1\}$, where $\mathbb{P}_G$ is the set of parabolic subgroups of $G$, sending every $\xi\in (\partial_\infty T_v)^{\para}$ to its elliptic stabilizer (see above \cref{lemma:reg-nonempty}, \cpageref{lemma:reg-nonempty}). Since $\Gamma$ is irreducible we can apply \cref{lemma:elliptic-stabilizer}, to obtain that $\theta_v(\xi)\neq G$ for all $\xi$.

Assume towards a contradiction that there exists a positive measure Borel subset $U\subseteq X_2$ such that $\nu(x)((\partial_\infty T_v)^{\para})>0$ for all $x\in U$. We now restrict each of the probability measures $\nu(x)$ to $(\partial_\infty T_v)^{\para}$, renormalize it to get a probability measure, and push it forward through the map $\theta_v:(\partial_\infty T_v)^{\para}\to\mathbb{P}_G\setminus\{1,G\}$. This yields a Borel $(\calg_{|U},\rho)$-equivariant map $U\to\Prob(\mathbb{P}_G\setminus\{1,G\})$. As $\mathbb{P}_G$ is countable, as above, we deduce a Borel $(\calg_{|U},\rho)$-equivariant map $U\to\calp_{<\infty}(\mathbb{P}_G\setminus\{1,G\})$. Therefore, there exist a positive measure Borel subset $U'\subseteq U$ and a non-empty finite set $\mathcal{Q}=\{Q_1,\dots,Q_k\}$ of proper non-trivial parabolic subgroups of $G$ such that, denoting by $H$ the setwise stabilizer of $\mathcal{Q}$ (for the conjugation action of $G$ on $\mathbb{P}_G$), one has $\rho(\calg_{|U'})\subseteq H$. Let $H^0$ be the finite-index subgroup of $H$ that fixes each $Q_i$ (again, for the conjugation action of $H$ on $\mathbb{P}_G$). Thus $H^0$ normalizes $Q_1$, i.e.\ $H^0\subseteq Q_1\times Q_1^{\perp}$. As $\Gamma$ is irreducible, $Q_1\times Q_1^{\perp}$ is a proper parabolic subgroup of $G$. Lemma~\ref{lemma:finite-index-0}, \cpageref{lemma:finite-index-0} thus ensures that $H$ is contained in a proper parabolic subgroup $Q$ of $G$. Thus $\rho(\calg_{|U'})\subseteq Q$, contradicting the fact that $(\calg_{|X_2},\rho)$ is tightly $G$-supported. This completes Step~2. 
\smallskip

\noindent \textbf{Step 3} We now prove that for almost every $x\in X_2$, the support of $\nu(x)$ has cardinality at most $2$. 

\smallskip

Otherwise, let $U\subseteq X_2$ be a positive measure Borel subset such that for every $x\in U$, the support of $\nu(x)$ has cardinality at least $3$. Then for every $x\in U$, the probability measure $\nu(x)\otimes\nu(x)\otimes\nu(x)$ on $(\partial_\infty T_v)^3$ gives positive measure to the subset $(\partial_\infty T_v)^{(3)}$ consisting of pairwise distinct triples. After renormalizing these measures, we obtain a Borel $(\calg_{|U},\rho)$-equivariant map $U\to\Prob((\partial_\infty T_v)^{(3)})$. Pushing this map forward through the continuous barycenter map $(\partial_\infty T_v)^{(3)}\to V(T_v)$, we derive a Borel $(\calg_{|U},\rho)$-equivariant map $U\to \Prob(V(T_v))$, which is a contradiction to Step~1. This completes Step~3.

\smallskip

By considering the support of $\nu(x)$, we deduce a Borel $(\calg_{|X_2},\rho)$-equivariant map $X_2\to\calp_{\le 2}((\partial_\infty T_v)^{\reg})$, as desired.
\end{proof}

\begin{Lmm}\label{lemma:maximal-map}
 Let $G$ be a graph product of countably infinite groups over a finite irreducible simple graph, not reduced to one vertex. Let $\calg$ be a measured groupoid over a standard probability space~$X$, equipped with a cocycle $\rho:\calg\to G$. Let $\cala\subseteq\calg$ be a measured subgroupoid, and assume that $(\cala,\rho)$ is boundary $G$-elementary and tightly $G$-supported.

Then there exists an essentially unique stably $(\cala,\rho)$-equivariant measurable map $\theta_{v,\max}: X\to \calp_{\le 2}((\partial_\infty T_v)^{\reg})$ such that for every stably $(\cala,\rho)$-equivariant measurable map $\theta:X\to\calp_{\le 2}((\partial_\infty T_v)^{\reg})$ and almost every $x\in X$, one has $\theta(x)\subseteq\theta_{v,\max}(x)$.

If in addition $\calh\subseteq\calg$ is a measured subgroupoid that stably normalizes $\cala$, then the map $\theta_{v,\max}$ is also stably $(\calh,\rho)$-equivariant.
\end{Lmm}

The argument for the existence and essential uniqueness of $\theta_{v,\max}$ comes from the work of Adams \cite[Lemmas~3.2 and~3.3]{Ada}; we review it here in our setting for convenience, following the presentation from \cite[Lemma~12.12]{GH-OutFn}. The argument for its invariance under a subgroupoid that stably normalizes $\cala$ comes from \cite[Lemma~3.4]{Ada}, we will follow \cite[Corollary~12.13]{GH-OutFn}.

\begin{proof}
The essential uniqueness is clear: if $\theta_{v,\max}^1$ and $\theta_{v,\max}^2$ are two such maps, then for almost every $x\in X$, one has $\theta_{v,\max}^1(x)\subseteq\theta_{v,\max}^2(x)$ and $\theta_{v,\max}^2(x)\subseteq\theta_{v,\max}^1(x)$.

For the existence, we first notice that if $\theta,\theta':X\to\calp_{\le 2}(\partial_\infty T_v)^{\reg}$ are two $(\calg,\rho)$-equivariant Borel maps, then for almost every $x\in X$, the subset $\theta(x)\cup\theta'(x)$ of $(\partial_\infty T_v)^{\reg}$ has cardinality at most $2$. Indeed otherwise, the barycenter argument as in Step~3 of the previous proof would yield a contradiction to the fact that $(\calg,\rho)$ is tightly $G$-supported.

Let now $U\subseteq X$ be a Borel subset of maximal measure such that there exists a stably $(\calg_{|U},\rho)$-equivariant Borel map $\varphi:U\to \calp_{=2}((\partial_\infty T_v)^{\reg})$, where $\calp_{=2}((\partial_\infty T_v)^{\reg})$ is the set of all subsets of $(\partial_\infty T_v)^{\reg}$ of cardinality exactly $2$. 
Furthermore, by boundary $G$-elementarity of $(\cala,\rho)$, there also exists a stably $(\calg_{|X\setminus U},\rho)$-equivariant Borel map $\psi:X\setminus U\to (\partial_\infty T_v)^{\reg}$.
We let $\theta_{v,\max}$ be the map that coincides with $\varphi$ on $U$ and with $\psi$ on $X\setminus U$, and we claim that it satisfies the first conclusion of the lemma. Indeed, let $\theta:X\to\calp_{\le 2}((\partial_\infty T_v)^{\reg})$ be a stably $(\cala,\rho)$-equivariant map. Then $\theta\cup\theta_{v,\max}$ is again $(\cala,\rho)$-equivariant, and therefore $\theta(x)\cup\theta_{v,\max}(x)$ must have cardinality $2$ on $U$ and cardinality $1$ on $X\setminus U$ (almost everywhere). It follows that $\theta(x)\subseteq\theta_{v,\max}(x)$ almost everywhere. 

For future use, we also make the following observation.
\smallskip

\noindent \textbf{Observation} For every positive measure Borel subset $U\subseteq X$, the map $(\theta_{v,\max})_{|U}$ is also the essentially unique maximal stably $(\cala_{|U},\rho)$-equivariant Borel map $U\to\calp_{\le 2}((\partial_\infty T_v)^{\reg})$.

\smallskip
 
 Let now $\calh\subseteq\calg$ be a measured subgroupoid that stably normalizes $\cala$. Using the above observation, we can (and will) assume up to a countable partition of $X$ that $\calh$ normalizes $\cala$.

By definition of normalization between measured subgroupoids, we can write $\calh$ as a countable union of Borel subsets $B_n$ such that for every $n\in\mathbb{N}$,
\begin{itemize}
\item $B_n$ induces a Borel isomorphism $f_n:U_n\to V_n$ between its source $U_n:=s(B_n)$ and its range $V_n:=r(B_n)$, and 
\item for every $\mathbf{a}\in\calg$ and every $\mathbf{b}_1,\mathbf{b}_2\in B_n$, if the composition $\mathbf{b}_1\mathbf{a}\mathbf{b}_2^{-1}$ is well-defined, then $\mathbf{a}\in\cala$ if and only if $\mathbf{b}_1\mathbf{a}\mathbf{b}_2^{-1}\in\cala$. 
\end{itemize}
Up to subdividing $B_n$ if necessary, we will also assume without loss of generality that the cocycle $\rho$ takes a single value on $B_n$, say $\rho(B_n)=\{g_n\}$. 

Let now $n\in\mathbb{N}$ be such that $U_n$ has positive measure. Then the map
\[\begin{array}{cccc}
    \theta'_{v,\max}: & U_n & \to & \calp_{\le 2}((\partial_\infty T_v)^{\reg}) \\
     & x & \mapsto & g_n^{-1}\theta_{v,\max}(f_n(x))
\end{array}
\]
is a maximal stably $(\cala_{|U_n},\rho)$-equivariant Borel map. So by the above observation, it coincides with $\theta_{v,\max}$ on a conull Borel subset of $U_n$. As this is true for every $n\in\mathbb{N}$ such that $U_n$ has positive measure, and the corresponding sets $B_n$ cover $\calh_{|X^*}$ for a conull Borel subset $X^*\subseteq X$, it follows that the map $\theta_{v,\max}$ is (stably) $(\calh,\rho)$-equivariant. 
\end{proof}

\begin{Lmm}\label{lemma:elementary-normal}
Let $G$ be a graph product of countably infinite groups over a finite simple graph. Let $\calg$ be a measured groupoid over a standard probability space $X$, equipped with a cocycle $\rho:\calg\to G$. Let $P\subseteq G$ be a parabolic subgroup whose type is irreducible, and let $\rho_P:\calg\to P$ be the cocycle obtained by postcomposing $\rho$ by the natural retraction $G\to P$. Let $\cala,\calh\subseteq\calg$ be measured groupoids, with $\cala$ normalized by $\calh$.  

If $(\cala,\rho_P)$ is $P$-elementary and $(\rho_P)_{|\cala}$ is nowhere trivial, then $(\calh,\rho_P)$ is $P$-elementary.
\end{Lmm}

\begin{proof}
Since $(\cala,\rho_P)$ is $P$-elementary, there exists a Borel partition $X=X_1\dunion X_2$ such that $(\cala_{|X_1},\rho_P)$ is properly $\mathbb{P}_P$-supported and $(\cala_{|X_2},\rho_P)$ is tightly $P$-supported and boundary $P$-elementary (see Remark~\ref{Rq:PartitionGElementary}).

Since $\cala$ is normalized by $\calh$ and the type of $P$ is irreducible, and $(\rho_P)_{|\cala}$ is nowhere trivial, Lemma~\ref{lemma:properly-supported} ensures that $(\calh_{|X_1},\rho_P)$ is properly $\mathbb{P}_P$-supported.    

Let $\Lambda\subseteq\Gamma$ be the type of $P$, and let $v\in V\Lambda$. Viewing $P$ as a graph product over $\Lambda$, we let $T_v^P$ be the splitting of $P$ associated to $v$ (see \cref{sec:trees}). Let  $\theta_{v,\max}:X_2\to\calp_{\le 2}((\partial_\infty T^P_v)^{\reg})$ be the map given by  Lemma~\ref{lemma:maximal-map}. Since $\cala$ is normalized by $\calh$, the map $\theta_{v,\max}$ is stably $(\calh_{|X_2},\rho_P)$-equivariant. This shows that $(\calh_{|X_2},\rho_P)$ is boundary $P$-elementary, which concludes our proof.
\end{proof}

\subsubsection{Exploiting normal amenable subgroupoids}\label{Sec:ExploitingNormalAmenable}

\begin{Lmm}\label{lemma:adams-argument-v2} Let $G$ be a graph product of countably infinite groups over a finite simple graph. Let $\calg$ be a measured groupoid over a standard probability space $X$, equipped with a cocycle $\rho:\calG\to G$. 

Let $\calA,\caln$ be measured subgroupoids of $\calg$. Let $A,N\subseteq G$ be parabolic subgroups. Assume that 
\begin{itemize}
\item $\cala$ is amenable and normalized by $\caln$, and $\caln$ is non-amenable,
\item $(\cala,\rho)$ is tightly $A$-supported, and $(\caln,\rho)$ is tightly $N$-supported, with $N\subseteq A$,
\item $\rho_{|\caln}$ has trivial kernel.
\end{itemize}
Then the clique factor $C_0$ of $A$ is non-amenable, and in fact $N\cap C_0$ is non-amenable.
\end{Lmm}

\begin{proof}
 Up to a global conjugation of the cocycle $\rho$, we will assume without loss of generality that $A=G_\Lambda$ for some induced subgraph $\Lambda\subseteq\Gamma$.

Consider a join decomposition $\Lambda=\Lambda_0\circ\Lambda_1\circ\dots\circ\Lambda_n$, where $\Lambda_0$ is the clique factor of $\Lambda$, and $\Lambda_1,\dots,\Lambda_n$ do not admit any non-trivial join decomposition. Remark that for all $j\geq 1$, the graph $\Lambda_j$ contains at least $2$ vertices (otherwise it could be included in the clique factor $\Lambda_0$). This join decomposition of $\Lambda$ induces a direct product decomposition $A=C_0\times F_1\times\dots\times F_n$. Here $C_0$ is the clique factor of $A$.

Let $j\in\{1,\dots,n\}$. Up to replacing $X$ by a conull Borel subset, we can (and will) assume that $\rho(\cala)\subseteq A$. Let $\rho_j:\rho^{-1}(A)\to F_j$ be the cocycle obtained by postcomposing $\rho$ with the $j^{\text{th}}$ projection (see \cref{Fig:DemoAmenableGrpdg}). Since $N\subseteq A$, the cocycle $\rho_j$ is defined on $\cala$ and on $\caln$. Let $v\in V\Lambda_j$, and let $T_v$ be the Bass-Serre tree of the splitting of $F_j$ associated to $v$. Since $\cala$ is amenable, there exists an $(\cala,\rho_j)$-equivariant Borel map $\nu:X\to\Prob(\partial_\infty T_v\cup V(T_v))$, see Remark~\ref{rk:amenable}.

\begin{figure}
    \centering
    \includegraphics[width=\textwidth]{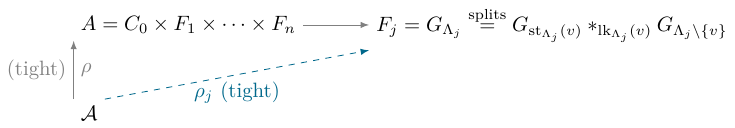}
    \caption{Definition of the map $\rho_j$}\label{Fig:DemoAmenableGrpdg}
\end{figure}

Since $(\cala,\rho)$ is tightly $A$-supported, necessarily $(\cala,\rho_j)$ is tightly $F_j$-supported. Therefore, Lemma~\ref{lemma:alternative} ensures that $(\cala,\rho_j)$ is boundary $F_j$-elementary. And Lemma~\ref{lemma:maximal-map} ensures that there is an essentially unique {maximal} Borel stably $(\cala,\rho_j)$-equivariant map $\theta^{j}_{v,\max}:X\to\calp_{\le 2}((\partial_\infty T_v)^{\reg})$, i.e.\ such that for any such map $\theta$ and almost every $x\in X$, one has $\theta(x)\subseteq\theta^{j}_{v,\max}(x)$. Since $\cala$ is normalized by $\caln$, Lemma~\ref{lemma:maximal-map} also ensures that $\theta^{j}_{v,\max}$ is stably $(\caln,\rho_j)$-equivariant.

Combining the maps $\theta^{j}_{v,\max}$ for all $j\in\{1,\dots,n\}$ and all $v\in V\Lambda_j$, we get a Borel stably $(\caln,\rho)$-equivariant map 
\begin{equation}\label{eq:Produit}
    X\to\prod_{j,v}\calp_{\le 2}((\partial_\infty T_v)^{\reg})
\end{equation}
By Lemma~\ref{lemma:decomposition-boundary}, \cpageref{lemma:decomposition-boundary}, for every $j\in\{1,\dots,n\}$, the action of $F_j$ on the non-empty set $(\partial_\infty T_v)^{\reg}$ is Borel amenable, and therefore so is its action on $\calp_{\le 2}((\partial_\infty T_v)^{\reg})$. So the action of $F_1\times\dots\times F_n$ on the above product is Borel amenable. 

We now prove that $N\cap C_0$ is not amenable, so assume towards a contradiction that it is. Then $N\subseteq (N\cap C_0)\times F_1\times\dots\times F_n$, and as $N\cap C_0$ is amenable it follows that the action of $N$ on the product in \cref{eq:Produit} is again Borel amenable. Since $\rho_{|\caln}$ has trivial kernel, \cite[Proposition~3.38]{GH-OutFn} (which recasts \cite[Proposition~4.33]{Kid-memoir}) ensures that $\caln$ is amenable, a contradiction.  
\end{proof}

\section{The irreducible case}\label{sec:strongly-reduced}
In this section and the next, we prove the key Proposition~\ref{prop:vertex-recognition} establishing the Vertex Recognition Property for the class of graph products $G$ with countably infinite vertex groups over finite simple graphs $\Gamma$ with no transvection and no partial conjugation. In the present section, we only consider the case where $\Gamma$ is irreducible, in other words $G$ does not split as a direct product non-trivially. And in the next section, we will provide the extra arguments to allow for $\Gamma$ to be reducible. When $\Gamma$ is irreducible, Lemma~\ref{lemma:strongly-reduced} ensures that $\Gamma$ is strongly reduced in the sense of Definition~\ref{de:strongly-reduced}, that is, it is not possible to obtain a new decomposition of $G$ as a graph product by collapsing a proper subgraph of $\Gamma$. The main goal of the present section is to prove the following proposition.

\begin{Prop}\label{prop:vertex-recognition-strongly-reduced}
The class of all graph products of countably infinite groups over transvection-free strongly reduced finite simple graphs (not reduced to one vertex) satisfies the Vertex Recognition Property. 
\end{Prop}

Note that here we allow graphs to have partial conjugations.

\subsection{Overview of the proof}\label{subsec:OverviewOfTheProofGroupoids}

Let $G$ be a graph product of countably infinite vertex groups over a transvection-free strongly reduced finite simple graph, not reduced to a vertex. Let $\calg$ be a measured groupoid over a standard probability space, and let $\rho:\calg\to G$ be an action-like cocycle. Given a measured subgroupoid $\calh\subseteq\calg$, our goal is to “recognize” when $(\calh,\rho)$ is of vertex type (see \cref{Def:VertexType}, \cpageref{Def:VertexType}) using a property that is purely phrased in terms of the structure of $\calh$, i.e.\ with no reference to the cocycle $\rho$. In this way, if $\calg$ comes equipped with two action-like cocycles $\rho_G:\calg\to G$ and $\rho_H:\calg\to H$, then subgroupoids of vertex type with respect to $\rho_G$, are also of vertex type with respect to $\rho_H$, and vice versa. 

The group-theoretic analogue of this problem is to recognize conjugates of vertex groups of $G$ using a purely group-theoretic property. Here we will formulate our strategy in the group-theoretic setting; in a sense, most of the technical work in the sequel of the present section will consist in “translating” this strategy to the groupoid-theoretic framework. With this in mind, we will formulate our characterization using amenability and normalization of subgroups, as these notions have their counterparts for groupoids.

Our starting point is the combinatorial statement provided by Lemma~\ref{lemma:graph}, \cpageref{lemma:graph}: starting from $G$, vertex groups and their conjugates are obtained by successively passing to maximal product parabolic subgroups and zooming in their factors (and at the last step, taking the clique factor). Thus, our goal decomposes into two tasks:
\begin{enumerate}
    \item Characterize maximal product parabolic subgroups of $G$.
    \item When $P$ is a product parabolic subgroup, characterize either 
    \begin{itemize}
    \item the factors of $P$, if $P$ has trivial clique factor, or
    \item the clique factor of $P$, if it is non-trivial. 
    \end{itemize}
\end{enumerate}

\paragraph{Task 1: Characterizing maximal products.}
 This task is carried in Section~\ref{sec:maximal-product}, following a strategy that was already used in \cite{HHI} in the context of right-angled Artin groups.

The idea is the following. If $P=C_0\times F_1\times\dots\times F_k$ is a maximal product of parabolic subgroups (where $C_0$ is the clique factor, and the $F_j$ are irreducible for all $j$), letting $A\subseteq F_j$ be an amenable subgroup and $N=C_0\times F_1\times\dots\times \hat{F}_j\times \dots\times F_k$ (where we have just removed the $j^{\text{th}}$ factor), then $A$ is normalized by $N$, and $N$ is normal in $P$. It turns out that conversely, maximal product parabolic subgroups of $G$ are characterized as those maximal subgroups $P$ of $G$ for which there exists an infinite amenable subgroup~$A$ of $P$, and a non-amenable subgroup $N\subseteq P$, such that $A$ is normalized by $N$, and $N\unlhd P$. This is proved in Lemma~\ref{lemma:maximal-product}, at the groupoid-theoretic level (Property~$\Pprod$ below is the groupoid-theoretic translation of the above property).

\paragraph{Task 2: Characterizing (clique) factors.}

This task is carried in Section~\ref{sec:factors}. Let $P$ be a product parabolic subgroup. 

As a warm-up, let us assume that we know in advance that $P=F_1\times\dots\times F_k$ has trivial clique factor (here we decomposed $P$ as a product of irreducible parabolic subgroups $F_i$). Then the \emph{co-factors} $F_1\times\dots\times \hat{F}_j\times\dots\times F_k$ are characterized as the maximal normal subgroups $N\unlhd P$ that normalize an infinite amenable subgroup (namely, any $A\subseteq F_j$ infinite amenable). And the factors $F_j$ are then characterized as the minimal non-trivial intersections of co-factors.

This approach is in general a bit too naive for two reasons:
\begin{itemize}
    \item It does not work if the clique factor $C_0$ of $P$ is non-trivial (indeed $C_0$ could contain a normal infinite amenable subgroup $A$, and then $N=P$ in the above would work).
    \item This would also necessit to be able to characterize in advance the existence or not of a non-trivial clique factor.
\end{itemize}
We follow a slightly different approach. This necessits to introduce some extra terminology. We write $P=\mathsf{F}_1\times \dots\times \mathsf{F}_k$, where each $\mathsf{F}_j$ is an irreducible parabolic subgroup (in particular the clique factor has been decomposed as a product of conjugates of vertex groups).

A subgroup $S\subseteq P$ is \emph{special} if $S\unlhd P$ and $N_G(S)\nsubseteq N_G(P)$. One should typically think that certain subproducts of $P$ (i.e.\ products of only some of the subgroups $\mathsf{F}_j$) could be special. See Figure~\ref{fig:Strategie_sous_prod_sepciaux}, \cpageref{fig:Strategie_sous_prod_sepciaux} for an illustration, where the special subroups are exactly the \textcolor{Orange}{orange} vertex groups in the clique factor. In fact in general, if the type of the clique factor is a complete graph on at least two vertices, then the subgroups $\mathsf{F}_j$ whose type is reduced to one vertex are automatically special, using that $\Gamma$ is transvection-free.

Now we say that an infinite subgroup $B\subseteq P$ is \emph{appropriate} if it satisfies the following two properties:
\begin{itemize}
    \item $N_G(B)\nsubseteq N_G(P)$;
    \item if $B$ is contained in a special subgroup $S\unlhd P$, then $N_G(B)\nsubseteq N_G(S)$.
\end{itemize}
We make two important observations:
\begin{itemize}
     \item For every $j\in\{1,\dots,k\}$, there exists an appropriate subgroup $B\subseteq \mathsf{F}_j$, which can be taken to be a conjugate of a vertex group. This is a consequence of our crucial assumption that $\Gamma$ is strongly reduced: we can find $B$ whose normalizer is not contained in $\mathsf{F}_j\times \mathsf{F}_j^{\perp}$, so regardless of whether $\mathsf{F}_j$ is special or not, the above conditions are satisfied. See Figure~\ref{fig:Strategie_sous_prod_sepciaux}, \cpageref{fig:Strategie_sous_prod_sepciaux} for an illustration, where the \textcolor{MFCB}{brown} vertex groups are appropriate.
    \item On the other hand, the clique factor $C_0$ does not contain any appropriate subgroup: indeed, if $B\subseteq C_0$ satisfies $N_G(B)\nsubseteq P\times P^{\perp}$, then the smallest subclique $C'_0\subseteq C_0$ containing $B$ also satisfies $N_G(C'_0)\nsubseteq P\times P^{\perp}$. So $C'_0$ is special, and $N_G(B)\subseteq N_G(C'_0)$. 
\end{itemize}
Armed with these observations, we now deduce that \emph{clique-inclusive cofactors} of $P$ (i.e.\ subgroups of the form $C_0\times F_1\times \dots\times \hat{F}_j\times\dots\times F_k$) are characterized as follows: they are the maximal subgroups $N\unlhd P$ that normalize an appropriate infinite subgroup of $P$.  

And the clique factor is then characterized as the intersection of all clique-inclusive cofactors. (If $C_0=\{1\}$, then factors are characterized as the minimal non-trivial intersection of clique-inclusive cofactors.)

\begin{figure}[htbp]
    \centering
    \includegraphics{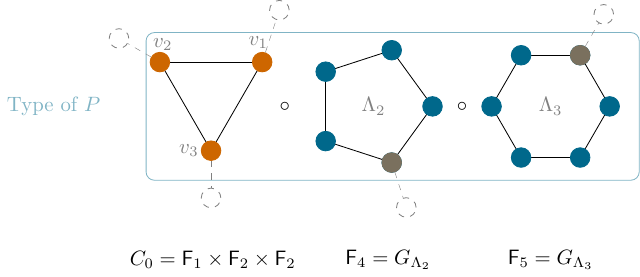}
    \caption{The \textcolor{Orange}{orange} vertex groups are special in $P$, and the \textcolor{MFCB}{brown} vertex groups are appropriate.}
    \label{fig:Strategie_sous_prod_sepciaux}
\end{figure}

\paragraph{Conclusion.} We finish this overview by mentioning that in Section~\ref{sec:vertex-groups}, we complete our proof of Proposition~\ref{prop:vertex-recognition-strongly-reduced} by combining the combinatorial zooming lemma (Lemma~\ref{lemma:graph}) with the previously obtained characterizations of maximal product parabolic subgroups and their factors.

Recall that \cref{Appendix:Isom}, \cpageref{Appendix:Isom} provides a group theoretic analogue of \cref{prop:vertex-recognition-strongly-reduced} and can be read as a warm up for the proof the latter proposition.

\subsection{Recognizing maximal products}\label{sec:maximal-product}

Let now $\calg$ be a measured groupoid over a standard probability space $X$, equipped with an action-like cocycle $\rho:\calg\to G$. Let $\calh$ be a measured subgroupoid of $\calg$. We say that $(\calh,\rho)$ is \emph{of parabolic type}\index{Parabolic!Parabolic type (groupoid framework)} if there exist a partition $X^*=\dunion_{i\in I} X_i$ of a conull Borel subset $X^*\subseteq X$ into at most countably many Borel subsets, and for every $i\in I$, a parabolic subgroup $P_i\subseteq G$ such that  $\calh_{|X_i}=\rho^{-1}(P_i)_{|X_i}$. We say that $(\calh,\rho)$ is \emph{of product type}\index{Product!Product type (groupoid framework)} (resp.\ \emph{of maximal product type}\index{Maximal product!Maximal product type (groupoid framework)}, resp.\ \emph{of clique type}\index{Clique!Clique type (groupoid framework)}, resp.\ \emph{of isolated clique type}\index{Isolated clique type (groupoid framework)} if in the above, one can take $P_i$ to be a product parabolic subgroup (resp.\ a maximal product parabolic subgroup, resp.\ a parabolic subgroup whose type is a clique, resp.\ a parabolic subgroup whose type is an isolated clique) for every $i\in I$. We say that $(\calg,\rho)$ is \emph{nowhere of clique type}\index{Nowhere!Nowhere of clique type} (resp.\ \emph{nowhere of isolated clique type})\index{Nowhere!Nowhere of isolated clique type} if there does not exist any positive measure Borel subset $U\subseteq X$ such that $(\calg_{|U},\rho)$ is of clique type (resp.\ of isolated clique type).

Our goal is now to give a partial characterization of pairs $(\calp,\rho)$ of maximal product type with no reference to the cocycle $\rho$ (see Remark~\ref{rk:prod} below for why we only reach a \emph{partial} characterization). For this we introduce the following property.

\begin{Def}[Property~$\Pprod$]\index{Property!$\Pprod$}\label{de:pprod}
Let $\calg$ be a measured groupoid over a standard probability space, and let $\calp$ be a measured subgroupoid of $\calg$. We say that the pair $(\calg,\calp)$ satisfies \emph{Property $\Pprod$} if it verifies the following assertions.
\begin{description}
    \item[$\Pprod_1$] There exist measured subgroupoids $\cala,\caln\subseteq\calp$ such that 
    \begin{enumerate}
        \item[(a)] $\cala$ is amenable and of infinite type, and is stably normalized by $\caln$;
        \item[(b)] $\caln$ is everywhere non-amenable, and is stably normalized by $\calp$;
    \end{enumerate}
    \item[$\Pprod_2$] Whenever $\calp'$ is a measured subgroupoid of $\calg$ that satisfies Property~$\Pprod_1$, and such that $\calp$ is stably contained in $\calp'$, then $\calp$ and $\calp'$ are stably equal. 
\end{description}
\end{Def}

\begin{Rq}
We mention the following crucial stabilities for Property~$\Pprod$.
\begin{enumerate}
    \item If $(\calg,\calp)$ satisfies Property~$\Pprod$, then for every positive measure Borel subset $U\subseteq X$, the pair $(\calg_{|U},\calp_{|U})$ also satisfies Property~$\Pprod$.
    \item If there exists a Borel partition $X^*=\dunion_{i\in I}X_i$ of a conull Borel subset $X^*\subseteq X$ into at most countably many Borel subsets, such that for every $i\in I$, the pair $(\calg_{|X_i},\calp_{|X_i})$ satisfies Property~$\Pprod$, then $(\calg,\calp)$ satisfies Property~$\Pprod$.
\end{enumerate}
These stabilities will allow us to take restrictions and work up to a countable Borel partition in the proofs below. All other properties that will appear later in this section (Properties~$\Pspec,\Padm,\Pfact,\Pvert$) and the next two also satisfy the same stabilities.
\end{Rq}

Property~$\Pprod$ is the groupoid-theoretic translation of the following idea: if $P=F_1\times\dots\times F_k$ is a maximal direct product in $G$, then (unless possibly if this product comes from a clique and no factor contains an infinite amenable subgroup) there exists an infinite amenable subgroup $A$ in some $F_i$, normalized by $N=F_1\times\dots\times\hat{F}_j\times\dots\times F_k$, which is itself normalized by $P$. It turns out that conversely, maximal subgroups of $G$ with this property are either product parabolic subgroups, or else they could be contained in a parabolic subgroup whose type is an isolated clique. The following lemma, which is inspired from \cite[Lemma~5.4]{HHI}, establishes this fact at the groupoid-theoretic level.

\begin{Lmm}\label{lemma:maximal-product}
Let $G$ be a graph product of countably infinite groups over a finite simple graph $\Gamma$. 
Let $\calg$ be a measured groupoid over a standard probability space $X$, equipped with an action-like cocycle $\rho:\calg\to G$. Let $\calp\subseteq\calg$ be a measured subgroupoid. Then 
\begin{enumerate}
\item If $(\calp,\rho)$ is of maximal product type and nowhere of isolated clique type, then $(\calg,\calp)$ satisfies Property~$\Pprod$.
\item If $(\calg,\calp)$ satisfies Property~$\Pprod$, then there exists a Borel partition $X=X^1\dunion X^2$ such that $(\calp_{|X^1},\rho)$ is of maximal product type, and $\calp_{|X^2}$ is stably contained in a subgroupoid $\calq$ of $\calg_{|X^2}$ such that $(\calq,\rho)$ is of isolated clique type.
\end{enumerate}
\end{Lmm}

\begin{Rq}\label{rk:prod}
When $\Gamma$ does not contain any isolated clique, then (as in \cite{HHI}) the above lemma gives a characterization of pairs $(\calp,\rho)$ of maximal product type that is independent of the choice of an action-like cocycle $\rho:\calg\to G$. But when $\Gamma$ has isolated cliques, whether or not subgroupoids of isolated clique type satisfy $\Pprod$ depends on the vertex groups. This is why there is an asymmetry in the statement of the lemma.

In particular, we mention that with our definitions, it is unclear to us whether or not a subgroupoid of the form $\rho^{-1}(C)$, where $C$ is a parabolic subgroup associated to an isolated clique on several vertices, with nonamenable vertex groups, always satisfies Property~$\Pprod$. This is the case as soon as one of the vertex groups contains an infinite amenable subgroup $A$, by taking $\cala=\rho^{-1}(A)$ and $\caln=\rho^{-1}(N)$, where $N$ is the direct product of the other vertex groups of the clique. The difficulty comes from the case where the vertex groups are torsion groups, e.g.\ free Burnside groups. This uncertainty required us to make a few adjustments in the statements and in the arguments.
\end{Rq}

\begin{proof}
    \textbf{Step 1} Let us first prove that if $(\calp,\rho)$ is of maximal product type and nowhere of isolated clique type, then $(\calg,\calp)$ satisfies Property~$\Pprod_1$.
    \smallskip
    
    Up to passing to a conull Borel subset of $X$ and performing a countable Borel partition of $X$, we can assume that $\calp=\rho^{-1}(P)$ for some maximal product parabolic subgroup $P$ not of isolated clique type. Write $P=P_1\times P_2$ as a product of two infinite parabolic subgroups. Up to changing the roles of $P_1$ and $P_2$, Lemma~\ref{lemma:maximal-product-nonamenable}, \cpageref{lemma:maximal-product-nonamenable} lets us assume that $P_2$ contains a non-abelian free subgroup. If $P_1$ is amenable, then the conclusion follows by letting $\cala=\rho^{-1}(P_1)$ and $\caln=\rho^{-1}(P_2)$: indeed 
    \begin{itemize}
        \item  $\cala$ is amenable because $A$ is and $\rho$ has trivial kernel (Lemma~\ref{lemma:amenable-subgroup-subgroupoid}, \cpageref{lemma:amenable-subgroup-subgroupoid}), of infinite type because $A$ is infinite and $\rho$ is action-like, and (stably) normalized by $\caln$ because $A$ is normalized by $P_2$ (Example~\ref{ex:normal}, \cpageref{ex:normal}). 
        \item $\caln$ is everywhere non-amenable because $P_2$ is non-amenable and $\rho$ is action-like, and $\caln$ is (stably) normalized by $\calp$ because $N$ is normalized by $P$ (Example~\ref{ex:normal}).
    \end{itemize}
    Otherwise, if $P_1$ is non-amenable, let $A\subseteq P_2$ be an infinite cyclic subgroup; then the conclusion follows by letting $\cala=\rho^{-1}(A)$ and $\caln=\rho^{-1}(P_1)$. 
    \medskip
    
    \textbf{Step 2} Let us prove that if $\calp$ satisfies $\Pprod_1$, then there exist a countable partition $X^*=\dunion_{i\in I}X_i$ of a conull Borel subset $X^*\subseteq X$ into at most countably many Borel subsets $X_i$, and for every $i\in I$, a parabolic subgroup $Q_i$ which is either a product or of isolated vertex type, such that $\calp_{|X_i}\subseteq\rho^{-1}(Q_i)_{|X_i}$. 
   
  Let $\cala,\caln\subseteq\calp$ be measured subgroupoids given by $\Pprod_1$. Up to replacing $X$ by a conull Borel subset and taking a countable Borel partition, we will assume that       
\begin{itemize}
\item $\cala$ is normalized by $\caln$, and $\caln$ is normalized by $\calp$;
\item $(\cala,\rho)$ is tightly $A$-supported for some parabolic subgroup $A\subseteq G$ (Lemma~\ref{lemma:support}, \cpageref{lemma:support}) --~notice that $A\neq\{1\}$ because $\cala$ is of infinite type and $\rho$ has trivial kernel; 
\item $(\caln,\rho)$ is tightly $N$-supported for some parabolic subgroup $N\subseteq G$ (Lemma~\ref{lemma:support}) -- notice that $N$ is non-amenable because $\caln$ is everywhere non-amenable and $\rho$ has trivial kernel, see Lemma~\ref{lemma:amenable-subgroup-subgroupoid}, \cpageref{lemma:amenable-subgroup-subgroupoid};
\item $\caln\subseteq\rho^{-1}(A\times A^{\perp})$ and $\calp\subseteq\rho^{-1}(N\times N^{\perp})$ using the invariance of the support by the normalizer (Lemma~\ref{lemma:support-normal}).
\end{itemize}
Notice in particular that $N\subseteq A\times A^{\perp}$, by definition of the support. Now there are two cases.
    
\begin{itemize}    
    \item Case 1: Either $N$ splits as a direct product of two infinite parabolic subgroups, or $N^{\perp}\neq\{1\}$. In this case $Q:=N\times N^{\perp}$ is a product parabolic subgroup, we are done since $\calp\subseteq\rho^{-1}(Q)$. 
    \item Case 2: $N$ does not split as a direct product of two infinite parabolic subgroups, and $N^{\perp}=\{1\}$. Then $\calp \subseteq \rho^{-1}(N)$, and we will now show that in this case $N$ is of isolated clique type, which will conclude. As $N\subseteq A\times A^{\perp}$, we deduce using Lemma~\ref{lemma:product-parabolic-subgroup}, \cpageref{lemma:product-parabolic-subgroup} that $N=(N\cap A) \times (N\cap A^\perp)$. But we assumed that $N$ does not split as a direct product, so either $N\subseteq A$ or $N\subseteq A^\perp$. If the latter holds we have $\{1\}\subsetneq A \subseteq N^\perp$ which contradicts our assumption that $N^\perp$ is trivial. Hence $N\subseteq A$, and $A^{\perp}=\{1\}$ because $N^{\perp}=\{1\}$. Lemma~\ref{lemma:adams-argument-v2}, \cpageref{lemma:adams-argument-v2} (applied to the amenable groupoid $\cala$, which is normalized by the non-amenable groupoid~$\caln$) thus ensures that the clique factor $A_0$ of $A$ is non-amenable, and $N\cap A_0\neq\{1\}$. As $N\subseteq A$, as $N\cap A_0\neq\{1\}$, as above using again Lemma~\ref{lemma:product-parabolic-subgroup}, \cpageref{lemma:product-parabolic-subgroup}, and that $N$ does not split as a direct product we obtain that 
    $N$ is equal to a conjugate of a vertex group. Since $N^{\perp}=\{1\}$, this vertex group is isolated. This completes the proof in Case~2. 
    \end{itemize}
    \medskip
    
    \textbf{Step 3} We now prove that if $(\calp,\rho)$ is of maximal product type, then $(\calg,\calp)$ satisfies Property~$\Pprod_2$: this will complete the proof of the first part of the lemma.
    \smallskip
    
    Let $\calp'$ be a measured subgroupoid of $\calg$ that satisfies Property~$\Pprod_1$, such that $\calp$ is stably contained in $\calp'$. By Step~2, we can find a countable partition $X^*=\dunion_{i\in I}X_i$ of a conull Borel subset $X^*\subseteq X$ into at most countably many Borel subsets $X_i$, and for every $i\in I$, a parabolic subgroup $Q_i$ which is either a product or of isolated clique type, such that $\calp'_{|X_i}\subseteq\rho^{-1}(Q_i)_{|X_i}$. Without loss of generality, we will assume that all subsets $X_i$ have positive measure. Since $\calp$ is stably contained in $\calp'$, up to passing again to a conull Borel subset and performing a countable partition, we will assume that $\calp_{|X_i}\subseteq\rho^{-1}(Q_i)_{|X_i}$. 
    Since $(\calp,\rho)$ is of maximal product type, up to a further partition again, we have $\calp_{|X_i}=\rho^{-1}(P_i)_{|X_i}$ for some maximal product parabolic subgroup $P_i$. It thus follows from Lemma~\ref{lemma:inclusion-parabolics}, \cpageref{lemma:inclusion-parabolics} that $P_i\subseteq Q_i$. In particular, the type of $Q_i$ is not reduced to one vertex. It thus follows from Step~2 that $Q_i$ is a product. The maximality of $P_i$ therefore implies that $P_i=Q_i$.
    It follows from the above that $\calp'$ is stably contained in $\calp$, showing that $\calp$ satisfies Property~$\Pprod_2$.
    \medskip
    
    \textbf{Step 4} We finally prove the second assertion of the lemma.
    \smallskip
    
    So assume that $(\calg,\calp)$ satisfies Property~$\Pprod$. By Step~2, there exist a countable partition $X^*=\dunion_{i\in I}X_i$ of a conull Borel subset $X^*\subseteq X$ into at most countably many Borel subsets $X_i$, and for every $i\in I$, a parabolic subgroup $Q_i$ which is either of product type or of isolated clique type, such that $\calp_{|X_i}\subseteq\rho^{-1}(Q_i)_{|X_i}$. Up to increasing $Q_i$, we will assume that $Q_i$ is either a {maximal} product or of isolated clique type (in particular $Q_i$ is not a proper subgroup of a parabolic subgroup of isolated clique type). Let $I_2\subseteq I$ be the subset consisting of all $i$ such that $Q_i$ is of isolated clique type, and let $I_1=I\setminus I_2$. Let $X^1$ (resp.\ $X^2$) be the union of all $X_i$ with $i\in I_1$ (resp.\ with $i\in I_2$). Let us now prove that $\calp_{|X_i}$ is actually (stably) equal to $\rho^{-1}(Q_i)_{|X_i}$ for all $i\in I_1$. So let $\calp'$ be a subgroupoid of $\calg$ such that for every $i\in I_1$, one has $\calp'_{|X_i}=\rho^{-1}(Q_i)_{|X_i}$, with $\calp'_{|X^2}=\calp_{|X^2}$. Then $\calp$ is stably contained in $\calp'$. In addition $\calp'$ satisfies Property~$\Pprod_1$: this uses Step~1 of this proof for $\calp'_{|X^1}$, and the fact that $\calp$ satisfies Property~$\Pprod$ for $\calp'_{|X^2}$. So $\Pprod_2$ for $(\calg,\calp)$ implies that $\calp$ and $\calp'$ are stably equal. In particular the second assertion of the lemma follows with $\calq$ a groupoid verifying $\calq_{|X_i}:=\rho^{-1}(Q_i)_{|X_i}$ for every $i\in I_2$.  
\end{proof}

\subsection{Recognizing factors in products}\label{sec:factors}

\paragraph{Special subproducts.} 

Let $P=\mathsf{F}_1\times\dots\times \mathsf{F}_n$ be a product parabolic subgroup, where the subgroups $\mathsf{F}_j$ with $j\ge 1$ are parabolic subgroups that do not further split as direct products of parabolic subgroups (in particular the clique factor of $P$ is fully decomposed). A \emph{subproduct}\index{Subproduct} is a subgroup of $P$ equal to the product of finitely many (possibly only one) factors $\mathsf{F}_j$. A subproduct $S$ of $P$ is \emph{special within $G$}\index{Special subproduct!Special subproduct of a parabolic group}\index{Subproduct!Special subproduct of a parabolic group} if $N_G(S)\supsetneq N_G(P)$, 
that is to say $S\times S^\perp \supsetneq P \times P^\perp$. 

Let now $\calg$ be a measured groupoid over a standard probability space $X$, and let $\rho:\calg\to G$ be an action-like cocycle. Let $\calp$ be a measured subgroupoid such that $(\calp,\rho)$ is of product type. Let $X^*=\dunion_{i\in I}X_i$ be a partition of a conull Borel subset $X^*\subseteq X$, and for every $i\in I$, $P_i$ be a  parabolic subgroup of $G$ of product type, such that $\calp_{|X_i}=\rho^{-1}(P_i)_{|X_i}$. Let $\cals$ be a measured subgroupoid of $\calp$. We say that $(\cals,\rho)$ is \emph{of special subproduct type within $(\calp,\rho)$}\index{Special subproduct!Special subproduct type (groupoid framework)} if, up to passing to a further conull Borel subset and refining the above partition, for every $i\in I$, there exists a special subproduct $S_i$ of $P_i$ such that $\cals_{|X_i}=\rho^{-1}(S_i)_{|X_i}$. We observe that this definition does not depend on the decomposition of $X$ we chose.

We highlight the following property, which translates the notion of a special subproduct at the groupoid-theoretic level.

\begin{Def}[Property $\Pspec$]\index{Property!$\Pspec$}\label{de:pspec}
Let $\calg$ be a measured groupoid over a standard probability space $X$, and let $\cals\subseteq\calp$ be two measured subgroupoids of $\calg$. We say that the triple $(\calg,\calp,\cals)$ satisfies \emph{Property $\Pspec$} if the following two properties hold.
\begin{description}
\item[$\Pspec_1$] The subgroupoid $\cals$ is stably normal in $\calp$.
\item[$\Pspec_2$] There exists a measured subgroupoid $\caln\subseteq\calg$ that stably normalizes $\cals$, such that for every measured subgroupoid $\caln'$ of $\calg$ that stably normalizes $\calp$, and every Borel subset $U\subseteq X$ of positive measure, one has $\caln_{|U}\nsubseteq\caln'_{|U}$.
\end{description}
\end{Def} 

\begin{Lmm}\label{lemma:spec}
Let $G$ be a graph product of countably infinite groups over a finite simple graph $\Gamma$. Let $\calg$ be a measured groupoid over a standard probability space $X$, equipped with an action-like cocycle $\rho:\calg\to G$. Let $\cals\subseteq\calp\subseteq\calg$ be measured subgroupoids, with $(\calp,\rho)$ of product type.  
\begin{enumerate}
    \item If $(\cals,\rho)$ is of special subproduct type within $(\calp,\rho)$, then $(\calg,\calp,\cals)$ satisfies Property~$\Pspec$.
    \item If $(\calg,\calp,\cals)$ satisfies Property~$\Pspec$, then there exist a countable Borel partition $X^*=\dunion_{i\in I}X_i$ of a conull Borel subset $X^*\subseteq X$, and for every $i\in I$, a product parabolic subgroup $P_i$, and a special subproduct $S_i\subseteq P_i$, such that $\calp_{|X_i}=\rho^{-1}(P_i)_{|X_i}$ and $(\cals_{|X_i},\rho)$ is tightly $S_i$-supported.
\end{enumerate}
\end{Lmm}

Remark that in our definition of $\Pspec$ --~in contrast with $\Pprod$~-- we do not have any maximality condition. This explains the asymmetry in the above lemma.

\begin{proof}
Without loss of generality, we can assume that $\calp=\rho^{-1}(P)$ for some product parabolic subgroup $P\subseteq G$.

We first prove Assertion~1, so assume that $(\cals,\rho)$ is of special subproduct type within $(\calp,\rho)$. That $(\calg,\calp,\cals)$ satisfies $\Pspec_1$ follows from the fact that subproducts are normal in the ambient product, and Example~\ref{ex:normal}, \cpageref{ex:normal}. We now prove that it satisfies $\Pspec_2$. Up to passing to a conull Borel subset and taking a countable partition of $X$, we will assume that $\cals=\rho^{-1}(S)$ for some special subproduct $S\subseteq P$. Let $\caln=\rho^{-1}(S\times S^{\perp})$, which normalizes $\cals$. Let $\caln'$ be a measured subgroupoid of $\calg$ that stably normalizes $\calp$. Let $U\subseteq X$ be a positive measure Borel subset. Since $(\calp,\rho)$ is tightly $P$-supported (Lemma~\ref{lemma:tight-support}, \cpageref{lemma:tight-support}) and stably normalized by $\caln'$, Lemma~\ref{lemma:support-normal}, \cpageref{lemma:support-normal} ensures that there exists a positive measure Borel subset $V\subseteq U$ such that $\caln'_{|V}\subseteq \rho^{-1}(P\times P^\perp)_{|V}$. Since $S$ is a special subproduct, we have $S\times S^{\perp}\nsubseteq P\times P^{\perp}$. It thus follows from Lemma~\ref{lemma:inclusion-parabolics}, \cpageref{lemma:inclusion-parabolics} that $\caln_{|V}\nsubseteq\caln'_{|V}$; in particular $\caln_{|U}\nsubseteq\caln'_{|U}$, so $\Pspec_2$ holds.

We now prove the second assertion, so assume that $(\calg,\calp,\cals)$ satisfies Property $\Pspec$. Without loss of generality, we can assume that $(\cals,\rho)$ is tightly $S$-supported for some parabolic subgroup $S\subseteq P$. Since $\cals$ is stably normal in $\calp$ (by $\Pspec_1$), ~\cref{lemma:support-normal}, \cpageref{lemma:support-normal}, ensures that $\calp=\rho^{-1}(P)$ is stably contained in $\rho^{-1}(S\times S^{\perp})$, and \cref{lemma:inclusion-parabolics}, \cpageref{lemma:inclusion-parabolics} ensures that $P\subseteq S\times S^{\perp}$. Therefore $S$ is a subproduct of $P$ (see \cref{lemma:product-parabolic-subgroup}).
Let $\caln$ be a measured groupoid as in $\Pspec_2$. Let $\caln'=\rho^{-1}(P\times P^\perp)$, which normalizes $\calp$. Since $\caln$ stably normalizes~$\cals$, by Lemma~\ref{lemma:support-normal}, there exists a positive measure Borel subset $U\subseteq X$ such that $\caln_{|U}\subseteq\rho^{-1}(S\times S^{\perp})_{|U}$. It follows from $\Pspec_2$ that $S\times S^{\perp}$ is not contained in $P\times P^{\perp}$. So $S$ is a special subproduct, which completes the proof.
\end{proof}

\paragraph{Clique-inclusive co-factors.}

Let $P$ be a parabolic subgroup of product type. Write $P=C_0\times {F}_1\times\dots\times {F}_k$ as a direct product of parabolic subgroups, where the $F_j$ do not further split as a direct product of proper parabolic subgroups and $C_0$ denotes the clique factor of $P$. A \emph{clique-inclusive co-factor}\index{Clique-inclusive!Co-factor (subgroup)} of $P$ is a subgroup of the form $C_0\times F_1\times\dots\times \hat {{F}}_{j}\times\dots\times F_k$ for some $j\in\{1,\dots,k\}$. By convention, if $P=C_0$, there is no clique-inclusive co-factor.

Let now $\calg$ be a measured groupoid over a standard probability space $X$, and let $\rho:\calg\to G$ be an action-like cocycle. Let $\calp$ be a measured subgroupoid such that $(\calp,\rho)$ is of product type. Let $X^*=\dunion_{i\in I}X_i$ be a partition of a conull Borel subset $X^*\subseteq X$, and for every $i\in I$, $P_i$ be a product parabolic subgroup of $G$, such that $\calp_{|X_i}=\rho^{-1}(P_i)_{|X_i}$. Let $\calq$ be a measured subgroupoid of $\calp$. We say that $(\calq,\rho)$ is \emph{of clique-inclusive co-factor type within $(\calp,\rho)$}\index{Clique-inclusive!Co-factor type (groupoid framework)} if, up to passing to a further conull Borel subset and refining the above partition, for every $i\in I$, there exists a clique-inclusive co-factor $Q_i$ of $P_i$ such that $\calq_{|X_i}=\rho^{-1}(Q_i)_{|X_i}$. 

We now aim to give a characterization of subgroupoids of clique-inclusive co-factor type. For this, we highlight the following property, an informal explanation is given below~it.

\begin{Def}[Property~$\Padm$]\index{Property!$\Padm$}\label{de:padm}
 Let $\calg$ be a measured groupoid over a standard probability space $X$. Let $\calq\subseteq\calp$ be measured subgroupoids of $\calg$. We say that the triple $(\calg,\calp,\calq)$ satisfies \emph{Property~$\Padm$} if the following three properties hold.
\begin{description}
\item[$\Padm_1$] The groupoid $\calq$ is stably normal in $\calp$.
    \item[$\Padm_2$] There exists a measured subgroupoid $\calb\subseteq \calp$ of infinite type such that 
    \begin{enumerate}
        \item[(a)] the groupoid $\calb$ is stably normalized by a measured subgroupoid $\caln$ of $\calg$ such that 
        \begin{enumerate}
        \item[i.] whenever $\caln'$ is a measured subgroupoid of $\calg$ that stably normalizes $\calp$, for every positive measure Borel subset $U\subseteq X$, one has $\caln_{|U}\nsubseteq\caln'_{|U}$;
        \item[ii.] for every positive measure Borel subset $U\subseteq X$, if $\calb_{|U}$ is contained in a measured subgroupoid $\cals$ of $\calp_{|U}$ such that $(\calg_{|U},\calp_{|U},\cals)$ satisfies Property $\Pspec$, and if~$\caln'$ is a measured subgroupoid of $\calg_{|U}$ that stably normalizes~$\cals$, then for every positive measure Borel subset $V\subseteq U$, one has $\caln_{|V} \nsubseteq \caln'_{|V}$;
        \end{enumerate}
        \item[(b)] the groupoid $\calq$ stably normalizes $\calb$.
        \end{enumerate}
    \item[$\Padm_3$] Whenever $\calq'$ is a measured subgroupoid of $\calp$ such that $(\calg,\calp,\calq')$ satisfies Properties~$\Padm_1$ and~$\Padm_2$, and such that $\calq$ is stably contained in $\calq'$, then $\calq$ and $\calq'$ are stably equal. 
\end{description}
\end{Def}

The idea behind Property~$\Padm$ is the following. Let $P=F_1\times\dots\times F_n$ be a product parabolic subgroup, and for simplicity let us assume that its clique factor is trivial. Let $j\in\{1,\dots,n\}$. The group $Q=F_1\times\dots\times\hat{F}_j\times\dots\times F_n$ is normal in $P$, as in $\Padm_1$. Under our assumption that the defining graph $\Gamma$ is strongly reduced, the subgraph $\Lambda_j$ that gives the type of $F_j$ is not collapsible. We can therefore find $B\subseteq F_j$, conjugate to a vertex group, whose normalizer is not contained in the normalizer of $F_j$. In particular $N_G(B)$ is not contained in $N_G(P)$ (as expressed in its groupoid-theoretic form by $\Padm_2(a.i)$). And if $F_j$ is contained in a special subproduct~$S$, then $N_G(B)$ is also not contained in $N_G(S)$ (as expressed in its groupoid-theoretic form by $\Padm_2(a.ii)$). Finally notice that $B$ is normalized by $Q$, as in $\Padm_2(b)$. Conversely, we will prove (in a groupoid-theoretic context) that these properties are enough to ensure that a subgroup(oid) is contained in one of co-factor type. And the maximality property~$\Padm_3$ is here to ensure that we get the full co-factor, not just a proper subgroup(oid). (In the case where the clique factor of $P$ is non-trivial, the above remains true, except that it now only provides a characterization of the co-factors which are clique-inclusive.)  This is the contents of our next lemma, which as explained above crucially uses the assumption that the underlying graph $\Gamma$ is strongly reduced.

\begin{Lmm}\label{lemma:admissible}
Let $G$ be a graph product of countably infinite groups over a strongly reduced finite simple graph $\Gamma$. Let $\calg$ be a measured groupoid over a standard probability space $X$, equipped with an action-like cocycle $\rho:\calg\to G$. Let $\calq\subseteq\calp\subseteq\calg$ be measured subgroupoids, with $(\calp,\rho)$ of  product type and nowhere of clique type.

Then $(\calq,\rho)$ is of clique-inclusive co-factor type within $(\calp,\rho)$ if and only if $(\calg,\calp,\calq)$ satisfies Property~$\Padm$.
\end{Lmm}

We summarize in \cref{fig:Demo612} the objects appearing in the proof and the link between them.

\begin{figure}
    \centering
    \includegraphics[width=\textwidth]{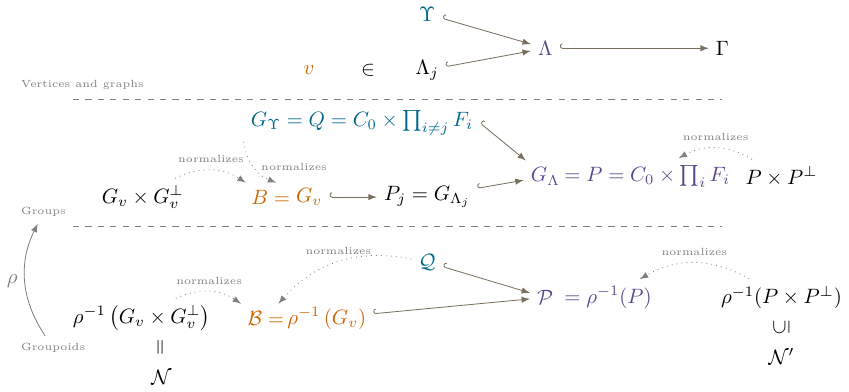}
    \caption{Objects appearing in Step~1 of the proof of \cref{lemma:admissible}}
    \label{fig:Demo612}
\end{figure}

\begin{proof}
\textbf{Step 1} Let us prove that if $(\calq,\rho)$ is of clique-inclusive co-factor type within $(\calp,\rho)$, then $(\calg,\calp,\calq)$ satisfies $\Padm_1$ and $\Padm_2$.
\smallskip

Notice that if $(\calq,\rho)$ is of clique-inclusive co-factor type within $(\calp,\rho)$, then $(\calg,\calp,\calq)$ satisfies $\Padm_1$, because any co-factor is normal in the ambient product (see Example~\ref{ex:normal}). We now prove that $(\calg,\calp,\calq)$ also satisfies Property~$\Padm_2$.  

Up to replacing $X$ by a conull Borel subset and partitioning it into countably many Borel subsets, we can (and will) assume that $\calp=\rho^{-1}(P)$ and $\calq=\rho^{-1}(Q)$, where 
\begin{itemize}
\item $P=C_0\times F_1\times\dots\times F_k$ is a product parabolic subgroup, with (possibly trivial) clique factor $C_0$, with $k\ge 1$, and no $F_i$ decomposes further as a direct product;
\item And $Q=C_0\times F_1\times\dots\times\hat{F}_j\times\dots\times F_k$ for some $j\in\{1,\dots,k\}$. 
\end{itemize}

We now fix $j\in\{1,\dots,k\}$ given by the description of $Q$. Up to a global conjugation, we will also assume that $P=G_\Lambda$ and $Q=G_\Upsilon$ for some induced subgraphs $\Upsilon\subseteq\Lambda\subseteq\Gamma$. We let $\Lambda_j\subseteq\Lambda$ be the subgraph such that $F_j=G_{\Lambda_j}$ --~notice that $\Lambda_j$ is not reduced to a point. We recall from Lemma~\ref{lemma:tight-support}, \cpageref{lemma:tight-support} that $(\calp,\rho)$ is tightly $P$-supported.

Since $\Gamma$ is strongly reduced, the subgraph $\Lambda_j$ (which is not reduced to a point) is not collapsible. Therefore, there exists a vertex $v\in V\Lambda_j$ which is joined to a vertex $w\notin V\Lambda_j$, such that $w$ is not joined to all vertices of $\Lambda_j$. Notice in particular that 
\begin{equation}\label{eq:inclusionGw}
    G_w \subseteq G^\perp_v \quad \text{and}
     \quad G_w\nsubseteq F_j\times F^\perp_j \quad \text{and} 
    \quad G_w\nsubseteq P\times P^\perp.
\end{equation}
Let $\calb=\rho^{-1}(G_v)$. Then $\calb$ is contained in $\calp$, is of infinite type (because $G_v$ is infinite and $\rho$ is action-like) and is normalized by $\calq$ (because $G_v$ is normalized by $Q$, and using Example~\ref{ex:normal}). This proves $\Padm_2(b)$. 

We now prove $\Padm_2(a)$. Let $\caln=\rho^{-1}(G_v\times G_v^{\perp})$, which normalizes $\calb$. 

To check $\Padm_2(a.i)$, let $\caln'$ be a measured subgroupoid of $\calg$ that stably normalizes~$\calp$. Without loss of generality, we will assume that $\caln'$ normalizes $\calp$. Since $(\calp,\rho)$ is tightly $P$-supported, Lemma~\ref{lemma:support-normal}, \cpageref{lemma:support-normal} ensures that $\caln'\subseteq\rho^{-1}(P\times P^{\perp})$ (up to restricting to a conull Borel subset). And \cref{eq:inclusionGw} implies that $G_v\times G_v^{\perp}\nsubseteq P\times P^\perp$. Lemma~\ref{lemma:inclusion-parabolics}, \cpageref{lemma:inclusion-parabolics}, thus ensures that for every positive measure Borel subset $U\subseteq X$, one has $\caln_{|U}\nsubseteq\caln'_{|U}$, as desired.

To check $\Padm_2(a.ii)$, let $U,\cals,\caln',V$ be as in the statement. Let $W\subseteq V$ be a positive measure Borel subset such that $\caln'_{|W}$ normalizes $\cals_{|W}$. Since $(\calg_{|W},\calp_{|W},\cals_{|W})$ satisfies Property~$\Pspec$, the second point of Lemma~\ref{lemma:spec} ensures that, up to replacing $W$ by a positive measure subset, $(\cals_{|W},\rho)$ is tightly $S$-supported for some special subproduct $S$ of $P$. Since by assumption $\calb_{|W}=\rho^{-1}(G_v)_{|W}$ is contained in $\cals_{|W}$, it follows from Lemma~\ref{lemma:inclusion-parabolics}, \cpageref{lemma:inclusion-parabolics} that $G_v\subseteq S$. So $F_j\subseteq S$ since $S$ is a subproduct. By Lemma~\ref{lemma:support-normal}, up to replacing $W$ by a conull Borel subset, we have $\caln'_{|W}\subseteq \rho^{-1}(S\times S^{\perp})_{|W}$, which in turn is contained in $\rho^{-1}(F_j\times F_j^{\perp})_{|W}$. The first two inclusions of \cref{eq:inclusionGw} give that $G_v\times G_v^{\perp}\nsubseteq F_j\times F_j^\perp$, so Lemma~\ref{lemma:inclusion-parabolics} implies that $\caln_{|W}\nsubseteq\caln'_{|W}$. And therefore $\caln_{|V}\nsubseteq\caln'_{|V}$, as desired.
\medskip

\textbf{Step 2} We now assume that $(\calg,\calp,\calq)$ satisfies Properties~$\Padm_1$ and~$\Padm_2$, and we prove that there exists a measured subgroupoid $\calq^\prime$ of $\calg$ such that $(\calq^\prime,\rho)$ is of clique-inclusive co-factor type within $(\calp,\rho)$, and $\calq$ is stably contained in $\calq'$.
\smallskip

Let $\calb,\caln$ be as in $\Padm_2$. Again, up to passing to a conull Borel subset of $X$ and partitioning it into at most countably many Borel subsets, we will assume that $\calp=\rho^{-1}(P)$ where $P=\mathsf{F}_1\times\dots\times \mathsf{F}_n$ is a product parabolic subgroup of $G$, where the $\mathsf{F}_i$ are chosen not to split further (in particular the clique subgroup of $P$ is decomposed as the product of its vertex groups). We will further assume that $\calb$ is normalized by $\caln$ and $\calq$, and that $\calq$ is normal in $\calp$. Using Lemma~\ref{lemma:support}, we can (and will) also assume that $(\calb,\rho)$ is tightly $B$-supported for some parabolic subgroup $B\subseteq P$, and that $(\calq,\rho)$ is tightly $Q$-supported for some parabolic subgroup $Q\subseteq P$.

Since $\calb$ is normalized by $\caln$, Lemma~\ref{lemma:support-normal} ensures (up to restricting to a conull Borel subset) that $\caln\subseteq\rho^{-1}(B\times B^{\perp})$. For $j\in\{1,\dots,n\}$, let $B_j=B\cap \mathsf{F}_j$. 

We claim that there exists $j\in\{1,\dots,n\}$ such that $B_j\neq\{1\}$ and $B_j\neq \mathsf{F}_j$. Indeed, otherwise $B$ would be a subproduct. But in this case, if $B$ was a nonspecial subproduct, then $B\times B^{\perp}\subseteq P\times P^\perp$, so $\caln\subseteq\rho^{-1}(B\times B^{\perp})\subseteq \rho^{-1}(P\times P^\perp)$. Since $\caln':= \rho^{-1}(P\times P^{\perp})$ normalizes $\calp$, this gives a contradiction to Property~$\Padm_2(a.i)$. If $B$ was special, then $(\calg,\calp,\rho^{-1}(B))$ would satisfy Property~$\Pspec$ (Lemma~\ref{lemma:spec}), and the inclusion $\caln\subseteq\rho^{-1}(B\times B^\perp)$ would give a contradiction to Property~$\Padm_2(a.ii)$, applied this time to $\caln^\prime:=\rho^{-1}(B\times B^\perp)$. Hence our claim.

Let $j\in\{1,\dots,n\}$ be as above. The above claim shows that $\mathsf{F}_j$ contains $B_j$ as a proper parabolic subgroup, so necessarily $\mathsf{F}_j$ is not conjugate to a vertex group. Since $\calb$ is normalized by $\calq$, we have $\calq\subseteq\rho^{-1}(B\times B^{\perp})$ (Lemma~\ref{lemma:support-normal}), so $Q\subseteq B\times B^{\perp}$. In particular $Q\cap \mathsf{F}_j$ is contained in $B_j\times B_j^{\perp}$. So $Q\subseteq \mathsf{F}_1\times\dots\times \mathsf{F}_{j-1}\times ((B_j\times B_j^{\perp})\cap \mathsf{F}_j)\times \mathsf{F}_{j+1}\times\dots\times \mathsf{F}_n$. Since $\calq$ is normal in $\calp=\rho^{-1}(P)$ and $(\calq,\rho)$ is tightly $Q$-supported, we have $P\subseteq Q\times Q^{\perp}$, and we deduce that $Q\subseteq \mathsf{F}_1\times\dots\times \hat {\mathsf{F}}_{j}\times\dots\times \mathsf{F}_n$. Let $\calq'=\rho^{-1}(\mathsf{F}_1\times\dots\times \hat{\mathsf{F}}_{j}\times\dots\times \mathsf{F}_n)$. Since $\mathsf{F}_j$ is not conjugate to a vertex group, the subgroup $\mathsf{F}_1\times\dots\times \hat{\mathsf{F}}_{j}\times\dots\times \mathsf{F}_n$ is a clique-inclusive co-factor, and therefore $(\calq',\rho)$ is of clique-inclusive co-factor type within $(\calp,\rho)$. So Step 1 of this proof shows that $\calq'$ satisfies Properties~$\Padm_1$ and~$\Padm_2$. Moreover $\calq\subseteq\calq'$ (up to a conull Borel subset). Hence the wanted assertion. 
\medskip

\noindent\textbf{Step 3} We now assume that $(\calg,\calp,\calq)$ satisfies Property~$\Padm$, and prove that $(\calq,\rho)$ is of clique-inclusive co-factor type within $(\calp,\rho)$.
\smallskip

Step 2 gives a measured subgroupoid $\calq^\prime$ of $\calg$ such that $(\calq^\prime,\rho)$ is of clique-inclusive type within $(\calp,\rho)$, and $\calq$ is stably contained in $\calq'$. Step~1 then ensures that $(\calg,\calp,\calq')$ satisfies $\Padm_1$ and $\Padm_2$. And Assertion~$\Padm_3$ then implies that $\calq$ is stably equal to $\calq'$, showing that $(\calq,\rho)$ is of clique-inclusive co-factor type within $(\calp,\rho)$.
\medskip

\noindent\textbf{Step 4} We are left with showing that if $(\calq,\rho)$ is of clique-inclusive co-factor type within $(\calp,\rho)$, then $(\calg,\calp,\calq)$ satisfies $\Padm_3$.
\smallskip

So assume that $\calq=\rho^{-1}(Q)$ for some clique-inclusive co-factor $Q$ of $P$, and let $\calq'$ be a measured subgroupoid of $\calp$ such that $(\calg,\calp,\calq')$ satisfies $\Padm_1$ and $\Padm_2$, such that $\calq$ is stably contained in $\calq'$. By Step~2 applied to the triple $(\calg,\calp,\calq')$, up to a conull Borel subset and a countable partition of the base space, $\calq'$ is contained in $\rho^{-1}(Q')$ for some clique-inclusive co-factor $Q'$ of $P$. Lemma~\ref{lemma:inclusion-parabolics} implies that $Q\subseteq Q'$, and since $Q$ and $Q'$ are two co-factors of $P$ it follows that $Q=Q'$, which concludes our proof.
\end{proof}

\paragraph{Factors.}

Let $\calg$ be a measured groupoid over a standard probability space $X$, and let $\rho:\calg\to G$ be an action-like cocycle. Let $\calp$ be a measured subgroupoid of $\calg$ such that $(\calp,\rho)$ is of parabolic type. Let $X^*=\dunion_{i\in I}X_i$ be a partition of a conull Borel subset $X^*\subseteq X$ into at most countably many Borel subsets, and for every $i\in I$, let $P_i$ be a  parabolic subgroup of $G$, such that $\calp_{|X_i}=\rho^{-1}(P_i)_{|X_i}$. Let $\calf$ be a measured subgroupoid of $\calp$. We say that $(\calf,\rho)$ is \emph{of factor type within $(\calp,\rho)$}\index{Factor!Factor type (groupoid framework)} if, up to passing to a further conull Borel subset and refining the above partition, for every $i\in I$, there exists a factor $F_i$ of $P_i$ such that $\calf_{|X_i}=\rho^{-1}(F_i)_{|X_i}$. We say that $(\calf,\rho)$ is \emph{of clique factor type within $(\calp,\rho)$}\index{Clique!Clique factor type (groupoid framework)} if additionally, for every $i\in I$ such that $X_i$ has positive measure, $F_i$ is the clique factor of~$P_i$.

We now aim to characterize these notions using a purely groupoid-theoretic property, which is as follows.

\begin{Def}[Property~$\Pfact$]\index{Property!$\Pfact$}\label{de:pfact}
Let $\calg$ be a measured groupoid over a standard probability space $X$. Let $\calf\subseteq\calp$ be measured subgroupoids of $\calg$. We say that the triple $(\calg,\calp,\calf)$ satisfies \emph{Property~$\Pfact$} if the following three properties hold.
\begin{description}
\item[$\Pfact_1$] $\calf$ is of infinite type.
    \item[$\Pfact_2$] There exist $n\ge 1$ and measured subgroupoids $\calq_1,\dots,\calq_n$ of $\calp$ such that $(\calg,\calp,\calq_j)$ satisfies Property~$\Padm$ for every $j\in\{1,\dots,n\}$, and a countable Borel partition $X^*=\dunion_{i\in I} X_i$ of a conull Borel subset $X^*\subseteq X$ such that for every $i\in I$, one has $\calf_{|X_i}=(\calq_1\cap\dots\cap\calq_n)_{|X_i}$.
    \item[$\Pfact_3$] For every measured subgroupoid $\calq$ of $\calp$ such that $(\calg,\calp,\calq)$ satisfies Property~$\Padm$, and every positive measure Borel subset $U\subseteq X$, there exists a positive measure Borel subset $V\subseteq U$ such that either $\calf_{|V}\subseteq\calq_{|V}$ or $(\calf\cap\calq)_{|V}$ is trivial.
\end{description}
\end{Def}

The idea behind Property~$\Pfact$ is the following. If $P=F_1\times\dots\times F_n$ is a product parabolic subgroup with trivial clique factor, then the factors $F_j$ are characterized as the intersections of co-factors (translated in $\Pfact_2$) which are infinite (as in $\Pfact_1$) and minimal (minimality is the aim of $\Pfact_3$). And if $P$ has a non-trivial  clique factor, then all clique-inclusive co-factors contain it, and therefore the clique factor is characterized as being the intersection of all clique-inclusive co-factors. The following lemma expresses this idea at the groupoid-theoretic level.

\begin{lemma}\label{lemma:factor}
Let $G$ be a graph product of countably infinite groups over a strongly reduced finite simple graph $\Gamma$. Let $\calg$ be a measured groupoid over a standard probability space $X$, equipped with an action-like cocycle $\rho:\calg\to G$. Let $P\subseteq G$ be a product parabolic subgroup with clique factor $C_0$, with $P\neq C_0$, and let $\calp=\rho^{-1}(P)$. Let $\calf\subseteq\calp$ be a measured subgroupoid. 
\begin{enumerate}
    \item If $C_0=\{1\}$, then $(\calf,\rho)$ is of factor type within $(\calp,\rho)$ if and only if $(\calg,\calp,\calf)$ satisfies Property~$\Pfact$.
    \item If $C_0\neq\{1\}$, then $(\calf,\rho)$ is of clique factor type within $(\calp,\rho)$ if and only if $(\calg,\calp,\calf)$ satisfies Property~$\Pfact$.
\end{enumerate}
\end{lemma}

\begin{proof}
We first assume that $C_0=\{1\}$ and that $(\calf,\rho)$ is of factor type within $(\calp,\rho)$. Write $P=F_1\times\dots\times F_n$, where the factors $F_j$ do not further decompose as products. Up to a conull Borel subset and a countable partition, and a permutation of the factors $F_j$, we will assume that $\calf=\rho^{-1}(F_1)$. In particular $\calf$ is of infinite type because $F_1$ is infinite and $\rho$ is action-like, and hence $\Pfact_1$ is verified. Now, for every $j\in\{2,\dots,k\}$, let $Q_j=F_1\times\dots\times \hat{F}_j\times\dots\times F_k$, and let $\calq_j=\rho^{-1}(Q_j)$. Lemma~\ref{lemma:admissible}, \cpageref{lemma:admissible} ensures that $(\calg,\calp,\calq_j)$ satisfies Property~$\Padm$. In addition $\calf=\calq_2\cap\dots\cap\calq_k$, showing that $\Pfact_2$ holds. Finally, let $\calq$ be a measured subgroupoid of $\calp$ such that $(\calg,\calp,\calq)$ satisfies $\Padm$. By Lemma~\ref{lemma:admissible}, there exist a positive measure Borel subset $U\subseteq X$ and a clique-inclusive co-factor $Q\subseteq P$ such that $\calq_{|U}=\rho^{-1}(Q)_{|U}$. Either $Q=F_2\times\dots\times F_n$, in which case $(\calf\cap\calq)_{|U}$ is trivial, or else $F_1\subseteq Q$, in which case $\calf_{|U}\subseteq\calq_{|U}$. This shows that $\Pfact_3$ holds.

We now assume that $C_0\neq\{1\}$ and that $(\calf,\rho)$ is of clique factor type within $(\calp,\rho)$. Write $P=C_0\times F_1\times\dots\times F_k$, where the $F_j$ do not further decompose as direct products. Up to a conull Borel subset and a countable partition, we will assume that $\calf=\rho^{-1}(C_0)$. In particular $\calf$ is of infinite type because $C_0$ is infinite and $\rho$ is action-like, showing that $\Pfact_1$ holds. For every $j\in\{1,\dots,k\}$, let $Q_j=C_0\times F_1\times\dots\times \hat{F}_j\times\dots\times F_k$, and let $\calq_j=\rho^{-1}(Q_j)$. Lemma~\ref{lemma:admissible} ensures that $(\calg,\calp,\calq_j)$ satisfies Property~$\Padm$. In addition $\calf=\calq_1\cap\dots\cap\calq_k$, showing that $\Pfact_2$ holds. Finally, let $\calq$ be a measured subgroupoid of $\calp$ such that $(\calg,\calp,\calq)$ satisfies $\Padm$. By Lemma~\ref{lemma:admissible}, for every positive measure Borel subset $U\subseteq X$ there exist a positive measure Borel subset $V\subseteq U$ and a clique-inclusive co-factor $Q\subseteq P$ such that $\calq_{|V}=\rho^{-1}(Q)_{|V}$. Then $C_0\subseteq Q$, so $\calf_{|V}\subseteq\calq_{|V}$, showing that $\Pfact_3$ holds.

Conversely, let us now assume that $(\calg,\calp,\calf)$ satisfies Property~$\Pfact$. Let $\calq_1,\dots,\calq_n$ be measured subgroupoids of $\calp$ as in $\Pfact_2$. Up to a conull Borel subset and a countable partition, we will assume that $\calf=\calq_1\cap\dots\cap\calq_n$, and, using Lemma~\ref{lemma:admissible}, that $\calq_j=\rho^{-1}(Q_j)$ for some clique-inclusive co-factor $Q_j$. Thus $\calf=\rho^{-1}(F)$, where $F=Q_1\cap\dots\cap Q_n$ is a subproduct of $P$. Then $F$ is non-trivial as $\calf$ is of infinite type by $\Pfact_1$ while $\rho$ has trivial kernel, and $F$ contains $C_0$. To conclude the proof, there only remains to show that
\begin{itemize}
\item if $C_0=\{1\}$, then $F$ does not split as a product of at least two factors;
\item if $C_0\neq\{1\}$, then $F$ does not split as a product of $C_0$ and at least another factor.
\end{itemize}
 But in both cases, if it did, then we could find a clique-inclusive co-factor $Q$ such that $F\cap Q$ is an infinite proper parabolic subgroup of $F$ (in particular $F\nsubseteq Q$). We now reach a contradiction to $\Pfact_3$ by letting $\calq=\rho^{-1}(Q)$: indeed for every positive measure Borel subset $U\subseteq X$, Lemma~\ref{lemma:inclusion-parabolics} shows that $\calf_{|U}\nsubseteq\calq_{|U}$, while the fact that $F\cap Q$ is infinite and $\rho$ is action-like ensures that $(\calf\cap\calq)_{|U}$ is non-trivial. This completes our proof.
\end{proof}

\subsection{Proof of the Vertex Recognition Property}\label{sec:vertex-groups}

We finally introduce the property that will enable us to characterize subgroupoids of vertex type\footnote{Recall that groupoids of vertex type are defined in \cref{subsec:VRP}, \cpageref{Def:VertexType}} with no reference to an action-like cocycle, under the condition that the defining graph of $G$ is transvection-free and strongly reduced. We first state the property, and then comment it below --~it is inspired from  Lemma~\ref{lemma:graph}.

\begin{Def}[Property~$\Pvert$]\index{Property!$\Pvert$}\label{de:pvert}
Let $\calg$ be a measured groupoid over a standard probability space $X$, and let $\calv$ be a measured subgroupoid of $\calg$. We say that the pair $(\calg,\calv)$ satisfies \emph{Property $\Pvert$} if $\calv$ is of infinite type and the following four properties hold.
\begin{description}
    \item[$\Pvert_1$] There exists a countable Borel partition $X^*=\dunion_{i\in I}X_i$ of a conull Borel subset $X^*\subseteq X$ such that for every $i\in I$, there exist measured subgroupoids \[\calg_{|X_i}=\calf_0\supseteq\calp_1\supseteq\calf_1\supseteq\calp_2\supseteq\calf_2\supseteq\dots\supseteq\calp_n\supseteq\calf_n=\calv_{|X_i}\] of $\calg_{|X_i}$ (with $n\ge 1$) such that
    \begin{enumerate}
        \item[(a)] for every $j\in\{1,\dots,n\}$, the pair $(\calf_{j-1},\calp_j)$ satisfies Property~$\Pprod$;
        \item[(b)] for every $j\in\{1,\dots,n\}$, the triple $(\calg,\calp_j,\calf_j)$ satisfies Property~$\Pfact$.
    \end{enumerate}
    \item[$\Pvert_2$] For every measured subgroupoid $\calw$ of $\calg$ of infinite type which is stably contained in $\calv$, and every measured subgroupoid $\calh\subseteq\calg$, if $\calw$ is stably normalized by $\calh$, then $\calv$ is stably normalized by $\calh$.
    \item[$\Pvert_3$] There exists a measured subgroupoid $\caln$ of $\calg$ that stably normalizes $\calv$, such that for every measured subgroupoid $\calv'$ of $\calg$ such that $(\calg,\calv')$ satisfies Properties~$\Pvert_1$ and $\Pvert_2$, every measured subgroupoid $\caln'$ of $\calg$ that stably normalizes $\calv'$, and every positive measure Borel subset $U\subseteq X$, there exists a positive measure Borel subset $V\subseteq U$ such that either $\caln_{|V}=\caln'_{|V}$ or $\caln_{|V}\nsubseteq\caln'_{|V}$.
    \item[$\Pvert_4$] Whenever $\calv'$ is a measured subgroupoid of $\calg$ such that $(\calg,\calv')$ satisfies Properties~$\Pvert_1,\Pvert_2$ and~$\Pvert_3$, if $\calv$ is stably contained in $\calv'$, then they are stably equal.
    \end{description}
\end{Def}

A few comments regarding the meaning of this property are in order. Property~$\Pvert_1$ is a groupoid translation of the contents of Lemma~\ref{lemma:graph}, \cpageref{lemma:graph}, saying informally that vertex groups can be obtained by successively passing to maximal product parabolic subgroups and their factors. However, if we only considered Property~$\Pvert_1$ and end this process too early, we could end up with a subgroupoid “of parabolic type”\index{Parabolic!Parabolic type (groupoid framework)}, i.e.\ of the form $\rho^{-1}(B)$ for some parabolic subgroup $B\subseteq G$ (possibly after taking a conull subset and a countable Borel partition), where $B$ could be larger than a vertex group. Also, because of the slight asymmetry in our lemma that deals with subgroupoids of maximal product type (Lemma~\ref{lemma:maximal-product}, \cpageref{lemma:maximal-product}), we could possibly end up with a subgroupoid contained in one ``of clique type'' (possibly, a subgroupoid strictly contained in one of vertex type), with no precise control. The role of Properties~$\Pvert_2$,~$\Pvert_3$ and~$\Pvert_4$ is to exclude these extra possibilities, as we now explain. 

Under our assumption that the defining graph $\Gamma$ is strongly reduced, a parabolic subgroup $B$ that is not conjugate to a vertex group, always contains a proper (parabolic) subgroup $A\subseteq B$ with $N_G(A)\nsubseteq N_G(B)$. On the other hand a vertex group does not. Property~$\Pvert_2$ is the groupoid translation of this fact, which will allow us to exclude subgroupoids of parabolic type but not of vertex type.

Assuming that $\Gamma$ is transvection-free, if $B$ is a proper subgroup of a clique subgroup~$C$, then the normalizer of $B$ is strictly contained in the normalizer of any vertex subgroup of~$C$. On the other hand, the normalizer of a vertex group is not strictly contained in the normalizer of any other vertex group. A groupoid version of this distinction is exploited in Property~$\Pvert_3$, to exclude subgroupoids contained in one ``of clique type'' but not contained in one of vertex type.

Finally, a maximality assumption is exploited in Property~$\Pvert_4$ to exclude proper subgroupoids of subgroupoids of vertex type.

Having explained the general idea behind it, we can now turn to the proof of our main lemma, from which the Vertex Recognition Property (see \cref{de:VRP}) follows.

\begin{lemma}\label{lemma:characterization-vertex-groups}
    Let $G$ be a graph product of countably infinite groups over a strongly reduced transvection-free finite simple graph $\Gamma$ not reduced to a vertex.
    Let $\calg$ be a measured groupoid over a standard probability space $X$, equipped with an action-like cocycle $\rho:\calg\to G$. Let $\calv\subseteq\calg$ be a measured subgroupoid. 
    
    Then $(\calv,\rho)$ is of vertex type if and only if $(\calg,\calv)$ satisfies Property~$\Pvert$.
\end{lemma}

\begin{proof}
The proof has four steps.
\bigskip

\textbf{Step 1.} Let us show that if $(\calv,\rho)$ is of vertex type, then $\calv$ is of infinite type and $(\calg,\calv)$ satisfies $\Pvert_1$ and $\Pvert_2$.
\smallskip

Since $\rho$ is action-like and $(\calv,\rho)$ is of vertex type, and since vertex groups are infinite, $\calv$ is of infinite type. Up to passing to a conull Borel subset, taking a countable partition, and conjugating the cocycle, we can assume that $\calv=\rho^{-1}(G_v)$ for some vertex $v\in V\Gamma$. The vertex $v$ is untransvectable because $\Gamma$ is transvection-free.

For $\Pvert_1$, by Lemma~\ref{lemma:graph}, \cpageref{lemma:graph}, there exists a chain of parabolic subgroups \[G=F_0\supseteq P_1\supseteq F_1\supseteq P_2\supseteq F_2\supseteq\dots\supseteq P_n\supseteq F_n=G_v\] such that 
    \begin{enumerate}
    \item for every $j\in\{1,\dots,n\}$, $P_j$ is a maximal product parabolic subgroup of $F_{j-1}$ which is not of isolated clique type, and
    \item for every $j\in\{1,\dots,n-1\}$, the clique factor of $P_j$ is trivial, and $F_j$ is a factor of $P_j$, and
    \item $F_n$ is the clique factor of $P_{n}$.
    \end{enumerate}
For every $j\in\{1,\dots,n\}$, let $\calp_j=\rho^{-1}(P_j)$, and let $\calf_j=\rho^{-1}(F_j)$. These come equipped with action-like cocycles to $P_j,F_j$, respectively. Notice that for every $j\in\{1,\dots,n\}$, the pair $(\calp_j,\rho)$ is nowhere of isolated clique type. By the first assertion of Lemma~\ref{lemma:maximal-product}, \cpageref{lemma:maximal-product}, for every $j\in\{1,\dots,n\}$, the pair $(\calf_{j-1},\calp_j)$ satisfies Property~$\Pprod$. By Lemma~\ref{lemma:factor}, using that $\Gamma$ is strongly reduced, for every $j\in\{1,\dots,n\}$, the triple $(\calg,\calp_j,\calf_j)$ satisfies Property~$\Pfact$.

For $\Pvert_2$, let $\calw$ be a measured subgroupoid of infinite type which is stably contained in $\calv$, and let $\calh\subseteq\calg$ be a measured subgroupoid that stably normalizes $\calw$. Up to a countable partition of the base space, we will assume that $\calw\subseteq\calv$ and that $\calh$ normalizes $\calw$. Then $(\calw,\rho)$ is tightly $G_v$-supported, so Lemma~\ref{lemma:support-normal}, \cpageref{lemma:support-normal}, ensures (up to restricting to a conull Borel subset) that $\calh\subseteq\rho^{-1}(G_v\times G_v^\perp)$. Therefore $\calh$ normalizes $\calv=\rho^{-1}(G_v)$, as desired.
\bigskip

\textbf{Step 2.} Let us show that if $(\calg,\calv)$ satisfies Properties~$\Pvert_1$ and~$\Pvert_2$, then there exist a countable Borel partition $X^*=\dunion_{i\in I}X_i$ of a conull Borel subset $X^*\subseteq X$, and for every $i\in I$, a parabolic subgroup $C_i$ that is conjugate to $G_{\Upsilon_i}$, for some non-empty complete subgraph $\Upsilon_i\subseteq\Gamma$ (possibly reduced to one vertex), such that $\calv_{|X_i}\subseteq\rho^{-1}(C_i)_{|X_i}$.
\smallskip

Up to a countable partition, we can assume that there exist subgroupoids \[\calg=\calf_0\supseteq\calp_1\supseteq\calf_1\supseteq\calp_2\supseteq\calf_2\supseteq\dots\supseteq\calp_n\supseteq\calf_n=\calv\] as in $\Pvert_1$. By inductively applying Lemmas~\ref{lemma:maximal-product} and~\ref{lemma:factor}, we show that for every $i\in\{1,\dots,n\}$, up to passing to a conull Borel subset and taking a countable partition,  
\begin{itemize}
\item Either $\calp_i=\rho^{-1}(P_i)$ for some maximal product parabolic subgroup $P_i$ whose type is not a clique, and $\calf_i=\rho^{-1}(F_i)$ for some factor $F_i$ of $P_i$;
\item Or $\calp_i\subseteq\rho^{-1}(C)$ for some parabolic subgroup $C$ that is conjugate to $G_{\Upsilon}$ for some non-empty complete subgraph $\Upsilon\subseteq\Gamma$.
\end{itemize}
In particular, either $\calv=\rho^{-1}(B)$ for some proper parabolic subgroup $B$, or else $\calv\subseteq\rho^{-1}(C)$ for some parabolic group $C$ that is conjugate to $G_{\Upsilon}$ for some non-empty complete subgraph $\Upsilon\subseteq\Gamma$. 
To show the wanted assertion, it suffices to exclude the case where $\calv=\rho^{-1}(B)$ for a parabolic subgroup $B$ whose type is an induced subgraph $\Lambda\subseteq\Gamma$ that contains at least $2$ vertices. Up to conjugating the cocycle we will assume that $\calv=\rho^{-1}(G_\Lambda)$ and aim for a contradiction. Since $\Gamma$ is strongly reduced, there exists a vertex $w\in V\Lambda$ such that $G_w\times G_w^{\perp}\nsubseteq G_{\Lambda}\times G_{\Lambda}^{\perp}$, as otherwise the subgraph $\Lambda$ would be collapsible. We will now use Property~$\Pvert_2$ to reach a contradiction. For this, let $\calw=\rho^{-1}(G_w)$. Then $\calw$ is contained in $\calv$, and $\calw$ is of infinite type because $G_w$ is infinite and $\rho$ is action-like. Let $\calh=\rho^{-1}(G_w\times G_w^{\perp})$, which normalizes $\calw$ (Example~\ref{ex:normal}, \cpageref{ex:normal}). So Property~$\Pvert_2$ implies that $\calv$ is stably normalized by $\calh$. As $(\calv,\rho)$ is tightly $G_\Lambda$-supported (Lemma~\ref{lemma:tight-support}, \cpageref{lemma:tight-support}), it follows from Lemma~\ref{lemma:support-normal}, \cpageref{lemma:support-normal} that $\calh\subseteq\rho^{-1}(G_{\Lambda}\times G_{\Lambda}^{\perp})$. But since $G_w\times G_w^{\perp}\nsubseteq G_{\Lambda}\times G_{\Lambda}^{\perp}$, this is a contradiction to Lemma~\ref{lemma:inclusion-parabolics}, \cpageref{lemma:inclusion-parabolics}. This contradiction completes our proof of Step~2.  

\bigskip

\textbf{Step 3.} Let us show that if $(\calv,\rho)$ is of vertex type, then $(\calg,\calv)$ satisfies Property~$\Pvert_3$. 
\smallskip

Again we assume that $\calv=\rho^{-1}(G_v)$ for some $v\in V\Gamma$. Let $\caln=\rho^{-1}(G_v\times G_v^{\perp})$, which normalizes $\calv$ (Example~\ref{ex:normal}, \cpageref{ex:normal}). Let $\calv'$ be a measured subgroupoid of $\calg$ that satisfies $\Pvert_1$ and $\Pvert_2$, let $\caln'$ be a measured subgroupoid of $\calg$ that stably normalizes~$\calv'$ and let $U\subseteq X$ be a positive measure Borel subset. By Lemma~\ref{lemma:support}, \cpageref{lemma:support}, we can find a positive measure Borel subset $V\subseteq U$ such that $(\calv'_{|V},\rho)$ is tightly $C$-supported for some parabolic subgroup $C\subseteq G$. Up to partitioning, we can (and will) also assume that $\caln'_{|V}$ normalizes $\calv'_{|V}$. By the previous step, $C$ is conjugate to $G_{\Upsilon}$ for some non-empty complete subgraph $\Upsilon\subseteq\Gamma$. Since $v$ is untransvectable, it follows from Lemma~\ref{lemma:parabolic-untransvectable}, \cpageref{lemma:parabolic-untransvectable}, that either $C=G_v$, or else $G_v\times G_v^{\perp}\nsubseteq C\times C^{\perp}$. If $C=G_v$, then Lemma~\ref{lemma:support-normal}, \cpageref{lemma:support-normal} ensures that $\caln'_{|V}\subseteq\caln_{|V}$ (up to a null set), and we are done. If $G_v\times G_v^{\perp}\nsubseteq C\times C^{\perp}$, then Lemma~\ref{lemma:inclusion-parabolics}, \cpageref{lemma:inclusion-parabolics} ensures that $\caln_{|V}\nsubseteq\rho^{-1}(C\times C^{\perp})_{|V}$. As $\caln'_{|V}$ normalizes $\calv'_{|V}$, we also have $\caln'_{|V}\subseteq\rho^{-1}(C\times C^{\perp})_{|V}$ (Lemma~\ref{lemma:support-normal}, \cpageref{lemma:support-normal}). So we deduce that $\caln_{|V}\nsubseteq\caln'_{|V}$, and again Property~$\Pvert_3$ holds.
\bigskip

\textbf{Step 4.} Let us finally show that if $(\calg,\calv)$ satisfies Properties~$\Pvert_1$,~$\Pvert_2$ and~$\Pvert_3$, then there exist a countable Borel partition $X^*=\dunion_{i\in I}X_i$ of a conull Borel subset $X^*\subseteq X$, and for every $i\in I$, a parabolic subgroup $A_i$ that is conjugate to a vertex subgroup, such that $\calv_{|X_i}\subseteq \rho^{-1}(A_i)_{|X_i}$.
\smallskip

Otherwise, in view of Step~2 (and Lemma~\ref{lemma:support}, \cpageref{lemma:support}), there would exist a positive measure Borel subset $U\subseteq X$ such that $(\calv_{|U},\rho)$ is tightly $C$-supported for some parabolic subgroup $C$ conjugate to $G_{\Upsilon}$, where $\Upsilon\subseteq\Gamma$ is a complete subgraph of size at least $2$. Let $\caln$ be a measured subgroupoid given by $\Pvert_3$. Since $\caln$ stably normalizes $\calv$, Lemma~\ref{lemma:support-normal}, \cpageref{lemma:support-normal}, ensures that, up to replacing $U$ by a positive measure subset, we have $\caln_{|U}\subseteq\rho^{-1}(C\times C^{\perp})_{|U}$. Let $A\subseteq C$ be a parabolic subgroup that is equal to a conjugate of a vertex group. Let $\calv'$ be a measured subgroupoid of $\calg$ such that $\calv'_{|U}=\rho^{-1}(A)_{|U}$, and $\calv'_{|X\setminus U}=\calv_{|X\setminus U}$. Using Step~1, we see that $(\calg,\calv')$ satisfies Properties~$\Pvert_1$ and~$\Pvert_2$. Let $\caln'$ be such that $\caln'_{|U}=\rho^{-1}(A\times A^{\perp})_{|U}$, and $\caln'_{|X\setminus U}=\caln_{|X\setminus U}$. Then $\caln'$ stably normalizes $\calv'$. As every vertex of $\Gamma$ is untransvectable, the inclusion $C\times C^{\perp}\subseteq A\times A^{\perp}$ is strict. Lemma~\ref{lemma:inclusion-parabolics}, \cpageref{lemma:inclusion-parabolics} thus implies that $\caln_{|V}\subsetneq\caln'_{|V}$ for every positive measure Borel subset $V\subseteq U$, contradicting $\Pvert_3$.
\bigskip

\textbf{Conclusion.} If $(\calv,\rho)$ is of vertex type, then $\Pvert_1$, $\Pvert_2$ and $\Pvert_3$ are verified by Steps 1 and 3. Finally, take $\calv'$ verifying the hypothesis of $\Pvert_4$. By Step~4, up to taking a conull Borel subset and a countable partition, there exists a subgroup $A\subseteq G$ that is conjugate to a vertex group, such that $\calv'\subseteq\rho^{-1}(A)$. Up to refining the partition, we also assume that $\calv=\rho^{-1}(B)$ for some subgroup $B\subseteq G$ that is conjugate to a vertex group. Since $\calv$ is stably contained in $\calv'$, we deduce that $\rho^{-1}(B)$ is stably contained in $\rho^{-1}(A)$. By Lemma~\ref{lemma:inclusion-parabolics}, \cpageref{lemma:inclusion-parabolics}, this implies that $B\subseteq A$, and in fact $A=B$ since they are both conjugate to vertex groups. Hence $(\calg,\calv)$ satisfies $\Pvert$. 

Conversely, if $(\calg,\calv)$ verifies $\Pvert$ then using the maximality assumption provided by Property~$\Pvert_4$ and the above Step 4, we get that $(\calv,\rho)$ is of vertex type.
\end{proof}

We are now in position to complete our proof of Proposition~\ref{prop:vertex-recognition-strongly-reduced}.

\begin{proof}
    Let $G,H$ be two graph products of countably infinite groups over transvection-free strongly reduced finite simple graphs. Let $\calg$ be a measured groupoid over a standard probability space, equipped with action-like cocycles $\rho_G:\calg\to G$ and $\rho_H:\calg\to H$. Let $\calv\subseteq\calg$ be a measured subgroupoid, and assume that the pair $(\calv,\rho_G)$ is of vertex type; we aim to prove that $(\calv,\rho_H)$ is of vertex type.

    By Lemma~\ref{lemma:characterization-vertex-groups}, applied to the cocycle $\rho_G$, the pair $(\calg,\calv)$ satisfies Property~$\Pvert$. And by Lemma~\ref{lemma:characterization-vertex-groups}, applied to the cocycle $\rho_H$, the pair $(\calv,\rho_H)$ is of vertex type, as desired.  
\end{proof}

\section{The reducible case and proof of Proposition~\ref{prop:vertex-recognition}}\label{sec:reducible}
In this section, we complete our proof of Proposition~\ref{prop:vertex-recognition}, by treating the case of reducible graph products.

\subsection{Goal of the section}

The main result of this section is the following proposition, which extends Proposition~\ref{prop:vertex-recognition-strongly-reduced} to the case of direct products.

\begin{Prop}\label{prop:vertex-recognition-join}
The class of graph products of countably infinite groups over finite simple graphs $\Gamma$ that splits as a join $\Gamma=\Gamma_1\circ\dots\circ\Gamma_k$ (with $k\ge 1$) of transvection-free strongly reduced graphs $\Gamma_j$ on at least $2$ vertices, satisfies the Vertex Recognition Property.
\end{Prop}

In view of Corollary~\ref{cor:strongly-reduced}, this covers all graph products over transvection-free graphs with no partial conjugation, and therefore Proposition~\ref{prop:vertex-recognition} follows from Proposition~\ref{prop:vertex-recognition-join}. In the previous section, we have treated the case where $k=1$, we now need to tackle the reducible case.

\subsection{Proof of the Vertex Recognition Property}

Given a graph product $G$ that splits as a direct product with trivial clique factor, the following property will enable us to recognize the co-factors of the product, at a groupoid-theoretic level.

\begin{Def}[Property~$\Pcofact$]\index{Property!$\Pcofact$}\label{de:pcofact}
Let $\calg$ be a measured groupoid over a standard probability space, and let $\calq$ be a measured subgroupoid of $\calg$. We say that $(\calg,\calq)$ satisfies \emph{Property~$\Pcofact$} if the following two properties hold. 
 \begin{description}
    \item[$\Pcofact_1$] The groupoid $\calq$ is of infinite type and stably normal in $\calg$, and it stably normalizes an amenable measured subgroupoid $\cala$ of $\calg$ of infinite type.
     \item[$\Pcofact_2$] Whenever $\calq'$ is a measured subgroupoid of $\calg$ such that $(\calg,\calq')$ satisfies Property~$\Pcofact_1$, and such that $\calq$ is stably contained in $\calq'$, then $\calq$ and $\calq'$ are stably equal.
     \end{description}
\end{Def}

The idea behind Property~$\Pcofact$ is the following. Assume that a graph product $G$ splits non-trivially as a product $G=F_1\times\dots\times F_k$, with trivial clique factor. Then each $Q_j=F_1\times\dots\times\hat{F}_j\times\dots\times F_k$ is normal in $G$ and normalizes an infinite amenable subgroup $A\subseteq F_j$. Conversely, any maximal subgroup of $G$ with these properties is a co-factor. The following lemma is the groupoid-theoretic version of this fact (since all clique factors are trivial in this section, for simplicity, we will talk about ``subgroupoids of co-factor type'' as a shortcut for ``subgroupoids of clique-inclusive co-factor type'').

\begin{Lmm}\label{lemma:cofactor-2}
Let $G$ be a graph product over a finite simple graph $\Gamma$ with countably infinite vertex groups and trivial clique factor. Let $\calg$ be a measured groupoid over a standard probability space $X$, equipped with an action-like cocycle $\rho:\calg\to G$. Let $\calq\subseteq\calg$ be a measured subgroupoid. 

 Then $(\calg,\calq)$ satisfies Property~$\Pcofact$ if and only if $G$ is reducible and $(\calq,\rho)$ is of co-factor type within $(\calg,\rho)$.
\end{Lmm}

\begin{proof}
Let $G=F_1\times\dots\times F_k$ be the decomposition of $G$ into its irreducible components (possibly with $k=1$). 

\textbf{Step 1} We first assume that $G$ is reducible (i.e.\ $k\ge 2$) and that $(\calq,\rho)$ is of co-factor type within $(\calg,\rho)$, and prove that $(\calg,\calq)$ satisfies $\Pcofact_1$.
\smallskip

Up to a conull Borel subset and a countable partition, we will assume that $\calq=\rho^{-1}(Q)$, where $Q=F_1\times\dots\times\hat F_j\times\dots\times F_k$ for some $j\in\{1,\dots,k\}$. Then $\calq$ is of infinite type (because $Q$ is infinite and $\rho$ is action-like) and normalized by $\calg$ (because $Q$ is normal in $G$). Let $A\subseteq F_j$ be an infinite amenable subgroup: this always exists, because the type of $F_j$ (i.e.\ the underlying subgraph of $\Gamma$) contains at least $2$ vertices. The conclusion follows by letting $\cala=\rho^{-1}(A)$.
\medskip

 \textbf{Step 2} Let us prove that if $(\calg,\calq)$ satisfies $\Pcofact_1$ then $G$ is reducible (i.e.\ $k\ge 2$) and $\calq$ is stably contained in a subgroupoid $\calq'$ of $\calg$ such that $(\calq',\rho)$ is of co-factor type within $(\calg,\rho)$.
\smallskip

Let $\cala$ be given by $\Pcofact_1$. Without loss of generality, we will assume that $\cala$ is normalized by $\calq$. For every $j\in\{1,\dots,k\}$, let $\rho_j:\calg\to F_j$ be the cocycle obtained by post-composing $\rho$ by the projection $G\to F_j$. Up to a countable Borel partition of $X$, we will assume that $(\cala,\rho)$ is tightly $A$-supported for some parabolic subgroup $A\subseteq G$ (and $A\neq\{1\}$ because $\cala$ is of infinite type). 

\noindent $\bullet$ Let $j\in\{1,\dots,k\}$ be such that $A\cap F_j\neq\{1\}$ (such a $j$ exists by \cref{lemma:product-parabolic-subgroup}, \cpageref{lemma:product-parabolic-subgroup}). In particular $(\rho_j)_{|\cala}$ is nowhere trivial (by the second point of the definition of tight support, \cpageref{tight-support}).\\
\noindent $\bullet$ Let us prove that $(\rho_j)_{|\calq}$ is stably trivial.\\ 
So assume towards a contradiction that $(\rho_j)_{|\calq}$ is not stably trivial.
\begin{itemize}
    \item We first show that there exists a Borel subset $V$ of positive measure such that  $\rho_j$ is nowhere trivial on~$\calq_{|V}$.\\
   Let $U\subseteq X$ be a Borel subset of maximal measure such that $\rho_j$ is stably trivial on $\calq_{|U}$. Our assumption ensures that $V=X\setminus U$ has positive measure. And $\rho_j$ is nowhere trivial on~$\calq_{|V}$. 
    \item We now show that $(\rho^{-1}(F_j)_{|V},\rho_j)$ is $F_j$-elementary (in the sense of \cref{de:elementary}).\\
    Lemma~\ref{lemma:alternative}, \cpageref{lemma:alternative} and Remark~\ref{rk:amenable}, \cpageref{rk:amenable} ensure that $(\cala,\rho_j)$ is $F_j$-elementary. Applying Lemma~\ref{lemma:elementary-normal}, \cpageref{lemma:elementary-normal} twice then shows that $(\calq,\rho_j)$ and $(\calg_{|V},\rho_j)$ are $F_j$-elementary. In particular $(\rho^{-1}(F_j)_{|V},\rho_j)$ is $F_j$-elementary. 
    \item Since $\rho_j:\rho^{-1}(F_j)_{|V} \rightarrow F_j$ is action-like and the type of $F_j$ is not a clique, it contradicts  Lemma~\ref{lemma:nonelementary-example}, \cpageref{lemma:nonelementary-example}. 
    
    This contradiction shows that $(\rho_j)_{|\calq}$ is stably trivial.
\end{itemize}
\noindent $\bullet$ Let us show that $G$ is reducible.\\
Assume that $G$ is irreducible, namely that $k=1$. Then $\rho_1=\rho$ and since $\rho$ has trivial kernel, the previous point implies that $\calq$ is stably trivial. This contradicts the fact that $\calq$ is of infinite type (by $\Pcofact_1$). 

\noindent $\bullet$ Taking $Q:=F_1\times \cdots \times \hat{F}_j\times \cdots F_k$ and $\calq^\prime:=\rho^{-1}(Q)$ gives the wanted conclusion for Step 2.
\medskip

\textbf{Conclusion} If $G$ is reducible and $(\calq,\rho)$ is of co-factor type within $(\calg,\rho)$, then it verifies $\Pcofact_1$ by Step 1. To check $\Pcofact_2$ take $\calq^\prime$ such that $(\calg,\calq')$ satisfies $\Pcofact_1$ with $\calq$ stably contained in $\calq^\prime$. 
By Step $2$, there exists $\calq''$ such that $\calq^\prime$ is also stably contained in $\calq''$ and $(\calq'',\rho)$ is of co-factor type within $(\calg,\rho)$. Since an inclusion of co-factors is always an equality, we deduce that $\calq$ and $\calq''$ are stably equal, in particular $\calq$ and $\calq'$ are stably equal.

Now assume that $(\calg,\calq)$ satisfies $\Pcofact$. By Step 2, $G$ is reducible, and there exists $\calq'$ such that $\calq$ is stably contained in $\calq'$ and $(\calq',\rho)$ is of co-factor type within $(\calg,\rho)$. By Step 1, $(\calg,\calq^\prime)$ verifies $\Pcofact_1$ and thus $\Pcofact_2$ for $(\calg,\calq)$ gives that $\calq$ and $\calq^\prime$ are stably equal. In particular $(\calq,\rho)$ is of co-factor type within $(\calg,\rho)$.
\end{proof}

We now consider the following property which is a very slight variation over Property~$\Pfact$ from the previous section (Definition~\ref{de:pfact}), with $\Pcofact$ instead of $\Padm$.

\begin{Def}[Property~$\Pfactp$]\index{Property!$\Pfactp$}\label{de:pfact2}
Let $\calg$ be a measured groupoid over a standard probability space $X$. Let $\calf\subseteq\calg$ be a measured subgroupoid. We say that the pair $(\calg,\calf)$ satisfies \emph{Property~$\Pfactp$} if the following three properties hold.
\begin{description}
\item[$\Pfactp_1$] The subgroupoid $\calf$ is of infinite type.
\item[$\Pfactp_2$] There exist measured subgroupoids $\calq_1,\dots,\calq_n$ of $\calg$, with $n\ge 1$, and such that $(\calg,\calq_j)$ satisfies Property~$\Pcofact$ for every $j\in\{1,\dots,n\}$, and there exists a countable Borel partition $X^*=\dunion_{i\in I} X_i$ of a conull Borel subset $X^*\subseteq X$ such that for every $i\in I$, one has $\calf_{|X_i}=(\calq_1\cap\dots\cap\calq_n)_{|X_i}$.
\item[$\Pfactp_3$] For every subgroupoid $\calq$ of $\calg$ such that $(\calg,\calq)$ satisfies Property~$\Pcofact$, and every positive measure Borel subset $U\subseteq X$, there exists a positive measure Borel subset $V\subseteq U$ such that either $\calf_{|V}\subseteq\calq_{|V}$ or $(\calf\cap\calq)_{|V}$ is trivial.
\end{description}
\end{Def}

Replacing the use of \cref{lemma:admissible}, \cpageref{lemma:admissible} by \cref{lemma:cofactor-2} in the proof of the first point of Lemma~\ref{lemma:factor}, \cpageref{lemma:factor} we obtain the following statement.

\begin{lemma}\label{lemma:factor-2}
Let $G$ be a reducible graph product of countably infinite groups over a finite simple graph $\Gamma$ with trivial clique factor. Let $\calg$ be a measured groupoid over a standard probability space $X$, equipped with an action-like cocycle $\rho:\calg\to G$. Let $\calf\subseteq\calg$ be a measured subgroupoid. 

Then $(\calf,\rho)$ is of factor type within $(\calg,\rho)$ if and only if $(\calg,\calf)$ satisfies Property~$\Pfactp$. \qed
\end{lemma}

For the following statement, we refer to Definition~\ref{de:pvert}, \cpageref{de:pvert} from the previous section for the definition of Property~$\Pvert$.

\begin{lemma}\label{lemma:full-characterization}
Let $G$ be a reducible graph product of countably infinite groups over a finite simple graph $\Gamma$ with no transvection and no partial conjugation. Let $\calg$ be a measured groupoid over a standard probability space $X$, and let $\rho:\calg\to G$ be an action-like cocycle. Let $\calv\subseteq\calg$ be a measured subgroupoid.

Then $(\calv,\rho)$ is of vertex type if and only if there exist a measured subgroupoid $\calf\subseteq\calg$ such that $(\calg,\calf)$ satisfies Property~$\Pfactp$, and a countable Borel partition $X^*=\dunion_{i\in I}X_i$ of a conull Borel subset $X^*\subseteq X$ such that for every $i\in I$, one has $\calv_{|X_i}\subseteq\calf_{|X_i}$, and the pair $(\calf_{|X_i},\calv_{|X_i})$ satisfies Property~$\Pvert$. 
\end{lemma}

\begin{proof}
Let $\Gamma=\Gamma_1\circ\dots\circ\Gamma_k$ be the decomposition of $\Gamma$ as a join of irreducible subgraphs (well-defined up to permutation of the factors), and let $G=F_1\times\dots\times F_k$ be the corresponding decomposition of $G$ as a direct product (since $G$ is reducible, $k\geq 2$). Notice that as $\Gamma$ is transvection-free, no $\Gamma_i$ is reduced to a point, in other words $G$ has trivial clique factor. 

We first assume that $(\calv,\rho)$ is of vertex type. Up to a countable partition of the base space $X$, and up to conjugating the cocycle, we will assume that $\calv=\rho^{-1}(G_v)$ for some $v\in V\Gamma$. Let $j\in\{1,\dots,k\}$ be such that $v\in V\Gamma_j$. Let $\calf=\rho^{-1}(F_j)$. By Lemma~\ref{lemma:factor-2}, \cpageref{lemma:factor-2} the pair $(\calg,\calf)$ satisfies Property~$\Pfactp$. In addition, $\Gamma_j$ has no transvection and no partial conjugation (Remark~\ref{rk:transvection-free-factors}, \cpageref{rk:transvection-free-factors}), so it is strongly reduced by Lemma~\ref{lemma:strongly-reduced}, \cpageref{lemma:strongly-reduced}. It thus follows from Lemma~\ref{lemma:characterization-vertex-groups}, \cpageref{lemma:characterization-vertex-groups} (applied to the groupoid $\calf$ and to the cocycle $\rho$, viewed as an action-like cocycle $\calf\to F_j$), that $(\calf,\calv)$ satisfies Property~$\Pvert$.

Conversely, assume that there exist a subgroupoid $\calf$ as in the lemma -- up to reasoning on each subset of the partition, we will assume that $(\calf,\calv)$ satisfies Property~$\Pvert$. Since $(\calg,\calf)$ satisfies Property~$\Pfactp$,  Lemma~\ref{lemma:factor-2} ensures that, up to a countable Borel partition of $X$ and passing to a conull Borel subset, we have $\calf=\rho^{-1}(F_j)$ for some $j\in\{1,\dots,k\}$. Since $\Gamma_j$ is transvection-free and strongly reduced, and $(\calf,\calv)$ satisfies Property~$\Pvert$, it follows from Lemma~\ref{lemma:characterization-vertex-groups}, \cpageref{lemma:characterization-vertex-groups} that $(\calv,\rho)$ is of vertex type, as desired.
\end{proof}

We are now in position to complete the proof of the main proposition of the section.

\begin{proof}[Proof of Proposition~\ref{prop:vertex-recognition-join}]
    Let $G$ and $H$ be two graph products of countably infinite groups over finite simple graphs that split (possibly trivially) as joins of strongly reduced transvection-free finite simple graphs on at least $2$ vertices.
    
    Let $\calg$ be a measured groupoid equipped with two action-like cocycles $\rho_G:\calG\rightarrow G$ and $\rho_H:\calG\rightarrow H$.
    
    Let us first show that $G$ is reducible if and only if $H$ is.\\
    If $G$ is reducible there exists a subgroupoid $\calq$ in $\calg$ such that $(\calq,\rho_G)$ is of co-factor type within $(\calg,\rho_G)$. By \cref{lemma:cofactor-2}, \cpageref{lemma:cofactor-2}, the pair $(\calg,\calq)$ verifies $\Pcofact$ and therefore $H$ is also reducible. 
    
    If $G$ is reducible, the conclusion of Proposition~\ref{prop:vertex-recognition-join}, \cpageref{prop:vertex-recognition-join} is now an immediate consequence of Lemma~\ref{lemma:full-characterization}, \cpageref{lemma:full-characterization}, as we have characterized subgroupoids of $\calg$ of vertex type with respect to some action-like cocycle $\rho$, using a groupoid-theoretic property that does not refer to the cocycle $\rho$.
    
    And if $G$ is irreducible, the conclusion of \cref{prop:vertex-recognition-join}, \cpageref{prop:vertex-recognition-join} follows from \cref{lemma:characterization-vertex-groups}, \cpageref{lemma:characterization-vertex-groups}.
\end{proof}

\begin{proof}[{Proof of Proposition~\ref{prop:vertex-recognition}}]\label{ProofPropVRPGen}
    By Corollary~\ref{cor:strongly-reduced}, if a finite simple graph $\Gamma$ is transvection-free and has no partial conjugation, then $\Gamma$ splits as $\Gamma=\Gamma_1\circ\dots\circ\Gamma_k$ (possibly with $k=1$), in such a way that every $\Gamma_j$ is transvection-free, strongly reduced, and has at least two vertices. Proposition~\ref{prop:vertex-recognition} thus follows from Proposition~\ref{prop:vertex-recognition-join}.
\end{proof}

This completes the proof of the main theorem of \cref{Part:MEClassification} (\cref{theo:classificationversionintro}, \cpageref{theo:classificationversionintro}).
\section{Amenable vertex groups and untransvectable vertices}\label{Sec:AmenalbeUntransvectable}
\subsection{Main result}

The contents of this section is not needed for the main theorems of Part~\ref{Part:MEClassification}, but will be used in Section~\ref{Sec:InverseProblem}. In this section specifically, we will consider graph products where all vertex groups are amenable. In this case, more can be said without assuming that the graph has no transvection and no partial conjugation. Namely, given a subgroupoid $\calg$ over a standard probability space $X$ equipped with an action-like cocycle $\rho:\calg\to G$, we can still recognize subgroupoids that correspond to parabolic subgroups of untransvectable vertex type, in the following sense.

Let $\Gamma$ be a finite simple graph. Recall that a vertex $v\in V\Gamma$ is \emph{untransvectable}\index{Untransvectable!Untransvectable vertex}\label{Def:UntransvectableVertex} if there does not exist any vertex $w\in V\Gamma$ distinct from $v$ and such that $\lk(v)\subseteq\st(w)$. A parabolic subgroup $P\subseteq G$ is \emph{of untransvectable vertex type}\index{Untransvectable!Untransvectable vertex type (parabolic subgroup)} if it is conjugate to $G_v$ for some untransvectable vertex $v$.

Given $\calg$ and $\rho$ as above, and a measured subgroupoid $\calh\subseteq\calg$, we say that $(\calh,\rho)$ is \emph{of untransvectable vertex type}\index{Untransvectable!Untransvectable vertex type (groupoid framework)}  if there exist a conull Borel subset $X^*\subseteq X$, a partition $X^*=\dunion_{i\in I}X_i$ into at most countably many Borel subsets, and for every $i\in I$, a parabolic subgroup $P_i$ of untransvectable vertex type such that for every $i\in I$, one has $\calh_{|X_i}=\rho^{-1}(P_i)_{|X_i}$.

We say that a finite simple graph $\Gamma$ is \emph{clique-reduced}\index{Clique-reduced graph} if there do not exist distinct vertices $v,w\in V\Gamma$ with $\st(v)=\st(w)$. Equivalently $\Gamma$ is clique-reduced if and only if there does not exist any collapsible clique (in the sense of Definition~\ref{de:strongly-reduced}, \cpageref{de:strongly-reduced}) on at least $2$ vertices in $\Gamma$.

The main result of this section is the following proposition.

\begin{Prop}\label{prop:untransvectable}
    Let $G,H$ be two graph products over clique-reduced finite simple graphs with countably infinite amenable vertex groups. Let $\calg$ be a measured groupoid over a standard probability space, equipped with two action-like cocycles $\rho_G:\calg\to G$ and $\rho_H:\calg\to H$. Let $\calv\subseteq\calg$ be a measured subgroupoid.

    Then, the pair $(\calv,\rho_G)$ is of untransvectable vertex type if and only if the pair $(\calv,\rho_H)$ is of untransvectable vertex type.
\end{Prop}

\begin{Rq}
 The assumption that the two graphs are clique-reduced is necessary. Otherwise, here is a counterexample. Let $G$ be the graph product over the graph $\bar\Gamma$ from Figure~\ref{fig:Collapsible}, \cpageref{fig:Collapsible}, with vertex groups $\{G_v\}_v$ isomorphic to $\mathbb{Z}^2$. Then $G$ is also isomorphic to the graph product $H$ over the graph $\Gamma$ (from the same figure) with vertex groups isomorphic to $\mathbb{Z}$. We fix an isomorphism $\theta:H\to G$ sending vertex groups inside vertex groups. Let $\calg$ be a measured groupoid with an action-like cocycle $\rho_G:\calg\to G$, and let $\rho_H=\rho_G$ (through the chosen isomorphism $H\to G$). If $v\in V\bar\Gamma$, then $(\rho_G^{-1}(G_v),\rho_G)$ is of untransvectable vertex type, while $(\rho_G^{-1}(G_v),\rho_H)$ is not.    
\end{Rq}

Our motivation behind proving Proposition~\ref{prop:untransvectable} is that it is a useful in the measure equivalence classification of right-angled Artin groups. In fact, as recorded in Corollary~\ref{cor:untransvectable-extension-graph} below, it provides finer information than \cite{HH21}, which only dealt with right-angled Artin groups with finite outer automorphism groups. Among concrete applications, let us mention the following.
\begin{itemize}
    \item Corollary~\ref{cor:untransvectable-extension-graph} below will be used in Section~\ref{Sec:ExemplesMEVRP} to provide new concrete examples of right-angled Artin groups that can be distinguished from the viewpoint of measure equivalence.
    \item In Part~\ref{Part:QuantitativeResults} of the present work, Proposition~\ref{prop:untransvectable} will enable us to obtain quantitative estimates and solve the inverse problem in quantitative orbit equivalence for a large family of right-angled Artin groups.
    \item In work of Huang and the second-named author (which has largely motivated the present section), the authors use the results obtained here towards finer results in measure equivalence classification of the right-angled Artin groups \cite{HHFiniteOut}.
\end{itemize}

To state our main corollary, we make the following definition.

\begin{Def}[Untransvectable extension graph]\label{de:untransvectable-extension}
Let $G$ be a graph product over a finite simple graph $\Gamma$. The \emph{untransvectable extension graph}\index{Untransvectable!Untransvectable extension graph}\index{Extension graph!Untransvectable extension graph} is the subgraph $\Gamma_G^{ue}\subseteq\Gamma_G^e$ spanned by all vertices corresponding to parabolic subgroups of untransvectable vertex type. 
\end{Def}

Notice that $\Gamma_G^{ue}$ is a $G$-invariant subgraph of $\Gamma_G^e$, for the $G$-action on $\Gamma_G^e$ coming from the $G$-action by conjugation on itself. We warn the reader that $\Gamma_G^{ue}$ might be disconnected, and might be empty. 

We also refer to Definition~\ref{Def:MEWitness}, \cpageref{Def:MEWitness} for the definition of an ME-witness.

\begin{Cor}\label{cor:untransvectable-extension-graph}
Let $G$ and $H$ be two graph products over clique reduced finite simple graphs $\Gamma_G,\Gamma_H$, with countably infinite amenable vertex groups.

Then $(\Gamma_G^{ue},\Gamma_H^{ue})$ is an ME-witness for $G,H$. In particular, if $G$ and $H$ are measure equivalent, then $\Gamma_G^{ue}$ and $\Gamma_H^{ue}$ are isomorphic.
\end{Cor}

\begin{proof}
Let $\calg$ be a measured groupoid over a standard probability space, and let $\rho_G:\calg\to G$ and $\rho_H:\calg\to H$ be two action-like cocycles. Notice that $(\calg,\rho_G)$ is stably $\Gamma_G^{ue}$ in the sense of Section~\ref{sec:blueprint} if and only if it is of untransvectable vertex type. Therefore, the first point of Definition~\ref{Def:MEWitness} is exactly the contents of Proposition~\ref{prop:untransvectable}, \cpageref{prop:untransvectable}. And the second point of Definition~\ref{Def:MEWitness} is a consequence of Lemma~\ref{lemma:adjacency}, \cpageref{lemma:adjacency}, which characterizes adjacency in the (untransvectable) extension graph with no reference to the action-like cocycle. Finally, the additional part of the corollary follows from Proposition~\ref{prop:witness}, \cpageref{prop:witness}.
\end{proof}

In particular, Corollary~\ref{cor:untransvectable-extension-graph} establishes the untransvectable extension graph as a new invariant in the measure equivalence classification of right-angled Artin groups. Notice that when $\Gamma_G$ and $\Gamma_H$ are transvection-free, then $\Gamma_G^{ue}=\Gamma_G^e$ and $\Gamma_H^{ue}=\Gamma_H^e$, and Corollary~\ref{cor:untransvectable-extension-graph} recovers \cite[Theorem~2]{HH21}. Let us mention that, although we believe that the techniques from \cite{HH21} could have been extended to prove the results from the present section, we decided to use a different path, more in the spirit of the previous sections of the present work.

\subsection{Detecting clique subgroups}

We make the following definition.

\begin{Def}[Property~$\Pclique$]\index{Property!$\Pclique$}\label{de:pclique}
    Let $\calp$ be a measured groupoid over a standard probability space. Let $\calc\subseteq\calp$ be a measured subgroupoid. We say that $(\calp,\calc)$ satisfies Property~$\Pclique$ if the following two properties hold.
    \begin{description}
        \item[$\Pclique_1$] The subgroupoid $\calc$ is amenable, of infinite type, and stably normal in $\calp$.
        \item[$\Pclique_2$] If $\calc'$ is a measured subgroupoid of $\calp$ such that $(\calp,\calc')$ satisfies Property~$\Pclique_1$, then $\calc'$ is stably contained in $\calc$. 
    \end{description}
\end{Def}

\begin{Lmm}\label{lemma:clique}
    Let $G$ be a graph product of countably infinite amenable groups over a finite simple graph $\Gamma$. Let $\calg$ be a measured groupoid over a standard probability space, equipped with an action-like cocycle $\rho:\calg\to G$. Let $P\subseteq G$ be a product parabolic subgroup, and let  $\calp=\rho^{-1}(P)$. 

     Let $C\subseteq P$ be the clique factor of $P$.
\begin{itemize}
    \item If $C=\{1\}$ then no measured subgroupoid $\calc$ of $\calp$ is such that $(\calp,\calc)$ satisfies Property~$\Pclique$.
    \item If $C$ is non-trivial then a measured subgroupoid $\calc\subseteq\calp$ is such that $(\calp,\calc)$ satisfies Property~$\Pclique$ if and only if $\calc$ is stably equal to $\rho^{-1}(C)$.
\end{itemize}
\end{Lmm}

\begin{proof}
    We first assume that $C=\{1\}$, and assume towards a contradiction that there exists a measured subgroupoid $\calc\subseteq\calp$ such that $(\calp,\calc)$ satisfies Property~$\Pclique$. Then $(\calp,\calp)$ satisfies Property~$\Pcofact$ (recall Definition~\ref{de:pcofact}, \cpageref{de:pcofact}). So Lemma~\ref{lemma:cofactor-2}, \cpageref{lemma:cofactor-2} applies and shows that $\calp$ is stably equal to $\rho^{-1}(Q)$ for some cofactor $Q\subseteq P$. This is a contradiction to Lemma~\ref{lemma:inclusion-parabolics}, \cpageref{lemma:inclusion-parabolics} using that $\rho$ is action-like.

    We now assume that $C\neq\{1\}$.
    \begin{itemize}
        \item We first show that $\rho^{-1}(C)$ satisfies Property $\Pclique$.\\
    The subgroup $C$ is infinite and amenable because vertex groups are infinite and amenable. Let $\calc=\rho^{-1}(C)$. Then $\calc$ is amenable because $C$ is amenable and $\rho$ has trivial kernel (Lemma~\ref{lemma:amenable-subgroup-subgroupoid}, \cpageref{lemma:amenable-subgroup-subgroupoid}), of infinite type because $C$ is infinite and $\rho$ is action-like, and normal in $\calp$ because $C$ is normal in $P$ (Example~\ref{ex:normal}, \cpageref{ex:normal}). So $(\calp,\calc)$ satisfies Property~$\Pclique_1$.

    To check Property~$\Pclique_2$, let $\calc'\subseteq\calp$ be an amenable subgroupoid of infinite type which is stably normal in $\calp$. We aim to prove that $\calc'$ is stably contained in~$\calc$. Without loss of generality, we will assume that $\calc'$ is normalized by $\calp$. We write $P=C\times F_1\times\dots\times F_k$, where the subgroups $F_j$ do not further split as a direct product. For every $j\in\{1,\dots,k\}$, let $\rho_j:\calp\to F_j$ be the cocycle obtained by post-composing $\rho$ by the projection $G\to F_j$. Up to a countable Borel partition of $X$, we will assume using Lemma~\ref{lemma:support}, \cpageref{lemma:support} that $(\calc',\rho)$ is tightly $A$-supported for some parabolic subgroup $A\subseteq P$.
    We aim to prove that $A\subseteq C$. So assume by contradiction that this is not the case. Using Lemma~\ref{lemma:product-parabolic-subgroup}, \cpageref{lemma:product-parabolic-subgroup} we can find $j\in\{1,\dots,k\}$ such that $A\cap F_j\neq\{1\}$. In particular $(\rho_j)_{|\calc'}$ is nowhere trivial. Since $\calc'$ is amenable, Lemma~\ref{lemma:alternative}, \cpageref{lemma:alternative} and Remark~\ref{rk:amenable}, \cpageref{rk:amenable} ensure that $(\calc',\rho_j)$ is $F_j$-elementary. Lemma~\ref{lemma:elementary-normal}, \cpageref{lemma:elementary-normal} therefore implies that $(\calp,\rho_j)$ is also $F_j$-elementary. In particular $(\rho^{-1}(F_j),\rho_j)$ is $F_j$-elementary, which contradicts Lemma~\ref{lemma:nonelementary-example}, \cpageref{lemma:nonelementary-example} (since $F_j$ is not a clique). Therefore $A\cap F_j=\{1\}$ for all $j$ and thus $A\subseteq C$ as wanted.  
    \item The stability of Property $\Pclique$ under restriction and stabilization implies that any subgroupoid stably equal to $\rho^{-1}(C)$, satisfies itself Property~$\Pclique$.
    \item Now if $\calc$ satisfies Property~$\Pclique$, we showed in the first point that it is stably contained in $\rho^{-1}(C)$. By the maximality property $\Pclique_2$ it is therefore stably equal to $\rho^{-1}(C)$.\qedhere
    \end{itemize}
    \end{proof}

\subsection{Proof of Proposition~\ref{prop:untransvectable}}

When all vertex groups are amenable, subgroupoids of isolated clique type are amenable, and therefore they do not satisfy Property~$\Pprod$ from Definition~\ref{de:pprod}, \cpageref{de:pprod}.\\ Lemma~\ref{lemma:maximal-product}, \cpageref{lemma:maximal-product} thus reformulates as follows in this case.

\begin{Lmm}\label{lemma:maximal-product-amenable-vertex-groups}
Let $G$ be a graph product of countably infinite amenable groups over a finite simple graph $\Gamma$. 
Let $\calg$ be a measured groupoid over a standard probability space $X$, equipped with an action-like cocycle $\rho:\calg\to G$. Let $\calp\subseteq\calg$ be a measured subgroupoid. 

Then $(\calp,\rho)$ is of maximal product type and nowhere of isolated clique type if and only if $(\calg,\calp)$ satisfies Property~$\Pprod$. \qed
\end{Lmm}

We now introduce the following variation over Property~$\Pvert$ (Definition~\ref{de:pvert}). 

\begin{Def}[Property~$\Pamen$]\index{Property!$\Pamen$}\label{de:pamen}
Let $\calg$ be a measured groupoid over a standard probability space $X$, and let $\calv$ be a measured subgroupoid of $\calg$. We say that the pair $(\calg,\calv)$ satisfies \emph{Property $\Pamen$} if the following two properties hold.
\begin{description}
\item[$\Pamen_1$] There exists a countable Borel partition $X^*=\dunion_{i\in I}X_i$ of a conull Borel subset $X^*\subseteq X$ such that for every $i\in I$, there exist measured subgroupoids \[\calg_{|X_i}=\calf_0\supseteq\calp_1\supseteq\calf_1\supseteq\calp_2\supseteq\calf_2\supseteq\dots\supseteq\calp_n\supseteq\calc_n=\calv_{|X_i}\] of $\calg_{|X_i}$ such that
        \begin{enumerate}
            \item for every $j\in\{1,\dots,n\}$, the pair $(\calf_{j-1},\calp_j)$ satisfies Property~$\Pprod$ (see Definition~\ref{de:pprod}, \cpageref{de:pprod});
            \item for every $j\in\{1,\dots,n-1\}$, the pair $(\calp_j,\calf_j)$ satisfies Property~$\Pf$ (see Definition~\ref{de:pfact2}, \cpageref{de:pfact2});
            \item for every $j\in\{1,\dots,n-1\}$ and every positive measure Borel subset $U\subseteq X_i$, there does not exist any subgroupoid $\calc$ of $(\calp_j)_{|U}$ such that $((\calp_j)_{|U},\calc)$ satisfies Property~$\Pclique$ (see Definition~\ref{de:pclique}, \cpageref{de:pclique});
            \item the pair $(\calp_n,\calc_n)$ satisfies Property~$\Pclique$.
        \end{enumerate}
\item[$\Pamen_2$] For every measured subgroupoid $\calw$ of $\calg$ of infinite type which is stably contained in $\calv$, and every measured subgroupoid $\calh\subseteq\calg$, if $\calw$ is stably normalized by $\calh$, then $\calv$ is stably normalized by $\calh$.
        \end{description}
\end{Def}

\begin{Lmm}\label{lemma:untransvectable-vertex}
    Let $G$ be a graph product of countably infinite amenable groups over a clique-reduced finite simple graph $\Gamma$.  
    Let $\calg$ be a measured groupoid over a standard probability space $X$, equipped with an action-like cocycle $\rho:\calg\to G$. Let $\calv\subseteq\calg$ be a measured subgroupoid. 
    
    Then $(\calv,\rho)$ is of untransvectable vertex type if and only if $(\calg,\calv)$ satisfies Property~$\Pamen$.
\end{Lmm}

\begin{proof}
We first assume that $(\calv,\rho)$ is of untransvectable vertex type, and prove that $(\calg,\calv)$ satisfies Property~$\Pamen$.
\smallskip

Up to passing to a conull Borel subset of $X$, taking a countable partition, and conjugating the cocycle, we will assume that $\calv=\rho^{-1}(G_v)$, where $v\in V\Gamma$ is untransvectable. 
 
We start with $\Pamen_1$. Let \[G=F_0\supseteq P_1\supseteq F_1\supseteq\dots\supseteq P_n\supseteq C_n=G_v\] be parabolic subgroups provided by Lemma~\ref{lemma:graph}, \cpageref{lemma:graph} (here we write $F_n=C_n$). For every $j\in\{1,\dots,n\}$, let $\calp_j=\rho^{-1}(P_j)$ and $\calf_{j-1}=\rho^{-1}(F_{j-1})$, and let $\calc_n=\rho^{-1}(C_n)$. Since $P_j$ is a maximal product parabolic subgroup of $F_{j-1}$ and is not of isolated clique type, Lemma~\ref{lemma:maximal-product}, \cpageref{lemma:maximal-product} ensures that $(\calf_{j-1},\calp_j)$ satisfies Property~$\Pprod$ for every $j\in\{1,\dots,n\}$. Since $P_j$ has trivial clique factor for every $j\in\{1,\dots,n-1\}$, and $F_j$ is a factor of $P_j$, Lemma~\ref{lemma:factor-2}, \cpageref{lemma:factor-2} applies and shows that $(\calp_j,\calf_j)$ satisfies Property~$\Pf$. Since $C_n$ is the clique factor of $P_n$, Lemma~\ref{lemma:clique}, \cpageref{lemma:clique} ensures that $(\calp_n,\calc_n)$ satisfies Property~$\Pclique$. On the other hand, for every $j\in\{1,\dots,n-1\}$, since $P_j$ has trivial clique factor, Lemma~\ref{lemma:clique} ensures that for every positive measure Borel subset $U\subseteq X$, there does not exist any measured subgroupoid $\calc$ of $(\calp_j)_{|U}$ such that $((\calp_j)_{|U},\calc)$ satisfies Property~$\Pclique$. This proves Property~$\Pamen_1$.

We now prove that $(\calg,\calv)$ satisfies Property~$\Pamen_2$. Let $\calw$ and $\calh$ be as in $\Pamen_2$. Since $\calw$ is of infinite type and stably contained in $\calv$, up to a countable Borel partition of a conull Borel subset of $X$, we will assume that $(\calw,\rho)$ is tightly $G_v$-supported. As $\calw$ is stably normalized by $\calh$, Lemma~\ref{lemma:support-normal}, \cpageref{lemma:support-normal} allows us to further assume that $\calh\subseteq\rho^{-1}(G_v\times G_v^{\perp})$. But since $G_v$ is normal in $G_v\times G_v^{\perp}$, it follows that $\calv$ is (stably) normalized by $\calh$ (Example~\ref{ex:normal}, \cpageref{ex:normal}), as desired. 

\medskip

 Conversely, let us assume that $(\calg,\calv)$ satisfies Property~$\Pamen$, and prove that $(\calv,\rho)$ is of untransvectable vertex type.

Up to a countable partition, we can assume that there exist measured subgroupoids \[\calg=\calf_0\supseteq\calp_1\supseteq\calf_1\supseteq\calp_2\supseteq\calf_2\supseteq\dots\supseteq\calp_n\supseteq\calc_n=\calv\] 
as in $\Pamen_1$. Using Lemmas~\ref{lemma:factor-2},~\ref{lemma:clique} and~\ref{lemma:maximal-product-amenable-vertex-groups}, up to a further partition, we can assume that there exists a chain of parabolic subgroups \[G=F_0\supseteq P_1\supseteq F_1\supseteq\dots\supseteq P_n\supseteq C_n\] such that
    \begin{enumerate}
    \item for every $j\in\{1,\dots,n\}$, one has $\calp_j=\rho^{-1}(P_j)$, and $\calf_{j-1}=\rho^{-1}(F_{j-1})$, and $\calc_n=\rho^{-1}(C_n)$, and 
    \item for every $j\in\{1,\dots,n\}$, $P_j$ is a maximal product parabolic subgroup of $F_{j-1}$ and not of isolated clique type (Lemma~\ref{lemma:maximal-product-amenable-vertex-groups}, \cpageref{lemma:maximal-product-amenable-vertex-groups} using the first assertion of $\Pamen_1$),
    \item for every $j\in\{1,\dots,n-1\}$, the clique factor of $P_j$ is trivial (Lemma~\ref{lemma:clique}, using the third assertion of $\Pamen_1$), and $F_j$ is a factor of $P_j$ (Lemma~\ref{lemma:factor-2}, \cpageref{lemma:factor-2} using the second assertion of $\Pamen_1$), and
    \item $C_n$ is the clique factor of $P_n$ (Lemma~\ref{lemma:clique}, \cpageref{lemma:clique} using the fourth assertion of $\Pamen_1$).
    \end{enumerate} 
    \smallskip
    
    We now claim that $C_n$ is conjugate to $G_v$ for some $v\in V\Gamma$. This will be enough to conclude our proof, as Lemma~\ref{lemma:graph}, \cpageref{lemma:graph} then implies that $v$ is untransvectable.

    We are thus left with proving the above claim. Assume that $C_n$ is conjugate to $G_{\Upsilon}$, for a clique subgraph $\Upsilon\subseteq \Gamma$ which contains at least two vertices. Up to conjugating the cocycle, we will assume without loss of generality that $C_n=G_{\Upsilon}$. Since $\Gamma$ is clique-reduced, we can find a vertex $w\in V\Upsilon$ such that $\lk(w)\nsubseteq\Upsilon\circ\Upsilon^{\perp}$. We now aim for a contradiction using Property~$\Pamen_2$. So let $\calw=\rho^{-1}(G_w)$, and let $\calh=\rho^{-1}(G_w\times G_w^{\perp})$. Then $\calw\subseteq\calv$, and $\calw$ is normalized by $\calh$ (Example~\ref{ex:normal}, \cpageref{ex:normal}). Property~$\Pamen_2$ implies that $\calv$ is stably normalized by $\calh$. Since $(\calv,\rho)$ is tightly $G_{\Upsilon}$-supported (Lemma~\ref{lemma:tight-support}, \cpageref{lemma:tight-support}), Lemma~\ref{lemma:support-normal} enables us to find a positive measure Borel subset $U\subseteq X$ such that $\calh_{|U}\subseteq\rho^{-1}(G_{\Upsilon}\times G_{\Upsilon}^{\perp})$. It follows from Lemma~\ref{lemma:inclusion-parabolics}, \cpageref{lemma:inclusion-parabolics} that $G_w\times G_w^{\perp}\subseteq G_{\Upsilon}\times G_{\Upsilon}^{\perp}$, which contradicts our choice of~$w$. This contradiction completes our proof.
\end{proof}

We are now in position to complete the proof of the main proposition of this section.

\begin{proof}[Proof of Proposition~\ref{prop:untransvectable}]
    Assume that $(\calv,\rho_G)$ is of untransvectable vertex type. By Lemma~\ref{lemma:untransvectable-vertex} applied to the cocycle $\rho_G$, the pair $(\calg,\calv)$ satisfies Property~$\Pamen$. And Lemma~\ref{lemma:untransvectable-vertex} applied to the cocycle $\rho_H$ then proves that $(\calv,\rho_H)$ is of untransvectable vertex type. By symmetry the other implication also holds. 
\end{proof}

\section{Examples}
\label{Sec:ExemplesMEVRP}
To conclude this part, we provide some examples illustrating the different cases treated in the past sections. All the graph products presented in this section are in fact right-angled Artin groups (though they are sometimes viewed as graph products with non-cyclic vertex groups). So our work provides new cases of the measure equivalence classification of right-angled Artin groups, beyond those treated in \cite{HH21}.
On the other hand, some cases remain open, such as right-angled Artin groups defined over trees.

\begin{Ex}\label{Ex:PentagoneF2F3} The following example is an illustration of \cref{theo:classificationversionintro}.

    The two graph products defined in \cref{fig:Pentagone1} are not measure equivalent. 
    
    Indeed, the graph (namely the pentagon) is transvection-free and has no partial conjugation. Therefore \cref{theo:classificationversionintro} applies: if $G$ and $H$ were measure equivalent, then we would have a graph isomorphism sending the vertex group $F_2$ to an orbit equivalent vertex group in $H$. Such a vertex group does not exist in $H$. Therefore $G$ and $H$ cannot be measure equivalent.
\end{Ex}
\begin{figure}[htbp]
    \centering
    \includegraphics{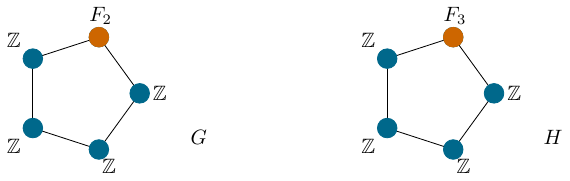}
    \caption{Groups from \cref{Ex:PentagoneF2F3}}
    \label{fig:Pentagone1}
\end{figure}

\begin{Ex}\label{Ex:PentaDeplie}
    The following example is an illustration of Corollary~\ref{Cor:VRPandGraphMorphism}.

    Let $G$ and $H$ be as in \cref{fig:ExemplePentagoneDeplie} (the two leftmost graphs). If $G$ and $H$ were measure equivalent, then we would have a graph homomorphism $\sigma:\Gamma_G \rightarrow \Gamma_H$ such that $G_v$ is measure equivalent to $H_{\sigma(v)}$ for all $v\in V\Gamma_G$. In particular $\sigma$ must send \textcolor{Turquoise}{blue} vertices to \textcolor{Turquoise}{blue} vertices. In particular, by looking at the restriction of $\sigma$ to the leftmost pentagon in $\Gamma_G$ (all of whose vertices are \textcolor{Turquoise}{blue}), there would exist a graph homomorphism from a pentagon to a segment of length $3$, which is not the case. Therefore $G$ and $H$ are not measure equivalent.
\end{Ex}

\begin{figure}[htbp]
    \centering
    \includegraphics[width=\textwidth]{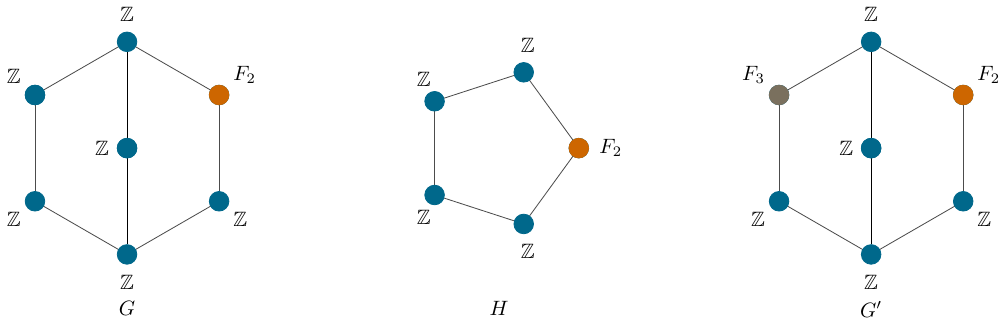}
    \caption{Graphs and groups from \cref{Ex:PentaDeplie,Rq:PbOuvertPentagones}}
    \label{fig:ExemplePentagoneDeplie}
\end{figure}

\begin{Rq}\label{Rq:PbOuvertPentagones} 
We do not know however whether or not the groups $H$ and $G'$ are measure equivalent. There do exist graph homomorphisms $\Gamma_G\to\Gamma_H$ and $\Gamma_H\to\Gamma_G$, such that a vertex and its image always carry measure equivalent vertex groups, so Corollary~\ref{Cor:VRPandGraphMorphism} does not provide any obstruction. In fact, if the \textcolor{MFCB}{brown} vertex group was isomorphic to $F_2$, then $G'$ would be isomorphic to an index $2$ subgroup of $H$ (whence measure equivalent to $H$). However, with the \textcolor{MFCB}{brown} vertex group isomorphic to $F_3$, we have no reason to expect that the groups $G'$ and $H$ should be measure equivalent.
\end{Rq}

\begin{Ex}\label{Ex:3}
    The following example is an illustration of \cref{prop:untransvectable}.
    
    Let $G$ and $H$ be the two groups defined in the left half of \cref{fig:Ex3}, \cpageref{fig:Ex3}. Then $G$ and $H$ are not measure equivalent. 
    
    Indeed, first remark that in \cref{fig:Ex3} all untransvectable vertices are drawn in \textcolor{Orange}{orange}. Second, note that $H$ can be rewritten as a right-angled Artin group over the graph drawn at the top of the right half of \cref{fig:Ex3} (in particular, with this decomposition all vertex groups are now amenable). In the latter graph, all vertices are transvectable. 
    If $G$ and $H$ were measure equivalent then we could apply \cref{prop:untransvectable}, \cpageref{prop:untransvectable} and get the existence of an \textcolor{Orange}{untransvectable} vertex in the rightmost graph defining $H$. Hence $G$ and $H$ can not be measure equivalent. For the same reason, the groups $H$ and $H'$ are not measure equivalent.
    
    In fact, as a consequence of Corollary~\ref{cor:untransvectable-extension-graph}, we can also prove that $G$ and $H'$ are not measure equivalent. Indeed, the untransvectable extension graph $\Gamma_{H'}^e$ (see Definition~\ref{de:untransvectable-extension}) is totally disconnected: its vertices represent the conjugates of the leftmost \textcolor{Orange}{orange} vertex group and of the rightmost \textcolor{Orange}{orange} vertex group, and no two distinct conjugates of these vertex groups commute. On the other hand $\Gamma_G^e$ contains an edge, because all \textcolor{Orange}{orange} vertex groups are untransvectable.
    
    One could provide many variations over the above example. But many cases remain open: the most basic ones is that our methods do not allow to decide whether or not right-angled Artin groups defined over trees of diameter at least $3$ are measure equivalent or not.
\end{Ex}

\begin{figure}[htbp]
  \centering
  \includegraphics[width=0.9\textwidth]{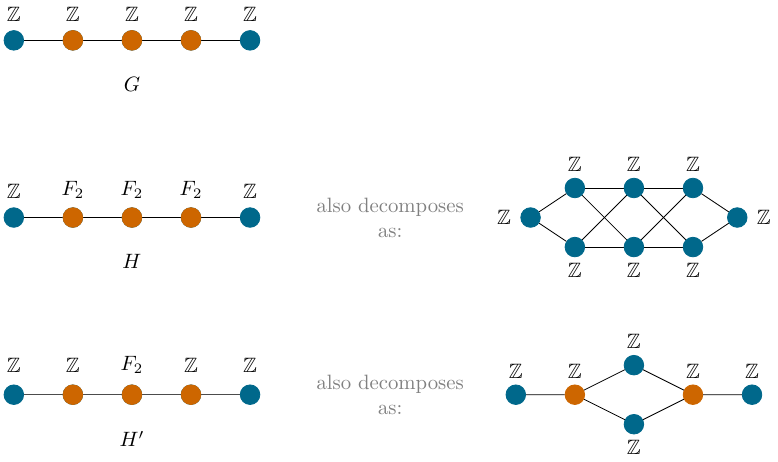}
  \caption{Groups from \cref{Ex:3}}
  \label{fig:Ex3}
\end{figure}


\newpage

\part{Quantitative aspects and the inverse problem for RAAGs}\label{Part:QuantitativeResults}

The goal of the present part is to obtain a quantitative version of the previous classification theorem (\cref{theo:classificationversionintro}, \cpageref{theo:classificationversionintro}) and solve the inverse problem in quantitative measure equivalence for a large family of right-angled Artin groups. We refer to \cref{Sec:QuantitativeDefinitions}, \cpageref{Sec:QuantitativeDefinitions} for the definition of quantitative orbit equivalence.

In \cref{Sec:FactToProd} we prove that when two graph products are defined over the same graph, the integrability at the level of the vertex groups can be lifted at the level of the graph products themselves (\cref{Prop:CouplageGraphProd}). This actually does not need any assumption on the defining graph and relies on a previous result of Huang and the second named author \cite{HH21}, generalizing an argument of Gaboriau \cite[{$\mathbf{P}_{\mathrm{ME}}\mathbf{6}$}]{Gab}. In \cref{Sec:ProdToFact} we show a generalized version of \cref{Prop:DescenteDuProfil}, proving for a large family of graph products that the integrability at the level of the graph products passes at the level of the vertex groups.

The inverse problem is finally solved in \cref{Sec:InverseProblem}.

\section{From the factors to the graph product}\label{Sec:FactToProd}
We prove here \cref{Prop:CouplageGraphProd} which we recall below --~in the introduction, we used the language of integrable orbit equivalence couplings instead of pairings, but these are equivalent, see Section~\ref{Sec:QuantitativeDefinitions}.

\begin{Th3}
 Let $G$ and $H$ be two graph products over a finite simple graph $\Gamma$, with finitely generated vertex groups $(G_v)_{v\in V\Gamma}$ and $(H_v)_{v\in V\Gamma}$, respectively.
 
  Let $\varphi,\psi :[1,+\infty[\rightarrow[1,+\infty[ $ be two non-decreasing functions. Assume that for every $v\in V\Gamma$, there exists a $(\varphi,\psi)$-integrable orbit equivalence pairing from $G_v$ to~$H_v$. 
  
  Then there exists a $(\varphi,\psi)$-integrable orbit equivalence pairing from
  $G$ to $H$.
\end{Th3}

This theorem is a quantitative version of \cite[Proposition~4.2]{HH21}. Our proof is actually based on the pairing built in \cite{HH21} to prove the latter proposition. In order to
keep this article self contained we recall the construction below.

\paragraph{Construction of the pairing} From now on we fix an enumeration $\{v_1,\ldots,v_n\}$ of the vertices of $\Gamma$.

For all $i\in\{1,\ldots,n\}$, let $(Y_i,\nu_i)$ be an orbit equivalence pairing between $G_{v_i}$ and~$H_{v_i}$. Consider also a free measure-preserving action of $G$ on a standard probability space~$Z$. Let $X=Z\times Y_1 \times \cdots \times Y_n$ and endow $X$ with the product probability measure denoted by $\mu$. An element $x$ of $X$ is denoted by $x=(z,y_1,\dots,y_n)$ where $z\in Z$ and $y_i\in Y_i$ for all $i\in \{1,\ldots,n\}$. We also denote by $r_i:G\to G_{v_i}$ the retraction which is the identity on $G_{v_i}$ and sends all other factors to $\{1\}$. We can thus define a free measure-preserving $G$-action on $X$, where the action on each $Y_i$ is through $r_i$, that is to say:
\begin{equation}\label{eq:DefActionGsurXProd}
g(z,y_1,\ldots,y_n)=\left(gz, r_1(g)y_1,\ldots,r_n(g)y_n\right). 
\end{equation}

To define the action of $H$ on $X$ we first define the action of $H_{v_i}$ on
$X$ for $i\in \{1,\ldots,n\}$. So let $i\in \{1,\ldots,n\}$ and denote by $c_i:G_{v_i}\times Y_i \rightarrow H_{v_i}$ and $c^{\prime}_i:H_{v_i}\times Y_{i} \rightarrow G_{v_i}$ the two orbit equivalence cocycles associated to the pairing between $G_{v_i}$ and $H_{v_i}$. Recall that $c_i(g,y_i)$ is defined for a.e.\ $y_i\in Y_i$, as the unique $h\in H_{v_i}$ such that $gy_i=hy_i$. Similarly $c^{\prime}_i(h,y_i)$ is the unique $g\in G_{v_i}$ such that $gy_i=hy_i$. Hence, for $h\in H_{v_i}$ we define for a.e.\ $(z,y_1,\ldots,y_n)\in X$, 
\begin{equation*}
    h(z,y_1,\ldots,y_n):=c^{\prime}_i(h,y_i)(z,y_1,\ldots,y_n).
\end{equation*}
This leads to a well defined free action of $H_{v_i}$ on $X$.

To obtain an action of the entire group $H$, we now show that if $h_i,h_j$ belong to adjacent vertex groups $H_{v_i}$, $H_{v_j}$ respectively, then $h_i \big(h_j(z,y_1,\ldots,y_n)\big)=h_j\big(h_i (z,y_1,\ldots,y_n)\big)$.\\
First remark that, by definition of the retractions, when $g$ belongs to some vertex group $G_{v_l}$ the above \cref{eq:DefActionGsurXProd} becomes
\begin{equation}\label{eq:ActionGsurX}
    g(z,y_1,\ldots,y_n)
    =\left(gz,\, y_1,\, \ldots,\, y_{l-1},\, g y_l,\, y_{l+1},\, \ldots,\, y_n\right). 
\end{equation}
Finally, since for all $l$ the cocycle $c^\prime_l$ goes from $H_{v_l}\times Y_l$ to $G_{v_l}$ and since $v_i$ and $v_j$ are adjacent in $\Gamma$, we thus get that $c^{\prime}_i(h_i,y_i)$ and $c^{\prime}_j(h_j,y_j)$ belong to adjacent vertex groups. Therefore, using \cref{eq:ActionGsurX} for the second and the fourth equalities, we get
\begin{align*}
    h_i \Big(h_j(z,y_1,\ldots,y_n)\Big)
    &=h_i \Big(c^{\prime}_j(h_j,y_j) (z,y_1,\ldots,y_n) \Big),\\
    &=c^{\prime}_i(h_i,y_i) \Big(c^{\prime}_j(h_j,y_j) (z,y_1,\ldots,y_n) \Big),\\
    &=c^{\prime}_j(h_j,y_j)  \Big( c^{\prime}_i(h_i,y_i)(z,y_1,\ldots,y_n) \Big),\\
    &=h_j \Big( c^{\prime}_i(h_i,y_i)(z,y_1,\ldots,y_n) \Big),\\
    &=h_j\Big(h_i (z,y_1,\ldots,y_n)\Big).
\end{align*}

The action defined at the level of the vertex groups thus extends at the level of the entire graph product $H$. The fact that this $H$-action is free follows from the freeness of the $G$-action on $Z$ and a normal form argument.

We can now use this construction to prove our result.

\begin{proof}[Proof of \cref{Prop:CouplageGraphProd}] 
  Keeping all notations from above, let $X=Z\times Y_1\times \cdots \times Y_n$ be the
  pairing defined in the last paragraph, where for all $i\in\{1,\ldots,n\}$  we choose $Y_i$ to be a $(\varphi,\psi)$-integrable orbit equivalence pairing from $G_{v_i}$ to $H_{v_i}$. Denote by $c:G\times X\rightarrow H$ and $c^\prime:H\times X \rightarrow G$ the two cocycles associated to the pairing $(X,\mu)$. Finally recall that an element $x$ of $X$ is denoted by $x=(z,y_1,\dots,y_n)$ where $z\in Z$ and $y_i\in Y_i$ for all $i\in \{1,\ldots,n\}$.

  For any $i\in\{1,\ldots,n\}$ denote by $S^i_{G}$ and $S^i_{H}$ 
  finite generating sets of
  $G_{v_i}$ and $H_{v_i}$ respectively. Then $S_G:=\sqcup^n_{i=1}S^i_{G}$ generates $G$ and $S_H:=\sqcup^n_{i=1}S^i_{H}$ generates $H$. Therefore, to prove that this pairing is $(\varphi,\psi)$-integrable we only need to show that for all $s\in S_G$ and all $s^\prime \in S_H$ we have
  \begin{align}
    \int_{x\in X} \varphi \left( \big|c(s,x)\big|_{S_H} \right)
    d\mu(x) &< +\infty,\label{eq:CondIntegrabilityPhi}\\
    \int_{x\in X} \psi \left(\big| c^\prime(s^\prime,x)\big|_{S_G}\right)
    d\mu(x) &< +\infty.\label{eq:CondIntegrabilityPsi}
  \end{align}

  \noindent \textbf{$\varphi$-integrability} 
  Consider $s\in S_{G}$ and let $i\in\{1,\ldots,n\}$ be such that $s\in S^i_{G}$. For a.e.\ $x\in X$, using first the definition of the action of $H$ on $X$ and then using that $c^{\prime}_i\big(c_i(s,y_i),y_i\big)=s$, we obtain
  \begin{equation*}
      c_i(s,y_i)x=c^\prime_i\big(c_i(s,y_i),y_i\big)x=sx,
  \end{equation*}
  and thus $c(s,x)=c_i(s,y_i)$. Therefore, by choice of $S_H$ we get that 
  \begin{equation*}
      |c(s,x)|_{S_H}=|c_i(s,y_i)|_{S_H}=|c_i(s,y_i)|_{S^i_H}.
  \end{equation*} 
  Finally, using that all spaces $Z$ and $Y_j$ have measure $1$ and that $c_i$ is $\varphi$-integrable on $Y_i$, we thus obtain that \cref{eq:CondIntegrabilityPhi} is verified.
  \medskip
  
  \noindent\textbf{$\psi$-integrability} Now consider $s^{\prime}\in S_{H}$ and let $i\in\{1,\ldots,n\}$ be such that $s^\prime\in S^i_{H}$. Then by definition of the $H$-action on $X$, we have $s^{\prime}x=c^{\prime}_i\left(s^{\prime},y_i\right)x$ for a.e.\ $x\in X$. Therefore $c^\prime(s^\prime,x)=c^\prime_i(s^\prime,y_i)$ for a.e.\ $x\in X$, and since the latter element belongs to $G_{v_i}$ we thus obtain
  \begin{equation*}
    \left\vert c^\prime\big(s^{\prime},x\big)\right\vert_{S_G}
    =\left\vert c^\prime_i\left(s^{\prime},y_i\right)\right\vert_{S_{G}}
    =\left\vert c^\prime_{i}\left(s^{\prime},y_i\right)\right\vert_{S^i_{G}}.
  \end{equation*}
  But by assumption, $c^{\prime}_i$ is $\psi$-integrable, thus \cref{eq:CondIntegrabilityPsi} is verified and hence the pairing constructed above is $(\varphi,\psi)$-integrable.
\end{proof}
\section{From the graph product to the factors}\label{Sec:ProdToFact}
The goal of this section is to prove \cref{Prop:DescenteDuProfil}, \cpageref{Prop:DescenteDuProfil}, which states that the quantification at the level of the graph products passes at the level of the vertex groups. 

\subsection{The general statement and applications}
\subsubsection{Main result}
What we actually show here is the following more general version of the latter theorem. We refer to \cref{de:VRP}, \cpageref{de:VRP} for the definition of the Vertex Recognition Property and to \cref{Def:Untransvectable}, \cpageref{Def:Untransvectable} for the definition of an untransvectable vertex.

\begin{Th}\label{Th:ProdToFactWithVRP}
    Let $\varphi,\psi :[1,+\infty[\rightarrow[1,+\infty[ $ be two increasing unbounded functions and let $G$ and $H$ be two graph products defined over $\Gamma_G,\Gamma_H$ respectively, with infinite finitely generated vertex groups. Assume that there exists a $(\varphi,\psi)$-integrable measure equivalence coupling from $G$ to $H$.

    \begin{enumerate}
        \item If for all $v\in V\Gamma_G$ and all $w\in V\Gamma_H$ the vertex groups $G_v$ and  $H_w$ are amenable then for all {untransvectable vertices} $v\in V\Gamma_G$, there exist $\sigma(v)\in V\Gamma_H$ and a $(\varphi,\psi)$-integrable measure equivalence coupling from $G_v$ to $H_{\sigma(v)}$.
        \item If $G$ and $H$ both belong to a same class of groups $\classC$ satisfying the Vertex Recognition Property, then there exist a graph homomorphism $\sigma:\Gamma_G \rightarrow \Gamma_H$ and for all $v\in V\Gamma_G$ a $(\varphi,\psi)$-integrable measure equivalence coupling from $G_v$ to $H_{\sigma(v)}$.
        \item If $\Gamma_G$ and $\Gamma_H$ are transvection-free with no partial conjugation, then there exists a graph isomorphism $\sigma: \Gamma_G\to \Gamma_H$ and for all $v\in V\Gamma_G$ a $(\varphi,\psi)$-integrable measure equivalence coupling from $G_v$ to $H_{\sigma(v)}$.
    \end{enumerate}
\end{Th}

The second assertion applies in particular whenever $\Gamma_G$ and $\Gamma_H$ are transvection-free and strongly reduced (Proposition~\ref{prop:vertex-recognition-strongly-reduced}, \cpageref{prop:vertex-recognition-strongly-reduced}).

The three conclusions of the above theorem provide quantitative versions of three situations already encountered in Part~\ref{Part:MEClassification} of this work. Namely, 
\begin{itemize}
\item the first part is a quantitative version of the work from Section~\ref{Sec:AmenalbeUntransvectable};
\item the second part is a version of Corollary~\ref{Cor:VRPandGraphMorphism}, \cpageref{Cor:VRPandGraphMorphism};
\item the third part (which is exactly Theorem~\ref{Prop:DescenteDuProfil}, \cpageref{Prop:DescenteDuProfil} from the introduction), is a quantitative version of Theorem~\ref{theo:classificationversionintro}.
\end{itemize}
We also refer to Section~\ref{Sec:ExemplesMEVRP} for illustrations of these various situations.

Recall also that when $G$ is a right-angled Artin group, asking $\Gamma_G$ to be transvection-free with no partial conjugation is equivalent to requiring $\mathrm{Out}(G)$ to be finite. So in particular the first and second assertions of our theorem apply to a larger family of right-angled Artin groups than the ones with finite outer automorphism group. 

Finally, we refer to Remark~\ref{Rq:AsymmetryOEMEQt} for a discussion on the asymmetry between Theorem~\ref{Prop:CouplageGraphProd} (stated in terms of orbit equivalence) and Theorem~\ref{Th:ProdToFactWithVRP} (stated in terms of measure equivalence).

\subsubsection{Obstructions in the non-amenable world}\label{sec:ApplicationOfProdToFact} As explained in the introduction, quantitative measure equivalence responds well to geometry, which sometimes provides us with upper bounds to the possible quantifications between two given groups. So far, most of the known obstructions were in the amenable world, our \cref{Th:ProdToFactWithVRP} above allows to extend them to a large family of non-amenable groups. We recall below some of the known geometric obstructions and detail the case of the isoperimetric profile that will be used in the next section.

Among the previously known obstructions were the results of Bowen and Austin in \cite{AustinBowen}. Indeed, Bowen proved that growth is stable under integrable measure equivalence. Austin moreover showed that two $L^1$-measure equivalent groups of polynomial growth have bi-Lipschitz asymptotic cones. This implies for example that $\bZ^4$ and the Heisenberg group $H_3$, although having the same growth rate, cannot be $L^1$-measure equivalent. Another obstruction relies on the notion of \emph{isoperimetric profile}. Let $S_G$ be a finite generating set of $G$ and recall that the \emph{$S_G$-boundary} of a finite subset $A$ of $G$ is defined as $\partial_{S_G}A:=\{g\in A \ : \ \exists s\in S_G, \ gs\notin A \}$. The \emph{isoperimetric profile}\index{Isoperimetric profile} of a finitely generated group $G$, with respect to a finite generating set $S$ is the function $I_G:\mathbb{N}^*\to\mathbb{R}$ defined by
\begin{equation*}
    I_{G}(n):= \sup_{A\subseteq G,\ 0<|A|\leq n} \frac{|A|}{|\partial_{S_G} A|}.
\end{equation*}
Note that (by the Følner criterion) a group is amenable if and only if its isoperimetric profile is unbounded.

Delabie, Koivisto, Le Maître and Tessera \cite{DKLMT} showed the monotonicity under quantitative measure equivalence of the isoperimetric profile and more precisely of its \emph{asymptotic behaviour}. More formally, if $f$ and $g$ are two real functions we denote $f \preccurlyeq g$\label{Nota:AsymptoBehaviour} if there exists some constant $C>0$ such that $f(x)=\mathcal{O}\big(g(Cx)\big)$ as $x$ tends to infinity. We write $f\simeq g$ if $f \preccurlyeq g$ and $g\preccurlyeq f$. 

\begin{Th}[{\cite[Th.1]{DKLMT}}] \label{Th:ProfiletOE}
  Let $G$ and $H$ be two finitely generated groups admitting a
  $(\varphi,L^0)$-integrable measure equivalence coupling from $G$ to $H$. If both $\varphi$ and $t/\varphi(t)$ are non-decreasing then $\varphi \circ I_{H} \preccurlyeq I_G$.
\end{Th}

As mentioned above, when $G$ and $H$ are non-amenable groups, both their isoperimetric profiles are bounded. Therefore, the latter theorem does not provide any obstruction for non-amenable groups. Our \cref{Th:ProdToFactWithVRP}, however, does: it allows us to extend the condition on the profile to the world of graph products as in the example below.

\begin{Ex}\label{Ex:ExampleUntransvectable}
  Let $G$ and $H$ be the graph products defined in \cref{fig:ExampleUntransvectable}.
  
  By \cite[Theorem 1.9]{DKLMT}, for all $\varepsilon>0$ there exists an $(L^{1/2-\varepsilon},L^0)$-integrable orbit equivalence pairing from $\bZ^2$ to $\bZ$. Therefore by \cref{Prop:CouplageGraphProd}, the groups $G$ and $H$ are $(L^{1/2-\varepsilon},L^0)$-integrably orbit equivalent for all $\varepsilon >0$.
  
  Now assume that we have a $(\varphi,L^0)$-integrable measure equivalence coupling from $G$ to $H$. The \textcolor{Orange}{orange} vertex on the left is untransvectable, therefore \cref{Th:ProdToFactWithVRP} provides us with a vertex group in $G$ that is $(\varphi,L^0)$-integrably measure equivalent to $\bZ^2$, namely a $(\varphi,L^0)$-integrable measure equivalence coupling from $\bZ^2$ to $\bZ$.
  
  Recall that $I_\bZ(x)\simeq x$ and $I_{\bZ^2}(x)\simeq \sqrt{x}$. 
  Applying \cref{Th:ProfiletOE} to the latter coupling gives that the integrability $\varphi$ cannot be better than~$L^{1/2}$. 
\end{Ex}

\begin{figure}[htbp]
    \centering
    \includegraphics[width=0.75\textwidth]{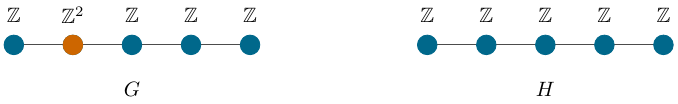}
    \caption{Graph products of \cref{Ex:ExampleUntransvectable}}
    \label{fig:ExampleUntransvectable}
\end{figure}

The next section contains the final tools needed to prove the theorem.

\subsection{Integrability at the vertex level} 
The strategy of the proof of Theorem~\ref{Th:ProdToFactWithVRP} is the following.
In the three cases we will have, inside the coupling $\Omega$, a subcoupling between a vertex group of $G$ and a vertex group of $H$. It will be obtained using the results from \cref{Part:MEClassification} and \cref{lmm:couplage-sommets-groupoides} below.

In the first case, we will be able to obtain such a coupling for all untransvectable vertices using \cref{Sec:AmenalbeUntransvectable}. 
When working with the second case, namely the Vertex Recognition Property, we use \cref{Cor:VRPandGraphMorphism} instead. And the third case relies on \cref{prop:coupling-restricted}. The final tool will be \cref{Lmm:ProdToParabolicsQt}, which provides a way to obtain the wanted integrability of the subcouplings between vertex groups. 

We first start by rewriting \cref{lemma:me-subgroups} in the graph product framework and with the groupoid theoretic vocabulary, as explained in Remark~\ref{Rq:CouplageRestriction}. We refer to \cref{sec:background-groupoids} for background on measured groupoids.

\begin{lemma}[Restating \cref{lemma:me-subgroups}]\label{lmm:couplage-sommets-groupoides} 
Let $G$ and $H$ be two graph products of countably infinite vertex groups, defined over finite simple graphs $\Gamma_G,\Gamma_H$. 
    Let $(\Omega,m,X_G,X_H)$ be a measure equivalence coupling from $G$ to $H$, chosen so that $X_G\cap X_H$ has positive measure. Let $\calg$ be the associated measured groupoid over $X_G\cap X_H$ and $\rho_G$ (resp.\ $\rho_H$) the corresponding cocycles. Let $v\in V\Gamma_G$ and $w\in V\Gamma_H$. Assume that there exists a positive measure Borel subset $U\subseteq X_G\cap X_H$ such that \[\rho_G^{-1}(G_v)_{|U}=\rho_H^{-1}(H_w)_{|U}.\]

    Then there exist a $(G_v\times H_w)$-invariant Borel subspace $\Omega'\subseteq\Omega$, and Borel fundamental domains $X_v,X_w$ for the actions of $G_v,H_w$ on $\Omega'$, such that
    \begin{enumerate}
        \item $(\Omega',m,X_v,X_w)$ is a measure equivalence coupling from $G_v$ to $H_w$, and
        \item $X_v,X_w$ are contained in fundamental domains for the respective actions of $G,H$ on $\Omega$.
    \end{enumerate}
\end{lemma}

The following lemma, contrarily to the previous one, relies on the graph product structure since it makes use of the normal form.

\begin{Lmm}\label{Lmm:ProdToParabolicsQt}
   Let $G$ and $H$ be two graph products over finite simple graphs $\Gamma_G,\Gamma_H$, and let $P_G\subseteq G$ and $P_H \subseteq H$ be two respective parabolic subgroups. 
   Let $(\Omega,m,X_G,X_H)$ be a $(\varphi,\psi)$-integrable measure equivalence coupling from $G$ to $H$. If there exist
   \begin{itemize}
       \item a $(P_G\times P_H)$-invariant Borel subset $\Omega^\prime\subseteq\Omega$, such that $(\Omega^\prime,m,X_{P_G},X_{P_H})$ is a measure equivalence coupling from $P_G$ to $P_H$; 
       \item and $X^\prime_G$, $X^\prime_H$ fundamental domains for the $G$- and $H$-actions respectively, such that $X_{P_G}\subseteq X^\prime_G$ and $X_{P_H}\subseteq X^\prime_H$;
   \end{itemize}
   then $P_G$ and $P_H$ are $(\varphi,\psi)$-integrably measure equivalent over $\Omega^\prime$ (for suitable fundamental domains).
\end{Lmm}

Recall that when $(\Omega,m,X_G,X_H)$ is a measure equivalence coupling from $G$ to $H$, we have an induced action of $G$ on $X_H$ and an induced action of $H$ on $X_G$; both actions are denoted by “$\cdot$” (see \cpageref{Nota:ActionDomFond} and \cref{fig:Orbites}, \cpageref{fig:Orbites}). Recall also that given a fundamental domain $X_H$ one can identify $\Omega/H$ with $X_H$ and --~through this identification~-- changing the fundamental domain amounts to conjugating the cocycles and vice versa (see \cref{Rq:CohomologousCocyclesFundamentalDomains}, \cpageref{Rq:CohomologousCocyclesFundamentalDomains}).

\begin{proof}
Up to replacing $\Omega'$ and the fundamental domains by translates, we can (and will) assume without loss of generality that $P_G=G_\Upsilon$ and $P_H=H_\Lambda$ for some induced subgraphs $\Upsilon\subseteq\Gamma_G$ and $\Lambda\subseteq\Gamma_H$.

    We will prove that $\Omega^\prime$ is $(\varphi,L^0)$-integrable (and thus exhibit an appropriate fundamental domain). By symmetry (i.e.\ interchanging the roles of $G$ and $H$), this will be enough to derive that $\Omega^\prime$ is also $(L^0,\psi)$-integrable (again, for an appropriate fundamental domain), and the lemma will follow.

    If $x$ belongs to $X_H$ we will denote by $x^\prime$ the unique element of $X^\prime_H$ in the same $H$-orbit as $x$.
    Denote by $c$ and $c^\prime$ the measure equivalence cocycles associated to the fundamental domains $X_H$ and $X^\prime_H$ respectively, namely
    \begin{align*}
        c:G\times X_H &\rightarrow H,\\
        c^\prime:G\times X^\prime_H &\rightarrow H.
    \end{align*}
    By assumption, $c$ is $\varphi$-integrable and since $X_{P_H}\subseteq X^\prime_H$ we have $c^\prime(P_G\times X_{P_H})\subseteq P_H$. Moreover, since $X^\prime_H$ and $X_H$ are two fundamental domains for the same $H$-action, the above two cocycles are cohomologous. That is to say, there exists $\alpha:X_H\rightarrow H$ such that $c^\prime(g,x^\prime)=\alpha(g\cdot x)c(g,x)\alpha(x)^{-1}$ for all $g\in G$ and a.e. $x\in X_H$. 

    Now let $\vartheta:H\to H$ be defined by letting $\vartheta(h)$ be the element with the smallest word length such that $P_H\vartheta(h)=P_Hh$. Equivalently $\vartheta(h)$ is the unique representative of the coset $P_Hh$ whose head (see \cpageref{Def:Tail} for the definition) contains no element in $P_H$. By definition $\vartheta$ is constant on all right cosets of $P_H$. Now let 
    \begin{equation*}
        X^{\prime \prime}_{H}:=\{ \vartheta\big(\alpha(x)\big) \cdot x \ : \ x\in X_{H}\}.
    \end{equation*}
    As above, if $x$ belongs to $X_H$ we denote by $x^{\prime \prime}$ the element of $X^{\prime \prime}_H$ corresponding to the same $H$-orbit. We denote by $X^{\prime \prime}_{P_H}$ the image of $X_{P_H}$ under the identification $x^\prime \mapsto x^{\prime \prime}$. Let us prove that $X^{\prime \prime}_{P_H}$ is the appropriate fundamental domain.
  
    For all $g\in P_G$ and a.e. $x\in X_H$ we define 
    \begin{equation*}
    \cpp (g,x^{\prime \prime}) 
    = \vartheta \left(\alpha(g\cdot x)\right)
    c(g,x) 
    \vartheta \big(\alpha(x)\big)^{-1}.
    \end{equation*}
    Remark that by definition of $\vartheta$ this cocycle $\cpp$ is $P_H$-cohomologous to $c^\prime$ (through the identification $x^\prime \mapsto x^{\prime \prime}$). In particular, its restriction to $P_G\times X^{\prime\prime}_{P_H}$ takes its values in $P_H$ and also defines a measure equivalence cocycle on $\Omega^\prime$ (namely, the measure equivalence cocycle associated to the fundamental domain $X^{\prime \prime}_{P_H}$). 
    
    Now, for all $v\in V\Gamma_H$ let $S_{H_v}$ be a finite generating set of $H_v$ and recall that $S_H=\cup_{v\in V\Gamma_H}S_{H_v}$ generates $H$. We also let $S_{P_H}:=\cup_{v\in V\Lambda}S_{H_v}$, where $\Lambda\subseteq\Gamma_H$ is the type of $P_H$ -- notice that $S_{P_H}$ is a finite generating set of $P_H$. Our definition of $\vartheta$, together with \cref{Lmm:LengthAndNormalForm}, \cpageref{Lmm:LengthAndNormalForm} imply that for a.e.\ $x^{\prime\prime} \in X^{\prime\prime}_{P_H}$ we have 
    \begin{equation*}
        \left\vert \cpp (g,x^{\prime \prime})\right\vert_{S_{P_H}}
        =\left\vert \cpp (g,x^{\prime \prime})\right\vert_{S_H}
        \leq |c(g,x)|_{S_H}.
    \end{equation*}
    Hence $\cpp$ is a $\varphi$-integrable cocycle since $c$ is itself $\varphi$-integrable, as wanted.
\end{proof}

\subsection{Proof of the main theorem}\label{subsec:ProofQtTh}
We now turn to the proof of \cref{Th:ProdToFactWithVRP}.

\begin{proof}[Proof of \cref{Th:ProdToFactWithVRP}]
  Let $(\Omega,m,X_G,X_H)$ be a $(\varphi,\psi)$-integrable measure equivalence coupling from $G$ to $H$, and let $c_1:G\times X_H\to H$ and $c_2:H\times X_G \rightarrow G$ be the two associated measure equivalence cocycles. Up to replacing $X_H$ by $hX_H$ for some $h\in H$, we can (and will) assume that $\mu(X_G\cap X_H)>0$ --~notice indeed that this changes the measure equivalence cocycle from $G$ to $H$ by a conjugation by $h$, and therefore does not affect its integrability. Let $X=X_G\cap X_H$.
    
  Let $\calg$ be the measured groupoid over $X$ given by the restriction of the $G$- and $H$-actions on $X_H$ and $X_G$ respectively (see \cref{Ex:GroupoidsAssociatedToCouplings}, \cpageref{Ex:GroupoidsAssociatedToCouplings}). It is naturally equipped with two cocycles $\rho_G:\calg\to G$ and $\rho_H:\calg\to H$, which are action-like (Lemma~\ref{lemma:action-like}, \cpageref{lemma:action-like}). We now distinguish three different cases corresponding to the three different sets of hypotheses in our theorem.
  
  \begin{description}
    \item[Case 1] If for all $v\in V\Gamma_G$ and all $w\in V\Gamma_H$ the vertex groups $G_v$ and $H_w$ are amenable then \cref{prop:untransvectable}, \cpageref{prop:untransvectable} ensures that for every untransvectable vertex $v\in V\Gamma_G$ there exist $\sigma(v)\in V\Gamma_H$, $h\in H$ and a Borel subset $U\subseteq X_G\cap X_H$ of positive measure such that
    \begin{equation*}
    \rho^{-1}_G(G_v)_{|U}=\rho^{-1}_H(hH_{\sigma(v)}h^{-1})_{|U}.
    \end{equation*} 
    Considering now $H^\prime$ the graph product over $\Gamma_H$ with vertex groups $\{h H_v h^{-1}\}_{v\in V\Gamma_H}$, we can apply \cref{lmm:couplage-sommets-groupoides} to $G$ and $H^\prime$ and obtain a measure equivalence coupling~$\Omega^\prime$ from $G_v$ to $hH_{\sigma(v)}h^{-1}$ verifying the hypotheses of \cref{Lmm:ProdToParabolicsQt}. The latter lemma gives the $(\varphi,\psi)$-integrability of~$\Omega^\prime$, for well-chosen fundamental domains.
    \item[Case 2] If $G$ and $H$ both belong to a class of groups $\classC$ having the Vertex Recognition Property then the proof of \cref{Cor:VRPandGraphMorphism}, \cpageref{Cor:VRPandGraphMorphism} ensures that there exists a graph homomorphism $\sigma:\Gamma_G\rightarrow \Gamma_H$ such that \emph{for all} $v\in V\Gamma_G$, the groups $G_v$ and $H_{\sigma(v)}$ verify the hypothesis of \cref{lmm:couplage-sommets-groupoides}. And therefore, as above there exists a $(\varphi,\psi)$-integrable coupling from $G_v$ to $H_{\sigma(v)}$.
    \item[Case 3] If we assume that $\Gamma_G$ and $\Gamma_H$ are transvection-free graphs with no partial conjugation, then \cref{prop:coupling-restricted,rk:requirement}
    ensure that there exist a graph isomorphism $\sigma:\Gamma_G\rightarrow \Gamma_H$, and for all $v\in V\Gamma_G$ a Borel subset $\Omega_v\subseteq \Omega$ which is a measure equivalence coupling from $G_v$ to $H_{\sigma(v)}$ and satisfy the hypotheses of \cref{Lmm:ProdToParabolicsQt}. The latter result then gives the desired integrability.\qedhere
  \end{description}
\end{proof}   
 
\section{The inverse problem for right-angled Artin groups}
\label{Sec:InverseProblem}
The goal of this section is to prove \cref{Th:PrescribedCouplingRAAG}, \cpageref{Th:PrescribedCouplingRAAG} which answers the inverse problem for right-angled Artin groups, namely: given a right-angled Artin group $H$ and a quantification $\varphi$, can we find a group $G$ and a measure equivalence coupling which is $(\varphi,L^0)$-integrable from $G$ to $H$? 

As discussed in the introduction (see \cref{Sec:OptimIPIntro}, \cpageref{Sec:OptimIPIntro}), we also ask this coupling to be optimal (or close to be). The optimality here comes from the \emph{geometry} of the involved groups, more precisely from the obstruction given by the isoperimetric profile of the vertex groups (see \cref{sec:ApplicationOfProdToFact}).

\subsection{Main theorem}
The following restates Theorem~\ref{Th:PrescribedCouplingRAAG} from the introduction.

\begin{Th6}
    Let $\Gamma$ be a finite simple graph that contains an untransvectable vertex 
    and let $\rho$ be a non-decreasing map such that $x\mapsto x/\rho(x)$ is non-decreasing. If $H$ is a right-angled Artin group over $\Gamma$ then there exists a group $G$ and an orbit equivalence pairing from $G$ to $H$ which is a $(\varphi_\varepsilon,L^0)$-integrable for all $\varepsilon>0$, where 
    \begin{equation*}
        \varphi_{\varepsilon}(x):=\frac{\rho\circ \log(x)}{\left( \log\circ\rho\circ\log(x) \right)^{1+\varepsilon}},
    \end{equation*}
    and such that this coupling is optimal up to a log factor in the following sense: if there exists a $(\varphi,L^0)$-integrable measure equivalence coupling from $G$ to $H$, then $\varphi\preccurlyeq \rho \circ \log$.
\end{Th6}

\paragraph{Discussion on the hypotheses}

First remark that the hypotheses on $\rho$ lead us to deal with quantifications having sub-logarithmic growth. On the other hand, Huang and the second-named author showed a superrigidity phenomenon under $(L^1,L^0)$-measure equivalence \cite[Theorem 1]{HH-L1}. More precisely, they showed that if $G$ is a right-angled Artin group with finite outer automorphism group and $H$ is any countable group with bounded torsion such that there exists an $(L^1,L^0)$-measure equivalence coupling from $H$ to $G$, then $H$ is finitely generated and quasi-isometric to $G$. In some cases this implies that $G$ and $H$ are commensurable \cite{Hua-QI}. We thus have here two completely different behaviors: a lot of flexibility for sub-logarithmic quantification versus superrigidity under $L^1$-integrability.

Second, we mention that one could extend the above theorem when $H$ is a graph product over $\Gamma$ of non-trivial finitely generated free abelian groups up to adding some assumptions on the map $\varphi_\varepsilon$ (concavity) as in \cite[Corollary 1.8]{Esc1}. 

On the other hand, we do not know whether the theorem holds when $H$ is a non-abelian free group, and more generally we cannot treat the case where $H$ splits as a graph product of non-abelian free groups. 

\subsection{\texorpdfstring{Proof of \cref{Th:PrescribedCouplingRAAG}}{Proof of the theorem on prescribed couplings}} 
The proof relies on two main tools: the resolution of the above inverse problem for $H=\bZ$ (or equivalently, when $\Gamma$ is reduced to one vertex) already established in \cite{Esc1}, and the results specific to graph products developed in the present article, namely \cref{Prop:CouplageGraphProd,Th:ProdToFactWithVRP} (\cpageref{Prop:CouplageGraphProd,Th:ProdToFactWithVRP}).

We start by recalling the known results for the case of $\Gamma$ reduced to one vertex and then turn to our proof of the above \cref{Th:PrescribedCouplingRAAG}.

\paragraph{Background on the inverse problem}
When $H=\bZ$ (or equivalently, when $\Gamma$ is reduced to one vertex) building couplings with prescribed integrability with $H$ has been done by the first named author in \cite[Theorem~1.7]{Esc1}, which we recall below. The optimality is obtained in this case using the obstruction given in \cref{Th:ProfiletOE}, which relies on the isoperimetric profile (see \cref{sec:ApplicationOfProdToFact} for the definition of the profile $I_G$).

\begin{Th}[Escalier \cite{Esc1}]\label{Th:CouplageOptimZ}
  For all non-decreasing maps $\rho$ such that $x\mapsto x/\rho(x)$ is non-decreasing, there exists a group $\Delta$ such that
  \begin{itemize}
  \item $I_{\Delta}\simeq {\rho}\circ \log$;
  \item There exists an orbit equivalence pairing that is $(\varphi_{\varepsilon},\exp \circ \rho)$-integrable from $\Delta$ to $\bZ$ for all 
    $\varepsilon>0$, where $\varphi_{\varepsilon}(x):={\rho\circ
      \log(x)}/{\left( \log\circ\rho\circ\log(x) \right)^{1+\varepsilon}}$.
  \end{itemize}
\end{Th}

We mention that the group $\Delta$ given by the above theorem, is a diagonal product of lamplighter groups as defined by Brieussel-Zheng in \cite{BZ}.

\begin{Claim}\label{Claim:OptimaliteCouplageZd} 
    The coupling given by the above theorem is optimal up to a log factor in the sense of \cref{Th:PrescribedCouplingRAAG}, namely: for every non-decreasing map $\varphi$ if there exists a $(\varphi,L^0)$-integrable measure equivalence coupling from $\Delta$ to $\bZ$ then $\varphi \preccurlyeq I_{\Delta}\simeq \rho\circ\log$.
\end{Claim}

\begin{proof}[Proof of the Claim]
    Let $\Delta$ be the group given by the above theorem. Assume that there exists a $(\varphi,L^0)$-integrable measure equivalence coupling from $\Delta$ to $H$, for some non-decreasing map $\varphi$ and let us prove that $\varphi \preccurlyeq I_{\Delta}\simeq \rho\circ\log$.
    \medskip

    \noindent$\bullet$ Let us first prove that $x\mapsto x/\varphi(x)$ is necessarily non-decreasing.\\
    So assume towards a contradiction that $x\mapsto x/\varphi(x)$ is decreasing. Then $\mathrm{id}\preccurlyeq \varphi$ and thus any $\varphi$-integrable cocycle is at least $L^1$-integrable. Therefore the coupling is also $(L^1,L^0)$-integrable from $\Delta$ to $H$. Thus, recalling that $I_{H}(x)\simeq x$ and $I_{\Delta}\simeq \rho \circ \log$ and using \cref{Th:ProfiletOE}, we get that in this case 
    \begin{equation*}
         x \simeq I_{H}(x) \preccurlyeq I_{\Delta}(x) \simeq \rho \circ \log(x)
    \end{equation*}
    But the map $x\mapsto x/\rho(x)$ is non-decreasing, which contradicts the above inequality. Hence the map $x\mapsto x/\varphi(x)$ has to be non-decreasing.
    \medskip
    
    \noindent$\bullet$ Since $t\mapsto t/\varphi(t)$ is non-decreasing, we can now apply \cref{Th:ProfiletOE} and obtain that $\varphi \circ I_{H}\preccurlyeq I_{\Delta}$. 
    Since $I_{H}\simeq \mathrm{id}$, the last inequality becomes $\varphi \preccurlyeq I_{\Delta}\simeq \rho\circ\log$. As wanted.
\end{proof}

\paragraph{The case of graph products} We now turn to the proof of the general case.
The existence of the group $G$ and the pairing are given by \cref{Th:CouplageOptimZ,Prop:CouplageGraphProd}. The optimality is then obtained using \cref{Claim:OptimaliteCouplageZd,Th:ProdToFactWithVRP}.

\begin{proof}[Proof of \cref{Th:PrescribedCouplingRAAG}]
  Let $\Gamma$ be a finite simple graph that contains an untransvectable vertex, denoted by $v$, and $H$ be right-angled Artin group over $\Gamma$. 
  
  \begin{description}
  \item[Construction of the coupling] Let $\rho$ and for all $\varepsilon>0$ let $\varphi_\varepsilon$ be as in the theorem. Then by \cref{Th:CouplageOptimZ} there exists a group $\Delta$ admitting a $(\varphi_\varepsilon,L^0)$-integrable orbit equivalence pairing from $\Delta$ to $\bZ$. 
  Denote by $G$ the graph product over $\Gamma$, with vertex groups $G_v=\Delta$ and $G_w=H_w$ for all $w\neq v$ in $V\Gamma$. Then by \cref{Prop:CouplageGraphProd}, \cpageref{Prop:CouplageGraphProd} there exists a $(\varphi_\varepsilon ,L^0)$-integrable orbit equivalence pairing from $G$ to $H$.
  \item[Optimality] Let $\varphi$ be a non-decreasing map and assume that we have a $(\varphi,L^0)$-integrable measure equivalence coupling from $G$ to $H$. By the first point of  \cref{Th:ProdToFactWithVRP}, \cpageref{Th:ProdToFactWithVRP} there exists $\sigma(v)\in V\Gamma$ and a $(\varphi,L^0)$-integrable measure equivalence coupling from $G_v=\Delta$ to $H_{\sigma(v)}\simeq \bZ$.
  So applying \cref{Claim:OptimaliteCouplageZd} to the latter coupling leads $\varphi \preccurlyeq \rho\circ\log$. Therefore, the $(\varphi_\varepsilon,L^0)$-integrable coupling constructed in the first part of the proof is optimal up to a logarithmic error, as wanted.\qedhere
  \end{description}
\end{proof}

\newpage
\part{Rigidity under ergodicity assumptions}
\label{Part:Superrigidity}
The goal of this part is to prove orbit equivalence rigidity theorems for actions of graph products that satisfy certain strong ergodicity assumptions. Our main theorems are Theorem~\ref{theo:strong-rigidity} (giving a rigidity statement among actions of graph products) and Theorem~\ref{theo:superrigidity} (giving a superrigidity statament, comparing two actions $G\actson X$ and $H\actson Y$ where only $G$ is a graph product, while $H$ is an arbitrary countable group). This is in the vein of the work of Monod and Shalom for direct products \cite{MS}, and generalizes the case of right-angled Artin groups \cite{HHI}. 

\section{Rigidity among graph products}\label{sec:rigidity-ergodicity}

\subsection{Some terminology}

We start by recalling some terminology that will be used throughout the section.

\paragraph*{Ergodicity assumptions.}\label{par:ergodicity}

Recall that a measure-preserving action of a countable group $G$ on a standard probability space $(X,\mu)$ is 
\begin{itemize}
\item \emph{weakly mixing}\index{Weakly mixing}\index{Mixing!Weakly mixing} if the diagonal action of $G$ on $X\times X$ is ergodic. 
\item \emph{totally ergodic}\index{Totally ergodic} if every infinite subgroup of $G$ acts ergodically on $X$. 
\item \emph{mildly mixing}\index{Mildly mixing}\index{Mixing!Mildly mixing} if it is ergodic and if for any measurable subset $U\subseteq X$ with $\mu(U),\mu(X\setminus U)>0$, and any sequence of pairwise distinct elements $g_n$ in $G$, one has $\liminf\mu(g_nU\triangle U)>0$. By \cite{SW}, this is equivalent to the following requirement, which is in fact the original definition of mild mixing due to Furstenberg and Weiss \cite{FW}: for every measure-preserving ergodic $G$-action on a standard Borel space $S$ equipped with a non-atomic $\sigma$-finite Borel measure, the diagonal $G$-action on $X\times S$ is ergodic.
\item \emph{strongly mixing} \index{Strongly!Strongly mixing}\index{Mixing!Strongly mixing} if given any two Borel subsets $U,V\subseteq X$, and any sequence $(g_n)_{n\in\mathbb{N}}\in G^{\mathbb{N}}$ of pairwise distinct elements, one has $\lim_{n\to +\infty}\mu(g_nU\cap V)=\mu(U)\mu(V)$.
\end{itemize}
We notice that strongly mixing actions are in particular mildly mixing, that mildly mixing actions of countably infinite groups are weakly mixing, and that weakly mixing actions are ergodic (and totally ergodic actions are also ergodic).

We also mention that if $G$ has a subgroup $H$ which acts weakly mixingly on $X$, then the $G$-action on $X$ is weakly mixing. This is however not true for the total ergodicity, mild mixing and strong mixing conditions.

\paragraph*{(Virtual) conjugation of actions.} 
Two measure-preserving actions $G\actson (X,\mu)$ and $H\actson (Y,\nu)$ on standard probability spaces are \emph{conjugate} if there exist a group isomorphism $\theta:G\to H$ and a measure space isomorphism $f:X\to Y$ such that $f(g \cdot x)=\theta(g)f(x)$ for every $g\in G$ and almost every $x\in X$.

There exists a weaker notion of \emph{virtually conjugate} actions, which takes into account the possibility of passing to finite-index subgroups and taking the quotient by a finite normal subgroup. We now recall the definition, see e.g.\ \cite[Definition~1.3]{Kid-oe}.

First, given a countable group $G$, a finite-index subgroup $H\subseteq G$, and a measure-preserving action of $H$ on a standard probability space $X$, there is a notion of an action of $G$ on a standard finite measure space $\Ind_H^G(X)$, induced from that of $H$, defined as follows. Let $\theta:H\to G$ be the inclusion map. The space $X\times G$ (where $G$ is equipped with the counting measure) is equipped with a measure-preserving action of $H\times G$ given by $(h,g)\cdot (x,g')=(hx,\theta(h)g'g^{-1})$. We then let $\Ind_H^G(X):= H\backslash (X\times G)$: this is naturally a standard finite measure space (isomorphic to $X\times G/H$), which we renormalize to a probability space. The $G$-action on $X\times G$ descends to a probability measure-preserving $G$-action on $\Ind_H^G(X)$.

Notice for future use that, following the notation from \cite[Notation~4.3]{Vae}, we can more generally define the induced action $G\actson \Ind_H^G(X,\theta)$ whenever $\theta:H\to G$ is an injective homomorphism with finite-index image. In the particular case where $\theta$ is an isomorphism, then the $H$-action on $X$ and the $G$-action on $\Ind_H^G(X,\theta)$ are conjugate (via the isomorphism $\theta:H\to G$ and the measure space isomorphism $x\mapsto [x,e]$, where $[x,e]$ denotes the class of $(x,e)$).

Notice also that, if a countable group $G$ has a measure-preserving action on a standard probability space $(X,\mu)$, and if $F_G\unlhd G$ is a finite normal subgroup, then there is a quotient measure-preserving action of $G/F_G\actson (X,\mu)/F_G$, where $(X,\mu)/F_G$ is the space of $F_G$-orbits in $X$, renormalized to a probability measure.

Two measure-preserving actions $G\actson (X,\mu)$ and $H\actson (Y,\nu)$ are \emph{virtually conjugate} if we can find 
\begin{itemize}
\item short exact sequences 
\[1\to F_G\to G\to G_1\to 1 \quad \text{and}\quad 1\to F_H\to H\to H_1 \to 1,\]
where $F_G\unlhd G$ and $F_H\unlhd H$ are finite;
\item finite-index subgroups $G_2\subseteq G_1$ and $H_2\subseteq H_1$;
\item and conjugate measure-preserving actions $G_2\actson (X_2,\mu_2)$ and $H_2\actson (Y_2,\nu_2)$, such that 
\begin{itemize}
    \item the action $G_1\actson (X,\mu)/F_G$ is conjugate to $G_1\actson \Ind_{G_2}^{G_1}(X_2)$, and \smallskip
    \item the action $H_1\actson (X,\mu)/F_H$ is conjugate to $H_1\actson \Ind_{H_2}^{H_1}(X_2)$.
\end{itemize}
\end{itemize}

\subsection{Rigidity of actions within the class of graph products}\label{Sec:StrongRigidity}

We make the following definition.

\begin{Def}
Let $G$ be a graph product of countable groups over a finite simple graph. A free, probability measure preserving action of $G$ on a standard probability space $X$ is \emph{vertexwise ergodic}\index{Vertexwise!Vertexwise ergodic} (resp.\ \emph{vertexwise weakly mixing}\index{Vertexwise!Vertexwise weakly mixing}\index{Weakly mixing!Vertexwise weakly mixing}\index{Mixing!Vertexwise weakly mixing}, resp. \emph{vertexwise totally ergodic}\index{Vertexwise!Vertexwise totally ergodic}) if for every vertex group $G_v$, the action of $G_v$ on $X$ is ergodic (resp.\ weakly mixing, resp.\ totally ergodic).
\end{Def}

When $G$ is a graph product with countably infinite vertex groups, notice that vertexwise weak mixing and vertexwise total ergodicity are satisfied for all strongly mixing actions of $G$ on $X$. Besides Bernoulli actions, this includes for example actions coming from embedding $G$ as a closed subgroup of $\mathrm{SL}(n,\mathbb{R})$ with non-compact image, and using the Howe--Moore property --~see \cite[Theorem~3.2]{HW} for examples of faithful representations of graph products into $\mathrm{SL}(n,\mathbb{Z})$. See more generally \cite[Corollary~B]{BdlN} for constructions of faithful representations of graph products into $\mathrm{GL}(n,\mathbb{C})$, assuming the linearity over $\mathbb{C}$ of the vertex groups.

Recall that a group $G$ has the \emph{unique root property}\index{Unique root property} if for every $g_1,g_2\in G$ and $n\ge 1$, if $g_1^n=g_2^n$, then $g_1=g_2$. In this section, we establish the following theorem. We refer to Section~\ref{sec:me-background-1} for the notions of stable orbit equivalence and compression constants.

\begin{Th}\label{theo:strong-rigidity}
  Let $\Gamma_G,\Gamma_H$ be two connected finite simple graphs, with no vertex joined to every other vertex. Let $G,H$ be graph products with countably infinite torsion-free vertex groups over $\Gamma_G,\Gamma_H$, respectively. Let $G\actson X$ and $H\actson Y$ be two free, vertexwise ergodic, measure-preserving actions on standard probability spaces. Assume that either
  \begin{itemize}
      \item $G\actson X$ is vertexwise weakly mixing, or
      \item $G\actson X$ is vertexwise totally ergodic, or
      \item $H$ has the unique root property.
  \end{itemize}
 If there is a stable orbit equivalence $f$ from $G\actson X$ to $H\actson Y$ with compression constant $\kappa(f)\ge 1$, then the two actions are conjugate (and every SOE cocycle $c:G\times X\to H$ associated to $f$ is cohomologous to a group isomorphism $G\to H$).
 \end{Th}  
 
 We mention that the last two cases will be useful in the next section to obtain superrigidity statements (\cref{Sec:Superrigidity}).
 
\begin{Rq}[Discussion on the connectedness hypothesis]
   The assumption that the defining graph $\Gamma_G$ is connected cannot be removed. In fact, Bernoulli actions of non-abelian free groups fail to satisfy the form of rigidity provided by Theorems~\ref{theo:strong-rigidity}, in view of works of Bowen \cite{Bow2,Bow3}.
\end{Rq}

\begin{Rq}[Examples of groups with the unique root property]
Assuming that $H$ satisfies the unique root property is equivalent to assuming that each of its vertex groups satisfies the unique root property \cite[Theorem~6.5]{BeFe}. The class of groups that satisfy this property is vast. It includes, among others, all torsion-free nilpotent groups \cite{Mal}, and more generally all bi-orderable groups (including all Baumslag--Solitar groups $\BS(1,n)$ with $n\ge 1$, Thompson's group or more generally diagram groups \cite{GS}), all torsion-free hyperbolic groups (using the fact that the centralizer of every non-trivial element is cyclic \cite[Corollary~III.$\Gamma$.3.10]{BH}) and more generally toral relatively hyperbolic groups (even more generally, torsion-free groups whose maximal abelian subgroups are malnormal, called \emph{CSA} groups, which also include limit groups \cite[Proposition~8]{MR}), and some torsion-free finite-index subgroups of surface mapping class groups \cite{BP} and $\Out(F_N)$ \cite{Gue}. The unique root property is also stable under certain group-theoretic operations like taking wreath products \cite{Bau}.
\end{Rq}

\paragraph*{Proof strategy and outline of the section.} Our proof of Theorem~\ref{theo:strong-rigidity} follows the strategy used by Huang, Ioana and the second-named author in \cite{HHI}. In particular we will follow very closely the presentation of \cite[Propositions~6.1 and~6.2]{HHI}. The idea is the following. Let $c$ be a stable orbit equivalence cocycle associated to $f$. 
\begin{itemize}
\item We first untwist $c$ (i.e.\ change it within its cohomology class) to a group homomorphism on some vertex group of $G$. The main tools are our recognition statements from Section~\ref{sec:strongly-reduced}, ensuring in particular that the orbit equivalence relation “remembers” the maximal product parabolic subgroups, together with the work of Monod--Shalom on direct products \cite{MS}.
\item Once this is done, a propagation argument enables to untwist the cocycle on the whole group $G$. We give three versions of this propagation argument: one uses weak mixing at the level of the vertex groups, one uses total ergodicity, and the third relies on the unique root property for $H$. 
\end{itemize}
The section is organized as follows. After presenting the propagation arguments in Section~\ref{sec:propagation}, we review and establish a few extra tools in Sections~\ref{sec:more-tools} to~\ref{sec:creg}, and complete our proof of Theorem~\ref{theo:strong-rigidity} in Section~\ref{sec:proof-strong-rigidity}.

\subsubsection{Three propagation arguments}\label{sec:propagation}

\paragraph*{Propagation via weak mixing.}

The idea of exploiting weak mixing to propagate rigidity from a normal subgroup to an ambient group arose in the work of Popa, see \cite[Lemma~3.6]{Pop}, which finds its roots in \cite[Lemma~5.7]{Pop3}. Popa's arguments were formulated in the language of von Neumann algebras; we will use Furman's formulation \cite[Lemma~3.5]{Fur-Popa} whose statement is more ergodic-theoretic. 

\begin{lemma}\label{lemma:propagation-weak-mixing}
    Let $G,H$ be countable groups. Let $G\actson X$ be a measure-preserving action on a standard probability space $X$, and $c:G\times X\to H$ be a cocycle. Let $G_1,\dots,G_k$ be subgroups of $G$ satisfying the following three assumptions: 
    \begin{itemize}
    \item the set $G_1\cup\dots\cup G_k$ generates $G$;
    \item for every $j\in\{1,\dots,k-1\}$, the subgroup $G_{j+1}$ normalizes $G_j$;
    \item for every $j\in\{1,\dots,k-1\}$, the subgroup $G_j$ acts weakly mixingly on $X$.
    \end{itemize}
    Assume that there exists a group homomorphism $\theta_1:G_1\to H$ such that for every $g\in G_1$ and almost every $x\in X$, one has $c(g,x)=\theta_1(g)$.

Then there exists a group homomorphism $\theta:G\to H$ such that for every $g\in G$ and almost every $x\in X$, one has $c(g,x)=\theta(g)$. 
\end{lemma}

\begin{proof}
We first prove by induction that for every $j\in\{1,\dots,k\}$, there exists a group homomorphism $\theta_j:G_j\to H$ such that for every $g\in G_j$ and almost every $x\in X$, one has $c(g,x)=\theta_j(g)$. This holds when $j=1$ by assumption. And assuming this holds for $G_j$, then \cite[Lemma~3.5]{Fur-Popa}, applied to the subgroup $G_j$ of $G$, with $Y$ reduced to a point and $\rho=\theta_j$, ensures that it also holds for $G_{j+1}$.

Now, as $G_1\cup\dots\cup G_k$ generates $G$, using the cocycle relation, we deduce that for every $g\in G$, the map $c(g,\cdot)$ is almost everywhere constant. Since $G$ is countable, there exists a conull Borel subset $X^*\subseteq X$ such that for every $g\in G$, the map $c(g,\cdot)_{|X^*}$ is constant. And the cocycle relation ensures that the map $\theta$ sending $g$ to the constant value of $c(g,\cdot)_{|X^*}$ is a group homomorphism, as desired.
\end{proof}

\paragraph*{Propagation via ergodicity and unique roots.}
 The following lemma exploits the unique root property to propagate rigidity. It is a very slight variation over \cite[Lemma~4.2]{HHI}.

\begin{lemma}\label{lemma:propagation-unique-root}
    Let $G,H$ be countable groups, and assume that $H$ satisfies the unique root property. Let $G\actson X$ be a measure-preserving action on a standard probability space $X$, and $c:G\times X\to H$ be a cocycle. Let $G_1,\dots,G_k$ be subgroups of $G$ satisfying the following three assumptions: 
    \begin{itemize}
    \item the set $G_1\cup\dots\cup G_k$ generates $G$;
    \item for every $j\in\{1,\dots,k-1\}$, the subgroup $G_{j+1}$ centralizes $G_j$;
    \item for every $j\in\{1,\dots,k-1\}$, the subgroup $G_j$ acts ergodically on $X$.
    \end{itemize}
    Assume that there exists a group homomorphism $\theta_1:G_1\to H$ such that for every $g\in G_1$ and almost every $x\in X$, one has $c(g,x)=\theta_1(g)$.

Then there exists a group homomorphism $\theta:G\to H$ such that for every $g\in G$ and almost every $x\in X$, one has $c(g,x)=\theta(g)$. 
\end{lemma}

\begin{proof}
As in the previous proof, it is enough to prove by induction that for every $j\in\{1,\dots,k\}$, there exists a group homomorphism $\theta_j:G_j\to H$ such that for every $g\in G_j$ and almost every $x\in X$, one has $c(g,x)=\theta_j(g)$. In fact, it is enough to prove that for every $g\in G_j$, the map $c(g,\cdot)$ is essentially constant: indeed, the map sending $g$ to the essential value of $c(g,\cdot)$ is then automatically a group homomorphism, in view of the cocycle relation. 

This holds for $j=1$ by assumption. Assume that it holds for some $j\in\{1,\dots,k-1\}$, and let $g\in G_{j+1}$. Let $X=\dunion_{i\in I} X_i$ be a countable Borel partition such that for every $i\in I$, the map $c(g,\cdot)_{|X_i}$ is constant, with value denoted by $h_i\in H$ (see also \cref{fig:Propagation1}). We aim to prove that $h_{i_1}=h_{i_2}$ for any $i_1,i_2\in I$.
\begin{figure}[htbp]
    \centering
    \includegraphics{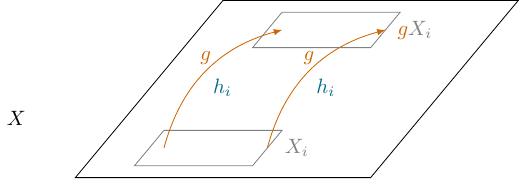}
    \caption{On $X_i$ the value of $c(g,\cdot)$ is constant, equal to $h_i$.}
    \label{fig:Propagation1}
\end{figure}

So let $i_1,i_2\in I$. Since $G_j$ acts ergodically on $X$, there exists $g'\in G_j$ such that $g'X_{i_1}\cap X_{i_2}=:V$ has positive measure. Let also $U:=(g')^{-1}(V)$ (see also \cref{fig:Propagation2}). Since $g\in G_{j+1}$ and $g'\in G_j$, we have $gg'=g'g$. So for every $x\in U$, we have $gg'\cdot x=g'g \cdot x$. Applying the cocycle $c$ gives
\[c(g,g'\cdot x)c(g',x)=c(g',g \cdot x)c(g,x),\] 
in other words 
\begin{equation}\label{eq:propagation}
h_{i_2}\theta_j(g')=\theta_j(g')h_{i_1}.
\end{equation}
On the other hand, using Poincaré recurrence, there also exists $k>0$ such that the subset $X_{i_1}\cap (g')^kX_{i_1}=:Z$ has positive measure. Let $W:=(g')^{-k}(Z)$ (see also \cref{fig:Propagation3}). Then for every $x\in W$, we have $g(g')^k\cdot x=(g')^kg \cdot x$, so applying the cocycle $c$ gives
\[c\big(g,(g')^k\cdot x\big)c\big((g')^k,x\big)=c\big((g')^k,g \cdot x\big)c(g,x),\] 
in other words
\begin{equation}\label{eq:hthetaj}
    h_{i_1}\theta_j(g')^k=\theta_j(g')^kh_{i_1}.
\end{equation}
This rewrites as $\left(h_{i_1}\theta_j(g')h_{i_1}^{-1}\right)^k=\theta_j(g')^k$. Using the unique root property for $H$, we deduce that $h_{i_1}$ and $\theta_j(g')$ commute. Therefore, Equation~\eqref{eq:propagation} yields $h_{i_2}=h_{i_1}$, as desired.  
\end{proof}

\begin{figure}[htbp]
    \centering
    \includegraphics[width=\textwidth]{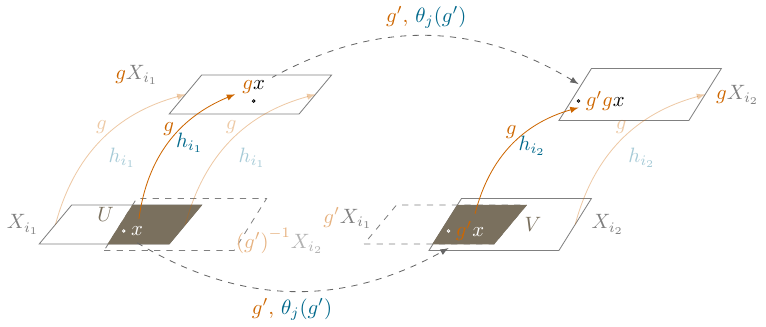}
    
    We represent here how we obtain \cref{eq:propagation}. In \textcolor{Orange}{orange} are represented the actions of elements in $G$, while in \textcolor{Turquoise}{blue} are written the corresponding values of the cocycle $c$. The sets $U$ and $V$ are drawn in \textcolor{MFCB}{brown}.\\
    \caption{Illustration for \cref{eq:propagation}}
    \label{fig:Propagation2}
\end{figure}

\begin{figure}[htbp]
    \centering
    \includegraphics[width=\textwidth]{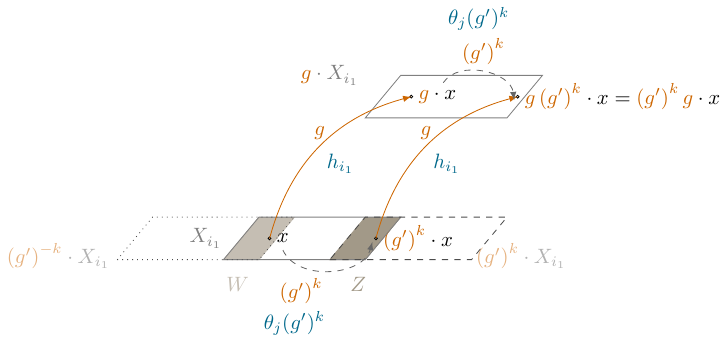}
    \caption{Illustration for \cref{eq:hthetaj}}
    \label{fig:Propagation3}
\end{figure}

\paragraph*{Propagation via total ergodicity.}
The following lemma uses total ergodicity to propagate rigidity. The argument is a slight variation over the previous one, but does not necessit to assume that $H$ satisfies the unique root property.

\begin{lemma}\label{lemma:propagation-totally-ergodic}
    Let $G,H$ be countable groups, and assume that $G$ is torsion-free. Let $G\actson X$ be a measure-preserving action on a standard probability space $X$, and let $c:G\times X\to H$ be a cocycle. Let $G_1,\dots,G_k$ be subgroups of $G$ satisfying the following three assumptions: 
    \begin{itemize}
    \item the set $G_1\cup\dots\cup G_k$ generates $G$;
    \item for every $j\in\{1,\dots,k-1\}$, the subgroup $G_{j+1}$ centralizes $G_j$;
    \item for every $j\in\{1,\dots,k-1\}$, the subgroup $G_j$ acts totally ergodically on $X$.
    \end{itemize}
    Assume that there exists a group homomorphism $\theta_1:G_1\to H$ such that for every $g\in G_1$ and almost every $x\in X$, one has $c(g,x)=\theta_1(g)$.

Then there exists a group homomorphism $\theta:G\to H$ such that for every $g\in G$ and almost every $x\in X$, one has $c(g,x)=\theta(g)$. 
\end{lemma}

\begin{proof}
As in the previous proof, it is enough to prove by induction that for every $j\in\{1,\dots,k\}$ and every $g\in G_j$ the map $c(g,\cdot)$ is essentially constant. 

This holds for $j=1$ by assumption. Assume that it holds for some $j\in\{1,\dots,k-1\}$, and let $g\in G_{j+1}$. Let $X=\dunion_{i\in I} X_i$ be a countable Borel partition such that for every $i\in I$, the map $c(g,\cdot)_{|X_i}$ is constant, with value denoted by $h_i\in H$ (as in \cref{fig:Propagation1}, \cpageref{fig:Propagation1}). We aim to prove that $h_{i_1}=h_{i_2}$ for any $i_1,i_2\in I$.
 
For any $i_1,i_2\in I$, as in the previous proof (see discussion above Equation~\eqref{eq:propagation}), the ergodicity of the action of $G_j$ enables us to find $g'\in G_j$ such that
\begin{equation}\label{eq:propagation-2}
h_{i_2}\theta_j(g')=\theta_j(g')h_{i_1}.
\end{equation}
And as in the previous proof (see the discussion above Equation~\eqref{eq:hthetaj}), we can also find $k>0$ such that
\begin{equation}\label{eq:hthetaj-2}
    h_{i_1}\theta_j(g')^k=\theta_j(g')^kh_{i_1}.
\end{equation}
Since $G$ is torsion-free and $G_j$ acts totally ergodically on $X$, in particular $(g')^k$ acts ergodically on $X$. So let $m>0$ be such that $(g')^{km}X_{i_1}\cap (g')^{-1}(X_{i_1})$ has positive measure. In particular, there exists a positive measure Borel subset $U\subseteq X_{i_1}$ such that $(g')^{km+1}(U)\subseteq X_{i_1}$. For almost every $x\in U$, since $gg'=g'g$, we have \[c\left(g,(g')^{km+1}\cdot x\right)c\left((g')^{km+1},x\right)=c\left((g')^{km+1},g\cdot x\right)c(g,x),\] which rewrites as 
\begin{equation}\label{eq:hthetaj-3}
h_{i_1}\theta_j(g')^{km+1}=\theta_j(g')^{km+1}h_{i_1}.
\end{equation}
Equation~\eqref{eq:hthetaj-2} implies that $(h_{i_1}\theta_j(g')h_{i_1}^{-1})^{km}=\theta_j(g')^{km}$, and Equation~\eqref{eq:hthetaj-3} gives
\[(h_{i_1}\theta_j(g')h_{i_1}^{-1})^{km+1}=\theta_j(g')^{km+1}.\] 
Combining these two equations yields $h_{i_1}\theta_j(g')h_{i_1}^{-1}=\theta_j(g')$. In other words $h_{i_1}$ and $\theta_j(g')$ commute. Therefore, Equation~\eqref{eq:propagation-2} yields $h_{i_2}=h_{i_1}$, as desired.  
\end{proof}

\paragraph{Untwisting the cocycle on one vertex group is enough.} 

The following lemma, which is essentially \cite[Lemma~4.3]{HHI}, reduces the proof of Theorem~\ref{theo:strong-rigidity} to untwisting the SOE cocycle to a group homomorphism on one vertex group.

\begin{lemma}\label{lemma:propagation}
     Let $G,H$ be two graph products with countably infinite vertex groups over connected finite simple graphs $\Gamma_G,\Gamma_H$ with trivial clique factors. Let $G\actson X$ and $H\actson Y$ be two free, ergodic, measure-preserving actions on standard probability spaces.
     Assume that either 
     \begin{itemize}
     \item $G\actson X$ is vertexwise weakly mixing or
     \item $G\actson X$ is vertexwise totally ergodic, and $G$ is torsion-free, or 
     \item $G\actson X$ is vertexwise ergodic, and $H$ has the unique root property.
     \end{itemize}
      Assume that there exists a stable orbit equivalence $f$ from $G\actson X$ to $H\actson Y$, with $\kappa(f)\ge 1$, and let $c:G\times X\to H$ be an SOE cocycle associated to $f$.
     
     Assume that $c$ is cohomologous to a cocycle $c'$ for which there exist $v\in V\Gamma_G$ and a group homomorphism $\theta_v:G_v\to H$, such that for every $g\in G_v$ and almost every $x\in X$, one has $c'(g,x)=\theta_v(g)$.
     
     Then the actions $G\actson X$ and $H\actson Y$ are conjugate (and $c$ is cohomologous to an isomorphism).
\end{lemma}

\begin{proof}
The proof is almost identical to that of \cite[Lemma~4.3]{HHI}, we include it for the convenience of the reader.

Since $\Gamma_G$ is connected, we can write $\{G_w\}_{w\in V\Gamma_G}=\{G_1,\dots,G_k\}$ (possibly with repetitions among the $G_j$) in such a way that for every $j\in\{1,\dots,k-1\}$, the group $G_{j+1}$ centralizes $G_j$. We can further assume that $G_1=G_v$, with $v\in V\Gamma_G$ as in the statement of the lemma. It follows that the subgroups $G_1,\dots,G_k$ satisfy all the assumptions from either Lemma~\ref{lemma:propagation-weak-mixing} (if we had assumed that $G \actson X$ is vertexwise weakly mixing) or Lemma~\ref{lemma:propagation-unique-root} (if we had assumed that $G\actson X$ is vertexwise ergodic and $H$ has the unique root property), or Lemma~\ref{lemma:propagation-totally-ergodic} (if we had assumed that $G\actson X$ is vertexwise totally ergodic and $G$ is torsion-free).

Lemma~\ref{lemma:propagation-weak-mixing} (or Lemma~\ref{lemma:propagation-unique-root}, or Lemma~\ref{lemma:propagation-totally-ergodic}) therefore implies that $c$ is cohomologous to a group homomorphism $\theta:G\to H$. As the clique factor of $\Gamma_G$ is trivial, $G$ has no non-trivial finite normal subgroup (see e.g. \cite[Corollary~2.9]{MV}). It thus follows from \cite[Lemma~4.7 and its proof]{Vae} that $\theta$ is injective, $\theta(G)$ has finite index in $H$, the action $H\actson Y$ is conjugate to the $H$-action on $\Ind_G^H(X,\theta)$, and $\kappa(f)=\frac{1}{[H:\theta(G)]}$. As $\kappa(f)\ge 1$, we deduce that $[H:\theta(G)]=1$. In particular $\theta$ is a group isomorphism, so the actions $G\actson X$ and $H\actson Y$ are conjugate, as desired.
\end{proof}

\subsubsection{Combinatorial tools}\label{sec:more-tools}

The following two lemmas are restatements of \cite[Lemmas~1.4 and~1.5]{HHI}. In \cite{HHI}, these two results are formulated for right-angled Artin groups, however their proofs do not use the structure of the vertex groups, and are also valid for graph products. The second case in the following lemma is illustrated in Figure~\ref{fig:penta-commuting-centers}.

\begin{lemma}[{\cite[Lemmas~1.4 and~1.5]{HHI}}]\label{lemma:HHI-1}
   Let $G$ be a graph product over a connected finite simple graph $\Gamma$, such that $\Gamma$ contains no vertex joined to every other vertex. Then 
    \begin{itemize}
        \item either $G$ contains a maximal product parabolic subgroup with trivial clique factor;
        \item or else $G$ contains two maximal product parabolic subgroups $P_1,P_2$ with non-trivial clique factors $C_1,C_2$, with $C_2\subseteq C_1^{\perp}$ and $C_1\subseteq C_2^{\perp}$.
    \end{itemize}
\end{lemma}

\begin{lemma}[{\cite[Lemma~1.5]{HHI}}]\label{lemma:HHI-2}
Let $G$ be a graph product over a finite simple graph, and let $P_1,P_2\subseteq G$ be two distinct maximal product parabolic subgroups, with respective clique factors $C_1,C_2$.
Then $C_1\cap C_2=\{1\}$.
\end{lemma}

\begin{figure}[htbp]
    \centering
    \includegraphics[width=0.75\textwidth]{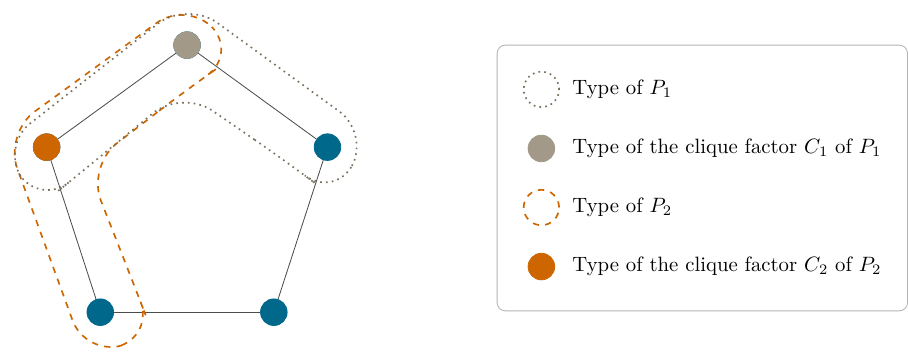}
    \caption{Example of products with commuting clique factors}
    \label{fig:penta-commuting-centers}
\end{figure}

\subsubsection{Groupoid tools: Detecting products with trivial clique factor}

We refer to \cref{de:strongly-reduced} for the notion of a strongly reduced graph. The following lemma is a variation over the results presented in Section~\ref{sec:strongly-reduced} of the present work -- the reader is refered to this section for all relevant definitions, in particular Definition~\ref{de:padm}, \cpageref{de:padm} for Property~$\Padm$. It gives a groupoid theoretic characterization of product parabolic subgroups with non-trivial clique factor.

 \begin{lemma}\label{lemma:recognize-star}
 Let $G$ be a graph product of countably infinite groups over a strongly reduced finite simple graph $\Gamma$. Let $\calg$ be a measured groupoid over a standard probability space $X$, equipped with an action-like cocycle $\rho:\calg\to G$. Let $P$ be a product parabolic subgroup whose type is not a clique. Let $\calp=\rho^{-1}(P)$.

 Then $P$ has a non-trivial clique factor if and only if given any non-empty finite set $\{\calq_1,\dots,\calq_n\}$ of measured subgroupoids of $\calp$, if $(\calg,\calp,\calq_j)$ satisfies Property~$\Padm$ for every $j\in\{1,\dots,n\}$, then $\calq_1\cap\dots\cap\calq_n$ is of infinite type.
 \end{lemma}

 \begin{proof}
 Write $P=C_0\times F_1\times\dots\times F_k$, where $C_0$ is the clique factor of $P$, and the $F_j$ do not further split as direct products of parabolic subgroups. We have $k\ge 1$ because the type of $P$ is not a clique. 
 
 We first assume that $C_0\neq\{1\}$. Let $\calq_1,\dots,\calq_n$ be finitely many measured subgroupoids of $\calg$ such that the triples $(\calg,\calp,\calq_j)$ all satisfy Property~$\Padm$. Since the type of $P$ is not a clique and $\rho$ is action-like, the pair $(\calp,\rho)$ is nowhere of clique type (Lemma~\ref{lemma:inclusion-parabolics}, \cpageref{lemma:inclusion-parabolics}). Therefore, we can apply Lemma~\ref{lemma:admissible}, \cpageref{lemma:admissible} and deduce that for every $j\in\{1,\dots,n\}$, $(\calq_j,\rho)$ is of clique-inclusive co-factor type within $(\calp,\rho)$. Thus, up to a conull Borel subset of $X$ and a countable Borel partition, we have $\rho^{-1}(C_0)\subseteq \calq_1\cap\dots\cap\calq_n$. In particular $\calq_1\cap\dots\cap\calq_n$ is of infinite type.

 We now assume that $C_0=\{1\}$. For every $j\in\{1,\dots,k\}$, let
 \[\calq_j=\rho^{-1}(F_1\times\dots\times \hat{F}_j\times\dots\times F_k).\]
 By Lemma~\ref{lemma:admissible}, \cpageref{lemma:admissible} for every $j\in\{1,\dots,k\}$, the triple $(\calg,\calp,\calq_j)$ satisfies Property~$\Padm$. In addition $\calq_1\cap\dots\cap\calq_n$ is trivial.
 \end{proof}

 \begin{lemma}\label{lemma:reducibility-invariant}
     Let $G,H$ be two graph products with countably infinite vertex groups, over finite simple graphs $\Gamma_G,\Gamma_H$ which are not cliques. If $G$ and $H$ are measure equivalent and $\Gamma_G$ is reducible, then $\Gamma_H$ is reducible. 
 \end{lemma}

\begin{proof}
    Since $G$ and $H$ are measure equivalent, we can find a measured groupoid $\calg$ over a standard probability space $X$, equipped with two action-like cocycles $\rho_G:\calg\to G$ and $\rho_H:\calg\to H$ (see Example~\ref{Ex:CocycleAssociatedToAnAction}, \cpageref{Ex:CocycleAssociatedToAnAction}, and Lemma~\ref{lemma:action-like-1}, \cpageref{lemma:action-like-1}). Since $\Gamma_G$ is reducible, $(\calg,\rho_G)$ itself is of maximal product type within $(\calg,\rho_G)$. And it is not of isolated clique type because $\Gamma_G$ is not a clique. By Lemma~\ref{lemma:maximal-product}(1), \cpageref{lemma:maximal-product} applied to the cocycle $\rho_G$, the pair $(\calg,\calg)$ satisfies Property~$\Pprod$. And by Lemma~\ref{lemma:maximal-product}(2) applied to the cocycle $\rho_H$, we deduce that $(\calg,\rho_H)$ is also of maximal product type (it is nowhere contained in a subgroupoid of isolated clique type because $\Gamma_H$ is not a clique). This implies that $\Gamma_H$ is reducible, as desired. 
\end{proof}

\subsubsection{\texorpdfstring{A theorem of Monod--Shalom}{Monod--Shalom's theorems for direct products of Creg groups}}\label{sec:creg}

We will need the following theorem which comes from the work of Monod--Shalom (see \cite[p.~862]{MS}). Its formulation comes from \cite{HHI}.

\begin{Th}[{Monod--Shalom, see \cite[Lemma~3.1]{HHI}}]\label{theo:ms-2}
	Let $G,H$ be countable groups. Let $G\actson X$ and $H\actson Y$ be free ergodic measure-preserving actions on standard probability spaces. Let $f:U\to V$ be a stable orbit equivalence from $G\actson X$ to $H\actson Y$ (where $U\subseteq X$ and $V\subseteq Y$ are positive measure Borel subsets). Let $c:G\times X\to H$ be an SOE cocycle associated to $f$.
	
	Let $A\unlhd G$ and $B\unlhd H$ be normal subgroups acting ergodically on $X,Y$, and assume that for every $x\in U$, one has $f((A\cdot x)\cap U)= (B\cdot f(x))\cap V$. 
	
Then there exist a group isomorphism $\alpha:G/A\to H/B$, and a cocycle $c':G\times X\to H$ such that
\begin{itemize}
    \item $c'$ is $H$-cohomologous to $c$;
    \item for every $g\in G$ and almost every $x\in X$, if $x,g\cdot x\in U$, then $c'(g,x)=c(g,x)$;
    \item for every $g\in G$ and almost every $x\in X$, one has $c^\prime(g,x)\in\alpha(gA)$. 
\end{itemize}
\end{Th}

\subsubsection{End of the proof}\label{sec:proof-strong-rigidity}

We now complete our proof of Theorem~\ref{theo:strong-rigidity}. We follow the proof of \cite[Propositions~6.1 and~6.2]{HHI} very closely, with only a few departures in the presentation (in particular, we have decided to treat the cases corresponding to \cite[Proposition~6.1]{HHI} and \cite[Proposition~6.2]{HHI} in a harmonized way).

\begin{proof}[Proof of Theorem~\ref{theo:strong-rigidity}]
We will prove the following: if $G\actson X$ and $H\actson Y$ are vertexwise ergodic free, probability measure-preserving actions, and if $c:G\times X\to H$ is an SOE cocycle, then there exists a vertex group $G_v\subseteq G$ such that $c_{|G_v\times X}$ is $H$-cohomologous to a group homomorphism $G_v\to H$. In view of Lemma~\ref{lemma:propagation}, this will be enough to conclude our proof.

By collapsing some subgraphs if necessary, we can write $G$ and $H$ as graph products over finite simple graphs $\bar\Gamma_G$ and $\bar\Gamma_H$ that satisfy the following two conditions:
\begin{itemize}
\item $\bar\Gamma_G$ and $\bar\Gamma_H$ are joins (possibly with only one join factor) of strongly reduced graphs in the sense of Definition~\ref{de:strongly-reduced}, \cpageref{de:strongly-reduced}.
\item The graph $\bar\Gamma_G$ does not contain any vertex $v$ such that $\bar\Gamma_G=\st(v)$, and neither does $\bar\Gamma_H$.
\end{itemize}
The vertex groups for the graph product structure of $G$ over $\bar\Gamma_G$ contain those for its graph product structure over $\Gamma_G$, so in particular the action of $G\actson X$ is still vertexwise ergodic with respect to the graph product structure over $\bar\Gamma_G$ (and likewise for $H$). Therefore $\bar\Gamma_G$ and $\bar\Gamma_H$ still satisfy all assumptions from Theorem~\ref{theo:strong-rigidity}. Thus, from now on, we will assume without loss of generality that $\Gamma_G$ and $\Gamma_H$ are joins of  strongly reduced graphs to start with.  

 Let $U\subseteq X$ and $V\subseteq Y$ be the positive measure Borel subsets defined as the source and range of the stable orbit equivalence $f$. In other words $f$ is a measure-scaling isomorphism $U\to V$ that sends orbits (restricted to $U$) to orbits (restricted to $V$). After renormalizing the measures on $U$ and $V$ to probability measures, the groupoid $(G\ltimes X)_{|U}$ is isomorphic (via $f$) to $(H\ltimes Y)_{|V}$. We denote by $\calg$ this common groupoid, which naturally comes equipped with two cocycles $\rho_G:\calg\to G$ and $\rho_H:\calg\to H$, which are action-like in view of Lemma~\ref{lemma:action-like-1}, \cpageref{lemma:action-like-1} and Remark~\ref{rk:action-like-restriction}, \cpageref{rk:action-like-restriction}. We view $\calg$ as a measured groupoid over $U$.

Our proof has three steps.
 \medskip
 
 \noindent \textbf{{Step 1.}} We show that there exist 
 \begin{itemize}
\item parabolic subgroups $P_G^1,P_G^2\subseteq G$, and $P_H^1,P_H^2\subseteq H$, 
\item normal infinite parabolic subgroups $C_G^i\unlhd P_G^i$ and $C_H^i\unlhd P_H^i$ for every $i\in\{1,2\}$, such that
\begin{itemize}
    \item $C_G^1$ and $C_G^2$ intersect trivially and commute,
    \item $C_H^1$ and $C_H^2$ intersect trivially and commute,
    \item for every $i\in\{1,2\}$, one has $P_G^i=C_G^i\times (C_G^i)^{\perp}$ and $P_H^i=C_H^i\times (C_H^i)^{\perp}$,
\end{itemize}
\item a positive measure Borel subset $W\subseteq U$, such that
\begin{itemize}
    \item for every $i\in\{1,2\}$, one has $\rho_G^{-1}(P_G^i)_{|W}=\rho_H^{-1}(P_H^i)_{|W}$,
    \item for every $i\in\{1,2\}$, one has $\rho_G^{-1}(C_G^i)_{|W}=\rho_H^{-1}(C_H^i)_{|W}$.
\end{itemize}
\end{itemize}

\smallskip

The proof of Step~1 decomposes into three cases.
 
 \smallskip
 
\noindent \textbf{Case 1:} The graph $\Gamma_G$ decomposes non-trivially as a join.

 \smallskip
 
  Write $\Gamma_G=\Gamma_1\circ\dots\circ\Gamma_k$, where each $\Gamma_j$ does not split further as a join -- notice that no $\Gamma_j$ is reduced to one vertex, as $\Gamma_G$ is not contained in the star of a vertex by assumption. Lemma~\ref{lemma:reducibility-invariant} ensures that $\Gamma_H$ also splits as a join: we write $\Gamma_H=\Lambda_1\circ\dots\circ\Lambda_n$ in such a way that no $\Lambda_j$ further splits as a join. Here again our assumption on $\Gamma_H$ ensures that no $\Lambda_j$ is reduced to one vertex. 
  
  We let $P_G^1=P_G^2=G$, and $P_H^1=P_H^2=H$, and we let $C_G^1$ and $C_G^2$ be two different factors of $G$ (so $C_G^1$ and $C_G^2$ intersect trivially and commute). 
  
  We now apply Lemma~\ref{lemma:factor-2}, \cpageref{lemma:factor-2}. This lemma gives a characterization of being of factor type with no reference to any action-like cocycle, and therefore it ensures that there exists a positive measure Borel subset $W\subseteq U$, and distinct factors $C_H^1,C_H^2$ of $H$, such that $\rho_G^{-1}(C_G^i)_{|W}=\rho_H^{-1}(C_H^i)_{|W}$ for every $i\in\{1,2\}$. This completes Step~1 in this case.

\medskip

\noindent \textbf{Case 2:} The graph $\Gamma_G$ does not decompose non-trivially as a join, and $G$ contains a maximal product parabolic subgroup $P_G$ with trivial clique factor.

\smallskip

We let $P_G^1=P_G^2:=P_G$, and let $C_G^1$ and $C_G^2$ be two distinct factors of $P_G$ (so $C_G^1$ and $C_G^2$ commute and intersect trivially). Notice that the maximality of $P_G$ ensures that $P_G=C_G^i\times (C_G^i)^{\perp}$ for every $i\in\{1,2\}$. 

 By Lemma~\ref{lemma:reducibility-invariant}, the graph $\Gamma_H$ does not decompose non-trivially as a join either. And since $\Gamma_G$ and $\Gamma_H$ do not contain isolated cliques (because they are connected), there is no subgroupoid of $\calg$ of isolated clique type for either $\rho_G$ or $\rho_H$, even in restriction to a positive measure subset. Lemma~\ref{lemma:maximal-product}, \cpageref{lemma:maximal-product} therefore implies that we can find a positive measure Borel subset $V\subseteq U$, and a maximal product parabolic subgroup $P_H\subseteq H$, such that $\rho^{-1}(P_G)_{|W}=\rho_H^{-1}(P_H)_{|W}$. We let $P_H^1=P_H^2:=P_H$.
 
 We now apply Lemma~\ref{lemma:factor}, \cpageref{lemma:factor}. This lemma gives a characterization of being of factor type with no reference to any action-like cocycle, and therefore it ensures that, up to replacing $W$ by a positive measure Borel subset, there exist two distinct factors $C_H^1,C_H^2$ of $P_H$ such that $\rho^{-1}(C_H^i)_{|W}=\rho^{-1}(C_G^i)_{|W}$ for every $i\in\{1,2\}$. And the maximality of $P_H$ ensures that $P_H=C_H^i\times (C_H^i)^{\perp}$ for every $i\in\{1,2\}$.  This completes the proof of Step~1 in this case.
 
 \medskip

\noindent \textbf{Case 3:} The graph $\Gamma_G$ does not decompose non-trivially as a join, and $G$ does not contain any maximal product parabolic subgroup $P_G$ with trivial clique factor.

\smallskip
 
 In this case, Lemma~\ref{lemma:HHI-1} ensures that $G$ contains two distinct maximal product parabolic subgroups $P_G^1,P_G^2$ whose clique factors $C_G^1,C_G^2$ are non-trivial and commute; in this case $C_G^1\cap C_G^2=\{1\}$ by Lemma~\ref{lemma:HHI-2}. For every $i\in\{1,2\}$, the maximality of $P_G^i$ ensures that $P_G^i=C_G^i\times (C_G^i)^{\perp}$. 
 
 Again by Lemma~\ref{lemma:reducibility-invariant}, the graph $\Gamma_H$ does not decompose non-trivially as a join either.
 
Since $\Gamma_G$ has no isolated clique, Lemma~\ref{lemma:maximal-product}, \cpageref{lemma:maximal-product} gives a characterization of being of maximal product type that does not depend on any action-like cocycle. Since $\Gamma_H$ has no isolated clique, Lemma~\ref{lemma:maximal-product}, applied to the cocycle $\rho_H$, therefore ensures that there exist maximal product parabolic subgroups $P_H^1,P_H^2$ and a positive measure Borel subset $W\subseteq U$ such that for every $i\in\{1,2\}$, we have $\rho_G^{-1}(P_G^i)_{|W}=\rho_H^{-1}(P_H^i)_{|W}$. Notice that, since $P_G^1\neq P_G^2$, we have $\rho_G^{-1}(P_G^1)_{|W}\neq \rho_G^{-1}(P_G^2)_{|W}$ (see Lemma~\ref{lemma:inclusion-parabolics}, \cpageref{lemma:inclusion-parabolics}), and therefore $P_H^1\neq P_H^2$. 

By Lemma~\ref{lemma:recognize-star}, the fact that $P_G^i$ has a non-trivial clique factor is characterized by a groupoid-theoretic property of $\rho_G^{-1}(P_G^i)$ which is independent from the action-like cocycle. It follows that $P_H^1$ and $P_H^2$ also have non-trivial clique factors, which we denote respectively by $C_H^1$ and $C_H^2$. By Lemma~\ref{lemma:HHI-2}, we have $C_H^1\cap C_H^2=\{1\}$. 

Likewise, Lemma~\ref{lemma:factor}, which characterizes the clique factor by a groupoid-theoretic property, ensures that, up to replacing $W$ by a positive measure Borel subset, we have $\rho_G^{-1}(C_G^i)_{|W}=\rho_H^{-1}(C_H^i)_{|W}$ for every $i\in\{1,2\}$. 


Since $C_G^1$ and $C_G^2$ commute, we deduce that $\rho_G^{-1}(C_G^1)_{|W}$ and $\rho_G^{-1}(C_G^2)_{|W}$ normalize each other, and thus the groups $C_H^1$ and $C_H^2$ commute (as in the proof of Lemma~\ref{lemma:adjacency}, \cpageref{lemma:adjacency}). 

Finally, the maximality of $P_H^1$ and $P_H^2$ as product parabolic subgroups ensures that $P_H^1=C_H^1\times (C_H^1)^{\perp}$ and $P_H^2=C_H^2\times (C_H^2)^{\perp}$.

\medskip

\noindent \textbf{Step 2.} We show that for every $i\in\{1,2\}$, there exist
\begin{itemize}
    \item a group homomorphism $\alpha_i:P_G^i\to P_H^i$, which restricts to an isomorphism from $C_G^i$ to $C_H^i$, with $(\alpha_1)_{|C_G^1\times C_G^2}=(\alpha_2)_{|C_G^1\times C_G^2}$, and
    \item measurable maps $\varphi_i:X\to H$ and $\kappa_i:P_G^i\times X\to C_H^i$,
\end{itemize}
such that \begin{equation}\label{eq:cocycle-2}
c(g,x)=\varphi_i(g \cdot x)\alpha_i(g)\kappa_i(g,x)\varphi_i(x)^{-1}, ~~~\forall g\in P_G^i, ~\text{a.e.~} x\in X.
\end{equation}
Informally, this means that we have untwisted the cocycle on $P_G^i$ to a group homomorphism, “up to an error in $C_H^i$”, for every $i\in\{1,2\}$.

\smallskip

Up to replacing $c$ by a cohomologous cocycle, we can assume that $c:G\times X\to H$ is an SOE cocycle associated to $f_{|W}$, and that whenever $x,g\cdot x\in W$, the element $c(g,x)$ is the unique element $h\in H$ such that $f(g\cdot x)=h\cdot f(x)$. Recall also that $P_G^i$ acts ergodically on $X$, and that $P_G^i$ and $P_H^i$ have the same orbits in restriction to $W$ (because $\rho_G^{-1}(P_G^i)_{|W}=\rho_H^{-1}(P_H^i)_{|W}$). It follows that $c$ is cohomologous to a cocycle $c_i:G\times X\to H$ having the following two properties (see e.g.\ the discussion at the end of Section~\ref{sec:cocycles}, \cpageref{sec:cocycles}):
\begin{itemize}
\item $(c_i)_{|P_G^i\times X}$ is an SOE cocycle associated to $f_{|W}$, viewed as a stable orbit equivalence between the actions $P_G^i\actson X$ and $P_H^i\actson Y$,
\item $c(g,x)=c_i(g,x)$ whenever $x,g\cdot x\in W$. 
\end{itemize}
For every $i\in\{1,2\}$, the normal subgroups $C_G^i\unlhd P_G^i$ and $C_H^i\unlhd P_H^i$ act ergodically on $X$ and on $Y$. In addition $C_G^i$ and $C_H^i$ have the same orbits in restriction to $W$ (because $\rho_G^{-1}(C_G^i)_{|W}=\rho_H^{-1}(C_H^i)_{|W}$). We can therefore apply Theorem~\ref{theo:ms-2} to the stable orbit equivalence $f_{|W}$, and deduce that, up to replacing $c_i$ by a cohomologous cocycle, we can assume that the following hold:
\begin{itemize}
    \item there is a group isomorphism $\bar\alpha_i:P_G^i/C_G^i\to P_H^i/C_H^i$ such that for every $g\in P_G^i$ and almost every $x\in X$, one has $c_i(g,x)\in\bar\alpha_i\left(gC_G^i\right)$;
    \item one has $c_i(g,x)=c(g,x)$ whenever $x,g \cdot x\in W$.
\end{itemize}
Let $r_i:P_H^i\to (C_H^i)^{\perp}$ be the retraction (sending $C_H^i$ to $\{1\}$). Let \[c'_i=(r_i\circ c_i)_{|(C_G^i)^{\perp}\times X},\] a cocycle with values in $(C_H^i)^{\perp}$. Through the natural isomorphisms $P_G^i/C_G^i\to (C_G^i)^{\perp}$ and $P_H^i/C_H^i\to (C_H^i)^{\perp}$ (given by the choice of the unique representative), $\bar\alpha_i$ yields an isomorphism $\alpha_i:(C_G^i)^{\perp}\to (C_H^i)^{\perp}$ such that 
\begin{equation}\label{eq:alpha}
c'_i(g,x)=\alpha_i(g),~~~\forall g\in (C_G^i)^{\perp},~ \text{a.e~} x\in X.
\end{equation}
Recall that $C_G^1\subseteq (C_G^2)^{\perp}$ (because the parabolic subgroups $C_G^1$ and $C_G^2$ intersect trivially and commute), and $C_H^1\subseteq (C_H^2)^\perp$. We will now prove the following claim. 

\begin{Claim}\label{claim:alphaCi}
We have $\alpha_2(C_G^1)=C_H^1$.
\end{Claim}

\begin{proof}[Proof of the Claim]
We first prove that $\alpha_2(C_G^1)\subseteq C_H^1$. So let $g\in C_G^1$. By Poincaré recurrence, for almost every $x\in W$, there exists $n>0$ such that $g^n\cdot x\in W$. Since $x,g^n\cdot x\in W$, our definition of the cocycles $c_i$ ensures that  $c_2(g^n,x)=c(g^n,x)$. Since $C_G^1$ and $C_H^1$ have the same orbits in restriction to $W$, we deduce that $c_2(g^n,x)\in C_H^1$, so $c_2(g^n,x)=c'_2(g^n,x)$. Since $g^n$ belongs to $C_G^1$, which is contained in $(C_G^2)^{\perp}$, our definition of $\alpha_2$ (see Equation~\eqref{eq:alpha}) shows that $c_2(g^n,x)=\alpha_2(g)^n$. It follows that $\alpha_2(g)^n\in C_H^1$, and Lemma~\ref{lemma:root}, \cpageref{lemma:root} implies that $\alpha_2(g)\in C_H^1$. This proves that $\alpha_2(C_G^1)\subseteq C_H^1$. 

Conversely, let us now prove that $C_H^1\subseteq\alpha_2(C_G^1)$. So let $h\in C_H^1$. By Poincaré recurrence again, for almost every $y\in W$, there exists $m>0$ such that $h^m\cdot f(y)\in f(W)$. Since the orbits of $C_G^1$ and of $C_H^1$ coincide in restriction to $W\approx f(W)$, we can find $g\in C_G^1$ such that 
$c_2(g,y)=h^m$. We deduce that
\begin{equation*}
    \alpha_2(g)=c'_2(g,y)=r_2\circ c_2(g,y)= r_2(h^m)=h^m.
\end{equation*}
Since $\alpha_2:(C_G^2)^{\perp}\to (C_H^2)^{\perp}$ is an isomorphism, we can also find $g_0\in (C_G^2)^\perp$ such that $\alpha_2(g_0)=h$. Then $\alpha_2(g_0^m)=\alpha_2(g)$, so $g_0^m=g$. Lemma~\ref{lemma:root} implies that $g_0\in C_G^1$, and we are done.
\end{proof}

The above claim shows that $\alpha_2$ restricts to an isomorphism from $C_G^1$ to $C_H^1$. Symmetrically $\alpha_1$ restricts to an isomorphism from $C_G^2$ to $C_H^2$.

We now extend $\alpha_1$ to $P_G^1$ by letting $(\alpha_1)_{|C_G^1}=(\alpha_{2})_{|C_G^1}$. Likewise, we extend $\alpha_2$ to $P_G^2$ by letting $(\alpha_2)_{|C_G^2}=(\alpha_1)_{|C_G^2}$. With this extension, 
for every $i\in\{1,2\}$, we have
\begin{equation}\label{eq:cocycle}
 c_i(g,x)\in\alpha_i(g)C_H^i \quad \forall g\in P_G^i, \ \text{a.e.}\ x\in X.
\end{equation}
Indeed, 
\begin{itemize}
\item if $g\in C_G^i$, then $c_i(g,x)$ belongs to $\bar\alpha_i(C_G^i)=C_H^i$, and by our extension of $\alpha_i$ to $C^i_G$ and \cref{claim:alphaCi} we get $\alpha_i(g)\in C_H^i$;
\item if $g\in (C_G^i)^{\perp}$, this comes from the fact that $c_i(g,x)\in \bar\alpha_i(gC^i_G)$ and the definition of~$\alpha_i$;
\item in general, every element $g\in P_G^i$ can be decomposed as $g=g_1g_2$ with $g_1\in (C_G^i)^{\perp}$ and $g_2\in C_G^i$, and Equation~\eqref{eq:cocycle} follows from the cocycle relation.
\end{itemize}

\medskip

At this point, for every $i\in\{1,2\}$, using Equation~\eqref{eq:cocycle} together with the fact that $c$ and $c_i$ are $H$-cohomologous, we have measurable maps $\varphi_i:X\to H$ and $\kappa_i:P_G^i\times X\to C_H^i$ such that for every $g\in P_G^i$ and a.e.\ $x\in X$, we have 
\begin{equation}\label{eq:cocycle-4}
    c(g,x)=\varphi_i(g \cdot x)c_i(g,x)\varphi_i(x)^{-1},
\end{equation}
and $c_i(g,x)=\alpha_i(g)\kappa_i(g,x)$, so \cref{eq:cocycle-2} is verified.
\medskip

\noindent \textbf{Step 3.} We show that $c_{|C_G^1\times X}$ is $H$-cohomologous to the homomorphism $\alpha_1$. This will be done by comparing the two ways of untwisting the cocycle $c$ given by Equation~\eqref{eq:cocycle-2}, for $i=1$ and $i=2$. 

Let $X^*=\sqcup_{j\in J}X_j$ be a partition of a conull Borel subset $X^*\subseteq X$ into at most countably many Borel subsets, such that for every $j\in J$, the maps $\varphi_1$ and $\varphi_2$ are constant when restricted to $X_j$, with respective values denoted by $\gamma_{1,j},\gamma_{2,j}$, and Equation~\eqref{eq:cocycle-2} holds everywhere on $X^*$.

Let $h\in C_H^1$ be such that $h$ does not belong to any proper parabolic subgroup of $C_H^1$, and neither do its proper powers (see e.g.\ the proof of \cref{lemma:reg-nonempty} for the existence of $h$). Let $g\in C_G^1$ such that $\alpha_1(g)=\alpha_2(g)=h$. For every $j\in J$, let $x\in X_j$ and $k_j>0$ be such that $g^{k_j}\cdot x\in X_j$ (these exist by Poincaré recurrence). Since $C_G^1\subseteq P_G^1\cap P_G^2$, we can use Equation~\eqref{eq:cocycle-2} with both $i=1$ and $i=2$ to write 
\begin{equation}\label{eq:cocycle-3}
c(g^{k_j},x)=\gamma_{1,j}\alpha_1(g)^{k_j}\kappa_1(g^{k_j},x)\gamma_{1,j}^{-1}=\gamma_{2,j}\alpha_2(g)^{k_j}\kappa_2(g^{k_j},x)\gamma_{2,j}^{-1}. 
\end{equation}
Let $r_2:H\to C_H^2$ be the natural retraction (see the last paragraph of Section~\ref{sec:first-facts}, \cpageref{sec:first-facts}). Since $C_H^1\subseteq (C_H^2)^{\perp}$ (by Step $1$), we have $r_2(C_H^1)=\{1\}$. In particular remark that $\alpha_1(g), \alpha_2(g),\kappa_1(g^{k_j},x)\in C^1_H$. By applying $r_2$ to the rightmost equality of Equation~\eqref{eq:cocycle-3}, we deduce that $\kappa_2(g^{k_j},x)=1$ for every $x\in X^*$.

Let $\beta_j=\gamma_{2,j}^{-1}\gamma_{1,j}$. Using that $\kappa_2(g^{k_j},x)=1$, Equation~\eqref{eq:cocycle-3} rewrites as 
\[\beta_{j}\alpha_1(g)^{k_j}\kappa_1(g^{k_j},x)\beta_j^{-1}=\alpha_2(g)^{k_j}.\]
Therefore $\beta_j C_H^1\beta_j^{-1}\cap C_H^1$ contains $\alpha_2(g)^{k_j}$, so our choice of $g$ ensures that $\beta_j C_H^1\beta_j^{-1}\cap C_H^1$ cannot be a proper parabolic subgroup of $C_H^1$. It follows that $\beta_jC_H^1\beta_j^{-1}\supseteq C_H^1$. The second point of Proposition~\ref{prop:am} implies that in fact  $\beta_jC_H^1\beta_j^{-1}=C_H^1$, i.e.\ $\beta_j$ belongs to the normalizer of $C_H^1$. 

In other words, for every $x\in X_j$, the constant value of $\varphi_2(x)^{-1}\varphi_1(x)$ belongs to $C_H^1\times (C_H^1)^{\perp}$. Since this is true for every $j\in J$, there exist measurable maps $\eta_1:X\to C_H^1$ and $\mu_1:X\to (C_H^1)^{\perp}$ such that for every $x\in X^*$, one has $\varphi_2(x)=\varphi_1(x)\eta_1(x)\mu_1(x)$. Now, by rewriting Equation~\eqref{eq:cocycle-2} with $i=2$, we deduce that for every $x\in X^*$ and every $g\in P_G^1$, one has \[c(g,x)=\varphi_1(g \cdot x)\eta_1(g \cdot x)\mu_1(g \cdot x)\alpha_2(g)\kappa_2(g,x)\mu_1(x)^{-1}\eta_1(x)^{-1}\varphi_1(x)^{-1},\] in other words 
\begin{equation}\label{eq:4}
\varphi_1(g\cdot x)^{-1}c(g,x)\varphi_1(x)=\eta_1(g\cdot x)\mu_1(g \cdot x)\alpha_2(g)\kappa_2(g,x)\mu_1(x)^{-1}\eta_1(x)^{-1}. 
\end{equation}
Now, when $g\in C_G^1$, using Equations~\eqref{eq:cocycle} and~\eqref{eq:cocycle-4}, the leftmost term belongs to $C_H^1$. We now apply the retraction $r_1:H\to C_H^1$ (which sends $(C_H^1)^{\perp}$ to $\{1\}$, so in particular sends $C_H^2$ to $\{1\}$), to Equation~\eqref{eq:4}. We deduce that \[\varphi_1(g \cdot x)^{-1}c(g,x)\varphi_1(x)=\eta_1(g \cdot x)\alpha_2(g)\eta_1(x)^{-1}, \quad \forall g\in C_G^1,~\text{a.e.}~x\in X .\] This proves that there exists a measurable map $\psi:X\to H$ and a homomorphism $\alpha_1:C_G^1\to H$ such that for every $g\in C_G^1$ and a.e.\ $x\in X$, one has \[c(g,x)=\psi(g \cdot x)\alpha_1(g)\psi(x)^{-1}.\] In other words, we have untwisted the cocycle $c$ to a group homomorphism on some vertex group of $G$. An application of Lemma~\ref{lemma:propagation} then concludes our proof of Theorem~\ref{theo:strong-rigidity}.
\end{proof}

\subsection{An application to von Neumann algebras}\label{sec:von-neumann-2}

When combined with a uniqueness theorem for Cartan subalgebras established by Chifan and Kunnawalkam Elayavalli in the framework of Popa's deformation/rigidity theory, Theorem~\ref{theo:strong-rigidity} has the following consequence at the level of von Neumann algebras. We refer to \cref{sec:von-neumann}, \cpageref{sec:von-neumann} for the needed definitions.

\begin{Cor}\label{cor:strong-rigidity-avn}
  Let $\Gamma_G,\Gamma_H$ be two connected finite simple graphs, with no vertex joined to every other vertex. Let $G,H$ be graph products with countably infinite torsion-free vertex groups over $\Gamma_G,\Gamma_H$, respectively. Let $G\actson X$ and $H\actson Y$ be two free, vertexwise ergodic, measure-preserving actions on standard probability spaces. Assume that either
  \begin{itemize}
 \item $G\actson X$ is vertexwise weakly mixing, or
 \item $G\actson X$ is vertexwise totally ergodic, or
 \item $H$ has the unique root property.
 \end{itemize}
 If the $\mathrm{II}_1$ factors $L(G\actson X)$ and $L(H\actson Y)$ are isomorphic (or more generally, if $L(H\actson Y)$ is isomorphic to an amplification $L(G\actson X)^t$ with $t\in (0,1]$), then the two actions are conjugate.
 \end{Cor}

\begin{proof}
By \cite[Theorem~1.3]{CKE}, the subalgebra $L^\infty(Y)$ is, up to unitary conjugation, the unique Cartan subalgebra of $L(H\actson Y)$. Let $U\subseteq X$ be a Borel subset of measure~$t$, and let $\calr$ be the restriction to $U$ of the orbit equivalence relation for the $G$-action on~$X$. Then $L(G\actson X)^t\simeq L(\calr)$, and $L^\infty(U)$ is a Cartan subalgebra. So there exists an isomorphism $\theta:L(H\actson Y)\to L(G\actson X)^t$ sending $L^\infty(Y)$ to $L^\infty(U)$. By \cite{SingerOE,FM}, this implies that the orbit equivalence relation of the $H$-action on $Y$ is isomorphic to $\calr$. In particular, there exists a stable orbit equivalence with compression $\frac{1}{t}\ge 1$ from $G\actson X$ to $H\actson Y$. Theorem~\ref{theo:strong-rigidity} thus implies that the two actions are conjugate.
\end{proof}

\begin{Rq}
As in Corollary~\ref{cor:fundamental}, we could also derive the triviality of the fundamental group of any action $G\actson X$ which is either vertexwise weakly mixing (or vertexwise totally ergodic), with $G$ torsion-free, or vertexwise ergodic, with $G$ having the unique root property. Likewise, the associated von Neumann algebra $L(G\actson X)$ has trivial fundamental group.   
\end{Rq}

\section{Superrigidity}\label{Sec:Superrigidity}

The goal of this section is to show the following theorem. 

\begin{Th}\label{theo:superrigidity}
 Let $\Gamma$ be a connected finite simple graph, with no vertex joined to every other vertex. Let $G$ be a graph product with countably infinite vertex groups over $\Gamma$, satisfying the unique root property, and let $H$ be an arbitrary countable group. Let $G\actson X$ and $H\actson Y$ be two free, ergodic, measure-preserving actions on standard probability spaces. Assume that
 \begin{itemize}
     \item one of the following two options holds:
     \begin{itemize}
         \item either $G\actson X$ is vertexwise ergodic, and $G$ has the unique root property, 
         \item or $G\actson X$ is vertexwise totally ergodic, and $G$ is torsion-free,
     \end{itemize}
     \item and $H\actson Y$ is mildly mixing. 
 \end{itemize}

 If the actions $G\actson X$ and $H\actson Y$ are stably orbit equivalent, then they are virtually conjugate.
\end{Th}

Recall that a group is \emph{ICC}\index{ICC (infinite conjugacy classe)} (standing for \emph{infinite conjugacy classes}\index{Infinite conjugacy classes (ICC)}) if the conjugacy class of every non-trivial element of $G$ is infinite. We will need the following lemma.

\begin{Lmm}\label{Lmm:ICC}
  If $G$ is a graph product over a finite simple graph $\Gamma$ such that no vertex is joined to every other vertex, then $G$ is ICC.
\end{Lmm}

\begin{proof}
    Write $\Gamma=\Gamma_1\circ\dots\circ\Gamma_k$, where each $\Gamma_j$ is irreducible (possibly $k=1$). It is enough to prove that each $G_{\Gamma_j}$ is ICC. So let $j\in\{1,\dots,k\}$. Our assumption that $\Gamma$ is not equal to the star of one of its vertices implies that $\Gamma_j$ is not a clique. By \cite[Corollary~2.13]{MO}, the group $G_{\Gamma_j}$ is acylindrically hyperbolic (equivalently by \cite[Theorem~1.2]{Osi}, the group $G$ contains a proper infinite hyperbolically embedded subgroup in the sense of \cite[Definition~2.1]{DGO}). In addition, by \cite[Corollary~2.9]{MV}, the group $G_{\Gamma_j}$ has no non-trivial finite normal subgroup. So by \cite[Theorem~8.14]{DGO}, the group $G_{\Gamma_j}$ is ICC. So $G$ is ICC, as desired.
\end{proof}

We are now in position to prove Theorem~\ref{theo:superrigidity}.

\begin{proof}[Proof of Theorem~\ref{theo:superrigidity}]
We first assume that the action $G\actson X$ is vertexwise ergodic and $G$ has the unique root property. In the terminology from \cite[Section~7]{HHI}, vertexwise ergodicity of $G\actson X$ is exactly $\calf$-ergodicity, with $\calf$ the collection of all vertex groups of $G$. And Theorem~\ref{theo:strong-rigidity} shows that $(G,\calf)$ is strongly cocycle rigid in the sense of \cite[Section~7]{HHI}. Since $G$ is ICC (Lemma~\ref{Lmm:ICC}), the conclusion therefore follows from \cite[Theorem~7.1]{HHI}.

We now assume that the action $G\actson X$ is vertexwise totally ergodic and $G$ is torsion-free. In the terminology from \cite[Section~7]{HHI}, vertexwise total ergodicity is exactly $\calf'$-ergodicity, with $\calf'$ the collection of all infinite subgroups of vertex groups of $G$. As above, Theorem~\ref{theo:strong-rigidity} shows that $(G,\calf')$ is strongly cocycle rigid in the sense of \cite[Section~7]{HHI}, and the conclusion follows from \cite[Theorem~7.1]{HHI}.
\end{proof}

\begin{Rq}
Again, using the uniqueness of the Cartan subalgebra, up to unitary conjugacy, for $L(G\actson X)$, Theorem~\ref{theo:superrigidity} can be reformulated as a $W^*$-rigidity statement.
\end{Rq}

Let us finally mention the following corollary to Theorem~\ref{theo:superrigidity}, which states the rigidity of graph product actions within the class of mixing group actions.

\begin{Cor}[\cref{theo:superrigidityintro}]
  Let $G,H$ be two torsion-free countable groups. Let $G\actson X$ and $H\actson Y$ be two mixing measure-preserving actions on standard probability spaces. Assume that $G$ splits as a graph product over a connected finite simple graph $\Gamma$, with no vertex joined to every other vertex.
  
  If the two actions $G\actson X$ and $H\actson Y$ are orbit equivalent, then they are conjugate.
\end{Cor}

\begin{proof}
Being mixing, the two actions are in particular totally ergodic (and in particular the action $G\actson X$ is vertexwise totally ergodic) and mildly mixing. They are also essentially free by \cite[Theorem~1.4]{TD}. Therefore Theorem~\ref{theo:superrigidity} applies and shows that the two actions are virtually conjugate. But $G$ and $H$ have no finite normal subgroup, and an induced action of the form $\Ind_{G'}^G(X)$ is never mixing if $G'\subsetneq G$ (because $G'$ does not act ergodically on $\Ind_{G'}^G(X)$). Therefore the actions $G\actson X$ and $H\actson Y$ are conjugate.
\end{proof}

\newpage
\part{Measure equivalence rigidity of some graph products}
\label{Part:SuperrigityHigman}
This chapter deals with measure equivalence rigidity of graph products. Contrary to the previous chapter, we will not make any ergodicity assumption on the couplings we consider; but on the other hand, we will impose some strong rigidity conditions on the vertex groups.

A general question is the following: to what extent is it true that if the vertex groups are rigid in measure equivalence, then the graph product is rigid in measure equivalence? 

We will not answer this question in full generality, but we will give a sufficient criterion in Section~\ref{sec:rigidity-criterion}, and an example in Section~\ref{sec:higman}. Our example is, to our knowledge, the first example of a group which is rigid in measure equivalence, but not in quasi-isometry, among torsion-free groups. On this matter, let us make two remarks.
\begin{itemize}
\item It is not true that a graph product of quasi-isometrically rigid groups is quasi-isometrically rigid. A basic observation is that in a graph product $G$, changing a vertex group $G_v$ by $G_v\times\mathbb{Z}/2\mathbb{Z}$ can often yield a group which is quasi-isometric but not commensurable to $G$. So the phenomena established in this chapter are radically different from their counterparts in quasi-isometry.
\item In the converse direction, examples of groups that are rigid in quasi-isometry but not in measure equivalence exist. First, some amenable groups are quasi-isometrically rigid, while countably infinite amenable groups are all measure equivalent by the Ornstein--Weiss theorem \cite{OW80}. Quasi-isometrically rigid amenable groups include $\mathbb{Z}^n$ for every $n\in\mathbb{N}$ (Pansu \cite{Pan}) or solvable Baumslag--Solitar groups (Farb--Mosher \cite{FaMo}). Examples of non-amenable groups that are rigid in quasi-isometry but not in measure equivalence include many right-angled Artin groups, see work of Huang \cite{Hua-QI} on the quasi-isometry side.
\end{itemize}

When dealing with the above question, it is in fact useful to impose a very strong form of rigidity on the vertex groups $G_v$, and require at least that for every $v\in V\Gamma$, every self measure equivalence coupling of $G_v$ factors through the tautological coupling on $G_v$ itself, by left-right multiplication. This will be an important assumption in our criterion in Section~\ref{sec:rigidity-criterion}.

To illustrate the subtleness of the question, let us make the following observation in the context of $L^1$-measure equivalence: it is possible to have a graph product $G$ of countably infinite groups over a finite simple graph with no transvection and no partial conjugation (and not reduced to a point), where all vertex groups are rigid in $L^1$-measure equivalence in the sense that every countable group that is $L^1$-measure equivalent to $G_v$ is virtually isomorphic to $G_v$, but $G$ is not. Indeed, it is enough to take the vertex groups to be isomorphic to $\mathbb{Z}$, and observe that some right-angled Artin groups are not rigid in $L^1$-measure equivalence \cite{HH-L1}, while $\mathbb{Z}$ is \cite[Corollary~1.2]{AustinBowen}.

\section{An ME-rigidity criterion for graph products}\label{sec:rigidity-criterion}

The goal of the present section is to prove the following theorem. The definition of Monod--Shalom's class $\Creg$ is given just below the statement.

\begin{Th}\label{theo:rigidity-criterion}
    Let $\Gamma$ be a finite simple graph with no transvection and no partial conjugation, with $\Aut(\Gamma)=\{1\}$. For every $v\in V\Gamma$, let $G_v$ be a countably infinite group, and let $G$ be the graph product over $\Gamma$ with vertex groups $(G_v)_{v\in V\Gamma}$. Assume that the following four hypotheses hold.
    \begin{enumerate}
        \item[$(\mathrm{H}_1)$] For every vertex $v\in V\Gamma$, there exists an extension $\hat{G}_v=G_v\rtimes F_v$ with $F_v$ finite, such that for every self measure equivalence coupling $\Omega_v$ of $G_v$, there exists a Borel $(G_v\times G_v)$-equivariant map $\Omega_v\to \hat{G}_v$.
        \item[$(\mathrm{H}_2)$] For every vertex $v\in V\Gamma$, the group $G_v$ is torsion-free, has no proper finite-index subgroup, and belongs to the class $\calc_\reg$.
        \item[$(\mathrm{H}_3)$] The vertex groups are pairwise not measure equivalent.
        \item[$(\mathrm{H}_4)$] For every clique $C$ with $VC=\{v_1,\dots,v_k\}$, the image of every injective homomorphism $G_{v_1}\to G_{v_1}\times\dots\times G_{v_k}$ is equal to $G_{v_1}\times \{1\}\times\dots\times\{1\}$. 
    \end{enumerate}
    Then every torsion-free countable group which is measure equivalent to $G$, is in fact isomorphic to $G$.
\end{Th}

As usual, throughout the section, all the stated equivariances are valid almost everywhere.

The class $\mathcal{C}_{\mathrm{reg}}$, introduced by Monod--Shalom in \cite{MS2}, is the class of all countable groups $G$ such that $\mathrm{H}^2_{\mathrm{b}}(G,\ell^2(G))\neq\{0\}$. Of particular importance for us is the following example.

\begin{Ex}\label{ex:creg}
    Recall that a group is \emph{acylindrically hyperbolic} in the sense of Osin \cite{Osi} if it admits a non-elementary acylindrical action on a hyperbolic space (see also Appendix~\ref{sec:acyl-hyp}, \cpageref{sec:acyl-hyp}). By a theorem of Hamenstädt \cite[Theorem~A]{Ham}, all acylindrically hyperbolic groups belong to the class $\Creg$. 
\end{Ex}

\begin{Rq}\label{rk:icc}
Hypothesis~$(\mathrm{H}_2)$ implies in particular that all vertex groups are ICC (i.e.\ every non-trivial conjugacy class in $G_v$ is infinite), since a countable torsion-free $\calc_\reg$ group is always ICC by \cite[Proposition~7.11]{MS}.
\end{Rq}

\paragraph*{Outline of the proof, and organization of the section.} 

Recall that the group $G$ acts on its right-angled building $\bD_G$ (see Section~\ref{sec:def-building} on \cpageref{sec:def-building} for its definition). We first prove that the $G$-stabilizer $G_u$ of every vertex $u\in V\bD_G$ is rigid in measure equivalence in a strong sense: every self-coupling of $G_u$ factors through the action on (some finite-index extension of) $G_u$ by left-right multiplication. When $u$ is a rank $1$ vertex, the group $G_u$ is conjugate to a vertex group for the graph product structure of $G$, and rigidity is exactly the contents of Hypothesis~$(\mathrm{H}_1)$. Stabilizers of higher rank vertices split as direct products of stabilizers of rank $1$ vertices, and we exploit an argument from the successive works of Monod--Shalom \cite{MS} and Kida \cite{Kid-me} to establish the measure equivalence rigidity of these direct products. This step is carried in Section~\ref{sec:self-couplings}.

Let now $H$ be a torsion-free group which is measure equivalent to $G$. We know from Part~\ref{Part:MEClassification} of the present work that every self-coupling of $G$ factors through $\Aut(\bD_G)$. By a standard technique, originating from the work of Furman \cite{Fur-me}, we can transfer the rigidity from self-couplings to arbitrary couplings, and deduce that $H$ also acts on $\bD_G$, and that (through this action) every measure equivalence coupling $\Omega$ between $G$ and $H$ factors through $\Aut(\bD_G)$. This is proved in Section~\ref{sec:factoring}. A key technical ingredient in the proof is to establish a strong ICC property for the Polish group $\Aut(\bD_G)$ -- our proof of this fact relies on previous work of Huang and the second-named author \cite{HH-L1}.

In Section~\ref{sec:action-of-h}, we then analyze this $H$-action on $\bD_G$. In a way that is similar to our work in Part~\ref{Part:MEClassification}, we prove that for every vertex $u\in V\bD_G$, the coupling $\Omega$ (between $G$ and $H$) restricts to a measure equivalence coupling between $G_u$ and $H_u$. And the combination of Hypotheses~$(\mathrm{H}_1)$ and~$(\mathrm{H}_2)$ then forces $H_u$ to be isomorphic to $G_u$ (the facts that $H_u$ is torsion-free and that $G_u$ has no proper finite-index subgroup are used to obtain an actual isomorphism, not a virtual one). In other words $H$ acts on $\bD_G$ with the same vertex stabilizers as $G$. And, by crucially using our assumption that all groups considered have no proper finite-index subgroup, we also prove in Section~\ref{sec:same-fundamental-domain} that the $H$-action on $\bD_G$ shares a common fundamental domain (which is a subcomplex) with the action of $G$.

In the terminology of Bridson--Haefliger \cite{BH}, this says that $H$ has the structure of a simple complex of groups over the fundamental domain $Y$ for the $G$-action on $\bD_G$. In other words $H$ is exactly the group which is obtained from the free product of all $H$-stabilizers of cells in $Y$, by taking the quotient by the relations coming from the natural inclusions of stabilizers (coming from face inclusions of the cells). Recall that each cell stabilizer splits as a direct product of some of the vertex groups $G_v$, and there remains to understand the inclusion homomorphisms. This is where we crucially use our Hypothesis~$(\mathrm{H}_4)$: it forces all inclusion homomorphisms to be factor inclusions. In other words $H$ has exactly the presentation of a graph product, and therefore $H$ is isomorphic to $G$.

\subsection{Self-couplings of clique subgroups}\label{sec:self-couplings}

Using the fact that all vertex groups belong to the class $\calc_\reg$, we extend Assumption~$(\mathrm{H}_1)$ to all self-couplings of clique subgroups of $G$ (and not just to self-couplings of vertex groups). This is the contents of the following proposition, which was essentially established by Kida \cite[Section~7]{Kid-me}, relying on work of Monod--Shalom \cite{MS}. We provide a brief argument as Kida's statements were formulated in the setting of mapping class groups.

\begin{Prop}\label{prop:product-me}
    Let $G_1,\dots,G_n$ be torsion-free countable groups in the class $\Creg$, and for every $i\in\{1,\dots,n\}$, let $\hat{G}_i$ be a finite-index extension of $G_i$.  Let $G=G_1\times\dots\times G_n$, and let $\hat G=\hat{G}_1\times\dots\times\hat{G}_n$. Assume that 
    \begin{itemize}
    \item for every $i\in\{1,\dots,n\}$, and every self measure equivalence coupling $\Omega_i$ of $G_i$, there exists a $(G_i\times G_i)$-equivariant Borel map $\Omega_i\to \hat{G}_i$, and
    \item for any distinct $i,j\in\{1,\dots,n\}$, the groups $G_i$ and $G_j$ are not measure equivalent.
    \end{itemize}
    Then for every self measure equivalence coupling $\Omega$ of $G$, there exists a $(G\times G)$-equivariant Borel map $\Omega\to \hat{G}$.
\end{Prop}

\begin{proof}
By considering the ergodic decomposition, every self measure equivalence coupling of $G$ can be disintegrated into ergodic self measure equivalence couplings, see the second point of \cite[Lemma~2.2]{Fur-me}. Therefore, in view of \cite[Corollary~5.4]{FMW} (applied to $G=H$, with $\Omega$ in place of their space $X$, and $\hat G$ in place of their space $Y$), it is enough to treat the case where the action of $G\times G$ on $\Omega$ is ergodic. From now on we work under this extra assumption.

Using an argument from the proof of \cite[Theorem~1.16]{MS}, Kida proved in \cite[Section~7, p.~1895]{Kid-me} that for every $i\in\{1,\dots,n\}$, there exists a standard measure space $\Omega_i$, equipped with 
\begin{itemize}
    \item a measure-preserving action of $G_i\times G_i$;
    \item a measurable projection $p_i:\Omega\to\Omega_i$ which is both $(G_i\times G_i)$-equivariant, and $(G_j\times G_j)$-invariant for every $j\neq i$;
    \item a $(G_i\times G_i)$-invariant Borel partition $\Omega_i^*=\dunion_k\Omega_{i,k}$ of a $(G_i\times G_i)$-invariant conull Borel subset $\Omega_i^*$ into at most countably many subspaces $\Omega_{i,k}$, such that each $\Omega_{i,k}$ is a self measure equivalence coupling of $G_i$.
\end{itemize}
(Notice that Kida has a space $\Omega_i$ which is invariant under $G_i\times G_{t(i)}$ for some bijection $t:\{1,\dots,n\}\to\{1,\dots,n\}$, but our second assumption implies that $t(i)=i$.) 

For every $i\in\{1,\dots,n\}$ and every $k\in\mathbb{N}$, let $\theta_{i,k}:\Omega_{i,k}\to \hat{G}_i$ be a $(G_i\times G_i)$-equivariant Borel map, given by our first assumption. Combining these maps as $k$ varies in $\mathbb{N}$ yields a $(G_i\times G_i)$-equivariant Borel map $\theta_i:\Omega_i\to \hat{G}_i$. Then the map $(\theta_1\circ p_1,\dots,\theta_n\circ p_n):\Omega\to \hat{G}$ gives the desired almost $(G\times G)$-equivariant Borel map.
\end{proof}

\subsection{Factoring the coupling through the right-angled building}

The goal of this section is to prove the following proposition (we refer to Section~\ref{sec:def-building} on \cpageref{sec:def-building}, for the definition of the right-angled building $\bD_G$).

\begin{Prop}\label{prop:factorize-through-building}
    Let $G$ be a graph product with countably infinite vertex groups, over a finite simple graph $\Gamma$ with no transvection and no partial conjugation. Let $H$ be a countable group, and let $\Omega$ be a measure equivalence coupling between $G$ and $H$. Let $\bD_G$ be the right-angled building of $G$.

    Then there exist a homomorphism $\iota: H\to\Aut(\bD_G)$ and a $(G\times H)$-equivariant Borel map $\theta:\Omega\to\Aut(\bD_G)$, where the action of $G\times H$ on $\Aut(\bD_G)$ is via $(g,h)\cdot f=gf\iota(h)^{-1}$.
\end{Prop}

We already know that self-couplings of $G$ factor through $\Aut(\bD_G)$ (Corollary~\ref{cor:factor-through-building}, \cpageref{cor:factor-through-building}). Our proof relies on a general strategy to transfer this fact to an information about arbitrary couplings between $G$ and $H$. This strategy was initiated by Furman in \cite{Fur-me}, and we will use a version found in \cite{Kid-amalgam} (or alternatively \cite{BFSIntegrability}). This argument requires establishing a strong ICC property for the Polish group $\Aut(\bD_G)$, which is the contents of our next lemma.

\begin{lemma}\label{lemma:icc}
Let $G$ be a graph product of countably infinite groups over a finite simple graph $\Gamma$, such that no vertex $v\in V\Gamma$ satisfies $\Gamma=\st(v)$. Let $\bD_G$ be its right-angled building.

Then the only Borel probability measure on $\Aut(\bD_G)$ which is invariant by the conjugation by any element of $G$ is the Dirac mass at the identity.
\end{lemma}

The conclusion of the lemma precisely says that $\Aut(\bD_G)$ is \emph{strongly ICC} relative to $G$ in the sense of Bader--Furman--Sauer \cite[Definition~2.2]{BFSIntegrability} (or that the inclusion  $G\subseteq\Aut(\bD_G)$ is \emph{strongly ICC}\index{ICC!Strongly ICC}\index{Strongly ICC} in the terminology of \cite[Section~4.2]{HH-L1}).

\begin{proof} 
This was proved when $G$ is the right-angled Artin group $A_\Gamma$ over $\Gamma$ in \cite[Proposition~5.1]{HH-L1}. We now explain how to reduce the proof to the case of a right-angled Artin group. We will use the language of $\mathrm{CAT}(0)$ cube complexes and their half-spaces, for which we refer to \cite{Sag}. We will also need the notion of the \emph{regular boundary} of a $\mathrm{CAT}(0)$ cube complex $X$, a natural compact subspace of its Roller boundary equal to the closure of the set of regular points in the sense of Fern\'os in \cite[Definition~7.3]{Fer}, also denoted $S(X)$ in \cite{KS}. All the necessary background for this proof was also reviewed in \cite[Section~5.2]{HH-L1}.

Let $\mathbb{D}_\Gamma$ be the right-angled building of the right-angled Artin group over $\Gamma$. Recall from Lemma~\ref{lemma:building} on \cpageref{lemma:building}, that there exists an isomorphism from $\bD_G$ to $\bD_\Gamma$ that sends $G$-orbits to $A_\Gamma$-orbits.

\medskip

\noindent\textbf{Step 1.} Assuming that $\Gamma$ is irreducible (and not reduced to one vertex), we will prove that the $G$-action on the regular boundary $R(\bD_G)$ is minimal and strongly proximal in the sense of Furstenberg \cite{Furs}, i.e.\ it is minimal and the $G$-orbit of any probability measure on $R(\bD_G)$ contains a Dirac mass in its weak-$*$ closure. 

By \cite[Lemma~5.3]{HH-L1}, the building $\mathbb{D}_G\approx\mathbb{D}_\Gamma$ is irreducible (it does not split as a product), and 
\begin{itemize}
\item $A_\Gamma$ acts non-elementarily, i.e.\ with no global fixed point, and no finite orbit in the visual boundary $\partial_\infty\mathbb{D}_\Gamma$, and 
\item $A_\Gamma$ acts essentially, in the sense of Caprace--Sageev \cite{CS}, i.e.\ no $A_\Gamma$-orbit stays in a bounded neighborhood of a half-space of $\mathbb{D}_\Gamma$.
\end{itemize}
The first of these properties (non-elementarity) was in fact proved by Caprace--Hume \cite[Section~6.2]{CH} in the more general context of graph products, so it applies to the $G$-action on $\bD_G$. The second property (essentiality) for the $G$-action on $\bD_G$ follows from the same property for the $A_\Gamma$-action on $\bD_\Gamma$ as these two actions have the same orbits.

It thus follows from \cite[Proposition~1]{KS} that the $G$-action on the regular boundary $R(\bD_G)$ is minimal and strongly proximal, as claimed. 

\medskip

\noindent \textbf{Step 2.} We now conclude the proof of the lemma. Write $\Gamma=\Gamma_1\circ \dots\circ \Gamma_k$ as a join, where each $\Gamma_j$ is irreducible (possibly with $k=1$). Notice that no $\Gamma_j$ is reduced to a vertex, by our assumption that $\Gamma$ is not contained in the star of a vertex. Let $G=G_1\times\dots\times G_k$ be the corresponding direct product decomposition. For every $j\in\{1,\dots,k\}$, let $\bD_j$ be the right-angled building of $G_j$, and let $R=R(\bD_1)\times\dots\times R(\bD_k)$. By Step~1, for every $j\in\{1,\dots,k\}$, the action of $G_j$ on $R(\bD_j)$ is minimal and strongly proximal, and it is faithful in view of \cite[Lemma~5.6]{HH-L1}. Let now $\mu$ be a probability measure on $\Aut(\bD_G)$ which is invariant under conjugation by every element of $G$. By \cite[Lemma~5.6]{HH-L1}, we can pushforward $\mu$ to a probability measure $\bar\mu$ on $\Homeo(R)$ which is invariant under the conjugation by every element of $G$. It thus follows from \cite[Lemma~5.2]{HH-L1} that $\bar\mu$ is the Dirac mass at the identity. Using again that the map $\Aut(\bD_G)\to \Homeo(R)$ is injective (\cite[Lemma~5.6]{HH-L1}), it follows that $\mu$ is the Dirac mass at the identity. 
\end{proof}

We are now in position to prove Proposition~\ref{prop:factorize-through-building}.

\begin{proof}[Proof of Proposition~\ref{prop:factorize-through-building}]
    Let $\iota_G:G\to\Aut(\bD_G)$ be the natural inclusion. By Corollary~\ref{cor:factor-through-building} on \cpageref{cor:factor-through-building}, for every self measure equivalence coupling $\Sigma$ of $G$, there exists a $(G\times G)$-equivariant Borel map $\Sigma\to\Aut(\bD_G)$. Combined with Lemma~\ref{lemma:icc}, this precisely says that $G$ is coupling rigid with respect to $(\Aut(\bD_G),\iota_G)$ in the sense of Kida \cite[Definition~3.3]{Kid-amalgam}. 
    The conclusion thus follows from \cite[Theorem~3.5]{Kid-amalgam}.
\end{proof}

\subsection{\texorpdfstring{Properties of the $H$-action on $\bD_G$}{Step 2: A common fundamental domain for G and H on the Davis complex}}\label{sec:action-of-h}

In the whole section, we will work in the following setting. 

\paragraph{General setting.} Let $\Gamma$ be a finite simple graph with no transvection and no partial conjugation, with $\Aut(\Gamma)=\{1\}$, and $G$ be a graph product with countably infinite vertex groups over $\Gamma$ that satisfies Assumptions~$(\mathrm{H}_1)$,~$(\mathrm{H}_2)$ and~$(\mathrm{H}_3)$ from Theorem~\ref{theo:rigidity-criterion} (Assumption~$(\mathrm{H}_4)$ will only be used in Section~\ref{sec:complex-of-groups}). Let $H$ be a torsion-free countable group which is measure equivalent to $G$, and let $(\Omega,m)$ be a measure equivalence coupling between $G$ and $H$. Proposition~\ref{prop:factorize-through-building} gives us  
\begin{itemize}
    \item a homomorphism $\iota:H\to\Aut(\bD_G)$;
    \item a $(G\times H)$-equivariant Borel map $\theta:\Omega\to\Aut(\bD_G)$.
\end{itemize}
In particular $\iota$ gives an action of $H$ on $\bD_G$, and from now on we will work with this action of $H$. 

\subsubsection{The $H$-action is type preserving}

Recall from Section~\ref{sec:def-building} that vertices of $\bD_G$ correspond to left cosets of the form $gG_\Upsilon$, for some $g\in G$ and some clique subgraph $\Upsilon\subseteq\Gamma$. The rank of a vertex is the cardinality of $V\Upsilon$. We also define the \emph{type}\index{Type!Type of a vertex of $\bD_G$} of a vertex of $\bD_G$ as the corresponding subgraph $\Upsilon$. More generally, given a cube $\sigma\subseteq\bD_G$, we define the \emph{type}\index{Type!Type of a cube in $\bD_G$} of $\sigma$ as the set consisting of all types of vertices of $\sigma$. Notice that the type of $\sigma$ is determined by the type of its smallest and largest vertices, for the partial order given by  inclusion (when viewing vertices as cosets in $G$).

The fact that $\Aut(\Gamma)=\{1\}$ will be used in the form of the following lemma.

\begin{lemma}\label{lemma:same-orbits}
    The actions of $G$ and of $\Aut(\bD_G)$ on $\bD_G$ have the same orbits of cubes. In particular, the action of $\Aut(\bD_G)$ on $\bD_G$ preserves types of cubes.
\end{lemma}

\begin{proof}
Since the $G$-action on $\bD_G$ preserves types of vertices, it preserves types of cubes, and therefore it is enough to prove that $G$ and $\Aut(\bD_G)$ have the same orbits of cubes.

\medskip

\noindent\textbf{Step 1:} The actions of $G$ and $\Aut(\bD_G)$ on $\bD_G$ have the same orbits of vertices.

\smallskip

    When $G=A_\Gamma$ is the right-angled Artin group over $\Gamma$, this was proved in \cite[Corollary~6.3]{HH-L1}: indeed the group $\hat{G}$ appearing in \cite[Corollary~6.3]{HH-L1} splits as a semi-direct product $G\rtimes\Aut(\Gamma)$ (see its definition in the first paragraph of \cite[Section~6.1]{HH-L1}), so it coincides with $G$ under our assumption that $\Aut(\Gamma)=\{1\}$.

    The general case now follows from the fact that there always exists an isomorphism $\bD_G\to\bD_\Gamma$ (where $\bD_\Gamma$ is the right-angled building for $A_\Gamma$), sending $G$-orbits of vertices to $A_\Gamma$-orbits of vertices (Lemma~\ref{lemma:building}, \cpageref{lemma:building}).  
    
    \medskip

\noindent\textbf{Step 2:} The actions of $G$ and $\Aut(\bD_G)$ on $\bD_G$ have the same orbits of cubes.

\smallskip

Let $\sigma,\sigma'\subseteq\bD_G$ be two cubes in the same $\Aut(\bD_G)$-orbit; we aim to prove that they are in fact in the same $G$-orbit. Let $v_{\min},v_{\max}$ be the minimal and maximal vertices of $\sigma$, and let $v'_{\min},v'_{\max}$ be the minimal and maximal vertices of $\sigma'$. Since $\Aut(\bD_G)$ preserves types of vertices (as a consequence of Step~1), the vertices $v_{\min}$ and $v'_{\min}$ are in the same $\Aut(\bD_G)$-orbit. Using Step~1 again, they are in fact in the same $G$-orbit: there exists $g\in G$ such that $v'_{\min}=gv_{\min}$. Likewise, the vertices $v_{\max}$ and $v'_{\max}$ are in the same $\Aut(\bD_G)$-orbit. In particular they have the same type, which we denote by $\Lambda$ (a subgraph of $\Gamma$). So $v_{\max}$ (resp.\ $v'_{\max}$) is the unique coset of $G_{\Lambda}$ that contains the coset of $G$ corresponding to $v_{\min}$ (resp.\ to $v'_{\min}=gv_{\min}$). Therefore $v'_{\max}=gv_{\max}$, and this in enough to ensure that $\sigma'=g\sigma$.
\end{proof}

Let us record the following consequence of Lemma~\ref{lemma:same-orbits}.

\begin{Cor}\label{cor:stab-cube}
    Let $\sigma\subseteq\bD_G$ be a cube, and let $v$ be the unique vertex of $\sigma$ of minimal rank. Then the stabilizers of $v$ and of $\sigma$ in $\Aut(\bD_G)$ coincide.
\end{Cor}

\begin{proof}
Since $v$ is the unique vertex of $\sigma$ of its type and the action of $\Aut(\bD_G)$ preserves types of vertices (Lemma~\ref{lemma:same-orbits}), we have $\Stab_{\Aut(\bD_G)}(\sigma)\subseteq\Stab_{\Aut(\bD_G)}(v)$.

The converse inclusion follows from the observation that $\sigma$ is the unique cube of its type that contains $v$ as its minimal vertex. Since $\Aut(\bD_G)$ preserves types of cubes (Lemma~\ref{lemma:same-orbits}), this implies that the $\Aut(\bD_G)$-stabilizer of $v$ also preserves $\sigma$.
\end{proof}

\subsubsection{Stabilizers for the actions of $G$ and $H$ are measure equivalent}

Given a cube $\sigma\subseteq V\bD_G$, we let $\mathrm{Stab}_G(\sigma),\mathrm{Stab}_H(\sigma)$ and $\Stab_{\Aut(\bD_G)}(\sigma)$ be the respective stabilizers of $\sigma$ for the actions of $G$, of $H$, and of $\Aut(\bD_G)$ on $\bD_G$. We let $\Omega_\sigma:=\theta^{-1}(\Stab_{\Aut(\bD_G)}(\sigma))$, where we recall that $\theta:\Omega\to\Aut(\bD_G)$ is our given $(G\times H)$-equivariant Borel map. Then $\Omega_\sigma$ is a Borel subset of $\Omega$ which is $(\mathrm{Stab}_G(\sigma)\times \mathrm{Stab}_H(\sigma))$-invariant. The following lemma is a special case of \cite[Corollary~4.4]{HH-L1}.

\begin{lemma}\label{lemma:hh}
For every cube $\sigma\subseteq V\bD_G$, the space $\Omega_\sigma$ is a measure equivalence coupling between $\mathrm{Stab}_G(\sigma)$ and $\mathrm{Stab}_H(\sigma)$.
\end{lemma}

\begin{proof}
In view of Corollary~\ref{cor:stab-cube}, it is enough to prove the lemma when $\sigma$ is a vertex, so we will assume so.

    By Lemma~\ref{lemma:same-orbits}, the groups $G$ and $\Aut(\bD_G)$ have the same orbits of vertices for their actions on $\bD_G$. The lemma therefore follows from \cite[Corollary~4.4]{HH-L1}, applied with $K=\bD_G$, with $L=\Aut(\bD_G)$, with $\mathsf{G}=G$, and with $\mathsf{H}=H$.
\end{proof}

As a particular case, we obtain the following statement.

\begin{Cor}\label{cor:rank-0}
    The $H$-stabilizer of every rank $0$ vertex of $\bD_G$ is trivial.
\end{Cor}

\begin{proof}
    Let $v\in V\bD_G$ be a rank $0$ vertex. By Lemma~\ref{lemma:hh}, the stabilizer $\mathrm{Stab}_H(v)$ is measure equivalent to $\mathrm{Stab}_G(v)$, which is trivial. Therefore $\mathrm{Stab}_H(v)$ is finite, and in fact trivial because $H$ is torsion-free.  
\end{proof}

In fact we will deduce more generally that $\mathrm{Stab}_G(\sigma)$ and $\mathrm{Stab}_H(\sigma)$ are isomorphic for every cube $\sigma\subseteq\bD_G$, in view of the following lemma.

\begin{lemma}\label{lemma:index-rigid}
   Let $\sigma\subseteq \bD_G$ be a cube. Let $K$ be a torsion-free countable group which is measure equivalent to $\mathrm{Stab}_G(\sigma)$.
   
   Then $K$ is isomorphic to $\mathrm{Stab}_G(\sigma)$.
\end{lemma}

\begin{proof}
The group $\mathrm{Stab}_G(\sigma)$ is isomorphic to $G_\Upsilon=G_{v_1}\times\dots\times G_{v_k}$ for some clique $\Upsilon$ of $\Gamma$ with $V\Upsilon=\{v_1,\dots,v_k\}$. We also write $\hat G_\Upsilon:=\hat G_{v_1}\times\dots\times\hat G_{v_k}$. Recall that each $G_{v_i}$ is ICC (Remark~\ref{rk:icc}), in particular $G_\Upsilon$ is ICC.

Hypotheses~$(\mathrm{H}_1)$,~
$(\mathrm{H}_2)$ and~$(\mathrm{H}_3)$, together with Proposition~\ref{prop:product-me}, ensure that for every self measure equivalence coupling $\Omega_\Upsilon$ of $G_\Upsilon$, there exists a $(G_\Upsilon\times G_\Upsilon)$-Borel map $\Omega_\Upsilon\to \hat{G}_\Upsilon$. Therefore \cite[Theorem~6.1]{Kid-me} implies that there exists a homomorphism $\pi_\sigma:K\to \hat G_\Upsilon$ with finite kernel, and finite-index image. Since $K$ is torsion-free, $\pi_\sigma$ is in fact injective. Then $\pi_\sigma(K)\cap G_\Upsilon$ is a finite-index subgroup of $G_\Upsilon$, and is therefore equal to $G_\Upsilon$ in view of Hypothesis~$(\mathrm{H}_2)$. Recall also from Hypothesis~$(\mathrm{H}_1)$ that for every $i\in\{1,\dots,k\}$, we have a semi-direct product $\hat{G}_{v_i}=G_{v_i}\rtimes F_{v_i}$, with $F_{v_i}$ finite. Then the image of $\pi_\sigma(K)$ in each $F_{v_i}$ is trivial: otherwise $\pi_\sigma(K)$ would intersect a lift of $F_{v_1}\times\dots\times F_{v_k}$ to $\hat G_{\Upsilon}$ non-trivially, contradicting that $K$ is torsion-free. Therefore $\pi_\sigma(K)=G_\Upsilon$, which shows that $K$ is isomorphic to $\mathrm{Stab}_G(\sigma)$.
\end{proof}

\subsubsection{A common strict fundamental domain for the actions of $G$ and $H$}\label{sec:same-fundamental-domain}

Let $Y_G\subseteq\bD_G$ be the standard fundamental domain for $G$, spanned by all vertices of the form $G_\Upsilon$, where $\Upsilon\subseteq\Gamma$ is a clique. Notice that $Y_G$ is a \emph{strict fundamental domain}\index{Strict fundamental domain} for the $G$-action on $\bD_G$, i.e.\ it is a subcomplex that meets every $G$-orbit in exactly one point. We will prove in Lemma~\ref{lemma:same-fundamental-domain} below that $Y_G$ is also a strict fundamental domain for the $H$-action on $\bD_G$. We start with a lemma.

\begin{lemma}\label{lemma:coupling-index-1}
For every rank $1$ vertex $v\in VY_G$, every measure equivalence coupling between $\Stab_G(v)$ and $\Stab_H(v)$ has coupling index $1$.
\end{lemma}

\begin{proof}
In view of Lemmas~\ref{lemma:hh} and~\ref{lemma:index-rigid}, the groups $\Stab_G(v)$ and $\Stab_H(v)$ are isomorphic, so it is enough to show that every self measure equivalence coupling of $\Stab_G(v)$ has coupling index $1$. 

Recall that since $v\in VY_G$, the stabilizer $\Stab_G(v)$ is a vertex group of $G$ (for its graph product structure). Hypothesis~$(\mathrm{H}_1)$ gives us a finite-index extension $\widehat{\Stab}_G(v)$.

We first observe that every self measure equivalence coupling $\hat\Sigma_v$ of $\widehat{\Stab}_G(v)$ has coupling index $1$. Indeed, since $\Stab_G(v)$ has finite index in $\widehat{\Stab}_G(v)$, the space $\hat\Sigma_v$ is also a self measure equivalence coupling of $\Stab_G(v)$. So Hypothesis~$(\mathrm{H}_1)$ yields a Borel $(\Stab_G(v)\times\Stab_G(v))$-equivariant map $\theta_v:\hat\Sigma_v\to \widehat{\Stab}_G(v)$. Since $\Stab_G(v)$ is ICC (Remark~\ref{rk:icc}) and normal in $\widehat{\Stab}_G(v)$, it follows from \cite[Lemma~5.8]{Kid-me} that $\theta_v$ is in fact $(\widehat{\Stab}_G(v)\times\widehat{\Stab}_G(v))$-equivariant. Then $\theta_v^{-1}(\{1\})$ is a common Borel fundamental domain for the two actions of $\widehat{\Stab}_G(v)$ on $\hat \Sigma_v$, so $\hat{\Sigma}_v$ has coupling index $1$.

Let now $\Sigma_v$ be a self measure equivalence coupling of $\Stab_{G}(v)$. Starting from $\Sigma_v$ and from a measure equivalence coupling between $\Stab_G(v)$ and $\widehat{\Stab}_G(v)$ of coupling index $[\widehat{\Stab}_G(v):{\Stab}_G(v)]$, by composition of couplings, one obtains a self measure equivalence coupling of $\widehat{\Stab}_G(v)$ with the same coupling index as $\Sigma_v$ (see e.g.\ Item~c on page~300 of \cite{Fur-survey}, below Definition~2.1). It thus follows from the previous paragraph that the coupling index of $\Sigma_v$ is~$1$, as desired.
\end{proof}

\begin{lemma}\label{lemma:same-fundamental-domain}
   The subcomplex $Y_G$ is a strict fundamental domain for the $H$-action on~$\bD_G$.
\end{lemma}

\begin{proof}
Assume towards a contradiction that $Y_G$ is not a strict fundamental domain for the $H$-action on $\bD_G$. This means that 
\begin{itemize}
    \item either two distinct points $x,y$ of $Y_G$ belong to the same $H$-orbit,
    \item or there is a point in $\bD_G$ which is not in the $H$-orbit of any point of $Y_G$.
\end{itemize}

We start by excluding the first option. Assuming that $x,y\in Y_G$ belong to the same $H$-orbit, we will prove that $x=y$. Let $\sigma_x,\sigma_y$ be the cubes of minimal dimension that contain $x,y$; these are contained in $Y_G$ because $Y_G$ is a subcomplex of $\bD_G$. And $\sigma_x,\sigma_y$ also belong to the same $H$-orbit, in particular they belong to the same $\Aut(\bD_G)$-orbit. In view of Lemma~\ref{lemma:same-orbits}, the cubes $\sigma_x,\sigma_y$ are in the same $G$-orbit. Since $Y_G$ is a strict fundamental domain for the $G$-action, it follows that $\sigma_x=\sigma_y$. But since automorphisms of $\bD_G$ preserve types of vertices, every automorphism that preserves $\sigma_x$ acts as the identity on $\sigma_x$. It follows that $x=y$, as desired. The first option is excluded.

We now assume that there exists a point $x$ in $\bD_G$ which is not in the $H$-orbit of any point of $Y_G$, and aim for a contradiction. Let $\sigma$ be a cube that contains $x$ and contains a rank $0$ vertex $v$. Then $\sigma$ is not in the $H$-orbit of any cube of $Y_G$. Denoting by $v_0$ the unique rank $0$ vertex of $Y_G$, it follows that $v$ is not in the $H$-orbit of $v_0$. 

The rank $0$ vertices $v_0$ and $v$ can be joined by an edge path, such that vertices along the path alternate between rank $0$ and rank $1$ vertices (see \cref{fig:CheminDeligne} for an example). Let $v_1$ be the last rank $0$ vertex along this path that belongs to the $H$-orbit of $v_0$, and let $v_2$ be the next rank $0$ vertex along this path. Then $v_1$ and $v_2$ are rank $0$ vertices of $\bD_G$ in distinct $H$-orbits which are adjacent to a common rank $1$ vertex $w$. 

\begin{figure}[htbp]
    \centering
    \includegraphics[width=0.6\textwidth]{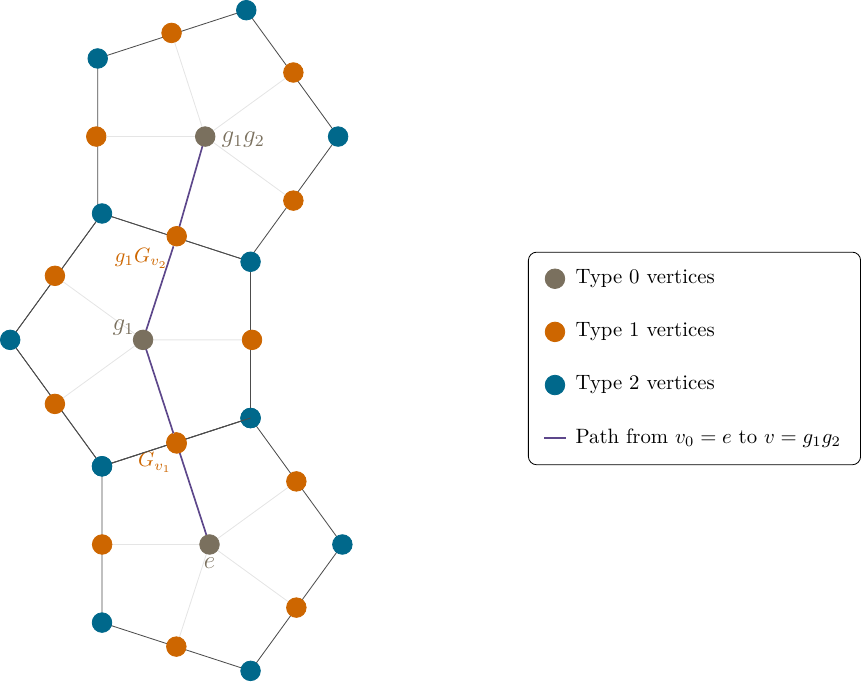}
    \caption{Illustration of a path going from $v_0=e$ to $v=g_1g_2$ (where $g_i\in G_{v_i}$, for $i\in \{1,2\}$) such that vertices on the path alternate between rank $1$ and $0$.}
    \label{fig:CheminDeligne}
\end{figure}

    Recall from Lemma~\ref{lemma:hh} that 
    \[\Omega_w:=\theta^{-1}(\Stab_{\Aut(\bD_G)}(w))\] is a measure equivalence coupling between the stabilizers $\Stab_G(w)$ and $\Stab_H(w)$. For every $i\in\{1,2\}$, let \[X_{w,i}:=\theta^{-1}(\{f\in\Stab_{\Aut(\bD_G)}(w)\mid f(v_i)=v_1\}).\] The action of $\Stab_G(w)$ on the set of rank $0$ vertices adjacent to $w$ is nothing but the action of $\Stab_G(w)$ by left translations on itself. In particular it is free and transitive, so for every $i\in\{1,2\}$, the set $X_{w,i}$ is a Borel fundamental domain for the action of $\Stab_G(w)$ on $\Omega_w$. In particular $m(X_{w,1})=m(X_{w,2})$.
    
    For every $i\in\{1,2\}$, the group $\Stab_H(v_i)$ is trivial (Corollary~\ref{cor:rank-0}). 
        Therefore, since $v_1$ and $v_2$ are not in the same $\Stab_H(w)$-orbit, the set $X_{w,1}\cup X_{w,2}$ meets almost every $\Stab_H(w)$-orbit in $\Omega_w$ at most once, and therefore it is contained in a Borel fundamental domain $Y_w$ for the action of $\Stab_H(w)$ on $\Omega_w$. In particular $m(Y_w)\ge 2 m(X_{w,1})$, so the coupling index satisfies  \[[\Stab_G(w):\Stab_H(w)]_{\Omega_w}\ge 2.\] This is a contradiction to Lemma~\ref{lemma:coupling-index-1}, which completes our proof.
\end{proof}

\subsection{Exploiting the structure of simple complex of groups}\label{sec:complex-of-groups}

In this final step, we will exploit the structure of $H$ as a simple complex of groups in the sense of Bridson--Haefliger \cite[Chapter~II.12]{BH} to get a fine control on the nature of the group $H$. We continue with the same general setting as in Section~\ref{sec:action-of-h}, and assume in addition that the group $G$ satisfies Hypothesis~$(\mathrm{H}_4)$ from Theorem~\ref{theo:rigidity-criterion}.

Let $\mathfrak{P}$ be the poset of cubes of $\bD_G$, ordered by inclusion. Then $\bD_G$ has a stratification indexed by $\mathfrak{P}$, whose strata are exactly the cubes of $\bD_G$: here \emph{stratification}\index{Stratification} is understood in the sense of \cite[Definition~II.12.1]{BH}, i.e.\ 
\begin{itemize}
\item $\bD_G$ is the union of its strata; 
\item the intersection of two strata, if non-empty, is again a stratum;
\item for every point in $\bD_G$, the intersection of all strata that contain $x$ is again a stratum (namely, it is the cube of smallest dimension that contains $x$).
\end{itemize}
The actions of $G$ and of $H$ on $\bD_G$ are strata-preserving, i.e.\ every element sends every stratum bijectively onto a stratum. And in view of Lemma~\ref{lemma:same-fundamental-domain}, the subcomplex $Y_G$ is a strict fundamental domain for the strata-preserving actions of $G$ and $H$ in the sense of \cite[Definition~II.12.7]{BH}: every cube of $\bD_G$ has a unique translate in $Y_G$.

In addition $\bD_G$ is simply connected (it is even $\mathrm{CAT}(0)$). Therefore, we can apply \cite[Corollary~II.12.22]{BH} and deduce that $H$ is the direct limit (amalgam) of the isotropy subgroups of the cubes contained in $Y_G$, in the sense of \cite[Definition~II.12.12]{BH}. The meaning of this is the following. Let \[\mathfrak{Q}=\{\sigma\in\mathfrak{P}\mid \sigma\subseteq Y_G\}.\] As before, for every cube $\sigma\subseteq Y_G$, we let $\mathrm{Stab}_H(\sigma)$ be the stabilizer of $\sigma$ for the $H$-action on $\bD_G$. Also, given cubes $\tau\subseteq\sigma$ both contained in $Y_G$, we let $\psi_{\tau\sigma}:\mathrm{Stab}_H(\sigma)\to \Stab_H(\tau)$ be the inclusion. Then $H$ is isomorphic to the group obtained from the free product of the groups $\mathrm{Stab}_H(\sigma)$, with $\sigma$ ranging in $\mathfrak{Q}$, by imposing as only extra relations that $\psi_{\tau\sigma}(h)=h$, whenever $\tau\subseteq\sigma$ both belong to $\mathfrak{Q}$ and $h\in \mathrm{Stab}_H(\sigma)$. 

We are now in position to prove the main theorem of this section.

\begin{proof}[Proof of Theorem~\ref{theo:rigidity-criterion}]
Let $v_1,\dots,v_n$ be the rank $1$ vertices in $Y_G$. Let $\sigma\subseteq Y_G$ be a cube, and let $w$ be its minimal vertex. Let $v_{i_1},\dots,v_{i_k}$ be the rank $1$ vertices that satisfy $v_{i_j}\le w$ (where the partial order $\le$ is just the inclusion of the cosets represented by these vertices). We say that $v_{i_1},\dots,v_{i_k}$ are the rank $1$ vertices \emph{below} $\sigma$ (notice that in the case where $\sigma$ contains the rank $0$ vertex of $Y_G$, there are no rank $1$ vertices below $\sigma$). For every $j\in\{1,\dots,k\}$, there is a cube $\sigma'_j$ in $Y_G$ that contains $v_{i_j}$ as its minimal vertex, and also contains $w$. So Corollary~\ref{cor:stab-cube} implies that, for every $j\in\{1,\dots,k\}$, we have
\[\Stab_{\Aut(\bD_G)}(v_{i_j})=\Stab_{\Aut(\bD_G)}(\sigma'_j)\subseteq\Stab_{\Aut(\bD_G)}(\sigma)=\Stab_{\Aut(\bD_G)}(w). \]
Therefore the $H$-stabilizer of $v_{i_j}$ is equal to $\Stab_H(\sigma'_j)$, and is contained in $\Stab_H(w)=\mathrm{Stab}_H(\sigma)$.

By Lemmas~\ref{lemma:hh} and~\ref{lemma:index-rigid}, we already know that the stabilizer $\mathrm{Stab}_H(\sigma)$ is abstractly isomorphic to $\mathrm{Stab}_G(\sigma)$, which splits as $\mathrm{Stab}_G(\sigma)=G_{v_{i_1}}\times\dots\times G_{v_{i_k}}$ (here we naturally identify the rank $1$ vertices $v_{i_1},\dots,v_{i_k}$ with vertices of the defining graph $\Gamma$, and the groups $G_{v_{i_j}}$ are exactly vertex groups for the graph product structure of $G$).

In addition, for every $j\in\{1,\dots,k\}$, we have an inclusion map $\Stab_H(v_{i_j})\to \mathrm{Stab}_H(\sigma)$, and $\Stab_H(v_{i_j})$ is isomorphic to $G_{v_{i_j}}$ (\cref{lemma:hh,lemma:index-rigid}). It thus follows from Hypothesis~$(\mathrm{H}_4)$ that the image of $\Stab_H(v_{i_j})$ in $\mathrm{Stab}_H(\sigma)$ is exactly equal to 
\[ \{1\}\times\dots\times\{1\}\times G_{v_{i_j}}\times\{1\}\times\dots\times\{1\},\]
under the above identification between $\mathrm{Stab}_H(\sigma)$ and $\mathrm{Stab}_G(\sigma)$. Therefore $\mathrm{Stab}_H(\sigma)$ splits as the direct product $\Stab_H({v_{i_1}})\times\dots\times  \Stab_H({v_{i_k}})$, in such a way that the inclusion maps $\Stab_H({v_{i_j}})\to  \mathrm{Stab}_H(\sigma)$ are just factor inclusions.

In other words, given any finite collection of rank $1$ vertices $v_{i_1},\dots,v_{i_k}$ spanning a clique when viewed as vertices of $\Gamma$, the subgroup generated by $\Stab_H(v_{i_1}),\dots, \Stab_H(v_{i_k})$ is equal to their direct product, and it is also equal to the stabilizer of any cube $\sigma$ such that the collection of rank $1$ vertices below $\sigma$ is exactly $\{v_{i_1},\dots,v_{i_k}\}$. From this, it follows that whenever $\sigma,\tau$ are two cubes with $\tau\subseteq\sigma$, the inclusion map $\mathrm{Stab}_H(\sigma)\to \Stab_H(\tau)$ is a factor inclusion.

It follows that the direct limit of the groups $\Stab_H(v_i)$ is nothing but their graph product over $\Gamma$, and it follows from \cite[Corollary~II.12.22]{BH} that this direct limit is isomorphic to $H$. Since $\Stab_H(v_i)$ is isomorphic to $G_{v_i}$, this shows that $H$ is isomorphic to $G$. 
\end{proof}

\section{An ME-rigid group which is not QI-rigid}\label{sec:higman}

In this section, we provide an example of a graph product that satisfies all assumptions from Theorem~\ref{theo:rigidity-criterion}. Our example will involve Higman groups \cite{Hig}, defined for $k\ge 4$ as \[\Hig_k=\langle a_1,\dots, a_k \mid a_ia_{i+1}a_i^{-1}=a_{i+1}^{2} \quad \forall i~(\text{mod~} k)\rangle.\] 
Our construction will enable us to prove the following theorem, which we recall from the introduction.

\begin{Th4}
    Let $\Gamma$ be a finite simple graph not reduced to one vertex, with no transvection and no partial conjugation, such that $\Aut(\Gamma)=\{1\}$. Let $(k_v)_{v\in V\Gamma}$ be a sequence of integers at least $5$, such that $k_v$ does not divide $k_w$ whenever $v\neq w$. Let $G$ be the graph product over $\Gamma$ with vertex groups $(\Hig_{k_v})_{v\in V\Gamma}$. Then
    \begin{enumerate}
        \item If a torsion-free countable group $H$ is measure equivalent to $G$, then $H$ is isomorphic to $G$.
        \item There exists an infinite family of torsion-free countable groups that are all quasi-isometric to $G$, and pairwise not commensurable.
    \end{enumerate}
\end{Th4}

\subsection{Homomorphisms between Higman groups}

The following proposition is an elaboration on work of Martin \cite{Mar2}, who proved the second assertion in the case where $k=4$. We will follow his proof very closely.

\begin{Prop}\label{prop:homomorphisms-higman}
Let $k,k'\ge 4$. Then
\begin{itemize}
    \item There exists a non-trivial homomorphism $\Hig_{k'}\to\Hig_k$ if and only if $k$ divides $k'$.
    \item Every non-trivial homomorphism $\Hig_k\to\Hig_k$ is an automorphism.
\end{itemize}
\end{Prop}

\begin{proof}
We let $\{a_1,\dots,a_k\}$ be a standard generating set of $\Hig_{k}$, and $\{a'_1,\dots,a'_{k'}\}$ be a standard generating set of $\Hig_{k'}$ (as in their standard presentations).

If $k$ divides $k'$, then the map sending the generator $a'_i$ of $\Hig_{k'}$ to the generator $a_{i}$ of $\Hig_k$ (with the index $i$ considered modulo $k$), extends to a non-trivial homomorphism from $\Hig_{k'}$ to $\Hig_k$.  

Conversely, assume that there exists a non-trivial homomorphism  $\theta:\Hig_{k'}\to\Hig_k$. We will prove that $k$ divides $k'$, and that if $k=k'$ then $\theta$ is an isomorphism, which will establish the proposition. Most of the proof follows \cite[Section~3]{Mar2} very closely. Martin's proof was written when $k=k'=4$, but the arguments readily extend to all $k,k'\ge 4$; we will provide a guide through the proof.

First, observe that $\theta$ cannot send any generator of $\Hig_{k'}$ to $1$, as otherwise the defining relations of $\Hig_{k'}$ would imply that every generator is sent to $1$, contradicting that $\theta$ is non-trivial.

As in \cite[Section~1.1]{Mar2} or \cite[Section~5]{Mar}, the group $\Hig_k$ (resp.\ $\Hig_{k'}$) acts on a $\mathrm{CAT}(0)$ polygonal complex $X_k$ (resp.\ $X_{k'}$) with strict fundamental domain a $2$-dimensional $k$-gon (resp.\ a $2$-dimensional $k'$-gon). For both actions, stabilizers of $2$-dimensional cells are trivial, edge stabilizers are isomorphic to $\mathbb{Z}$ (conjugate to the cyclic subgroups generated by the generators $a_i$ or $a'_i$), and vertex stabilizers are isomorphic to $\BS(1,2)$ (conjugate to the subgroups of the form $\langle a_i,a_{i+1}\rangle$, or $\langle a'_i,a'_{i+1}\rangle$, with indices understood modulo $k$ or $k'$). Moreover, for every edge $e=vw$ of $X_k$, exactly one of the homomorphisms $\Stab_G(e)\hookrightarrow \Stab_G(v)$ and $\Stab_G(e)\hookrightarrow \Stab_G(w)$ has undistorted image. As in \cite[Section~2.1]{Mar2}, we orient the edge towards the vertex $w$ such that the inclusion $\Stab_G(e)\hookrightarrow \Stab_G(w)$ is undistorted. If $e$ is oriented towards $w$, and $e'$ is any other edge incident on $w$, then $\Stab_G(e)\cap \Stab_G(e')=\{1\}$: this is the contents of \cite[Lemma~2.1]{Mar2}. In particular, as in \cite[Corollary~2.3]{Mar2}, the fixed point set $\mathrm{Fix}(h)$ of any non-trivial element of $\Hig_k$ is always contained in the outward star of a vertex, i.e.\ in the union of all edges emanating from that vertex ($\mathrm{Fix}(h)$ can be empty or reduced to one vertex).  

 As in \cite[Lemmas~2.4 and~3.6]{Mar2} (with the same proof), the homomorphism $\theta$ sends every subgroup $\langle a'_i,a'_{i+1}\rangle$ into a unique subgroup of the form $g \langle a_j,a_{j+1}\rangle g^{-1}$. In particular $\theta$ induces a map $\theta_\ast:VX_{k'}\to VX_k$. As in \cite[Lemma~3.7]{Mar2}, the map $\theta_\ast$ extends to an orientation-preserving map from the $1$-skeleton $X^1_{k'}$ to the $1$-skeleton $X^1_k$. And $\theta_\ast$ does not collapse any edge to a vertex (see the second paragraph of the proof of \cite[Lemma~3.8]{Mar2}).

Consider the vertices $v_1,\dots,v_{k'}$ of $X_{k'}$ corresponding to the subgroups $\langle a'_i,a'_{i+1}\rangle$ with $i$ varying in $\{1,\dots,k'\}$ (with indices considered modulo $k'$). They are aligned along a directed $k'$-gon $P_{k'}$. Therefore, there images under $\theta_\ast$ are aligned along an (immersed) oriented $k'$-gon $P$ in $X_k$. The projection of $P$ under the quotient map $X^1_{k}\to X^1_{k}/G$ is an immersed oriented $k'$-gon inside an oriented $k$-gon. This implies that $k$ divides $k'$, and proves the first part of the proposition.

We now prove the second assertion of the proposition, so from now on we assume that $k=k'$, and continue with the above notations. Since the case where $k=4$ was proved by Martin \cite{Mar2}, we will assume that $k\ge 5$. 

The map $\theta_\ast$ sends the boundary of any $2$-cell in $X_k$ to an immersed $k$-gon. And an immersed $k$-gon in $X_k$ is always the boundary of a $2$-cell in view of \cite[Lemma~4.4]{HH-Higman}. So $\theta_\ast$ sends the fundamental $k$-gon $P_k\subseteq X_k$ to a translate of $P_k$. Therefore, up to post-composing $\theta$ by an inner automorphism, we can assume that $\theta_\ast(P_k)=P_k$. In particular the map $(\theta_\ast)^k=(\theta^k)_\ast$ fixes $P_k$ pointwise. It follows from \cite[Lemma~3.9]{Mar2} (whose proof remains valid for every $k\ge 5$) that $\theta^k$ is the identity, in particular $\theta$ is an automorphism.
\end{proof}

\subsection{\texorpdfstring{Measure equivalence rigidity of the graph product $G$}{Measure equivalence rigidity of G}}

We are now in position to complete our proof of the first part of Theorem~\ref{theo:qi-me} on \cpageref{theo:qi-me}.

\begin{proof}[Proof of Theorem~\ref{theo:qi-me}(1)]
We apply the criterion provided by Theorem~\ref{theo:rigidity-criterion} to the graph product $G$. 

Hypothesis~$(\mathrm{H}_1)$ is given by \cite[Theorem~1.5]{HH-Higman}.

For Hypothesis~$(\mathrm{H}_2)$, notice that the vertex groups, being Higman groups, are torsion-free (as recorded for example in \cite[Lemma~2.6]{HH-Higman}) and have no proper finite-index subgroups (as proved by Higman \cite{Hig}). They are acylindrically hyperbolic (see \cite[Corollary~4.26]{MO} or \cite[Theorem~B]{Mar}), and therefore they belong to the class $\calc_\reg$ (see Example~\ref{ex:creg}). This verifies Hypothesis~$(\mathrm{H}_2)$ from Theorem~\ref{theo:rigidity-criterion}. 

For Hypothesis~$(\mathrm{H}_3)$, notice that if $\Hig_{k_v}$ is measure equivalent to $\Hig_{k_w}$, then \cite[Theorem~1.1]{HH-Higman} implies that $\Hig_{k_v}$ and $\Hig_{k_w}$ are virtually isomorphic (i.e.\ isomorphic up to taking the quotient by a finite normal subgroup, and passing to a finite-index subgroup). Since Higman groups are torsion-free and have no proper finite-index subgroup, it follows that $\Hig_{k_v}$ is isomorphic to $\Hig_{k_w}$. In view of Proposition~\ref{prop:homomorphisms-higman}, this is impossible unless $v=w$. This checks Hypothesis~$(\mathrm{H}_3)$. 

Finally, Hypothesis~$(\mathrm{H}_4)$ is a consequence of Proposition~\ref{prop:homomorphisms-higman}, as we are assuming that $k_v$ does not divide $k_w$ whenever $v\neq w$. 

The conclusion thus follows from Theorem~\ref{theo:rigidity-criterion}.
\end{proof}

\subsection{Flexibility in quasi-isometry}

Recall that two finitely generated groups $G,H$ are \emph{bi-Lipschitz equivalent} if there exists a bi-Lipschitz bijection from $G$ to $H$, equipped with word metrics associated to finite generating sets.

\begin{lemma}\label{lemma:qi}
Let $G,H$ be two graph products of finitely generated groups over a finite simple graph $\Gamma$. Assume that for every $v\in V\Gamma$, the vertex groups $G_v$ and $H_v$ are bi-Lipschitz equivalent.

Then $G$ and $H$ are bi-Lipschitz equivalent.
\end{lemma}

\begin{proof}
Through normal forms, a family of bi-Lipschitz bijections from $G_v$ to $H_v$ induces a bi-Lipschitz bijection from $G$ to $H$.  
\end{proof}

\begin{proof}[Proof of Theorem~\ref{theo:qi-me}(2)]
Let $v\in V\Gamma$. For every $n\in\mathbb{N}$, let $G_n$ be the graph product over $\Gamma$ whose vertex groups are the same as $G$, except that $G_v$ is replaced by $G_v\times\mathbb{Z}/n\mathbb{Z}$.

The group $G_v$ is non-amenable (it is even acylindrically hyperbolic \cite[Corollary~4.26]{MO}). Therefore, for every $n\in\mathbb{N}$, the groups $G_v$ and $G_v\times\mathbb{Z}/n\mathbb{Z}$ are bi-Lipschitz equivalent \cite{Why}, so Lemma~\ref{lemma:qi} ensures that all groups $G_n$ are bi-Lipschitz equivalent and therefore quasi-isometric.

On the other hand, since $G_v$ has no proper finite-index subgroup \cite{Hig}, when $n\neq m$ the groups $G_v\times\mathbb{Z}/n\mathbb{Z}$ and $G_v\times\mathbb{Z}/m\mathbb{Z}$ are never strongly commensurable (and they are not strongly commensurable to any $G_w$ either, as can be seen using Proposition~\ref{prop:homomorphisms-higman}). It thus follows from Theorem~\ref{theo:classification-commensurability} on \cpageref{theo:classification-commensurability}, that $G_n$ and $G_m$ are not commensurable for $n\neq m$.

Finally, the groups $G_n$ are not torsion-free, but they all admit a finite-index torsion-free subgroup. Indeed, a normal form argument shows that every torsion element of $G_n$ is $G_n$-conjugate to an element of the subgroup $\{e\}\times \mathbb{Z}/n\mathbb{Z} \subseteq G_v\times\mathbb{Z}/n\mathbb{Z}$. Let $\theta_n:G_n\to\mathbb{Z}/n\mathbb{Z}$ be the homomorphism which coincides with the second projection on the vertex group $G_v\times\mathbb{Z}/n\mathbb{Z}$, and is the identity on every vertex group $G_w$ with $w\neq v$. Then the kernel of $\theta_n$ is a torsion-free finite-index subgroup $G_n^0$ of $G_n$.

The groups $G_n^0$ are then torsion-free, and they are all quasi-isometric to $G$, but pairwise non-commensurable.
\end{proof}

\newpage
\part*{Appendix}
\addcontentsline{toc}{part}{Appendix}

\appendix \section{On the isomorphism problem for graph products}\label{Appendix:Isom}

We refer to Definition~\ref{de:strongly-reduced}, \cpageref{de:strongly-reduced} for the definition of a strongly reduced graph. The main goal of this appendix is to prove the following theorem, see also Theorem~\ref{theo:conjugating-automorphism-join} below for a version including the case of direct products.

\begin{Th}\label{theo:conjugating-automorphism}
    Let $G,H$ be graph products over strongly reduced finite simple graphs $\Gamma_G,\Gamma_H$, not reduced to a vertex.
    
    Then for every isomorphism $f:G\to H$ and every untransvectable vertex $v\in V\Gamma_G$, there exists an untransvectable vertex $w\in V\Gamma_H$ such that $f(G_v)$ is conjugate to $H_w$.
\end{Th}

When combined with \cite[Theorem~3.11]{GM}, Theorem~\ref{theo:conjugating-automorphism} has the following consequence to the isomorphism problem for graph products.

\begin{Cor}\label{cor:isomorphism}
Let $G,H$ be graph products over strongly reduced finite simple graphs $\Gamma_G,\Gamma_H$, not reduced to a vertex, with $\Gamma_G$ transvection-free. 

If $G$ and $H$ are isomorphic, then there exists an isomorphism $\sigma:\Gamma_G\to \Gamma_H$ such that for every $v\in V\Gamma_G$, the vertex groups $G_v$ and $H_{\sigma(v)}$ are isomorphic.
\end{Cor}

\begin{proof}
Let $f:G\to H$ be an isomorphism. Since $\Gamma_G$ is transvection-free, every vertex $v\in V\Gamma_G$ is untransvectable. Hence, Theorem~\ref{theo:conjugating-automorphism} implies that $f$ sends every vertex group $G$ to a conjugate of a vertex group of $H$. The conclusion therefore follows from \cite[Theorem~3.11]{GM}.
\end{proof}

Another consequence of Theorem~\ref{theo:conjugating-automorphism}, regarding the acylindrical hyperbolicity of $\Aut(G)$, will be derived in Section~\ref{sec:acyl-hyp}.

\begin{Rq}[Comparison with the framework of measured groupoids]\label{rk:groups-groupoids}
In the case where the defining graphs $\Gamma_G$ and $\Gamma_H$ are strongly reduced and transvection-free, Theorem~\ref{theo:conjugating-automorphism} ensures that $f$ sends every vertex group to a conjugate of a vertex group. When $G$ and $H$ are countable, this can be seen as an analogue of the Vertex Recognition Property from our measured group-theoretic framework (Definition~\ref{de:VRP}). Notice however that Theorem~\ref{theo:conjugating-automorphism} has two advantages compared to the work in Part~\ref{Part:MEClassification}, namely:
\begin{itemize}
    \item it allows for finite vertex groups, and also for uncountable vertex groups;
    \item even when the graph is not transvection-free, it enables to recognize untransvectable vertices; in this respect, this is similar to our work in Section~\ref{Sec:AmenalbeUntransvectable}, though in the latter section vertex groups were assumed to be amenable.
\end{itemize}
We will say more on the comparison between the group-theoretic and groupoid-theoretic frameworks in Remark~\ref{rk:groups-vs-groupoids}.
\end{Rq}

\begin{Rq}[Comparison with Genevois's work]\label{rk:genevois}
We now compare our assumptions in Theorem~\ref{theo:conjugating-automorphism} with those of \cite[Theorem~8.1]{Gen}. We have the extra assumption that $\Gamma_G$ and $\Gamma_H$ are strongly reduced. On the other hand, we do not assume that the vertex groups are graphically irreducible (a group $G_v$ is \emph{graphically irreducible}\index{Graphically irreducible} if whenever it decomposes as a graph product over a finite simple graph, it has to be over a clique \cite[Definition~3.2]{Gen}). 

Both assumptions (strongly reduced graphs and graphically irreducible vertex groups) cannot be removed simultaneously. In a sense, we decompose $G,H$ \emph{minimally} as non-trivial graph products, while Genevois decomposes them \emph{maximally} as graph products. 

None of the theorems implies the other. In particular, our theorem allows for infinitely generated vertex groups, which sometimes do not admit any maximal decomposition as a graph product over a finite simple graph (e.g.\ $F_\infty$). In fact, it is also unknown whether or not a finitely generated group always admits a decomposition as a graph product where every vertex group is graphically irreducible.  
\end{Rq}

\subsection{Proof of Theorem~\ref{theo:conjugating-automorphism}}

Our proof of Theorem~\ref{theo:conjugating-automorphism} will follow the one in Section~\ref{sec:strongly-reduced} quite closely, the main novelty being the use of thick free factors in Sections~\ref{sec:thick} and~\ref{sec:thick-free-factors}. It can be read independently from Section~\ref{sec:strongly-reduced}, and in fact can serve as a warm-up before reading its groupoid-theoretic version. To emphasize the parallel between this appendix and Section~\ref{sec:strongly-reduced}, the group-theoretic counterparts of Properties~$\Pprod,\Padm$, etc.\ from Section~\ref{sec:strongly-reduced} will be called $\Qprod,\Qadm$, etc.

The goal is to characterize untransvectable vertex groups by a purely group-theoretic property that will be preserved under group isomorphisms.

\subsubsection{Thick free factors}\label{sec:thick}

A \emph{free factor}\index{Free factor}\index{Factor!Free Factor} of a group $G$ is a subgroup $A\subseteq G$ such that there exists $B\subseteq G$ with $G=A\ast B$. By Bass--Serre theory, $A$ is a free factor of $G$ if and only if there exists an action of $G$ on a simplicial tree $T$ with trivial edge stabilizers, such that $A$ is equal to the stabilizer of a vertex of $T$.

Notice for future use that if $A$ is a free factor of $G$, and $P$ is a subgroup of $G$ that contains $A$, then $A$ is a free factor of $P$: indeed, letting $T$ be a $G$-tree with trivial edge stabilizers with one vertex stabilizer equal to $A$, the tree $T$ is also a $P$-tree with trivial edge stabilizers with one vertex stabilizer equal to $A$.

A group is \emph{freely indecomposable}\index{Freely indecomposable} if it does not admit any action on a simplicial tree with trivial edge stabilizers with no global fixed point. In particular, the only free factors of a freely indecomposable group $G$ are $\{1\}$ and $G$.

Let $G$ be a graph product over a finite simple graph $\Gamma$. Let $\Lambda_1,\dots,\Lambda_k$ be the connected components of $\Gamma$ that contain at least two vertices, and let $v_1,\dots,v_\ell$ be the isolated vertices of $\Gamma$. Then 
\[G=G_{\Lambda_1}\ast\dots\ast G_{\Lambda_k}\ast G_{v_1}\ast\dots\ast G_{v_\ell}.\] 
A subgroup of $G$ which is conjugate to $G_{\Lambda_j}$ for some $j\in\{1,\dots,k\}$ will be called a \emph{thick free factor}\index{Thick free factor}\index{Free factor!Thick free factor}\index{Factor!Thick free factor} of $G$. 

Notice that thick free factors of $G$ are indeed free factors of $G$, and they are freely indecomposable. Conversely, we make the following observation.

\begin{lemma}\label{lemma:freely-indecomposable}
    Let $G$ be a graph product over a finite simple graph $\Gamma$, and let $L\subseteq G$ be a freely indecomposable free factor of $G$.
    
    Then either $L$ is a thick free factor, or else $H$ is conjugate to a subgroup of $G_v$ for some isolated vertex $v\in V\Gamma$.
\end{lemma}

\begin{proof}
Write $G=G_{\Lambda_1}\ast\dots\ast G_{\Lambda_k}\ast G_{v_1}\ast\dots\ast G_{v_\ell}$, where the $G_{\Lambda_j}$ are thick, and each $v_j$ is an isolated vertex. Let $T$ be the Bass--Serre tree of this decomposition of $G$ as a free product: this is a $G$-tree with trivial edge stabilizers.

Being freely indecomposable, the group $H$ fixes a vertex of $T$. So
\begin{itemize}
\item either $L$ is conjugate to a subgroup of $G_{v_j}$ for some isolated vertex $v_j$ (and we are done), 
\item or else $L$ is non-trivial and conjugate to a subgroup of a thick free factor $G_{\Lambda_j}$ of $G$.
\end{itemize}
We assume that the latter case holds, and without loss of generality that $H\subseteq G_{\Lambda_j}$. Since $L$ is a free factor of $G$, it is also a free factor of $G_{\Lambda_j}$. And since $G_{\Lambda_j}$ is freely indecomposable and $L$ is non-trivial, we have $L=G_{\Lambda_j}$. 
\end{proof}

\subsubsection{A combinatorial lemma}

Our starting point for proving Theorem~\ref{theo:conjugating-automorphism} is the following variation over Lemma~\ref{lemma:graph}, \cpageref{lemma:graph}.

\begin{lemma}\label{lemma:graph-2}
    Let $\Gamma$ be a strongly reduced finite simple graph not reduced to one vertex, and let $G$ be a graph product over $\Gamma$. Let $B\subseteq G$ be a parabolic subgroup. The following two assertions are equivalent. 
    \begin{enumerate}
     \item $B$ is conjugate to $G_v$ for some untransvectable vertex $v\in V\Gamma$.
     \item There exist $n\geq 1$ and a chain of parabolic subgroups
    \[G=F_0\supseteq L_1\supseteq P_1\supseteq F_1\supseteq\dots\supseteq L_n\supseteq P_n\supseteq F_n=B\]
    such that
    \begin{enumerate}
    \item for every $j\in\{1,\dots,n\}$, $L_j$ is a thick free factor of $F_{j-1}$ whose type is not a clique,
        \item for every $j\in\{1,\dots,n\}$, $P_j$ is a maximal product parabolic subgroup of $L_{j}$,
        \item for every $j\in\{1,\dots,n-1\}$, the clique factor of $P_j$ is trivial, and $F_j$ is a factor of $P_j$, 
        \item $F_n$ is the clique factor of $P_n$,
        \item for every non-trivial subgroup $W\subseteq F_n$, one has $N_G(W)\subseteq N_G(F_n)$.
    \end{enumerate}
    \end{enumerate}
\end{lemma}

In our proof, we will use the following fact: every subgroup $A\subseteq G$ is contained in a unique smallest parabolic subgroup $\tilde{A}$, called the \emph{parabolic support}\index{Parabolic!Parabolic support} of $A$ (it is also the intersection of all parabolic subgroups that contain $A$). And one has $N_G(A)\subseteq N_G(\tilde{A})$. 

Remark for future use, that it implies that if $N$ is a subgroup of $G$ that normalizes $A$, then $\tilde{N}$ normalizes $\tilde{A}$.

\begin{proof}
We first prove that $(1)\Rightarrow (2)$.
\smallskip

Without loss of generality we will assume that $B=G_v$, with $v$ untransvectable. 
Consider subgroups $F_j,P_j$ given by Lemma~\ref{lemma:graph}. In particular they satisfy Assertions~(c) and~(d). Since $P_j$ is a maximal product parabolic subgroup of $F_{j-1}$, it is contained in a thick free factor $L_j$ of $F_{j-1}$. And $P_j$ is a maximal product parabolic subgroup of $L_j$. Additionally, since $P_j$ is not of isolated clique type, the type of $L_j$ is not a clique. So Assertions~(a) and~(b) are also satisfied. Finally, Assertion~(e) follows from the fact that $F_n=G_v$ is the parabolic support of any non-trivial subgroup $W\subseteq F_n$.
\medskip

We now prove that $(2)\Rightarrow (1)$.
\smallskip

Assume that we have subgroups $F_j,L_j,P_j$ as in the statement. We will prove that
\begin{itemize}
    \item the subgroups $F_j$ and $P_j$ satisfy the three conditions from Lemma~\ref{lemma:graph}, and
    \item $F_n=B$ is conjugate to a vertex group $G_v$. 
\end{itemize}
Once these two facts are proved, the untransvectability of $v$ follows from Lemma~\ref{lemma:graph}. 

We first prove the second fact. Assertion~(d) ensures that $F_n$ is conjugate to $G_{\Upsilon}$, for some complete subgraph $\Upsilon\subseteq\Gamma$. If $|V\Upsilon|\ge 2$, the fact that $\Gamma$ is strongly reduced ensures that $\Upsilon$ is not collapsible, so there exists a vertex $w\in V\Upsilon$ such that $N_G(G_w)\nsubseteq N_G(G_\Upsilon)$. This contradicts Assertion~(e). Therefore $|V\Upsilon|=1$, in other words $F_n$ is conjugate to a vertex group.

We now prove the first fact. The second and third conditions from Lemma~\ref{lemma:graph} are exactly Assertions~(c) and~(d). We focus on the first condition, so let $j\in \{1,\dots,n\}$. Assertions~(a) and~(b) ensure that $P_j$ is a product parabolic subgroup of $F_{j-1}$. We observe that as such, it is maximal. Indeed, if $P_j\subseteq P'$ with $P'\subseteq F_{j-1}$ a product parabolic subgroup, then the type of $P'$ is a connected subgraph of the type of $F_{j-1}$, and therefore $P'$ must be contained in the free factor $L_j$. There remains to prove that $P_j$ is not of isolated clique type. If $j\in\{1,\dots,n-1\}$, this follows from the fact that the clique factor of $P_j$ is trivial. If $j=n$, since $L_n$ is a thick free factor, $P_n$ is not conjugate to a vertex group. And since the clique factor of $P_n$ is $F_n$, which is conjugate to a vertex group (by the previous paragraph), the type of $P_n$ cannot be a clique on at least $2$ vertices either. So $P_n$ is not of isolated clique type. 
\end{proof}

In view of Lemma~\ref{lemma:graph-2}, we are left with giving a group-theoretic characterization of thick free factors, maximal products and their factors and clique factors.

\begin{Rq}[More on the comparison with the groupoid framework]\label{rk:groups-vs-groupoids}
As a follow-up to Remark~\ref{rk:groups-groupoids}, let us continue our comparison between the group-theoretic and groupoid-theoretic frameworks, and explain why we reach a stronger conclusion in this appendix than in Sections~\ref{sec:strongly-reduced} and~\ref{Sec:AmenalbeUntransvectable} (namely, a characterization of all conjugates of untransvectable vertex groups, without assuming $\Gamma$ to be transvection-free, and with no assumption on the vertex groups).

In the zoom-in process described in Lemmas~\ref{lemma:graph} and~\ref{lemma:graph-2}, it can happen that the type $\Upsilon_j$ of one of the factors $F_j$ is disconnected. A delicate situation for us is when one of the connected components of $\Upsilon_j$ is a clique (possibly just one vertex). In Section~\ref{sec:strongly-reduced}, when this happens, Property~$\Pprod$ did not allow us to recognize maximal products in $F_j$ from subgroup(oid)s that are contained in an isolated clique factor. This is why in Step~4 of the proof of Lemma~\ref{lemma:characterization-vertex-groups} (the key lemma that characterizes subgroupoids of vertex type), we had to use an extra argument, which crucially relied on the fact that $\Gamma$ was transvection-free, to distinguish actual subgroupoids of vertex type from those that are only contained in a subgroupoid of clique type.

In Section~\ref{Sec:AmenalbeUntransvectable}, the situation was slightly different: since we were assuming that all vertex groups were amenable, the isolated clique subgroups were always amenable, which enabled us to easily distinguish them from the other maximal product subgroups.

In this appendix, the notion of a thick free factor is the crucial tool to discard isolated vertex groups that arise in the process. This is why it is important for us to introduce the groups $L_j$ in Lemma~\ref{lemma:graph-2}, that we were not using in Part~\ref{Part:MEClassification} of this work.
\end{Rq}

\subsubsection{Maximal products are preserved}

The following lemma is a variation on Lemma~\ref{lemma:maximal-product-nonamenable}, \cpageref{lemma:maximal-product-nonamenable}.

\begin{lemma}\label{lemma:product}
    Let $G$ be a non-abelian graph product over a connected finite simple graph $\Gamma_G$ which is not reduced to a vertex.

    If $P_1,P_2$ are parabolic subgroups such that $P=P_1\times P_2$ is a maximal product parabolic subgroup, then either $P_1$ or $P_2$ is non-cyclic.  
\end{lemma}

 \begin{proof}
 We first treat the case where $\Gamma_G$ is a clique. Since $G$ is non-abelian, either $P_1$ or $P_2$ is non-cyclic.
 
 We now assume that $\Gamma_G$ is not a clique. Let $\Lambda$ be the type of $P$. Without loss of generality, we will assume that $P=G_{\Lambda}$, and that $\Lambda=\Lambda_1\circ\Lambda_2$, with $P_1=G_{\Lambda_1}$ and $P_2=G_{\Lambda_2}$.
 
 We first prove that $\Lambda$ is not a clique. 
 \begin{itemize}
 \item If $\Lambda=\Gamma_G$, by assumption its type is not a clique. 
 \item Otherwise, since $\Gamma_G$ is connected, we can find two adjacent vertices $v,w$, with $v\in V\Lambda$ and $w\notin V\Lambda$. Then $G_{\st(v)}$ is a parabolic subgroup of product type, which is different from $P$ because it contains $G_w$. By maximality of $P$, we have $P\nsubseteq G_{\st(v)}$, which implies that $\Lambda\nsubseteq\st(v)$. So $\Lambda$ is not a clique.
\end{itemize}
 
 We can therefore assume without loss of generality that the type $\Lambda_1$ of $P_1$ contains two non-adjacent vertices $v,w$. Then $P_1$ contains a subgroup isomorphic to $G_v*G_w$, and is therefore non-cyclic.
 \end{proof}

The following definition can be seen as a group theoretic analogue of Property~$\Pprod$ (see \cref{de:pprod}, \cpageref{de:pprod}).

\begin{Def}[Property~$\Qprod$]\index{Property!$\Qprod$}\label{de:qprod}
Let $G$ be a group, and $P\subseteq G$ be a subgroup. We say that the pair $(G,P)$ satisfies Property~$\Qprod$ if it verifies the following assertions.
\begin{description}
    \item[$\Qprod_1$] There exists a non-trivial element $g\in P$, and a non-cyclic subgroup $N\subseteq P$, such that $N$ centralizes $g$ and is normal in $P$. 
    \item[$\Qprod_2$] $P$ is maximal for inclusion among all subgroups $P'\subseteq G$ that satisfy Property~$\Qprod_1$. 
\end{description}
\end{Def}

\begin{Rq}\label{rk:iso}
This property is stable under group isomorphisms. That is, if $f:G\to H$ is an isomorphism between two groups, and if $P\subseteq G$ is a subgroup such that $(G,P)$ satisfies Property~$\Qprod$, then $(H,f(P))$ satisfies Property~$\Qprod$. All properties considered in the present section will be stable under group isomorphisms.
\end{Rq}

The following lemma is a group theoretic analogue of \cref{lemma:maximal-product}, \cpageref{lemma:maximal-product}. 

\begin{lemma}\label{lemma:qprod}
Let $G$ be a non-abelian graph product over a connected finite simple graph $\Gamma_G$ which is not reduced to a vertex. Let $P\subseteq G$ be a subgroup. 

Then $P$ is a maximal product parabolic subgroup if and only if $(G,P)$ satisfies Property~$\Qprod$.
\end{lemma}

\begin{proof} 
\textbf{Step 1.} Let us first prove that if $P$ is a maximal product parabolic subgroup, then $(G,P)$ satisfies Property~$\Qprod_1$.
\smallskip

Write $P=P_1\times P_2$ as a product of two parabolic subgroups, with $P_1$ non-cyclic (see Lemma~\ref{lemma:product}). Then Property~$\Qprod_1$ is satisfied by letting $g\in P_2$ be any non-trivial element, and letting $N=P_1$. 

\medskip

\noindent\textbf{Step 2.} Let us now prove that if $(G,P)$ satisfies Property~$\Qprod_1$, then $P$ is contained in a product parabolic subgroup.
\smallskip

So let $g\in P$ and $N\subseteq P$ be as in Property~$\Qprod_1$. Let $\tilde A\subseteq G$ be the parabolic support of $g$, and let $\tilde{N}$ be the parabolic support of $N$. Notice that $\tilde A$ and $\tilde{N}$ are non-trivial (because $g$ and $N$ are non-trivial). Then $\tilde{N}\subseteq \tilde A\times \tilde A^{\perp}$, and $P\subseteq \tilde{N}\times\tilde{N}^{\perp}$.
\begin{itemize}
    \item If $\tilde{N}^\perp\neq\{1\}$, then $P$ is contained in $\tilde{N}\times \tilde{N}^\perp$, which is a product of two non-trivial subgroups. 
    \item Now assume that $\tilde{N}^\perp=\{1\}$, then $P\subseteq \tilde{N}\subseteq \tilde{A}\times \tilde{A}^\perp$.
    \begin{itemize}
    \item If $\tilde{A}$ is conjugate to a vertex group, then $\tilde{A}^{\perp}\neq\{1\}$ because $\Gamma_G$ is connected and not reduced to one vertex. So we are done in this case.
    \item If $\tilde{A}$ splits non-trivially as a product, we are done.
    \item Otherwise, we can apply \cite[Theorem~56]{Bar}, which says that if $\tilde{A}$ is not conjugate to a vertex group and does not split as a product, then $Z_G(g)=\langle g\rangle\times \tilde{A}^{\perp}$. Since the centralizer of $g$ is non-cyclic, we deduce that $\tilde{A}^{\perp}\neq\{1\}$, and again we are done.
    \end{itemize}
\end{itemize}
    
    \medskip
    
\noindent \textbf{Step 3.} We now prove that if $P$ is a maximal product parabolic subgroup, then $(G,P)$ satisfies Property~$\Qprod_2$.\smallskip

So let $P'\subseteq G$ be a subgroup that satisfies Property~$\Qprod_1$, with $P\subseteq P'$. Step~2 implies that $P'$ is contained in a product parabolic subgroup $Q$. Then $P\subseteq Q$, and by maximality $P=Q$. In particular $P=P'$, as desired.
    
\medskip

\noindent \textbf{Step 4.} We are left with proving that if $(G,P)$ satisfies Property~$\Qprod$, then $P$ is a maximal product parabolic subgroup. 
\smallskip

By Step~2, $P$ is contained in a product parabolic subgroup $Q$. Up to increasing $Q$, we will assume that $Q$ is a \emph{maximal} product parabolic subgroup, so Step~1 ensures that $Q$ satisfies Property~$\Qprod_1$. And Property~$\Qprod_2$ for $P$ thus ensures that $P=Q$, so $P$ is a maximal product parabolic subgroup.
\end{proof}

\subsubsection{Thick free factors are preserved}\label{sec:thick-free-factors}

\begin{Def}[Property $\Qisol$]\index{Property!$\Qisol$}
    Let $F$ be a group, and let $L\subseteq F$ be a subgroup. We say that the triple $(G,F,L)$ satisfies Property~$\Qisol$ if the following three properties hold.
    \begin{description}
    \item[$\Qisol_1$] $L$ is a free factor of $F$.
    \item[$\Qisol_2$] $L$ is freely indecomposable. 
    \item[$\Qisol_3$] There exists a non-trivial subgroup $B\subseteq L$ such that $N_G(B)\nsubseteq N_G(L)$.
    \end{description}
\end{Def}

\begin{lemma}\label{lemma:isol}
Let $G$ be a graph product over a strongly reduced finite simple graph $\Gamma$, let $F\subseteq G$ be a parabolic subgroup, and let $L\subseteq F$ be a subgroup. 

Then $L$ is a thick free factor of $F$ if and only if $(G,F,L)$ satisfies Property~$\Qisol$.
\end{lemma}

\begin{proof}
Without loss of generality, we will assume that $F=G_{\Lambda}$, for some induced subgraph $\Lambda\subseteq\Gamma$.

We first assume that $L$ is a thick free factor -- in particular $L$ is a free factor and is freely indecomposable, so Properties~$\Qisol_1$ and~$\Qisol_2$ hold. To prove $\Qisol_3$, without loss of generality, we will assume that $L=G_\Upsilon$, for some induced subgraph $\Upsilon\subseteq\Lambda\subseteq\Gamma$ which is a connected component of $\Lambda$ on at least two vertices. Since $\Gamma$ is strongly reduced, there exists $v\in V\Upsilon$ such that $\lk_{\Gamma}(v)\nsubseteq\Upsilon\circ\Upsilon^{\perp}$. Letting $B=G_v$ shows that Property~$\Qisol_3$ holds. 

\medskip

Conversely, we now assume that $(G,F,L)$ satisfies Property~$\Qisol$. Since $L$ is a freely indecomposable free factor of $G_\Lambda$, Lemma~\ref{lemma:freely-indecomposable} ensures that
\begin{itemize}
\item either $L$ is thick, and we are done,
\item or else there exists an isolated vertex $v\in V\Lambda$, such that $L$ is contained in a conjugate of $G_v$.
\end{itemize}
We assume that we are in the second situation, and without loss of generality that $L\subseteq G_v$, and we aim for a contradiction. Since $L$ is a free factor of $G$ contained in $G_v$, we deduce that $L$ is a free factor of $G_v$, i.e.\ $G_v$ splits as $G_v=L\ast L'$ for some subgroup $L'\subseteq G_v$ (possibly with $L'=\{1\}$). 

We now observe that $G$ splits as a graph product over a graph $\tilde\Gamma$, with one vertex group equal to $L$: indeed,
\begin{itemize}
\item if $L=G_v$ and $L'=\{1\}$, we simply let $\tilde\Gamma=\Gamma$;
\item otherwise, $\tilde\Gamma$ is obtained from $\Gamma$ by replacing $v$ by two non-adjacent vertices $w,w'$, and joining both $w$ and $w'$ by an edge to all vertices of $\Gamma$ that are adjacent to $v$; then $G$ is a graph product over $\tilde{\Gamma}$ with $G_w=L$ and $G_{w'}=L'$.
\end{itemize}
Now, if $B\subseteq L$ is a non-trivial subgroup, then $L$ is the parabolic support of $B$ for the graph product structure over $\tilde{\Gamma}$. Therefore $N_G(B)\subseteq N_G(L)$, which contradicts Property~$\Qisol_3$.
\end{proof}

\subsubsection{Special subproducts}

Let $P=\mathsf{F}_1\times\dots\times\mathsf{F}_n$ be a product parabolic subgroup of $G$, written in such a way that no $\mathsf{F}_j$ further splits as a direct product of parabolic subgroups. In particular the clique factor of $P$ is decomposed as the product of factors that are conjugate to vertex groups. Recall from \cref{sec:factors}, \cpageref{sec:factors} that a \emph{subproduct} is a subgroup of $P$ equal to the product of finitely many factors $\mathsf{F}_j$ (possibly zero or one). A subproduct $S$ of $P$ is \emph{special} if $N_G(S)\nsubseteq N_G(P)$.

The following definition is the group theoretic analogue of \cref{de:pspec}, \cpageref{de:pspec}.

\begin{Def}[Property $\Qspec$]\index{Property!$\Qspec$}\label{de:qspec}
Let $G$ be a group, and $S\subseteq P$ be subgroups of~$G$. We say that the triple $(G,P,S)$ satisfies \emph{Property $\Qspec$} if $S\unlhd P$, and $N_G(S)\nsubseteq N_G(P)$.
\end{Def}

Notice that if $G$ is a graph product and $P$ is a product parabolic subgroup of $G$, and if $S$ is a special subproduct of $P$, then $(G,P,S)$ satisfies Property~$\Qspec$. Conversely, we have the following statement (analogous to \cref{lemma:spec}, \cpageref{lemma:spec}). 

\begin{lemma}\label{lemma:spec-groups}
Let $G$ be a graph product over a finite simple graph $\Gamma$. Let $P\subseteq G$ be a product parabolic subgroup, and $S\subseteq P$ be a subgroup. 

If $(G,P,S)$ satisfies Property~$\Qspec$, then the parabolic support of $S$ is a special subproduct of~$P$.
\end{lemma}

\begin{proof}
Let $\tilde{S}$ be the parabolic support of $S$. By assumption $S\unlhd P$, so $\tilde{S}\unlhd P$, and therefore $P\subseteq \tilde{S}\times \tilde{S}^{\perp}$. So $\tilde{S}$ is a subproduct of $P$. In addition, since $N_G(S)\nsubseteq N_G(P)$, we have $N_G(\tilde{S})\nsubseteq N_G(P)$, i.e.\ $\tilde{S}$ is special. In summary $S$ is contained in the special subproduct~$\tilde{S}$.
\end{proof}

\subsubsection{Clique-inclusive co-factors}

Recall from Section~\ref{sec:factors}, \cpageref{sec:factors} that if $P$ is a product parabolic subgroup of $G$, which splits as $P=C_0\times F_1\times\dots\times F_k$, where $C_0$ is the clique factor, and no $F_j$ further decomposes as a direct product of two parabolic subgroups, then a \emph{clique-inclusive co-factor} is a subgroup of $P$ of the form $C_0\times F_1\times\dots\times\hat{F}_j\times\dots\times F_k$, for some $j\in\{1,\dots,k\}$. By convention, if $P=C_0$, there is no clique-inclusive co-factor.
The following definition is analogous to \cref{de:padm}, \cpageref{de:padm}.

\begin{Def}[Property~$\Qadm$]\index{Property!$\Qadm$}\label{de:qadm}
 Let $G$ be a group, and $Q\subseteq P$ be subgroups. We say that the triple $(G,P,Q)$ satisfies \emph{Property~$\Qadm$}\index{Property!$\Qadm$} if the following three properties hold.
\begin{description}
\item[$\Qadm_1$] The subgroup $Q$ is normal in $P$.
    \item[$\Qadm_2$] There exists a non-trivial subgroup $B\subseteq P$, which is normalized both by $Q$ and by a subgroup $N\subseteq G$ such that 
    \begin{enumerate}
        \item[(a)] we have $N\nsubseteq N_G(P)$;
        \item[(b)] whenever $S\subseteq P$ is a subgroup such that $B\subseteq S$, and such that $(G,P,S)$ satisfies Property~$\Qspec$, we have $N\nsubseteq N_G(S)$.
    \end{enumerate}
    \item[$\Qadm_3$] The subgroup $Q$ is maximal among all subgroups $Q'\subseteq P$ such that $(G,P,Q')$ satisfies Properties~$\Qadm_1$ and $\Qadm_2$. 
\end{description}
\end{Def}

We now distinguish two cases depending on whether the type of $P$ is a clique or not.

\begin{lemma}\label{lemma:admissible-clique}
Let $G$ be a graph product over a finite simple graph $\Gamma$. Let $P\subseteq G$ be a product parabolic subgroup whose type is a clique.

Then there does not exist any subgroup $Q\subseteq P$ such that $(G,P,Q)$ satisfies Property~$\Qadm$.
\end{lemma}

\begin{proof}
It is enough to prove that whenever $B,N$ are subgroups of $P$, with $B$ non-trivial and normalized by $N$, one of the Assertions~(a) and~(b) from $\Qadm_2$ fails.

So let $B,N$ be as above, and let $\tilde{B}$ be the parabolic support of $B$. Notice that $N\subseteq N_G(B)\subseteq N_G(\tilde{B})$. Notice also that $\tilde{B}$ is a subproduct of $P$, because the type of $P$ is a clique.

If $\tilde{B}$ is non-special, then $N_G(\tilde{B})=N_G(P)$, and therefore $N\subseteq N_G(P)$, so Assertion~(a) fails.

If $\tilde{B}$ is special, then $(G,P,\tilde{B})$ satisfies Property~$\Qspec$, and the fact that $N\subseteq N_G(\tilde{B})$ shows that Assertion~(b) fails.
\end{proof}

The following lemma is analogous to \cref{lemma:admissible}, \cpageref{lemma:admissible}.

\begin{lemma}\label{lemma:admissible-groups}
Let $G$ be a graph product over a strongly reduced finite simple graph $\Gamma$. Let $P\subseteq G$ be a product parabolic subgroup whose type is not a clique. Let $Q\subseteq P$ be a subgroup. 

Then $Q$ is a clique-inclusive co-factor of $P$ if and only if $(G,P,Q)$ satisfies Property~$\Qadm$.
\end{lemma}

\begin{proof}
\textbf{Step 1.} We first prove that if $Q$ is a clique-inclusive co-factor of $P$, then $(G,P,Q)$ satisfies Properties~$\Qadm_1$ and~$\Qadm_2$.
\smallskip

We clearly have $Q\unlhd P$, i.e.\ $\Qadm_1$ holds. 

To prove Property~$\Qadm_2$, write $P=C_0\times F_1\times\dots\times F_k$, where $C_0$ is the clique factor, and no $F_j$ further decomposes as a direct product of two parabolic subgroups. Without loss of generality, we will assume that $P=G_\Lambda$ for some induced subgraph $\Lambda\subseteq \Gamma$, and each $F_j$ is equal to $G_{\Lambda_j}$, where $\Lambda_j$ is a join factor of $\Lambda$ that is not reduced to a point. Notice that $k\ge 1$ because the type of $P$ is not a clique. There thus exists $j\in\{1,\dots,k\}$ such that  $Q=C_0\times F_1\times\dots\times\hat{F}_j\times\dots\times F_k$. 

Since $\Gamma$ is strongly reduced, the subgraph $\Lambda_j$ is not collapsible. Therefore, there exists a vertex $v\in V\Lambda_j$ which is joined to a vertex $w\notin V\Lambda_j$, such that $w$ is not joined to all vertices of $\Lambda_j$. Notice in particular that 
    \[G_w \subseteq G^\perp_v \quad \text{and}
     \quad G_w\nsubseteq F_j\times F^\perp_j \quad \text{and} 
    \quad G_w\nsubseteq P\times P^\perp.\]
Let $B=G_v$. Then $B\subseteq P$ is non-trivial and normalized by $Q$. Let $N=G_v\times G_v^{\perp}$, which normalizes $B$. Then $G_w\subseteq N$, so in particular $N\nsubseteq N_G(P)$, which verifies $\Qadm_2$(a).

For $\Qadm_2$(b), let $S\subseteq P$ be a subgroup such that $B\subseteq S$, and such that $(G,P,S)$ satisfies Property~$\Qspec$. Lemma~\ref{lemma:spec-groups} ensures that the parabolic support $\tilde{S}$ of $S$ is a special subproduct of $P$. Since $G_v\subseteq \tilde{S}$ and $\tilde{S}$ is a subproduct, we have $F_j\subseteq \tilde{S}$. Since $G_w\nsubseteq N_G(F_j)$, we deduce that $G_w\nsubseteq N_G(\tilde{S})$. Therefore $N\nsubseteq N_G(\tilde{S})$ and in particular $N\nsubseteq N_G(S)$, as desired.

\medskip

\noindent \textbf{Step 2.} We now assume that $(G,P,Q)$ satisfies Properties~$\Qadm_1$ and~$\Qadm_2$, and prove that $Q$ is contained in a clique-inclusive co-factor of $P$.
\smallskip

Let $B,N$ be as in $\Qadm_2$. Write $P=\mathsf{F}_1\times\dots\times \mathsf{F}_n$, where the $\mathsf{F}_j$ are chosen not to split further (in particular the clique factor of $P$ is decomposed as the product of factors that are conjugate to vertex groups). Let $\tilde{B}$ and $\tilde{N}$ be the respective parabolic supports of $B$ and $N$. Since $N$ normalizes $B$, we have $N\subseteq N_G(\tilde{B})$. For $j\in\{1,\dots,n\}$, let $\tilde{B}_j=\tilde{B}\cap \mathsf{F}_j$. 

We claim that there exists $j\in\{1,\dots,n\}$ such that $\tilde{B}_j\neq\{1\}$ and $\tilde{B}_j\neq \mathsf{F}_j$. Indeed, otherwise $\tilde{B}$ would be a subproduct of $P$. Since $N\subseteq N_G(\tilde{B})$ and $N\nsubseteq N_G(P)$ by $\Qadm_2(a)$, we would have $N_G(\tilde{B})\nsubseteq N_G(P)$, in other words $\tilde{B}$ would be a special subproduct. So $(G,P,\tilde{B})$ would satisfy Property~$\Qspec$, and the inclusion $N\subseteq N_G(\tilde{B})$ would give a contradiction to $\Qadm_2$(b). Hence our claim.

Let $j\in\{1,\dots,n\}$ be as above. The above claim shows that $\mathsf{F}_j$ contains $\tilde{B}_j$ as a proper and non-trivial parabolic subgroup, in particular $\mathsf{F}_j$ is not conjugate to a vertex group. Let $\tilde{Q}$ be the parabolic support of $Q$ and let us show that $\tilde{Q}\cap \mathsf{F}_j=\{1\}$. Since $Q$ normalizes $B$, we have $Q\subseteq N_G(\tilde{B})$, and therefore $\tilde{Q}\subseteq N_G(\tilde{B})$. In particular $\tilde{Q}\cap \mathsf{F}_j$ is contained in $N_{\mathsf{F}_j}(\tilde{B}_j)$. Since $\tilde{B}_j$ is a proper non-trivial parabolic subgroup of $\mathsf{F}_j$ and $\mathsf{F}_j$ does not split as a product, the normalizer $N_{\mathsf{F}_j}(\tilde{B}_j)$ is a proper parabolic subgroup of $\mathsf{F}_j$. Using Lemma~\ref{lemma:product-parabolic-subgroup}, \cpageref{lemma:product-parabolic-subgroup} we have \[\tilde{Q}\subseteq \mathsf{F}_1\times\dots\times \mathsf{F}_{j-1}\times N_{\mathsf{F}_j}(\tilde{B}_j)\times \mathsf{F}_{j+1}\times\dots\times \mathsf{F}_n.\] 
Let $\tilde{Q}_j:=\tilde{Q}\cap\mathsf{F}_j$, a proper parabolic subgroup of $\mathsf{F}_j$, contained in $N_{\mathsf{F}_j}(\tilde{B}_j)$. Since $Q\unlhd P$, we have $P\subseteq N_G(\tilde Q)$, and we deduce that 
\[P\subseteq\mathsf{F}_1\times\dots\times \mathsf{F}_{j-1}\times N_{\mathsf{F}_j}(\tilde Q_j)\times \mathsf{F}_{j+1}\times\dots\times \mathsf{F}_n.\]
Therefore $N_{\mathsf{F}_j}(\tilde Q_j)=\mathsf{F}_j$, which implies that $\tilde Q_j=\{1\}$. We have thus proved that
\[Q\subseteq\tilde{Q}\subseteq \mathsf{F}_1\times\dots\times \hat {\mathsf{F}}_{j}\times\dots\times \mathsf{F}_n.\] 
This concludes Step~2.

\medskip

\noindent \textbf{Step 3.} We now assume that $(G,P,Q)$ satisfies Property~$\Qadm$, and prove that $Q$ is a clique-inclusive co-factor of $P$. 
\smallskip

By Step~2, $Q$ is contained in a clique-inclusive co-factor $Q'$ of $P$. By Step~1, the triple $(G,P,Q')$ satisfies Properties~$\Qadm_1$ and~$\Qadm_2$. So Property~$\Qadm_3$ for $(G,P,Q)$ implies that $Q=Q'$.

\medskip

\noindent\textbf{Step 4.} We are left with showing that if $Q\subseteq P$ is a clique-inclusive co-factor, then $(G,P,Q)$ satisfies Property~$\Qadm_3$. 
\smallskip

So let $Q'\subseteq P$ be such that $Q\subseteq Q'$ and $(G,P,Q')$ satisfies Properties~$\Qadm_1$ and $\Qadm_2$. By Step~2, $Q'$ is contained in a clique-inclusive co-factor $Q''$ of $P$. In particular $Q$ and $Q''$ are two clique-inclusive co-factors of $P$ with $Q\subseteq Q''$, so in fact $Q=Q''$. In particular $Q=Q'$, and Property~$\Qadm_3$ holds. 
\end{proof}

\subsubsection{Factors}

We now provide the group theoretic analogue of \cref{de:pfact}, \cpageref{de:pfact}.

\begin{Def}[Property~$\Qfact$]\index{Property!$\Qfact$}\label{de:qfact}
Let $G$ be a group, and $F\subseteq P$ be subgroups. We say that the triple $(G,P,F)$ satisfies \emph{Property~$\Qfact$} if the following two assertions hold. 
\begin{description}
\item[$\Qfact_1$] $F$ is non-trivial, and can be written as an intersection $Q_1\cap\dots\cap Q_\ell$ of subgroups $Q_j$ of $P$ with $\ell \ge 1$, where for every $j\in\{1,\dots,\ell \}$, the triple $(G,P,Q_j)$ satisfies Property~$\Qadm$.
\item[$\Qfact_2$] $F$ is minimal among all subgroups of $P$ that satisfy $\Qfact_1$. 
\end{description}
\end{Def}

The following lemma records the fact that factors arise as minimal non-trivial intersections of co-factors.

\begin{lemma}\label{lemma:factor-groups}
Let $G$ be a graph product over a strongly reduced finite simple graph $\Gamma$. Let $P\subseteq G$ be a product parabolic subgroup with clique factor $C_0$. Let $F\subseteq P$ be a subgroup.
\begin{enumerate}
\item If the type of $P$ is a clique, then $(G,P,F)$ never satisfies Property~$\Qfact$.    
\item If the type of $P$ is not a clique, then
\begin{enumerate}
\item If $C_0=\{1\}$, then $F$ is a factor of $P$ if and only if $(G,P,F)$ satisfies Property~$\Qfact$.
    \item If $C_0\neq\{1\}$, then $F$ is the clique factor of $P$ if and only if $(G,P,F)$ satisfies Property~$\Qfact$. 
    \end{enumerate}
\end{enumerate}
\end{lemma}

\begin{proof}
The first assertion is a consequence of the fact that there does not exist any subgroup $Q\subseteq P$ such that $(G,P,Q)$ satisfies Property~$\Qadm$ (Lemma~\ref{lemma:admissible-clique}).

The second assertion follows from Lemma~\ref{lemma:admissible-groups}, together with the fact that 
\begin{itemize}
    \item if $C_0=\{1\}$, then factors of $P$ are exactly the minimal non-trivial intersections of clique-inclusive co-factors;
    \item if $C_0\neq\{1\}$, then the clique factor of $P$ is the unique minimal non-trivial intersection of clique-inclusive co-factors. \qedhere
\end{itemize}
\end{proof}

We now aim at a characterization of product parabolic subgroups with trivial clique factor.

\begin{Def}[Property~$\Qtc$]\index{Property!$\Qtc$}
    Let $G$ be a group, and $P\subseteq G$ be a subgroup. We say that the pair $(G,P)$ satisfies \emph{Property~$\Qtc$} if there exist $\ell \geq 1$ and subgroups $Q_1,\dots,Q_\ell\subseteq P$ such that $(G,P,Q_j)$ satisfies Property~$\Qadm$ for every $j\in\{1,\dots,\ell\}$, and $Q_1\cap\dots\cap Q_\ell$ is trivial.
\end{Def}

The following lemma records the fact that the clique factor of $P$ is trivial if and only if the intersection of all clique-inclusive co-factors is trivial.

\begin{lemma}\label{lemma:trivial-clique}
    Let $G$ be a graph product over a strongly reduced finite simple graph $\Gamma$. Let $P$ be a product parabolic subgroup.

    Then the clique factor of $P$ is trivial if and only if $(G,P)$ satisfies Property~$\Qtc$. \qed
\end{lemma}

\subsubsection{Proof of Theorem~\ref{theo:conjugating-automorphism}}

The following property is inspired by Lemma~\ref{lemma:graph-2}. It can be seen as a group theoretic variation of \cref{de:pvert}, \cpageref{de:pvert}. Let us mention that it is closer to the version presented in \cref{de:pamen}, \cpageref{de:pamen}.  

\begin{Def}[Property~$\Qvert$]\index{Property!$\Qvert$}\label{de:qvert}
Let $G$ be a group, and $V\subseteq G$ be a non-trivial subgroup. We say that the pair $(G,V)$ satisfies \emph{Property $\Qvert$} if there exists a chain of subgroups
    \[G=F_0\supseteq L_1\supseteq P_1\supseteq F_1\supseteq\dots\supseteq L_n\supseteq P_n\supseteq F_n=V\]
    with $n\ge 1$ such that
    \begin{enumerate}
        \item for every $j\in\{1,\dots,n\}$, the triple $(G,F_{j-1},L_j)$ satisfies Property~$\Qisol$ and $L_j$ is non-abelian;
        \item for every $j\in\{1,\dots,n\}$, the pair $(L_j,P_j)$ satisfies Property~$\Qprod$;
        \item for every $j\in\{1,\dots,n-1\}$, the pair $(G,P_j)$ satisfies Property~$\Qtc$, while $(G,P_n)$ does not;
        \item for every $j\in\{1,\dots,n\}$, the triple $(G,P_j,F_j)$ satisfies Property~$\Qfact$;
    \item for every non-trivial subgroup $W\subseteq V$, one has $N_G(W)\subseteq N_G(V)$. 
    \end{enumerate}
\end{Def}

\begin{lemma}\label{lemma:qvert}
    Let $G$ be a graph product over a strongly reduced finite simple graph, not reduced to one vertex. A subgroup $V\subseteq G$ is conjugate to an untransvectable vertex group if and only if $(G,V)$ satisfies Property~$\Qvert$.
\end{lemma}
   
\begin{proof}
We first assume that $V$ is conjugate to an untransvectable vertex group, and prove that $(G,V)$ satisfies Property~$\Qvert$. Let 
 \[G=F_0\supseteq L_1\supseteq P_1\supseteq F_1\supseteq\dots\supseteq L_n\supseteq P_n\supseteq F_n=V\]
 be a chain of parabolic subgroups as in Lemma~\ref{lemma:graph-2}, \cpageref{lemma:graph-2}. Then
    \begin{enumerate}
    \item For every $j\in\{1,\dots,n\}$, the pair $(G,F_{j-1},L_j)$ satisfies Property~$\Qisol$ (this uses Assertion~2.(a) of Lemma~\ref{lemma:graph-2}, together with Lemma~\ref{lemma:isol}), and $L_j$ is non-abelian because its type is not a clique. 
    \item For every $j\in\{1,\dots,n\}$, the pair $(L_j,P_j)$ satisfies Property~$\Qprod$: this uses Assertion 2.(b) of Lemma~\ref{lemma:graph-2}, together with Lemma~\ref{lemma:qprod} (to apply Lemma~\ref{lemma:qprod}, one needs to know that the type of $L_j$ is a connected subgraph of $\Gamma$ which is not reduced to one vertex -- this is ensured by Assertion 2.(a) of Lemma~\ref{lemma:graph-2}).  
    \item For every $j\in\{1,\dots,n-1\}$, the pair $(G,P_j)$ satisfies Property~$\Qtc$, while $(G,P_n)$ does not: this follows from Assertions~2.(c) and~2.(d) of Lemma~\ref{lemma:graph-2}, together with Lemma~\ref{lemma:trivial-clique}.
    \item For every $j\in\{1,\dots,n\}$, the triple $(G,P_j,F_j)$ satisfies Property~$\Qfact$. Indeed, we saw that the type of $L_j$ is a connected subgraph which is not a clique; since $P_j$ is a maximal product parabolic subgroup inside $L_j$, its type is not a clique either. Therefore, Assertions~2.(c) and~2.(d) of Lemma~\ref{lemma:graph-2}, together with Lemma~\ref{lemma:factor-groups}, ensure that $(G,P_j,F_j)$ satisfies Property~$\Qfact$.
    \item For every non-trivial subgroup $W\subseteq F_n$, one has $N_G(W)\subseteq N_G(F_n)$: this is exactly Assertion 2.(e) of Lemma~\ref{lemma:graph-2}.
    \end{enumerate}
So Property~$\Qvert$ is satisfied.  
\medskip

Conversely, let us assume that $(G,V)$ satisfies Property~$\Qvert$. Let \[G=F_0\supseteq L_1\supseteq P_1\supseteq F_1\supseteq\dots\supseteq L_n\supseteq P_n\supseteq F_n=V\] be a chain of subgroups given by Property~$\Qvert$. Successive applications of Lemmas~\ref{lemma:isol},~\ref{lemma:qprod},~\ref{lemma:trivial-clique} and~\ref{lemma:factor-groups} show that 
\begin{itemize}
    \item For every $j\in\{1,\dots,n\}$, $L_j$ is a thick free factor of $F_{j-1}$ (in particular its type is not reduced to one vertex), and by $\Qvert_1$ the subgroup $L_j$ is non-abelian.
    \item For every $j\in\{1,\dots,n\}$, $P_j$ is a maximal product parabolic subgroup of $L_j$.
    \item For every $j\in\{1,\dots,n\}$, the type of $P_j$ (and therefore of $L_j$) is not a clique. This follows from $\Qvert_4$ and Assertion~1 of \cref{lemma:factor-groups}. 
    \item For every $j\in\{1,\dots,n-1\}$ by $\Qvert_3$ and \cref{lemma:trivial-clique} the clique factor of $P_j$ is trivial, while the clique factor of $P_n$ is non-trivial.
    \item For every $j\in\{1,\dots,n\}$, the subgroup $F_j$ is a factor of $P_j$, by $\Qvert_4$ and Assertion~2 of \cref{lemma:factor-groups}. Moreover $F_n$ is the clique factor of $P_n$.
\end{itemize}
Together with Assertion~5 from Property~$\Qvert$, this means that all assertions from Lemma~\ref{lemma:graph-2} hold. Lemma~\ref{lemma:graph-2} therefore implies that $V$ is conjugate to an untransvectable vertex group. 
\end{proof}

We are now in position to complete the proof of the main theorem of this appendix.

\begin{proof}[Proof of Theorem~\ref{theo:conjugating-automorphism}]
Let $v\in V\Gamma_G$ be an untransvectable vertex. Then $(G,G_v)$ satisfies Property~$\Qvert$ (Lemma~\ref{lemma:qvert}), and therefore so does $(H,f(G_v))$ -- notice indeed that all properties of the present section are stable under group isomorphisms (Remark~\ref{rk:iso}). Applying Lemma~\ref{lemma:qvert} again shows that $f(G_v)$ is conjugate to an untransvectable vertex group.
\end{proof}

\subsubsection{An extension to products}

We finally mention that Theorem~\ref{theo:conjugating-automorphism} can be extended to include direct products, as follows.

\begin{Th}\label{theo:conjugating-automorphism-join}
     Let $G,H$ be graph products over finite simple graphs $\Gamma_G,\Gamma_H$ that decompose as joins of strongly reduced graphs, none of which is reduced to a vertex or an edge.
    
    Then for every isomorphism $f:G\to H$ and every untransvectable vertex $v\in V\Gamma_G$, there exists an untransvectable vertex $w\in V\Gamma_H$ such that $f(G_v)$ is conjugate to $H_w$.
\end{Th}

\begin{proof}
Write $G=G_1\times\dots\times G_k$ and $H=H_1\times\dots\times H_\ell$, where the $G_j$ and $H_j$ are irreducible graph products over strongly reduced finite simple graphs, not reduced to a vertex. Notice that (by irreducibility) all groups $G_j$ and $H_j$ have trivial center, and by \cite[Lemma~3.5]{Gen} they admit a generating set consisting of pairwise non-commuting elements with maximal centralizers. Since $G$ and $H$ are isomorphic, we can therefore apply \cite[Lemma~3.12]{Gen} with $A_0=B_0=\{1\}$ and deduce that $k=\ell$, and that there exists a bijection $\sigma:\{1,\dots,k\}\to\{1,\dots,\ell \}$ such that for every $j\in\{1,\dots,k\}$, the image $f(G_j)$ is equal to $H_{\sigma(j)}$. The conclusion then follows from Theorem~\ref{theo:conjugating-automorphism}, applied to each subgroup~$G_j$.
\end{proof}

Again, together with \cite[Theorem~3.11]{GM}, we deduce the following consequence, which in particular establishes the version of Theorem~\ref{theo:classification-commensurability}, where “commensurable” and “strongly commensurable” are replaced by “isomorphic”.

\begin{Cor}\label{cor:isomorphism-join}
Let $G,H$ be graph products over finite simple graphs $\Gamma_G,\Gamma_H$ that decompose as joins of strongly reduced graphs, with $\Gamma_G$ transvection-free. 

If $G$ and $H$ are isomorphic, then there exist an isomorphism $\sigma:\Gamma_G\to \Gamma_H$ such that for every $v\in V\Gamma$, the vertex groups $G_v$ and $H_{\sigma(v)}$ are isomorphic. \qed
\end{Cor}


\subsection{Acylindrical hyperbolicity for automorphisms of graph products}\label{sec:acyl-hyp}

A group is \emph{acylindrically hyperbolic}\index{Acylindrically hyperbolic} \cite[Definition~1.3]{Osi} if it admits a non-elementary action on a Gromov-hyperbolic metric space $X$ (i.e.\ with unbounded orbits and no finite orbit in the Gromov boundary $\partial_\infty X$), which is acylindrical in the sense that for every $R>0$, there exist $L,N\ge 0$ such that given any two points $x,y\in X$ with $d(x,y)\ge L$, one has
\[\left\vert \left\{g\in G \mid d(x,gx)\le R \text{~and~} d(y,gy)\le R \right\} \right\vert\le N. \]

We show the following result, which is essentially due to Genevois \cite[Theorem~1.1]{Gen}. See also the remark below for more details.

\begin{Th}\label{theo:acyl-hyp}
    Let $\Gamma$ be a finite irreducible simple graph, not reduced to a vertex, and let $G$ be a finitely generated graph product over $\Gamma$. Assume that $G$ is not isomorphic to $\mathbb{Z}/2\mathbb{Z}\ast\mathbb{Z}/2\mathbb{Z}$.
    
    Then $\Aut(G)$ is acylindrically hyperbolic.
\end{Th}

\begin{Rq}[Parallel with Genevois’ work]
    The above theorem was proved by Genevois \cite[Theorem~1.1]{Gen} under the extra assumption that all vertex groups are graphically irreducible (see Remark~\ref{rk:genevois} for the definition). See also the work of Genevois--Martin \cite[Theorem~E]{GM}, where this was proved without any assumption on the vertex groups, but with more restrictions on the defining graph $\Gamma$. 
    In fact our proof still heavily relies on his work, and we only provide a new argument for Step~1 of Genevois' proof, carried in \cite[Section~3]{Gen}. 
\end{Rq}

In particular all consequences obtained from the acylindrical hyperbolicity of $\Aut(G)$ hold in this generality. Let us mention for instance the result by Fournier-Facio and Wade \cite[Theorem~5.5]{FFW}, that now holds for any graph product $G$ as in Theorem~\ref{theo:acyl-hyp}, and says that the space of $\Aut(G)$-invariant homogeneous quasimorphisms on $G$ is infinite-dimensional. Likewise, the theorem of Coulon and Fournier-Facio \cite[Corollary~3.5 and Example~3.6]{CFF}, saying that $G$ has an infinite simple characteristic quotient, also holds in this generality.

\begin{proof}[Proof of \cref{theo:acyl-hyp}]
If $\Gamma$ is disconnected, then $G$ has infinitely many ends, and the theorem follows from \cite[Theorem~1.1]{GH-Hyp}. From now on we assume that $\Gamma$ is connected.

By collasping some subgraphs of $\Gamma$ if needed, we can write $G$ as a graph product over a strongly reduced finite simple graph $\bar\Gamma$, not reduced to one vertex, and still irreducible. Therefore, without loss of generality, we can (and shall) assume that $\Gamma$ is strongly reduced to start with.

By \cite[Theorem~6.2]{Gen}, there exists an element $g\in G$ such that 
\begin{itemize}
    \item $g$ is not contained in any proper parabolic subgroup of $G$, and
    \item if $g$ is elliptic in some isometric $G$-action on an $\mathbb{R}$-tree $T$ whose arc stabilizers are contained in proper parabolic subgroups, then $G$ has a global fixed point in $T$.
\end{itemize}
In particular $g$ is \emph{irreducible} in the sense of \cite[p.~14]{Gen}, i.e.\ it is not conjugate in any vertex group or product parabolic subgroup, and up to conjugating $g$ we can further assume that it is \emph{graphically cyclically reduced} in the sense of \cite[Definition~2.2]{Gen}. By \cite[Theorem~5.1]{Gen}, the fixator of $g$ in $\Aut(G)$ has finite image in $\Out(G)$. Now, Theorem~\ref{theo:conjugating-automorphism} precisely shows that all assumptions from \cite[Proposition~4.33]{Gen} are satisfied (our \emph{untransvectable} vertices correspond to \emph{$\prec$-maximal} vertices in Genevois' terminology). By applying \cite[Proposition~4.33]{Gen}, we deduce that $g$ is a WPD element for some action of $\Aut(G)$ on a hyperbolic space (in the sense of Bestvina--Fujiwara \cite[Section~3]{BF}). Having a WPD element is enough to ensures that $\Aut(G)$ is either virtually cyclic or acylindrically hyperbolic \cite[Theorem~1.2]{Osi}. Finally, $\Aut(G)$ is not virtually cyclic: since $G$ is not virtually cyclic and has trivial center (see \cite[Theorem~3.34]{Gre}), the group of inner automorphisms of $G$ is not virtually cyclic.
\end{proof}

\markboth{Index}{Index}
\printindex
\addcontentsline{toc}{section}{Index}

\newpage
\section*{Notations index}
\markboth{Notations}{Notations}
\addcontentsline{toc}{section}{Notations}

\begin{description}
\item[$\preccurlyeq$] See \cpageref{Nota:AsymptoBehaviour};
\item[{$[G:H]_\Omega$}] Compression constant, equal to ${m(X_H)}/{m(X_G)}$ (see \cpageref{Def:CompressionConstant});
\item[$A_\Gamma$] Right-angled Artin group defined over the graph $\Gamma$;
\item[$\mathfrak{C}_G$] See \cpageref{Nota:FrakC};
\item[$\partial_\infty T_v$] Boundary of the tree $T_v$;
\item[$\left(\partial_\infty T_v\right)^\mathrm{para}$] See \cpageref{Def:Tvpara};
\item[$\left(\partial_\infty T_v\right)^\mathrm{reg}$] See \cpageref{Def:TvReg};
\item[$\bD_G$] Right-angled building of $G$;
\item[$E\Gamma$] Set of edges of the graph $\Gamma$;
\item[$\mathfrak{F}$] The set of all standard clique cosets (see \cpageref{Def:FrakF});
\item[$G \ltimes X$] Groupoid associated to an action of $G$ on a probability space $X$ (see \cref{Ex:MeasuredGroupoidAssociatedToAnAction};
\item[$\calG$, $\calH$] Groupoids;
\item[$\Gamma$, $\Lambda$, $\Upsilon$] Finite simple graphs;
\item[$\Gamma \circ \Lambda$] The graph join of $\Gamma$ and $\Lambda$;
\item[$\Gamma^e_G$] Extension graph of $G$ (see \cref{de:extension-graph});
\item[$\mathfrak{I}_G$] See \cpageref{Nota:FrakI};
\item[$\kappa(f)$] Compression constant of $f$, equal to ${\nu(V)}/{\mu(U)}$ (see \cpageref{Def:CompressionConstantf});
\item[$\lk(v)$] The link of the vertex $v$ (see \cpageref{Def:Link})
\item[$\Lambda^\perp$] See \cpageref{notation:lambdaperp};
\item[$\mathcal{P}_{<\infty}(Z)$] Set of all nonempty finite subsets of $Z$;
\item[$\mathcal{P}_{\leq 2}(Z)$] Set of all nonempty subsets of $Z$ of cardinality at most $2$;
\item[$(\Omega,m,X_G,X_H)$] A measure equivalence coupling (see \cref{Def:QuantitativeME});
\item[$\mathbb{P}_G$] Set of parabolic subgroups of $G$;
\item[$\mathrm{Prob}(K)$] Space of Borel probability measures on $K$;
\item[$\rho$] A cocycle (groupoid framework);
\item[$\st(v)$] The star of the vertex $v$ (see \cpageref{Def:Star})
\item[$T_v$] Bass-Serre tree associated to the splitting $G=G_{\st(v)}\ast_{G_\lk(v)}G_{\Gamma\setminus\{v\}}$ of a graph product (see \cref{NotationTv})
\item[$V\Gamma$] Set of vertices of the graph $\Gamma$;
\item[$\underline{\mathsf{w}}$] A word (see \cref{def:word});
\item[$X_G$] A fundamental domain for the action of a group $G$.
\end{description}

\newpage
\markboth{Bibliography}{Bibliography}
\footnotesize
\bibliographystyle{alpha}
\bibliography{quantitative}

@book {AB,
    AUTHOR = {Abramenko, P. and Brown, K.S.},
     TITLE = {Buildings. Theory and applications},
    SERIES = {Graduate Texts in Mathematics},
    VOLUME = {248},
 PUBLISHER = {Springer, New York},
      YEAR = {2008},
     PAGES = {xxii+747},
}

@Article{Ada,
author = {S. Adams},
title = {Indecomposability of equivalence relations generated by word hyperbolic groups},
journal = {Topol.},
year = {1994},
volume = {33},
number = {4},
pages = {785-798},
}

@article {AG,
    AUTHOR = {Alvarez, A. and Gaboriau, D.},
     TITLE = {Free products, orbit equivalence and measure equivalence
              rigidity},
   JOURNAL = {Groups Geom. Dyn.},
    VOLUME = {6},
      YEAR = {2012},
    NUMBER = {1},
     PAGES = {53--82},
}

@article{AM,
title={Tits alternatives for graph products},
author={Y. Antol\'in and A. Minasyan},
journal={J. Reine Angew. Math.},
year={2015},
pages={55-83},
volume={704},
}

@incollection {AD,
    AUTHOR = {Anantharaman-Delaroche, C.},
     TITLE = {The {H}aagerup property for discrete measured groupoids},
 BOOKTITLE = {Operator algebra and dynamics},
    SERIES = {Springer Proc. Math. Stat.},
    VOLUME = {58},
     PAGES = {1--30},
 PUBLISHER = {Springer, Heidelberg},
      YEAR = {2013},
}

@book {ADR,
    AUTHOR = {Anantharaman-Delaroche, C. and Renault, J.},
     TITLE = {Amenable groupoids},
    SERIES = {Monographies de L'Enseignement Math\'{e}matique},
    VOLUME = {36},
 PUBLISHER = {L'Enseignement Math\'{e}matique, Geneva},
      YEAR = {2000},
     PAGES = {196},
}

@article{AustinBowen,
title={Integrable measure equivalence for groups of polynomial growth},
author={Austin, T.},
journal={Groups, Geometry, and Dynamics},
year={2016},
month={February},
pages={117-154},
volume={10},
issue={1},
}

@article{BFSIntegrability,
title={Integrable measure equivalence and rigidity of hyperbolic lattices},
author={Bader, U. and Furman, A. and Sauer, R.},
journal={Invent. Math.},
year={2013},
volume={194},
issue={2},
URL={https://doi.org/10.1007/s00222-012-0445-9},
pages={313-379},
}

@article {Bar,
    AUTHOR = {Barkauskas, D.A.},
     TITLE = {Centralizers in graph products of groups},
   JOURNAL = {J. Algebra},
    VOLUME = {312},
      YEAR = {2007},
    NUMBER = {1},
     PAGES = {9--32},
}

@article {Bau,
    AUTHOR = {Baumslag, G.},
     TITLE = {Some aspects of groups with unique roots},
   JOURNAL = {Acta Math.},
    VOLUME = {104},
      YEAR = {1960},
     PAGES = {217--303},
}

@article {BeFe,
    AUTHOR = {Berlai, F. and Ferov, M.},
     TITLE = {Separating cyclic subgroups in graph products of groups},
   JOURNAL = {J. Algebra},
    VOLUME = {531},
      YEAR = {2019},
     PAGES = {19--56},
}

@article {BdlN,
    AUTHOR = {Berlai, F. and de la Nuez, J.},
     TITLE = {Linearity of graph products},
   JOURNAL = {Proc. Lond. Math. Soc. (3)},
    VOLUME = {124},
      YEAR = {2022},
    NUMBER = {5},
     PAGES = {587--600},
}

@article {BR,
    AUTHOR = {Berlyne, D. and Russell, J.},
     TITLE = {Hierarchical hyperbolicity of graph products},
   JOURNAL = {Groups Geom. Dyn.},
    VOLUME = {16},
      YEAR = {2022},
    NUMBER = {2},
     PAGES = {523--580},
}

@article {BF,
    AUTHOR = {Bestvina, M. and Fujiwara, K.},
     TITLE = {Bounded cohomology of subgroups of mapping class groups},
   JOURNAL = {Geom. Topol.},
    VOLUME = {6},
      YEAR = {2002},
     PAGES = {69--89},
}

@article{BGH,
title={Boundary amenability of $\mathrm{Out}({F}_{N})$},
author={M. Bestvina and V. Guirardel and C. Horbez},
journal={Ann. Sci. \'Ec. Norm. Supér. (4)},
year={2022},
volume = {55},
number = {5},
pages = {1379-1431},
}

@article {BP,
    AUTHOR = {Bonatti, C. and Paris, L.},
     TITLE = {Roots in the mapping class groups},
   JOURNAL = {Proc. Lond. Math. Soc. (3)},
    VOLUME = {98},
      YEAR = {2009},
    NUMBER = {2},
     PAGES = {471--503},
}

@article {Bow,
    AUTHOR = {Bowen, L.},
     TITLE = {Integrable orbit equivalence rigidity for free groups},
   JOURNAL = {Israel J. Math.},
    VOLUME = {221},
      YEAR = {2017},
    NUMBER = {1},
     PAGES = {471--480},
}

@article {Bow2,
    AUTHOR = {Bowen, L.},
     TITLE = {Orbit equivalence, coinduced actions and free products},
   JOURNAL = {Groups Geom. Dyn.},
    VOLUME = {5},
      YEAR = {2011},
    NUMBER = {1},
     PAGES = {1--15},
}

@article {Bow3,
    AUTHOR = {Bowen, L.},
     TITLE = {Stable orbit equivalence of {B}ernoulli shifts over free
              groups},
   JOURNAL = {Groups Geom. Dyn.},
    VOLUME = {5},
      YEAR = {2011},
    NUMBER = {1},
     PAGES = {17--38},
}

@book {BH,
    AUTHOR = {Bridson, M.R. and Haefliger, A.},
     TITLE = {Metric spaces of non-positive curvature},
    SERIES = {Grundlehren der mathematischen Wissenschaften},
    VOLUME = {319},
 PUBLISHER = {Springer-Verlag, Berlin},
      YEAR = {1999},
     PAGES = {xxii+643},
}

@article{BZ,
      title={{Speed of random walks, isoperimetry and compression of finitely generated groups}}, 
      author={Brieussel, J. and Zheng, T.},
      journal={Ann. of Math. (2)},
      volume={193},
      issue={1},
      pages={1-105},
      year={2021},
}

@article {CH,
    AUTHOR = {Caprace, P.-E. and Hume, D.},
     TITLE = {Orthogonal forms of {K}ac-{M}oody groups are acylindrically
              hyperbolic},
   JOURNAL = {Ann. Inst. Fourier (Grenoble)},
    VOLUME = {65},
      YEAR = {2015},
    NUMBER = {6},
     PAGES = {2613--2640},
}

@article {CS,
    AUTHOR = {Caprace, P.-E. and Sageev, M.},
     TITLE = {Rank rigidity for {CAT}(0) cube complexes},
   JOURNAL = {Geom. Funct. Anal.},
    VOLUME = {21},
      YEAR = {2011},
    NUMBER = {4},
     PAGES = {851--891},
}

@article {ChiKid,
    AUTHOR = {Chifan, I. and Kida, Y.},
    TITLE = {{$OE$} and {$W^*$} superrigidity results for actions by
              surface braid groups},
    JOURNAL = {Proc. Lond. Math. Soc. (3)},
    VOLUME = {111},
    YEAR = {2015},
    issue = {6},
    PAGES = {1431--1470},
}

@misc{CKE,
title={Cartan {S}ubalgebras in von {N}eumann {A}lgebras {A}ssociated with {G}raph {P}roduct {G}roups},
author={I. Chifan and Kunnawalkam Elayavalli, S.},
note={arXiv:2107.04710, (to appear in \emph{Groups Geom. Dyn.})},
year={2021},
}

@article{CDD1,
title={Rigidity for von {N}eumann algebras of graph product groups. {I}. {S}tructure of automorphisms},
author={I. Chifan and M. Davis and D. Drimbe},
journal={arXiv:2209.12996},
year={2022},
}

@article{CDD2,
title={Rigidity for von {N}eumann algebras of graph product groups. {II}. {S}uperrigidity results},
author={I. Chifan and M. Davis and D. Drimbe},
journal={arXiv:2304.05500},
year={2023},
}

@misc{CRKNG,
title={\textit{On the elementary theory of graph products of groups}},
author={M. Casals-Ruiz and I. Kazachkov and J. de la Nuez Gonzalez},
note={arXiv:2106.03782},
year={2021},
}

@article{CF,
 title={Graph products of operator algebras}, 
      author={M. Caspers and P. Fima},
      year={2017},
      eprint={1411.2799},
      archivePrefix={arXiv},
      journal={J. Noncommut. Geom.},
      volume={11},
      issue={1},
      pages={367–411},
      }

@article {Con,
    AUTHOR = {Connes, A.},
     TITLE = {A factor of type {${\rm II}\sb{1}$} with countable fundamental
              group},
   JOURNAL = {J. Operator Theory},
    VOLUME = {4},
      YEAR = {1980},
    NUMBER = {1},
     PAGES = {151--153},
}

@misc{CFF,
      title={Infinite simple characteristic quotients}, 
      author={Coulon,R and Fournier-Facio, F.},
      year={2023},
      eprint={2312.11684},
      archivePrefix={arXiv},
      primaryClass={math.GR},
      note={arxiv:2312.11684},
}

@article {DGO,
    AUTHOR = {Dahmani, F. and Guirardel, V. and Osin, D.},
     TITLE = {Hyperbolically embedded subgroups and rotating families in
              groups acting on hyperbolic spaces},
   JOURNAL = {Mem. Amer. Math. Soc.},
    VOLUME = {245},
      YEAR = {2017},
    NUMBER = {1156},
}

@incollection {Dan,
    AUTHOR = {Dani, P.},
     TITLE = {The large-scale geometry of right-angled {C}oxeter groups},
 BOOKTITLE = {Handbook of group actions. {V}},
    SERIES = {Adv. Lect. Math. (ALM)},
    VOLUME = {48},
     PAGES = {107--141},
 PUBLISHER = {Int. Press, Somerville, MA},
      YEAR = {2020},
}

@incollection {Dav,
    AUTHOR = {Davis, M.W.},
     TITLE = {Buildings are {${\rm CAT}(0)$}},
 BOOKTITLE = {Geometry and cohomology in group theory ({D}urham, 1994)},
    SERIES = {London Math. Soc. Lecture Note Ser.},
    VOLUME = {252},
     PAGES = {108--123},
 PUBLISHER = {Cambridge Univ. Press, Cambridge},
      YEAR = {1998},
}

@article {DL,
    AUTHOR = {Davis, M. W. and Leary, I. J.},
     TITLE = {The {$l^2$}-cohomology of {A}rtin groups},
   JOURNAL = {J. London Math. Soc. (2)},
    VOLUME = {68},
      YEAR = {2003},
    NUMBER = {2},
     PAGES = {493--510},
}

@article {DO,
    AUTHOR = {Davis, M.W. and Okun, B.},
     TITLE = {Cohomology computations for {A}rtin groups, {B}estvina-{B}rady
              groups, and graph products},
   JOURNAL = {Groups Geom. Dyn.},
    VOLUME = {6},
      YEAR = {2012},
    NUMBER = {3},
     PAGES = {485--531},
}

@article{DKLMT,
      author={Delabie, T. and Koivisto, J. and Le Maître, F. and Tessera, R.},
     title = {Quantitative measure equivalence between amenable groups},
     journal = {Ann. H. Lebesgue},
     pages = {1417--1487},
     publisher = {\'ENS Rennes},
     volume = {5},
     year = {2022},
     doi = {10.5802/ahl.155},
     language = {en},
     url = {https://ahl.centre-mersenne.org/articles/10.5802/ahl.155/}, 
}

@Article{Dye1,
author = {H.A. Dye},
title = {On groups of measure preserving transformation. {I}},
journal = {Amer. J. Math.},
year = {1959},
volume = {81},
pages = {119-159},
}

@Article{Dye2,
author = {H.A. Dye},
title = {On groups of measure preserving transformations. {II}},
journal = {Amer. J. Math.},
year = {1963},
volume = {85},
pages = {551-576},
}

@misc{Esc1,
author = {Escalier, A.},
title = {Building prescribed quantitative orbit equivalence with $\mathbb{Z}$},
note = {arXiv:2202.10312, (to appear in \emph{Groups Geom. Dyn.})},
year = {2022},
}

@article {FaMo,
    AUTHOR = {Farb, B. and Mosher, L.},
     TITLE = {Quasi-isometric rigidity for the solvable {B}aumslag-{S}olitar
              groups. {II}},
   JOURNAL = {Invent. Math.},
    VOLUME = {137},
      YEAR = {1999},
    NUMBER = {3},
     PAGES = {613--649},
}

@article {FM,
    AUTHOR = {Feldman, J. and Moore, C.C.},
     TITLE = {Ergodic equivalence relations, cohomology, and von {N}eumann
              algebras. {II}},
   JOURNAL = {Trans. Amer. Math. Soc.},
    VOLUME = {234},
      YEAR = {1977},
    NUMBER = {2},
     PAGES = {325--359},
}

@article {FSZ,
    AUTHOR = {Feldman, J. and Sutherland, C. E. and Zimmer, R. J.},
     TITLE = {Subrelations of ergodic equivalence relations},
   JOURNAL = {Ergodic Theory Dynam. Systems},
    VOLUME = {9},
      YEAR = {1989},
    NUMBER = {2},
     PAGES = {239--269},
}

@article {Fer,
    AUTHOR = {Fern\'{o}s, T.},
     TITLE = {The {F}urstenberg-{P}oisson boundary and {${\rm CAT}(0)$} cube
              complexes},
   JOURNAL = {Ergodic Theory Dynam. Systems},
    VOLUME = {38},
      YEAR = {2018},
    NUMBER = {6},
     PAGES = {2180--2223},
}

@article {FMW,
    AUTHOR = {Fisher, D. and Morris, D.W. and Whyte, K.},
     TITLE = {Nonergodic actions, cocycles and superrigidity},
   JOURNAL = {New York J. Math.},
    VOLUME = {10},
      YEAR = {2004},
     PAGES = {249--269},
}

@article {FFW,
    AUTHOR = {Fournier-Facio, F. and Wade, R.},
     TITLE = {Aut-invariant quasimorphisms on groups},
   JOURNAL = {Trans. Amer. Math. Soc.},
    VOLUME = {376},
      YEAR = {2023},
    NUMBER = {10},
     PAGES = {7307--7327},
}

@article{Fur,
title={Orbit equivalence rigidity},
author={Furman, A.},
year={1999},
journal={Ann. of Math. (2)},
volume={150},
number = {3},
pages={1083--1108},
}

@article{Fur-me,
author = {Furman, A.},
title = {Gromov's measure equivalence and rigidity of higher-rank lattices},
journal = {Ann. of Math. (2)},
year = {1999},
volume = {150},
pages = {1059-1081},
}

@article{Fur-Popa,
author = {Furman, A.},
title = {On {P}opa's cocycle superrigidity theorem},
journal = {Int. Math. Res. Not. IMRN},
year = {2007},
volume={2007},
number = {19},
pages = {Art. ID rnm073},
}

@incollection {Fur-survey,
    AUTHOR = {Furman, A.},
     TITLE = {A survey of measured group theory},
 BOOKTITLE = {Geometry, rigidity, and group actions},
    SERIES = {Chicago Lectures in Math.},
     PAGES = {296--374},
 PUBLISHER = {Univ. Chicago Press, Chicago, IL},
      YEAR = {2011},
}

@article {Furs,
    AUTHOR = {Furstenberg, H.},
     TITLE = {A {P}oisson formula for semi-simple {L}ie groups},
   JOURNAL = {Ann. of Math. (2)},
    VOLUME = {77},
      YEAR = {1963},
     PAGES = {335--386},
}

@incollection {FW,
    AUTHOR = {Furstenberg, H. and Weiss, B.},
     TITLE = {The finite multipliers of infinite ergodic transformations},
 BOOKTITLE = {The structure of attractors in dynamical systems ({P}roc.
              {C}onf., {N}orth {D}akota {S}tate {U}niv., {F}argo, {N}.{D}.,
              1977)},
    SERIES = {Lecture Notes in Math.},
    VOLUME = {668},
     PAGES = {127--132},
 PUBLISHER = {Springer, Berlin-New York},
      YEAR = {1978},
}

@article{Gab,
title={Examples of groups that are measure equivalent to the free group},
author={Gaboriau, D.},
year={2005},
journal={Ergodic Theory Dynam. Systems},
volume={25},
number = {6},
pages={1809--1827},
}

@article {Gab-l2,
    AUTHOR = {Gaboriau, D.},
     TITLE = {Invariants {$l^2$} de relations d'\'{e}quivalence et de
              groupes},
   JOURNAL = {Publ. Math. Inst. Hautes \'{E}tudes Sci.},
   volume={95},
   issue={1},
   YEAR = {2002},
   PAGES = {93--150},
}

@article {Gab-cost,
    AUTHOR = {Gaboriau, D.},
     TITLE = {Co\^{u}t des relations d'\'{e}quivalence et des groupes},
   JOURNAL = {Invent. Math.},
    VOLUME = {139},
      YEAR = {2000},
    NUMBER = {1},
     PAGES = {41--98},
}

@inproceedings {Gab-survey,
    AUTHOR = {Gaboriau, D.},
     TITLE = {Orbit equivalence and measured group theory},
 BOOKTITLE = {Proceedings of the {I}nternational {C}ongress of
              {M}athematicians. {V}olume {III}},
     PAGES = {1501--1527},
 PUBLISHER = {Hindustan Book Agency, New Delhi},
      YEAR = {2010},
}

@misc{Gen,
title={Automorphisms of graph products of groups and acylindrical hyperbolicity},
author={A. Genevois},
year={2018},
note = {arXiv:1807.00622, (to appear in \emph{Mem. Amer. Math. Soc.})},
}

@article {GH-Hyp,
    AUTHOR = {Genevois, A. and Horbez, C.},
     TITLE = {Acylindrical hyperbolicity of automorphism groups of
              infinitely ended groups},
   JOURNAL = {J. Topol.},
    VOLUME = {14},
      YEAR = {2021},
    NUMBER = {3},
     PAGES = {963--991},
}

@article {GM,
    AUTHOR = {Genevois, A. and Martin, A.},
     TITLE = {Automorphisms of graph products of groups from a geometric
              perspective},
   JOURNAL = {Proc. Lond. Math. Soc. (3)},
    VOLUME = {119},
      YEAR = {2019},
    NUMBER = {6},
     PAGES = {1745--1779},
}

@phdthesis{Gre,
title={Graph products of groups},
author={Green, E.R.},
year={1990},
school={University of Leeds},
}

@InProceedings{Gro,
author = {M. Gromov},
title = {Asymptotic invariants of infinite groups},
booktitle = {Geometric group theory},
pages = {1-295},
year = {1993},
volume = {2},
issue = {182},
series = {London Math. Soc. Lecture Note Ser.},
address = {Cambridge Univ. Press, Cambridge},
}

@article {GS,
    AUTHOR = {Guba, V. S. and Sapir, M. V.},
     TITLE = {Diagram groups are totally orderable},
   JOURNAL = {J. Pure Appl. Algebra},
    VOLUME = {205},
      YEAR = {2006},
    NUMBER = {1},
     PAGES = {48--73},
}

@misc{Gue,
title={Roots of outer automorphisms of free groups and centralizers of abelian subgroups of $\mathrm{Out}({F}_{N})$},
author={Y. Guerch},
year={2022},
note={arXiv:2212.07674 (to appear in \emph{Michigan Math. J.})},
}

@misc{GH-OutFn,
title={\textit{Measure equivalence rigidity of $\mathrm{Out}({F}_{N})$}},
author={V. Guirardel and C. Horbez},
year={2021},
note={arXiv:2103.03696},
}

@article {Ham,
    AUTHOR = {Hamenst\"{a}dt, U.},
     TITLE = {Bounded cohomology and isometry groups of hyperbolic spaces},
   JOURNAL = {J. Eur. Math. Soc. (JEMS)},
    VOLUME = {10},
      YEAR = {2008},
    NUMBER = {2},
     PAGES = {315--349},
}

@misc{HeH,
title={\textit{Measure equivalence rigidity of the handlebody groups}},
author={Hensel, S. and Horbez, C.},
year={2021},
note={arXiv:2111.10064},
}

@article {Hig,
    AUTHOR = {Higman, G.},
     TITLE = {A finitely generated infinite simple group},
   JOURNAL = {J. London Math. Soc.},
    VOLUME = {26},
      YEAR = {1951},
     PAGES = {61--64},
}

@article{HH21,
title={Measure equivalence classification of transvection-free right-angled {A}rtin groups},
author={Horbez, C. and Huang, J.},
year={2022},
journal={J. Éc. polytech. Math.},
volume={9},
pages={1021--1067},
note={arXiv:2010.03613 {[math.GR]}},
}

@misc{HH-Higman,
title={\textit{Measure equivalence rigidity among the {H}igman groups}},
author={Horbez, C. and Huang, J.},
year={2022},
note={arXiv:2206.00884},
}

@misc{HH-Artin,
title={\textit{Boundary amenability and measure equivalence rigidity among two-dimensional {A}rtin groups of hyperbolic type}},
author={Horbez, C. and Huang, J.},
year={2020},
note={arXiv:2004.09325},
}

@misc{HH-L1,
title={\textit{Integrable measure equivalence rigidity of right-angled {A}rtin groups via quasi-isometry}},
author={Horbez, C. and Huang, J.},
year={2023},
note={arXiv:2309.12147},
}

@article{HHI,
title={Orbit equivalence rigidity of irreducible actions of right-angled {A}rtin groups},
author={Horbez, C. and Huang, J. and Ioana, A.},
year={2023},
journal={Compos. Math.},
volume = {159},
number = {4},
pages= {860-887},
}

@article{HHFiniteOut,
 author = {Horbez, C. and Huang, J.},
 title = {Measure equivalence classification of right-angled {Artin} groups: the finite out classes},
 fjournal = {Tunisian Journal of Mathematics},
 journal = {Tunis. J. Math.},
 issn = {2576-7658},
 volume = {8},
 number = {1},
 pages = {135--156},
 year = {2026},
 language = {English},
 doi = {10.2140/tunis.2026.8.135},
 zbMATH = {8177280}
}

@article{Hua,
title={Quasi-isometric classification of right-angled {A}rtin groups {I}: the finite out case},
author={Huang, J.},
year={2017},
journal={Geom. Topol.},
volume={21},
number = {6},
pages={3467-3537},
}

@article {Hua-QI,
    AUTHOR = {Huang, J.},
     TITLE = {Commensurability of groups quasi-isometric to {RAAG}s},
   JOURNAL = {Invent. Math.},
    VOLUME = {213},
      YEAR = {2018},
    NUMBER = {3},
     PAGES = {1179--1247},
}

@Article{HKS,
 Author = {Huang, J. and Kleiner, B. and Stadler, S.},
 Title = {Morse quasiflats. {II}},
 Journal = {Adv. Math.},
 Volume = {425},
 Pages = {41},
 Note = {Id/No 109075},
 Year = {2023},
}

@article{HW,
title={On linear and residual properties of graph products},
author={T. Hsu and D.T. Wise},
year={1999},
journal={Michigan Math. J.},
volume={46},
number = {2},
pages={251-259},
}

@article {Iva,
    AUTHOR = {Ivanov, N.V.},
     TITLE = {Automorphism of complexes of curves and of {T}eichm\"{u}ller
              spaces},
   JOURNAL = {Internat. Math. Res. Notices},
      YEAR = {1997},
      volume={1997},
      issue={14},
     PAGES = {651--666},
}

@article {JS,
    AUTHOR = {Januszkiewicz, T. and \'{S}wi\c{a}tkowski, J.},
     TITLE = {Commensurability of graph products},
   JOURNAL = {Algebr. Geom. Topol.},
    VOLUME = {1},
      YEAR = {2001},
     PAGES = {587--603},
}

@article {KS,
    AUTHOR = {Kar, A. and Sageev, M.},
     TITLE = {Ping pong on {$\rm CAT(0)$} cube complexes},
   JOURNAL = {Comment. Math. Helv.},
    VOLUME = {91},
      YEAR = {2016},
    NUMBER = {3},
     PAGES = {543--561},
}

@Book{Kec,
author = {A.S. Kechris},
title = {Classical descriptive set theory},
publisher = {Springer-Verlag},
year = {1995},
volume = {156},
series = {Graduate Texts in Mathematics},
address = {New York},
}

@article {Kid-amalgam,
    AUTHOR = {Kida, Y.},
     TITLE = {Rigidity of amalgamated free products in measure equivalence},
   JOURNAL = {J. Topol.},
    VOLUME = {4},
      YEAR = {2011},
    NUMBER = {3},
     PAGES = {687--735},
}

@article{Kid-memoir,
title={The mapping class group from the viewpoint of measure equivalence theory},
author={Y. Kida},
year={2008},
journal={Mem. Amer. Math. Soc.},
volume={196},
number={916},
}

@article{Kid-me,
title={Measure equivalence rigidity of the mapping class group},
author={Y. Kida},
year={2010},
journal={Ann. of Math. (2)},
volume={171},
number={3},
pages={1851-1901},
}

@InProceedings{Kid-survey,
author = {Y. Kida},
title = {Introduction to measurable rigidity of mapping class groups},
booktitle = {Handbook of Teichmüller theory, volume II},
pages = {297-367},
year = {2009},
volume = {13},
series = {IRMA Lect. Math. Theor. Phys.},
publisher = {Eur. Math. Soc.},
}

@article {Kid-BS,
    AUTHOR = {Kida, Y.},
     TITLE = {Invariants of orbit equivalence relations and
              {B}aumslag-{S}olitar groups},
   JOURNAL = {Tohoku Math. J. (2)},
    VOLUME = {66},
      YEAR = {2014},
    NUMBER = {2},
     PAGES = {205--258},
}

@article {Kid-oe,
    AUTHOR = {Kida, Y.},
     TITLE = {Orbit equivalence rigidity for ergodic actions of the mapping
              class group},
   JOURNAL = {Geom. Dedicata},
    VOLUME = {131},
      YEAR = {2008},
     PAGES = {99--109},
}

@Article{KK,
author = {S.h. Kim and T. Koberda},
title = {Embedability between right-angled {A}rtin groups},
journal = {Geom. Topol.},
year = {2013},
volume = {17},
number = {1},
pages = {493-530},
}

@article{Laurence,
author = {Laurence, M.R.},
title = {A Generating Set for the Automorphism Group of a Graph Group},
journal = {Journal of the London Mathematical Society},
volume = {52},
number = {2},
pages = {318-334},
doi = {https://doi.org/10.1112/jlms/52.2.318},
year = {1995},
}

@misc{Laz,
    AUTHOR = {Lazarovich, N.},
     TITLE = {Finite index rigidity of hyperbolic groups},
   note = {arXiv:2302.04484},
      YEAR = {2023},
}

@book {LS,
    AUTHOR = {Lyndon, R.C. and Schupp, P.E.},
     TITLE = {Combinatorial group theory},
    SERIES = {Ergebnisse der Mathematik und ihrer Grenzgebiete},
    VOLUME = {Band 89},
 PUBLISHER = {Springer-Verlag, Berlin-New York},
      YEAR = {1977},
}

@article {Mal,
    AUTHOR = {Mal'cev, A. I.},
     TITLE = {Nilpotent torsion-free groups},
   JOURNAL = {Izv. Akad. Nauk SSSR Ser. Mat.},
    VOLUME = {13},
      YEAR = {1949},
     PAGES = {201--212},
}

@article {Mar,
    AUTHOR = {Martin, A.},
     TITLE = {Acylindrical actions on {CAT}(0) square complexes},
   JOURNAL = {Groups Geom. Dyn.},
    VOLUME = {15},
      YEAR = {2021},
    NUMBER = {1},
     PAGES = {335--369},
}

@article {Mar2,
    AUTHOR = {Martin, A.},
     TITLE = {On the cubical geometry of {H}igman's group},
   JOURNAL = {Duke Math. J.},
    VOLUME = {166},
      YEAR = {2017},
    NUMBER = {4},
     PAGES = {707--738},
}

@article {MO,
    AUTHOR = {Minasyan, A. and Osin, D.},
     TITLE = {Acylindrical hyperbolicity of groups acting on trees},
   JOURNAL = {Math. Ann.},
    VOLUME = {362},
      YEAR = {2015},
    NUMBER = {3-4},
     PAGES = {1055--1105},
}

@misc{MV,
author = {P. Möller and O. Varghese},
title = {On normal subgroups in automorphism groups},
note = {arXiv:2208.05677},
year = {2022},
}

@Article{MS,
author = {N. Monod and Y. Shalom},
title = {Orbit equivalence rigidity and bounded cohomology},
journal = {Ann. of Math. (2)},
year = {2006},
volume = {164},
number = {3},
pages = {825-878},
}

@article {MS2,
    AUTHOR = {Monod, N. and Shalom, Y.},
     TITLE = {Negative curvature from a cohomological viewpoint and cocycle
              superrigidity},
   JOURNAL = {C. R. Math. Acad. Sci. Paris},
    VOLUME = {337},
      YEAR = {2003},
    NUMBER = {10},
     PAGES = {635--638},
}

@article{MvN,
  title={On Rings of Operators},
  author={F.J. Murray and J. von Neumann},
  journal={Annals of Mathematics},
  year={1936},
  volume={37},
  pages={116-229},
}

@article {MvN2,
    AUTHOR = {Murray, F. J. and von Neumann, J.},
     TITLE = {On rings of operators. {IV}},
   JOURNAL = {Ann. of Math. (2)},
    VOLUME = {44},
      YEAR = {1943},
     PAGES = {716--808},
}

@article {MR,
    AUTHOR = {Myasnikov, A. G. and Remeslennikov, V. N.},
     TITLE = {Exponential groups. {II}. {E}xtensions of centralizers and
              tensor completion of {CSA}-groups},
   JOURNAL = {Internat. J. Algebra Comput.},
    VOLUME = {6},
      YEAR = {1996},
    NUMBER = {6},
     PAGES = {687--711},
}

@article{OW80,
title= {{Ergodic theory of amenable group actions. I: The Rohlin lemma}},
author = {Ornstein, D.S. and Weiss, B.},
year={1980},
mont={january},
journal={Bulletin of the American Mathematical Society},
volume ={2},
issue = {1},
pages = {161--164},
}

@article {Osi,
    AUTHOR = {Osin, D.},
     TITLE = {Acylindrically hyperbolic groups},
   JOURNAL = {Trans. Amer. Math. Soc.},
    VOLUME = {368},
      YEAR = {2016},
    NUMBER = {2},
     PAGES = {851--888},
}

@article{Oza,
title= {Boundary amenability of relatively hyperbolic groups},
author = {N. Ozawa},
year={2006},
journal={Topol. Appl.},
volume ={14},
pages = {2624-2630},
}

@incollection {Ozawa-ICM,
    AUTHOR = {Ozawa, N.},
     TITLE = {Amenable actions and applications},
 BOOKTITLE = {International {C}ongress of {M}athematicians. {V}ol. {II}},
     PAGES = {1563--1580},
 PUBLISHER = {Eur. Math. Soc., Z\"{u}rich},
      YEAR = {2006},
}

@article {Pan,
    AUTHOR = {Pansu, P.},
     TITLE = {Croissance des boules et des g\'{e}od\'{e}siques ferm\'{e}es
              dans les nilvari\'{e}t\'{e}s},
   JOURNAL = {Ergodic Theory Dynam. Systems},
    VOLUME = {3},
      YEAR = {1983},
    NUMBER = {3},
     PAGES = {415--445},
}

@article{Pop,
title= {Cocycle and orbit equivalence superrigidity for malleable actions of $w$-rigid groups},
author = {S. Popa},
year={2007},
journal={Invent. Math.},
volume ={170},
number = {2},
pages = {243-295},
}

@article{Popa-SpectralGap,
title= {On the superrigidity of malleable actions with spectral gap},
author = {S. Popa},
year={2008},
journal={J. Amer. Math. Soc.},
fjournal={Journal of the American Mathematical Society},
volume ={21},
number={4},
pages = {981-1000},
doi={10.1090/S0894-0347-07-00578-4},
issn={0894-0347,1088-6834},
}

@article {Pop2,
    AUTHOR = {Popa, S.},
     TITLE = {On a class of type {${\rm II}_1$} factors with {B}etti numbers
              invariants},
   JOURNAL = {Ann. of Math. (2)},
    VOLUME = {163},
      YEAR = {2006},
    NUMBER = {3},
     PAGES = {809--899},
}

@article {Pop3,
    AUTHOR = {Popa, S.},
     TITLE = {Some rigidity results for non-commutative {B}ernoulli shifts},
   JOURNAL = {J. Funct. Anal.},
    VOLUME = {230},
      YEAR = {2006},
    NUMBER = {2},
     PAGES = {273--328},
}

@incollection {Sag,
    AUTHOR = {Sageev, M.},
     TITLE = {{$\rm CAT(0)$} cube complexes and groups},
 BOOKTITLE = {Geometric group theory},
    SERIES = {IAS/Park City Math. Ser.},
    VOLUME = {21},
     PAGES = {7--54},
 PUBLISHER = {Amer. Math. Soc., Providence, RI},
      YEAR = {2014},
}

@article {SW,
    AUTHOR = {Schmidt, K. and Walters, P.},
     TITLE = {Mildly mixing actions of locally compact groups},
   JOURNAL = {Proc. London Math. Soc. (3)},
    VOLUME = {45},
      YEAR = {1982},
    NUMBER = {3},
     PAGES = {506--518},
}

@book {Ser,
    AUTHOR = {Serre, J.-P.},
     TITLE = {Arbres, amalgames, {${\rm SL}\sb{2}$}},
    SERIES = {Ast\'{e}risque},
    VOLUME = {No. 46},
 PUBLISHER = {Soci\'{e}t\'{e} Math\'{e}matique de France, Paris},
      YEAR = {1977},
}

@article{Servatius,
title = {Automorphisms of graph groups},
journal = {Journal of Algebra},
volume = {126},
number = {1},
pages = {34-60},
year = {1989},
issn = {0021-8693},
doi = {https://doi.org/10.1016/0021-8693(89)90319-0},
author = {H. Servatius},
}

@article{SingerOE,
  title={Automorphisms of Finite Factors},
  author={I.M. Singer},
  journal={American Journal of Mathematics},
  year={1955},
  volume={77},
  pages={117--133},
}

@article {TD,
    AUTHOR = {Tucker-Drob, R.D.},
     TITLE = {Mixing actions of countable groups are almost free},
   JOURNAL = {Proc. Amer. Math. Soc.},
    VOLUME = {143},
      YEAR = {2015},
    NUMBER = {12},
     PAGES = {5227--5232},
}

@InProceedings{Vae,
author = {S. Vaes},
title = {Rigidity results for {B}ernoulli actions and their von {N}eumann algebras (after {S}orin {P}opa)},
booktitle = {Séminaire Bourbaki. Vol. 2005/2006},
pages = {237-294},
year = {2007},
volume = {311},
notes= {Exp. No. 961, viii},
series = {Astérisque},
}

@article {Why,
    AUTHOR = {Whyte, K.},
     TITLE = {Amenability, bi-{L}ipschitz equivalence, and the von {N}eumann
              conjecture},
   JOURNAL = {Duke Math. J.},
    VOLUME = {99},
      YEAR = {1999},
    NUMBER = {1},
     PAGES = {93--112},
}

@book {Zim,
    AUTHOR = {Zimmer, R.J.},
     TITLE = {Ergodic theory and semisimple groups},
    SERIES = {Monographs in Mathematics},
    VOLUME = {81},
 PUBLISHER = {Birkh\"{a}user Verlag, Basel},
      YEAR = {1984},
}

@article {Zim-amen,
    AUTHOR = {Zimmer, R.J.},
    TITLE = {Amenable ergodic group actions and an application to {P}oisson
              boundaries of random walks.},
   JOURNAL = {J. Functional Analysis},
   volume = {27},
    YEAR = {1978},
    NUMBER = {3},
    PAGES = {350--372},
}
\addcontentsline{toc}{section}{Bibliography}

\bigskip

\begin{flushleft}
Amandine Escalier\\ 
Universit\'e Paris-Saclay, CNRS,  Laboratoire de math\'ematiques d'Orsay, 91405, Orsay, France \\
\emph{e-mail:~}\texttt{amandine.escalier@universite-paris-saclay.fr}\\[4mm]
\end{flushleft}

\begin{flushleft}
Camille Horbez\\ 
Universit\'e Paris-Saclay, CNRS,  Laboratoire de math\'ematiques d'Orsay, 91405, Orsay, France \\
\emph{e-mail:~}\texttt{camille.horbez@universite-paris-saclay.fr}\\[4mm]
\end{flushleft}
\end{document}